\documentclass[11pt]{article} 

\usepackage{fullpage}
\usepackage{amsthm}
\usepackage{amssymb}
\usepackage{amsmath}
\usepackage[mathscr]{eucal}
\usepackage{graphicx}
\usepackage{centernot}
\usepackage{textcomp}
\usepackage{afterpage}
\usepackage{esint}
\usepackage{xspace}
\usepackage{tikz}
\usepackage{hyperref}
\usepackage{bm}
\usepackage{wasysym}
\usepackage{xcolor}
\usepackage{setspace}

\usepackage{fouriernc}
\usepackage[T1]{fontenc}
\usepackage[utf8]{inputenc}
\usepackage{authblk}
\usepackage[title]{appendix}
\usepackage{euflag}

\newcommand{\abs}[1]{\left| #1 \right|}

\newcommand{\C}{\mathbb{C}}

\newcommand{\Z}{\mathbb{Z}}

\newcommand{\N}{\mathbb{N}}

\newcommand\norm[1]{\left\lVert#1\right\rVert}

\newcommand{\p}{\partial}

\newcommand{\R}{\mathbb{R}}

\newcommand{\set}[2]{\left\{{#1}:{#2}\right\}}

\newcommand{\supp}[1]{\textrm{supp}\left({#1}\right)}

\newtheorem{theorem}{Theorem}[section]
\newtheorem{proposition}[theorem]{Proposition}
\newtheorem{corollary}[theorem]{Corollary}

\newtheorem{definition}[theorem]{Definition}

\newtheorem{remark}[theorem]{Remark}

\numberwithin{equation}{section}

\usepackage{color}
\definecolor{bleu}{rgb}{0.,0.,0.5}

\title{Global subelliptic estimates for geometric Kramers-Fokker-Planck operators on closed manifolds}
\author{Francis~Nier\thanks{LAGA, Universit{\'e} de Sorbonne-Paris Nord (Paris XIII), 99 avenue
J.B.~Cl{\'e}ment, F-93430~Villetaneuse. nier@math.univ-paris13.fr}\\
Xingfeng~Sang\thanks{LAGA, Universit{\'e} de Sorbonne-Paris Nord (Paris XIII), 99 avenue
J.B.~Cl{\'e}ment, F-93430~Villetaneuse. sang@math.univ-paris13.fr}\\
Francis~White\thanks{Department of Mathematics, University of California Irvine, 400 Physical Sciences Quad, Irvine CA-92697.\\fgwhite@uci.edu}}

\begin{document}
\maketitle
\begin{abstract}
  In this article we reconsider the proof of subelliptic estimates for Geometric Kramers-Fokker-Planck operators, a class which includes Bismut's hypoelliptic Laplacian, when the base manifold is closed (no boundary).
  The method is significantly different from the ones proposed by Bismut-Lebeau in \cite{BiLe} and Lebeau in \cite{Leb1} and \cite{Leb2}. As a new result we are able to prove maximal subelliptic estimates with some control of the constants in the two asymptotic regimes of high ($b\to 0$) and low ($b\to +\infty$) friction. After a dyadic partition in the momentum variable, the analysis is essentially local in the position variable, contrary to the microlocal reduction techniques of the previous works. In particular this method will be easier to adapt on manifolds with boundaries. A byproduct of our analysis is the introduction of a very convenient double exponent Sobolev scale associated with globally defined differential operators. Applications of this convenient parameter dependent functional analysis to accurate spectral problems, in particular for Bismut's hypoelliptic Laplacian with all its specific geometry, is deferred to subsequent works.
\end{abstract}

\tableofcontents{}

\section{Introduction}

\subsection{Background and Motivations}
In \cite{Bis041}\cite{Bis042}\cite{Bis05} J.M.~Bismut introduced the hypoelliptic Laplacian which can be viewed alternatively as a deformation of Hodge theory on the cotangent bundle or as a generalization to the case of arbitrary degree differential forms of the Fokker-Planck equation (often specified as Kramers-Fokker-Planck equation in this case) associated with the Langevin process. Very rapidly J.M.~Bismut and G.~Lebeau made in \cite{BiLe} a careful analysis of families of such hypoelliptic Laplacians indexed by a parameteer $b>0$\,, proving in particular the spectral convergence  as $b\to 0^{+}$ of the hypoelliptic Laplacian to the Hodge Laplacian on the base manifold. It was in the mood of those times to develop the accurate spectral analysis of parameter dependent non self-adjoint hypoelliptic operators and we refer the reader to \cite{Dav}\cite{DSZ}\cite{HerNi}\cite{HHS} for works in this direction which still have many developments. In \cite{BiLe} J.M.~Bismut and G.~Lebeau combined such involved microlocal and semiclassical spectral analysis with the heavy geometrical construction of the hypoelliptic Laplacian and this was continued by S.~Shen in \cite{She}, for the double parameter asymptotics where the first limit $b\to 0^{+}$ allows to recover a semiclassical Witten Laplacian on the base manifold let say with a parameter $h\to 0^{+}$ which leads to Morse theory as in \cite{Wit}\cite{HeSj4}\cite{Zha}. In \cite{Leb1}\cite{Leb2}, G.~Lebeau introduced a general class of non self-adjoint hypoelliptic operators for which accurate subelliptic estimates, a corner stone of the spectral asymptotic analysis, can be proven.\\
About the Witten Laplacian on the base manifold, it was realized in \cite{LNV1}\cite{LNV2} that the introduction of artificial boundary value problems, associated with a suitable cutting and gluing of the manifold, was a very convenient and robust way for the generalization of the accurate spectral asymptotic analysis with a potential which can be more general than a Morse function. Such accurate spectral analysis was motivated by various questions related with molecular dynamics. This raised the question of understanding similar problems for the Langevin process, where the arbitrary degree form formulation, involves boundary value problems for the hypoelliptic Laplacian. The functional analysis was started in \cite{Nie}\cite{NiSh} for fixed hypoelliptic Laplacians, which means with no parameter. The asymptotic spectral analysis with respect to one parameter $b>0$ or two parameters $(b,h)\in (0,+\infty)^{2}$ remains to be done.\\
In this direction, the microlocal reduction approach proposed by in \cite{BiLe}\cite{Leb1}\cite{Leb2} appears to be not well adapted for a similar asymptotic analysis of boundary value problems. The present article proposes an alternative approach which relies on some local approximation of the hypoelliptic Laplacian on the cotangent $T^{*}Q$ of a general compact Riemannian manifold $Q$, by the euclidean version. This euclidean approximation relies on the use of normal coordinates around a point $q_{0}\in Q$ while the Taylor expansion of the metric leads to controlled error terms in balls $B(q_{0},r_{b,|p|})$ parametrized by the parameter $b>0$ and the size of momentum variable $p\in T^{*}_{q_{0}}Q$\,.\\
While doing so we are able to provide a control of constants in the global subelliptic estimates not only as $b\to 0^{+}$ but also as $b\to +\infty$ which can be of interest for further developments. The spectral analysis as $b\to 0^{+}$ proving the spectral convergence to the Hodge or Witten Laplacian on the base manifold, will be carried out in another text. Compared to the works of \cite{BiLe}\cite{She}, the approach proposed here  makes possible  an easier extension to boundary value problems and a more direct connexion with standard tools of spectral analysis like Grushin problems (see\cite{SjZw}) and other related works with similar spectral problems studied in \cite{ReTa}\cite{Dro}\cite{Smi}. In particular, global Sobolev scales $\tilde{\mathcal{W}}^{s_{1},s_{2}}(T^{*}Q)$\,, $s_{1},s_{2}\in \R$\,, are associated with a pair of commuting self-adjoint operators which make use of the scalar horizontal Laplacian associated with a Sasaki type metric. This synthetic formulation, of which the properties nevertheless rely on some global pseudodifferential techniques, was inspired by \cite{ReTa} and \cite{BeBo}. Again, it will allow a rather direct adaptation of techniques developed for simpler spectral problems, namely scalar operators on a compact total space within the framework of standard Sobolev spaces, to the more involved framework of the hypoelliptic Laplacian (and possibly other global analysis problems on the total space of  the cotangent bundle $X=T^{*}Q$).
\subsection{General Framework}

In this text we shall consider geometric Kramers-Fokker-Planck operators on $X = T^* Q$, where $(Q, g)$ is a smooth compact $d$-dimensional Riemannian manifold without boundary. Points in $X$ will be denoted by $x$, and $\pi_{X}: X \rightarrow Q$ will denote the natural projection $T^*_q Q \ni x \mapsto q \in Q$\,. Local coordinates on $Q$ will be denoted by $(q^1, \ldots, q^d)$\,. We shall use Einstein's convention of summing over repeated up and down indices. If $q \in Q$ and $(U, q^1, \ldots q^d)$ is a local coordinate system for $Q$, then an element of the fiber $p \in T^*_q Q$ will be written $p = p_j \, dq^j$, and $(q^1, \ldots, q^d, p_1, \ldots, p_d)$ will denote canonical coordinates in $U \times \R^d \sim T^* U \subset T^* Q$\,. If local canonical coordinates $(q^1, \ldots, q^d, p_1, \ldots, p_d)$ have been fixed in a neighborhood $T^* U$ of a point $x \in X$\,, then we shall write $x = (q,p)$\,.

The metric $g$ on $Q$, i.e. on the tangent bundle $\pi_{TQ}: TQ \rightarrow Q$, will be denoted by $g(q) = {}^{t}g(q) = \left(g_{jk}(q) \right)_{1 \le j, k \le d}$ or $g = g_{jk}(q) \, dq^j \, dq^k$\,. The corresponding dual metric on the cotangent bundle $\pi_{T^* Q}: T^* Q \rightarrow Q$ will be denoted by $g^{-1}(q) = (g^{jk}(q))_{1 \le j, k \le d}$\,. Let $\nabla^{LC}$ be the Levi-Civita connection on the tangent bundle $\pi_{TQ}: TQ \rightarrow Q$ associated with the metric $g$\,. By abuse of notation, we shall also denote by $\nabla^{LC}$ the connection on the tensor bundle
\begin{align}
	\mathcal{T}^{(k,\ell)} TQ = \underbrace{TQ \otimes \cdots \otimes TQ}_{k \ \textrm{times}} \otimes \underbrace{T^* Q \otimes \cdots \otimes T^* Q}_{\ell \ \textrm{times}}
\end{align}
induced by the Levi-Civita connection on $TQ$ for every $k, \ell \in \N$\,. If $q = (q^1, \ldots, q^d)$ are local coordinates for $Q$, then the Christoffel symbols associated the Levi-Civita connection $\nabla^{LC}$ are given by
\begin{align}
	\Gamma_{jk}^\ell = \frac{1}{2} g^{\ell a} \left(\p_{q^j} g_{a k} + \p_{q^k} g_{a j} - \p_{q^a} g_{jk} \right), \ \ 1 \le j, k, \ell \le d\,.
\end{align}
The connection $\nabla^{LC}$ gives rise to a global decomposition
\begin{align} \label{horizontal_vertical_decomposition}
	TX = T^H X \oplus T^V X\,,
\end{align}
where $T^V X = \ker{(d\pi_X)}$ is the vertical subbundle of $TX$ associated with the projection $\pi_X$ and $T^H X$ is the horizontal subbundle of $TX$ defined at each point $x \in X$ by
\begin{align}
\begin{split}
	T^H_x X = \{\gamma'(0): & \ \textrm{there exists $\epsilon>0$ and a smooth path $\gamma: (-\epsilon, \epsilon) \rightarrow X$ such that} \\
	&\ \ \textrm{$\gamma(0) = x$ and $\nabla^{LC}_{d\pi_X(\gamma'(t))} \gamma(t) = 0$ for all $-\epsilon<t<\epsilon$}\}\,.
\end{split}
\end{align}
In terms of local coordinates for $X$, the subbundles $T^HX$ and $T^V X$ may be described as follows. If $(q^1, \ldots, q^d, p_1, \ldots, p_d)$ are local canonical coordinates for $X$, let
\begin{align} \label{definition_e_j}
	e_j = \frac{\p}{\p q^j} + \Gamma^\ell_{jk}(q) p_\ell \frac{\p}{\p p_k}, \ \ \ 1 \le j \le d\,,
\end{align}
and let
\begin{align} \label{definiton_hat_e_j}
	\widehat{e}^j = \frac{\p}{\p p_j}, \ \ 1 \le j \le d\,.
\end{align}
Together these vector fields form a local frame $(e_1, \ldots, e_d, \widehat{e}^1, \ldots, \widehat{e}^d)$ for $TX$ and locally we have
\begin{align}
	T^H X = \textrm{span} \left(e_1, \ldots, e_d \right)
\end{align}
and
\begin{align}
	T^V X = \textrm{span} \left(\widehat{e}^1, \ldots, \widehat{e}^d\right)\,.
\end{align}
Since for every $x = (q,p) \in X$ the differential $\left. d \pi_X \right|_x$ restricts to a linear isomorphism $T^H_x X \rightarrow T_q Q$ while $T^V_x X \cong T^*_q Q$\,, the decomposition (\ref{horizontal_vertical_decomposition}) yields the identifications
\begin{align} \label{decomposition_of_tangent_bundle}
  TX \cong \pi_{X}^{*}(TQ \oplus T^* Q)\,.
\end{align}
The total tangent space $TX=T^{H}X\oplus T^{V}X$ is equiped with the metric $\pi_{X}^{*}(g\displaystyle\mathop{\oplus}^{\perp}g^{-1})$, simply written $g\oplus g^{-1}$, by using the above identification.
Clearly horizontal (resp. vertical) vector fields on $X$ are given as sections in $\mathcal{C}^{\infty}(X;T^{H}X)$ (resp. $\mathcal{C}^{\infty}(X;T^{V}X)$). Specific subspaces of horizontal (resp. vertical) sections are provided by the fact that \eqref{decomposition_of_tangent_bundle} induces a natural imbedding
$$
i_{g}:\mathcal{C}^{\infty}(Q;TQ\oplus T^{*}Q)\to \mathcal{C}^{\infty}(X;\pi^{*}_{X}(TQ\oplus T^{*}Q))=\mathcal{C}^{\infty}(X;TX)
$$ and we introduce
\begin{eqnarray*}
  && \mathcal{C}^{\infty}_{Q}(X;T^{H}X)=i_{g}(\mathcal{C}^{\infty}(Q;TQ))\subset \mathcal{C}^{\infty}(X;T^{H}X)\,,\\
 \text{resp.}&& \mathcal{C}^{\infty}_{Q}(X;T^{V}X)=i_{g}(\mathcal{C}^{\infty}(Q;T^{*}Q))\subset \mathcal{C}^{\infty}(X;T^{V}X)\,.
\end{eqnarray*}
These spaces $\mathcal{C}^{\infty}_{Q}(X;T^{H}X)$ and $\mathcal{C}^{\infty}_{Q}(X;T^{V}Q)$ are $\mathcal{C}^{\infty}(Q;\R)$ modules.  Additionally on $\mathcal{C}^{\infty}(Q;TQ\oplus T^{*}Q)$ a $\mathcal{C}^{k}$-norm can be fixed once and for all by using a finite partition of unity subordinate to an open chart covering $Q=\mathop{\cup}_{j=1}^{J}\Omega_{j}$ while changing the atlas and the partition of unity gives an equivalent norm.  We therefore can speak of $\|T\|_{\mathcal{C}^{k}}$ for $T\in \mathcal{C}^{\infty}_{Q}(X;T^{H}X)$ and $T\in \mathcal{C}^{\infty}_{Q}(X;T^{V}X)$ without specifying its expression.\\

It will be convenient to use the following families of vector fields.
\begin{definition}
  \label{de.famcalT}
  For any $N\in \N$ and any $k\in \N$\,, the set $\mathcal{T}_{N,k}^{H}$ (resp. $\mathcal{T}_{N,k}^{V}$) is defined by
  \begin{eqnarray*}
    && \mathcal{T}_{N,k}^{H}=\left\{(T_{1}^{H},\ldots, T_{N}^{H})\in \mathcal{C}^{\infty}_{Q}(X;T^{H}X)^{N}\,,\, \forall j\in \left\{1,\ldots, N\right\}\,, \|T_{j}^{H}\|_{\mathcal{C}^{k}}\leq 1\right\}\,,\\
 \text{resp.}&&\mathcal{T}_{N,k}^{V}=\left\{(T_{1}^{V},\ldots, T_{N}^{V})\in \mathcal{C}^{\infty}_{Q}(X;T^{V}X)^{N}\,,\, \forall j\in \left\{1,\ldots, N\right\}\,, \|T_{j}^{V}\|_{\mathcal{C}^{k}}\leq 1\right\}\,.
  \end{eqnarray*}
\end{definition}
\noindent By duality, we also have the identifications
\begin{align} \label{identification_of_T*X}
	T^* X \cong T^*Q \oplus T Q.
\end{align}
If $(q^1, \ldots, q^d, p_1, \ldots, p_d)$ are local canonical coordinates for $X$, we let
\begin{align}
	e^j = dq^j, \ \ 1 \le j \le d,
\end{align}
and
\begin{align}
	\widehat{e}_j = dp_j - \Gamma^\ell_{jk}(q) p_\ell d q^k, \ \ 1 \le j \le d.
\end{align}
It is clear that $(e^1, \ldots, e^d, \widehat{e}_1, \ldots, \widehat{e}_d)$ is a local coframe for $T^* X$, and locally it is true that
\begin{align}
	(T^H X)^* = \textrm{span} \left(e^1, \ldots, e^d\right)
\end{align}
and
\begin{align}
	(T^V X)^* = \textrm{span} \left(\widehat{e}_1, \ldots, \widehat{e}_d \right).
\end{align}

We also note that $X$ is naturally a symplectic manifold with respect to the usual symplectic form $\sigma$ given in local canonical coordinates $(q,p)$ by
\begin{align} \label{standard_symplectic_form}
	\sigma = \sum_{j=1}^d dp_j \wedge dq^j\,.
\end{align}
Since $\sigma^{\wedge d} \neq 0$, the manifold $X$ is orientable, and we orient $X$ so that every local canonical coordinate system $(q,p)$ is positively oriented. The volume form on $X$ for the metric $g \oplus g^{-1}$ is denoted by $d\textrm{vol}_X$ and given locally by
\begin{align}
	d\textrm{vol}_X = dq^1 \wedge \cdots \wedge dq^d \wedge dp_1 \wedge \cdots \wedge dp_d\,.
\end{align}
The volume form $d\textrm{vol}_X$ is related to $\sigma$ by
\begin{align}
	d \textrm{vol}_X = \frac{1}{d!}(-1)^\frac{d(d+1)}{2} \sigma^{\wedge d}\,.
\end{align}
In particular, if $H \in C^\infty(X; \R)$ and $\mathcal{Y}$ is the Hamilton vector field of $H$ with respect to the symplectic form $\sigma$\,, i.e. $\mathcal{Y}$ is the unique smooth vector field on $X$ such that $\iota_{\mathcal{Y}} \sigma = -dH$, then the flow $\Phi^t = \exp{(t\mathcal{Y})}$ on $X$ generated by $\mathcal{Y}$ preserves $d\textrm{vol}_X$\,. In this text, we will be primarily concerned with the situation in which $H$ is the kinetic energy
\begin{align} \label{kinetic_energy}
	H(q,p) = \frac{1}{2} \abs{p}^2_{q} = \frac{1}{2} g^{jk}(q)p_j p_k\,.
\end{align}
In this case, the Hamilton vector field $\mathcal{Y}$ of $H$ is given locally by
\begin{align}
	\mathcal{Y} = g^{jk}(q) p_j e_k\,,
\end{align}
where $e_k$ is as in (\ref{definition_e_j}), and the projections of the integral curves of $\mathcal{Y}$ to $Q$ by $\pi_X$ are precisely the smooth geodesic curves of the metric $g$\,.
We will use also the metric-dependent Japanese bracket
	\begin{align}
		\langle p \rangle_q = (1 + \abs{p}^2_{q})^{1/2}=(1+g^{ij}(q)p_i p_j)^{1/2}\,,
	\end{align}
        while the notation
	\begin{align}
          \langle p \rangle = (1 + \abs{p}^2)^{1/2}=(1+\delta^{ij}p_i p_j)^{1/2}(1+\sum_{i=1}^{d}p_{i}^{2})^{1/2}\,,
        \end{align}
will be used for the euclidean version.\\
Let $E \xrightarrow{\pi_E} Q$ be a smooth complex vector bundle over $Q$ of complex dimension $N$ that is equipped with an affine connection $\nabla^E$ and a Hermitian metric $g^E$\,. Let $\mathcal{E} := \pi_X^* E \xrightarrow{\pi_{\mathcal{E}}} X$ denote the pullback bundle of $E \xrightarrow{\pi_E} Q$ by the map $\pi_X: X \rightarrow Q$\,. Locally, smooth sections $u$ of $\mathcal{E} \xrightarrow{\pi_{\mathcal{E}}} X$ have the form
\begin{align} \label{local_form_of_section}
	u(x) = \sum_{\ell = 1}^{N} u_\ell(x) f^\ell(q), \ \ x = (q,p) \in X\,, \ \ 
\end{align}
where $\left(f^1, \ldots, f^{N}\right)$ is a smooth local frame for $E \xrightarrow{\pi_E} Q$ and $u_1, \ldots, u_N$ are smooth locally defined complex-valued functions on $X$\,. We equip $\mathcal{E} \xrightarrow{\pi_\mathcal{E}} X$ with the pullback connection $\nabla^{\mathcal{E}}$\,, which is defined using the decomposition (\ref{decomposition_of_tangent_bundle}) of $TX$ by the relations
\begin{align}
\begin{split}
	\left(\nabla^{\mathcal{E}}_{e_j} u \right)(x) &= \sum_{\ell=1}^{N} \left[(e_j u_\ell)(x) f^\ell(q) + u_\ell(x) \nabla^E_{\frac{\p}{\p q^j}} f^\ell(q)\right]\,, \\
	\left(\nabla^{\mathcal{E}}_{\widehat{e}^j} u \right)(x) &= \sum_{\ell=1}^{N} (\widehat{e}^j u_\ell)(x) f^\ell(q)=\sum_{\ell=1}^{N}(\partial_{p_{j}}u_{\ell})(x)f^{\ell}(q) \,, \quad x = (q,p) \in X, \ \ 1 \le j \le d\,,
\end{split}
\end{align}
whenever $u \in C^\infty(X; \mathcal{E})$ is of the form (\ref{local_form_of_section}). Because the connection $\nabla^{\mathcal{E}}_{T_{V}}$ is trivial for $T_{V}\in TX^{V}$\,, the covariant derivative with respect to a vertical vector field  will be identified with the associated  scalar first order differential operator\,. Accordingly the vertical Laplacian and the vertical harmonic oscillator, written locally as,
\begin{eqnarray}
  \label{vertical_laplacian}
  \Delta_p &=& \sum_{1 \le j, k \le d} \frac{1}{2} g_{jk}(q) \nabla^{\mathcal{E}}_{\widehat{e}^j} \nabla^{\mathcal{E}}_{\widehat{e}^k}=\sum_{1\leq j,k\le d} \frac{1}{2} g_{jk}(q) \partial_{p_{j}}\partial_{p_{k}}
  \\
  \label{scalar_vertical_harmonic_oscillator}
	\mathcal{O} &=& -\frac{1}{2} \Delta_p + \frac{1}{2} \abs{p}_{q}^2.
\end{eqnarray}
are globally defined operators, which happen to be scalar differential operators in the sense that in any local frame $(f^{1},\ldots,f^{N})$ of $E\stackrel{\pi_{E}}{\to}Q$\,,
$$
\Delta_{p}(\sum_{\ell=1}^{N}u_{\ell}(x)f^{\ell}(q))=\sum_{\ell=1}^{N}(\Delta_{p}u_{\ell})(x)f^{\ell}(q)\quad\text{and}\quad
\mathcal{O}(\sum_{\ell=1}^{N}u_{\ell}(x)f^{\ell}(q))=\sum_{\ell=1}^{N}(\mathcal{O}u_{\ell})(x)f^{\ell}(q)\,.
$$
We also equip the bundle $\mathcal{E}$ with the pulled back Hermitian metric $g^{\mathcal{E}}$ defined by
\begin{align}
	g^{\mathcal{E}}(u, u') = \sum_{\ell_1, \ell_2} \overline{u_{\ell_1}(x)} u'_{\ell_2}(x) g^E(f^{\ell_1}(q), f^{\ell_2}(q)), \ \ x = (q,p) \in X\,,
\end{align}
where $u = \sum u_\ell(x) f^\ell(q)$ and $u' = \sum u'_{\ell}(x) f^{\ell}(q)$\,. Using the Hermitian metric $g^{\mathcal{E}}$ on $\mathcal{E}$ and the volume form $d \textrm{vol}_X$ on $X$\,, we may introduce the Hilbert space $L^2(X; \mathcal{E})$ of square integrable sections of $\mathcal{E}$ as follows. The space $L^{2}(X;\mathcal{E})$ is the set of measurable sections $u$ such that
\begin{align}
	\langle u\,,\,  u\rangle_{L^{2}(X;\mathcal{E})} = \int_X g^{\mathcal{E}}_{x}\left(u(x), u(x)\right) \, d\textrm{vol}_X(x)<+\infty\,,
\end{align}
and it is a Hilbert space for the scalar product
\begin{align}
	\langle u_{1}\,,\,  u_{2}\rangle_{L^{2}(X;\mathcal{E})} = \int_X g^{\mathcal{E}}_{x}\left(u_{1}(x), u_{2}(x)\right) \, d\textrm{vol}_X(x)\,,
\end{align}
in which  $\mathcal{C}^{\infty}_{0}(X;\mathcal{E})$ is dense.\\
By recalling $d\mathrm{vol}_{X}=dqdp$\,, the operator $\mathcal{O}$ is clearly self-adjoint with its maximal domain $D(\mathcal{O})=\left\{u\in L^{2}(X,\mathcal{E})\,, \mathcal{O}u\in L^{2}(X;\mathcal{E})\right\}$\,, in which $\mathcal{C}^{\infty}_{0}(X;\mathcal{E})$ is dense with the graph norm. It is also bounded from below by $\frac{d}{2}$ and $\sqrt{\mathcal{O}}$ is well defined.\\
Let us now introduce some Sobolev type spaces, taking into account the different homogeneities of $e_{i}$\,, $p_{i}$ and $\partial_{p_{i}}$\,.
\begin{definition}
\label{de:tW}
	For $k \in \N$ and $u$ a sufficiently regular section of $\mathcal{E}$, we define
	\begin{align}
          \norm{u}_{\tilde{\mathcal{W}}^{k}}=
          \sup_{\substack{N_{1}+\frac{N_{2}+N_{3}}{2}\leq k\\
          (T_{1}^{H},\ldots, T_{N_{1}}^{H})\in \mathcal{T}_{N_{1},k}^{H}\\
          (T_{1}^{H},\ldots, T_{N_{2}}^{V})\in \mathcal{T}_{N_{2},k}^{V} \\
          }}
          \norm{\langle p \rangle_{q}^{N_3} \nabla^{\mathcal{E}}_{T_{1}^{H}}\ldots \nabla^{\mathcal{E}}_{T_{N_{1}}^{H}}
          \nabla^{\mathcal{E}}_{T_{1}^{V}}\ldots \nabla^{\mathcal{E}}_{T_{N_{2}}^{V}}
          u}_{L^2(X; \mathcal{E})}\,,
	\end{align}
	and we take
	\begin{align}
		\tilde{\cal W}^k(X; \mathcal{E}) = \overline{C^\infty_0(X; \mathcal{E})}^{\norm{\cdot}_{\tilde{\cal W}^k}}.
	\end{align}
	The space $\tilde{\cal W}^s(X; \mathcal{E})$ is then defined for all $s \ge 0$ by interpolation and for $s<0$ by setting $\tilde{\cal W}^s(X;\mathcal{E}) = (\tilde{\cal W}^{-s}(X;\mathcal{E}))^*$\,.\\
        Finally the space $\tilde{\mathcal{W}}^{1,s}(X;\mathcal{E})$ is the space
$$
\tilde{\mathcal{W}}^{1,s}(X;\mathcal{E})=\left\{u\in \mathcal{W}^{s}(X;\mathcal{E})\,, \quad \sqrt{\mathcal{O}}u\in \mathcal{W}^{s}(X;\mathcal{E})\right\}
$$
endowed with the norm $\|\sqrt{\mathcal{O}}u\|_{\tilde{\mathcal{W}}^{s}}$\,.
\end{definition}
\begin{remark}
  \label{re:compLeb}
The supremum norm over the families of vector fields ensure the geometrical global meaning of the functional spaces $\tilde{\mathcal{W}}^k(X;\mathcal{E})$ and therefore of $\tilde{\mathcal{W}}^{s}(X;\mathcal{E})$ and $\tilde{\mathcal{W}}^{1,s}(X;\mathcal{E})$\,. It is not the most convenient definition and in particular their Hilbert nature is not obvious here. A more convenient presentation in terms of local coordinates and then the use of a specific pseudo-differential calculus presented in Appendix \ref{sec:pseudodiff}  is detailed in Section~\ref{sec:spacesWs}.\\
Although those spaces are modelled on Lebeau's spaces in \cite{Leb1}\cite{Leb2} they slightly differ, e.g. the case $s=1$ allows $N_2=2$ with two vertical derivatives bounded in $L^2$\,.
\end{remark}

We shall define geometric Kramers-Fokker-Planck operators as second order differential operators acting on sections of the pullback bundle $\mathcal{E}$ that depend on a parameter $b \in (0,\infty)$. Our definition will be slightly more general than that of Lebeau \cite{Leb1}\cite{Leb2} but in the same spirit.
\begin{definition}[Geometric Kramers-Fokker-Planck Operator]
\label{de:GKFPO}
	A Geometric Kramers-Fokker-Planck (abbreviated as GKFP) operator  is a $b$-dependent operator $P_{\pm,b}+M(b)$ acting on $C^\infty_{0}(X; \mathcal{E})$ or $\mathcal{S}(X;\mathcal{E})$ with
	\begin{eqnarray*}
          &&P_{\pm,b} = \frac{1}{b^2} \mathcal{O} \pm \frac{1}{b} \nabla^{\mathcal{E}}_\mathcal{Y}\,,\\
          \text{and}&& \forall s\in \mathbb{R}\,,\quad M(b)\in \mathcal{L}(\tilde{\mathcal{W}}^{1,s}(X;\mathcal{E});\tilde{\mathcal{W}}^{s}(X;\mathcal{E}))
	\end{eqnarray*}
	where $\mathcal{Y}$ is Hamilton vector field of the kinetic energy (\ref{kinetic_energy}) with respect to the symplectic form $\sigma$ and $\mathcal{O}$ is the vertical harmonic oscillator.
      \end{definition}
Actually the term $M(b)$ will appear as a perturbative term for which the norm estimates of $\|M(b)\|_{\mathcal{L}(\tilde{\mathcal{W}}^{1,s};\tilde{\mathcal{W}}^{s})}$ with respect to the parameter $b>0$ can be discussed afterwards. Actually  all the analysis focuses on the case $M(b)=0$\,.
The H{\"o}rmander Theorem about sum of squares and type II operators (see \cite{Hor67}) provides the local hypoelliptic nature of the geometric Kramers-Fokker-Planck operator $P_{\pm, b}$ for every $b \in (0,\infty)$\,. By following the method of Lebeau in \cite{Leb1}\cite{Leb2} our aim is to provide accurate  subelliptic estimates with the best regularity exponents, that will also account for the behavior of $P_{\pm,b}$ as either $b \rightarrow 0^+$ (the large friction limit) or $b \rightarrow \infty$ (the low friction limit).

\subsection{Statement of the Main Result}

Remember the following notion.
\begin{definition}
\label{de:essmaxaccr} In a Hilbert space $\mathfrak{H}$ and densely  defined operator $A:D\mapsto \mathfrak{H}$ is called essentially maximal accretive, if it is accretive, therefore closable, and if it admits a unique maximal accretive extension equal to its closure $\overline{A}:D(\overline{A})\mapsto \mathfrak{H}$ with $D(\overline{A})=\overline{D}^{\|~\|_{A}}$\,, $\|u\|^{2}_{A}=\|u\|_{\mathfrak{H}}^{2}+\|Au\|_{\mathfrak{H}}^{2}$\,.
\end{definition}

The main result of this paper is the following subelliptic estimate for geometric Kramers-Fokker-Planck operators.\\

\begin{theorem}
  \label{th:mainOne}
  Let $P_{\pm,b}=\frac{1}{b^{2}}\mathcal{O}\pm \frac{1}{b}\nabla^{\mathcal{E}}_{\mathcal{Y}}$\,.
  There exists a constant $C_{g}\geq 1$
  determined by the geometric data $(g,E,g^{E},\nabla^{E})$ such that the operator $\frac{\kappa_{b}}{b^{2}}+P_{\pm,b}$ is essentially maximal accretive on $\mathcal{C}^{\infty}_{0}(X;\mathcal{E})$ (or on $\mathcal{S}(X;\mathcal{E})$), when $\kappa_{b}\geq C_{g}(1+b^{5})$\,.
  If $\overline{P}_{\pm,b}$ denotes its closure, the inequalities
  \begin{equation}
    \label{eq:IPPineqTh}
\mathrm{Re}~\langle u\,,\, (\frac{\kappa_{b}}{b^{2}}+\overline{P}_{\pm,b})u\rangle_{L^{2}}\geq \frac{1}{4b^{2}}\left[\|u\|_{\tilde{\mathcal{W}}^{1,0}}^{2}+\kappa_{b}\|u\|_{L^{2}}^{2}\right]\,.
\end{equation}
and
   \begin{align} \nonumber
    \left\| \left(\overline{P}_{\pm,b} - \frac{i\lambda}{b}\right)u\right\|_{L^{2}}+\frac{1}{b^2}\norm{u}_{L^2} 
    \geq  \frac{1}{C_{g}(1+b)^7}
    \Bigg(\left\| \frac{\mathcal{O}}{b^{2}}u \right\|_{L^{2}} + &
    \left\|
      \frac{1}{b}\left( \pm\nabla^{\mathcal{E}}_{\mathcal{Y}} - i\lambda \right)u
    \right\|_{L^{2}}
    \\
    \label{eq:principalInequalityOne}
    &+ \frac{1}{b^{4/3}}\left[||u||_{\tilde{\cal W}^{\frac{2}{3}}} + \norm{\left(\frac{\abs{\lambda}}{\langle p \rangle_q} \right)^{2/3} u}_{L^2}\right]
    \Bigg)
\end{align}
hold for  every $u\in D(\overline{P}_{\pm,b})$ and every $(\lambda,b)\in \R\times (0,+\infty)$\,.
\end{theorem}
The proof of the Theorem~\ref{th:mainOne} can be found in Section~\ref{sec:finalproof}. Other results involving the realizations of $P_{\pm,b}$ in the Sobolev spaces $\mathcal{W}^{s}(X;\mathcal{E})$ or other perturbative results will be deduced as corollaries in Section~\ref{sec:conseqOpt}.

\subsection{Outline of the article}
In Section~\ref{sec:redscal} an elementary integration by part provides the first a priori lower bound for $\mathrm{Re}~\langle u, P_{\pm,b}u\rangle$\,. This implies that the analysis of $P_{\pm,b}$ can be localized in the $q$-variable via partition of unity. Comparison of different connections can be done locally which reduces the problem to purely scalar operators and then the essential maximal accretivity on $\mathcal{C}^{\infty}_{0}(X;\mathcal{E})$ or $\mathcal{S}(X;\mathcal{E})$ is proved.\\
The Sobolev spaces 
$\tilde{\mathcal{W}}^{s_1,s_2}(X;\mathcal{E})$
are then studied in Section~{\ref{sec:spacesWs}}. After a localization via partition of unity, the Definition~\ref{de:tW} is characterized in term of the suitable pseudodifferential calculus, local in the $q$-variable but global in the $p$-variable. The construction of this pseudodifferential calculus relying on standard techniques, nevertheless to be adapted, is detailed in Appendix~\ref{sec:pseudodiff}. Section~\ref{sec:spacesWs} ends with a very convenient global characterization of these Sobolev spaces $\tilde{\mathcal{W}}^{s_1,s_2}(X;\mathcal{E})$ in terms of the functional calculus of two geometrically defined commuting self-adjoint operators, namely $\mathcal{O}$ and $W^2=C-\Delta_H +C\mathcal{O}^2$ where $\Delta_H$ is a scalar horizontal Laplacian.\\
Section~\ref{sec:localization} is devoted to the localization process. A dyadic partition of unity in the $p$-variable is used and then once the parameter $2^j$  of the dyadic partition if fixed, a grid partition in the $q$-variable with the spacing $2^{-j}$ is introduced. Near point of the grid, a Taylor expansion of the metric in normal coordinates expresses the scalar GKFP operator as the euclidean one with a $(2^j,b)$-dependent error term.\\
In Section~\ref{sec:EuclCase} the maximal subelliptic estimate, where the exponent $2/3$ is obtained via the model problem of the one dimensional complex Airy operator, is recalled. Actually the uniform estimates with respect to the parameters $(b,2^j,\lambda)\in (0,+\infty)^2\times R$ are carefully checked.\\
Section~\ref{sec:finalproof} gathers the local comparison of the scalar GKFP operator with the euclidean model of Section~\ref{sec:localization} with the uniform estimates of the euclidean model. Error terms due to the two partition of unities (dyadic in $p$ and $2^j$-dependent grid in $q$) happen to be controlled by the lower bounds of the parameter dependent euclidean model. While doing this, intermediate parameters of the grid partition must be tuned carefully according to the two regimes $2^j>>1$ or $2^j\leq C$\,.\\
Section~\ref{sec:conseqOpt} completes Theorem~\ref{th:mainOne} with various consequences or precisions. In particular the $b$-dependence of the perturbation $M(b)$ in Definition~\ref{de:GKFPO}, which allows the generalization of Theorem~\ref{th:mainOne} is specified. A corollary is the $\tilde{W}^{0,s}(X;\mathcal{E})$ version of Theorem~\ref{th:mainOne}, where a simple conjugation reduces the perturbed operator in $L^2(X,dqdp;\mathcal{E})$.\\
The Appendices gathers known material. A rather long paragraph is about the global pseudodifferential calculus on the total space $X=T^*Q$. As already said, it follows the general approach but things have to be specified in particular for proving, via the Helffer-Sjöstrand formula, that functions of self-adjoint globally elliptic operators in this class are pseudodifferential operators with a good asymptotic expansion.

\section{Reduction to a scalar operator}
\label{sec:redscal}
Here we write first a priori estimates for Geometric Kramers-Fokker-Planck (GKFP) operators coming from a simple integration by parts. The essential maximal accretivity of $\frac{\kappa_{b}}{b^{2}}+P_{\pm,b}$ is checked and all the perturbative terms coming from a partition of unity in the $q$-variable will be shown to be of lower order with a uniform control of the constants w.r.t $b$\,. Similarly a local change of connection happens to be of lower order and this reduces the problem to local scalar GKFP operators.
\subsection{Integration by parts and maximal accretivity}
\begin{proposition}\label{pr:IppIneqWithRealPart}
  \label{pr:IPPineq}
  Let $P_{\pm, b}=\frac{1}{b^{2}}\mathcal{O}\pm \frac{1}{b}\nabla^{\mathcal{E}}_{\mathcal{Y}}$\,. There exists $C_{0}\geq 1$\,, determined by the geometric data
  $(g,E,\nabla^{E},g^{E})$\,, such that for all $b>0$\,, $\lambda\in \R$  and for $\kappa_{b}\geq C_{0}(1+b^{2})$ the inequality
  \begin{eqnarray}
    \label{eq:IPPineq}
&&\mathrm{Re}~ \langle u\,,\, (\frac{\kappa_{b}}{b^{2}}+P_{\pm, b}-i\lambda) u\rangle_{L^{2}(X;\mathcal{E})}\geq \frac{1}{4b^{2}}
\left[
  \|u\|_{\tilde{\mathcal{W}}^{1,0}(X;\mathcal{E})}^{2}+\kappa_{b}\|u\|_{L^{2}(X;\mathcal{E})}^{2}\right]\\
\text{and}&&
\label{eq:secondIPPineq}
    \left\| (\frac{\kappa_{b}}{b^{2}}+P_{\pm, b}-i\lambda) u \right\|_{L^{2}(X;\mathcal{E})}^{2}\geq \frac{\kappa_{b}}{16b^{4}}
\left[
  \norm{u}_{\mathcal{W}^{1,0}(X;\mathcal{E})}^{2}
  +\kappa_{b}\|u\|_{L^{2}(X;\mathcal{E})}^{2}\right]
  \end{eqnarray}
  holds for all $ u\in \mathcal{C}^{\infty}_{0}(X;\mathcal{E})$\,.
\end{proposition}
Before proving this result let us specify the formal adjoint of $\nabla_{\mathcal{Y}}^{\mathcal{E}}$\,. Start with the vector bundle $\pi_{E}:E\to Q$ and the data $(\nabla^{E},g^{E})$ and consider the dual connection with respect to $g^{E}$ given by
$$
Xg^{E}(s,s')=g^{E}(s, \nabla^{E}_{X}s')+g^{E}(\nabla_{X}^{E,*}s, s')\,.
$$
The unitary connection
$$
\nabla^{E,u}=\frac{\nabla^{E}+\nabla^{E,*}}{2}
$$
differs from $\nabla^{E}$ by
$$
\nabla^{E,u}-\nabla^{E}=\frac{1}{2}\omega(\nabla^{E},g^{E})\quad\in \mathcal{C}^{\infty}(Q;T^{*}Q\otimes\mathrm{End}(E))
$$
With the pull back we obtain with $\mathcal{Y}=g^{ij}(q)p_{i}e_{j}$ written in a local canonical coordinates system
$$
\nabla^{\mathcal{E},u}_{\mathcal{Y}}-\nabla^{\mathcal{E}}_{\mathcal{Y}}=\frac{1}{2}
g^{ij}(q)p_{i}\omega(\nabla^{E},g^{E})(\frac{\partial}{\partial q^{j}})=a^{i}(q)p_{i}\quad,\quad a^{i}(q)\in \mathrm{End}(E_{q})
$$
We also recall the formula
$$
\forall v,w\in \mathcal{C}^{\infty}_{0}(X;\mathcal{E})\,,\forall T\in \mathcal{C}^{\infty}(X;TX)\,,\,
\int_{X}g^{\mathcal{E}}(v,\nabla^{\mathcal{E},u}_{T}w)~d\textrm{vol}_X=-\int_{X}g^{\mathcal{E}}(\nabla_{T}^{\mathcal{E},u}v,w)
+\mathrm{div}(T)g^{\mathcal{E}}(v,w)~d\textrm{vol}_{X}
$$
while here $d\textrm{vol}_{X}=dqdp$ and $\mathrm{div}\,\mathcal{Y}=0$\,.\\
We find that the formal adjoint of $\nabla_{\mathcal{Y}}^{\mathcal{E}}$ is nothing but
\begin{equation}
  \label{eq:formadjY}
(\nabla^{\mathcal{E}}_{\mathcal{Y}})^{*}=-\nabla^{\mathcal{E}}_{\mathcal{Y}}-g^{ij}(q)p_{i}\omega(\nabla^{E},g^{E})(\frac{\partial}{\partial q^{j}})
=-\nabla^{\mathcal{E}}_{\mathcal{Y}}-a^{i}(q)p_{i}\,.
\end{equation}
\begin{proof}[Proof of Proposition~\ref{pr:IPPineq}]
  Let $\mathop{\cup}_{j=1}^{J}\Omega_{j}= Q$ be a finite open chart covering of $Q$ and let $\sum_{j=1}^{J}\varrho_{j}(q)^{2}\equiv 1$ be a subordinate quadratic partition of unity, $\varrho_{j}\in \mathcal{C}^{\infty}_{0}(\Omega_{j};[0,1])$\,.  Because $\nabla^{\mathcal{E}}_{\mathcal{Y}}$\,, $P_{\pm,b}$\,, $\mathcal{O}$\,, are at most first order differential operators in $q$ we get
  \begin{eqnarray*}
&& \langle u\,,\, (\frac{\kappa_{b}}{b^{2}}+P_{\pm, b}-i\lambda)u\rangle_{L^{2}(X;\mathcal{E})}
=\sum_{j=1}^{J}\langle u_{j}\,,\, (\frac{\kappa_{b}}{b^{2}}+P_{\pm, b}-i\lambda)u_{j}\rangle_{L^{2}(X;\mathcal{E})}
    \\
&&     \underbrace{\|\sqrt{\mathcal{O}}u\|_{L^{2}(X;\mathcal{E})}^{2}}_{=\|u\|^{2}_{\tilde{\mathcal{W}}^{1,0}(X;\mathcal{E})}}   
   + \kappa_{b}\|u\|_{L^{2}(X;\mathcal{E})}^{2}
   =\langle u\,,\, \mathcal{O} u\rangle_{\mathcal{L}^{2}(X;\mathcal{E})}+\kappa_{b}\|u\|_{L^{2}(X,\mathcal{E})}^2
    =\sum_{j=1}^{J}\underbrace{\|\sqrt{\mathcal{O}}u_{j}\|_{L^{2}(X;\mathcal{E})}^{2}}_{=\|u_{j}\|^{2}_{\tilde{\mathcal{W}}^{1,0}(X;\mathcal{E})}}+ \kappa_{b} \|u_{j}\|_{L^{2}(X;\mathcal{E})}^{2}
  \end{eqnarray*}
  for all $u\in \mathcal{C}^{\infty}_{0}(X;\mathcal{E})$\,, by setting $u_{j}=\varrho_{j}u\in \mathcal{C}^{\infty}_{0}(T^{*}\Omega_{j};\mathcal{E})$\,.\\
  With canonical local coordinates $(q,p)$ in $T^{*}\Omega_{j}$\,, \eqref{eq:formadjY} implies
  \begin{align*}
    \mathrm{Re}~\langle u_{j}\,,\, (\frac{\kappa_{b}}{b^{2}}+P_{\pm,b}-i\lambda)u_{j}\rangle_{L^{2}(X;\mathcal{E})}
    &=
      \langle u_{j}\,,\, \frac{2\kappa_{b}-\Delta_{p}+|p|_{q}^{2}}{2b^{2}}u_{j}\rangle_{L^{2}(X;\mathcal{E})}
      \mp
      \frac{1}{b}\langle u_{j}\,,\, a^{i}(q)p_{i}u_{j}\rangle_{L^{2}(X;\mathcal{E})}
    \\
    &\geq \frac{1}{2b^{2}}\left[\|u_{j}\|_{\tilde{\mathcal{W}}^{1,0}(X;\mathcal{E})}^{2}+2\kappa_{b}\|u\|^{2}_{L^{2}(X;\mathcal{E})}\right]-\frac{C_{0}'}{b}\norm{u_{j}}_{L^{2}(X;\mathcal{E})}
      \|u_{j}\|_{\tilde{\mathcal{W}}^{1,0}(X;\mathcal{E})}
\\    &\hspace{-1cm}\geq
      \frac{1}{2b^{2}}\left[\|u_{j}\|_{\tilde{\mathcal{W}}^{1,0}(X;\mathcal{E})}^{2}+2\kappa_{b}\|u\|^{2}_{L^{2}(X;\mathcal{E})}\right]
      -\frac{1}{4b^{2}}\|u_{j}\|_{\tilde{\mathcal{W}}^{1,0}(X;\mathcal{E})}^{2} -2{C_{0}'}^{2}\|u_{j}\|_{L^{2}(X;\mathcal{E})}^{2}\,,
  \end{align*}
  for some $C_{0}'>0$ determined by the geometric data $(g,E,g^{E},\nabla^{E})$\,.
  With $\frac{\kappa_{b}}{2b^{2}}\geq C_{0}\frac{1+b^{2}}{2b^{2}}\geq \frac{C_{0}}{2}$\,, the first inequality \eqref{eq:IPPineq} is proved for $C_{0}\geq 4{C_{0}'}^{2}$\,.\\
  Using Cauchy-Schwarz inequality in the left hand side of \eqref{eq:IPPineq} yields
  $$ \| (\frac{\kappa_{b}}{b^{2}}+P_{\pm, b} - i\lambda) u \|_{L^{2}(X;\cal{E})} \|u\|_{L^{2}(X;\cal{E})} \geq \frac{1}{4b^{2}} 
  \left[
    \|u\|_{\tilde{\mathcal{W}}^{1,0}(X;\mathcal{E})}^{2} + \kappa_{b} \|u\|_{L^{2}(X;\cal{E})}^2
  \right] .$$
  We deduce at once $ \|(\frac{\kappa_{b}}{b^{2}}+P_{\pm, b} - i \lambda) u \|_{L^{2}(X;\cal{E})} \geq \frac{\kappa_{b}}{4b^{2}}\|u\|_{L^{2}(X,\cal{E})} $\,. The latter inequality multiplied by $\|(\frac{\kappa_{b}}{b^{2}}+P_{\pm, b} - i \lambda) u \|_{L^{2}(X;\cal{E})}$  yields \eqref{eq:secondIPPineq}.
\end{proof}

\begin{corollary} \label{corollary_ess_max_acc}
   Let $P_{\pm, b}=\frac{1}{b^{2}}\mathcal{O}\pm \frac{1}{b}\nabla^{\mathcal{E}}_{\mathcal{Y}}$ and let $C_{0}\geq 1$ be determined by the geometric data
   $(g,E,\nabla^{E},g^{E})$ according to Proposition~\ref{pr:IPPineq}\,. For $\kappa_{b}\geq C_{0}(1+b^{2})$ the operator $\frac{\kappa_{b}}{b^{2}}+P_{\pm,b}$
   is essentially maximal accretive on $\mathcal{C}^{\infty}_{0}(X;\mathcal{E})$ and therefore on $\mathcal{S}(X;\mathcal{E})$\,.
 \end{corollary}
 \begin{proof}
   Proposition~\ref{pr:IPPineq} says that the operator $(\frac{\kappa_{b}}{b^{2}}+P_{\pm, b})$ and its formal adjoint $\frac{\kappa_{b}}{b^{2}}+P_{\pm,b}^{*}$ are  accretive on $\mathcal{C}^{\infty}_{0}(X;\mathcal{E})$ with the lower bound \eqref{eq:IPPineq}\,.\\
   It suffices to prove that the range $(\frac{\kappa_{b}}{b^{2}}+P_{\pm b})\mathcal{C}^{\infty}_{0}(X;\mathcal{E})$ is dense in $L^{2}(X;\mathcal{E})$\,.
   It is equivalent to 
   $$
   \left.
     \begin{array}[c]{l}
     u\in L^{2}(X;\mathcal{E})\\
     (\frac{\kappa_{b}}{b^{2}}+P_{\pm, b}^{*})u=0\in \mathcal{D}'(X;\mathcal{E})
   \end{array}
   \right\}\Rightarrow u=0\,.
   $$
   With $P_{\pm,b}^{*}=P_{\mp,b}\mp a^{i}(q)b_{i}$ in local coordinates according to \eqref{eq:formadjY},
   H{\"o}rmander's hypoellipticity result for type~II operators (see~\cite{Hor67}) implies $u\in \mathcal{C}^{\infty}(X;\mathcal{E})$\,.
   For $\chi\in \mathcal{C}^{\infty}_{0}(\mathbb{R};[0,1])$ such that $\chi\equiv 1$ in a neighborhood of $0$ and for $\varepsilon>0$ set $u_{\varepsilon}=\chi(\varepsilon|p|_{q}^{2})u$\,.\\
   The above equation implies
   $$
   (\frac{\kappa_{b}}{b^{2}}+P_{\pm,b}^{*})u_{\varepsilon}=-\left[P_{\pm, b}^{*}\,,\,\chi(\varepsilon |p|_{q}^{2})\right]u
   =
   -\left[-\frac{\Delta_{p}}{2b^{2}}\,,\,\chi(\varepsilon|p|_{q}^{2})\right]u   
     $$
     because $\mathcal{Y}f(|p|_{q}^{2})=0$\,. The form of the last commutator allows to write
$$
(\frac{\kappa_{b}}{b^{2}}+P_{\pm,b}^{*}u_{\varepsilon})= -\left[-\frac{\Delta_{p}}{2b^{2}}\,,\,\chi(\varepsilon|p|_{q}^{2})\right](1-\tilde{\chi}(\varepsilon |p|_{q}^{2}))u  
$$
where $\tilde{\chi}\in \mathcal{C}^{\infty}_{0}(\mathbb{R};R)$ has a support included a neighborhood of $0$ where  $\chi\equiv 1$ and its derivatives vanish, while $\tilde{\chi}\equiv 1$ in a smaller neighborhood of $0$\,.
By taking the scalar product with $u_{\varepsilon}$\,, the inequality \eqref{eq:IPPineq} for $P_{\pm,b}^{*}$ implies
$$
\frac{1}{4b^{2}}\left[\|u_{\varepsilon}\|_{\tilde{\mathcal{W}}^{1,0}(X,\mathcal{E})}^{2}+\kappa_{b}\|u_{\varepsilon}\|^{2}_{L^{2}(X;\mathcal{E})}\right]\leq C_{g,\chi}\|u_{\varepsilon}\|_{\tilde{\mathcal{W}}^{1,0}(X;\mathcal{E})}\|(1-\tilde{\chi}(\varepsilon|p|_{q}^{2}))u\|_{L^{2}(X;\mathcal{E})}\,
$$
and
$$
\sqrt{\frac{d}{2}}\|u_{\varepsilon}\|_{L^{2}(X;\mathcal{E})}\leq \|u_{\varepsilon}\|_{\tilde{\mathcal{W}}^{1,0}(X;\mathcal{E})}\leq 4C_{g,\chi}b^{2}\|(1-\tilde{\chi}(\varepsilon)|p|_{q}^{2})u\|_{L^{2}(X;\mathcal{E})}\,.
$$
Lebesgue's theorem for the limit $\varepsilon\to 0$ gives
$$
\sqrt{\frac{d}{2}}\|u\|_{L^{2}(X;\mathcal{E})}\leq 4C_{g,\chi}b^{2}\lim_{\varepsilon\to 0}
\|(1-\tilde{\chi}(\varepsilon|p|_{q}^{2}))u\|_{L^{2}(X;\mathcal{E})}=0\,.
$$
 \end{proof}

\subsection{Localization}
\label{sec:partunitPb}
\begin{proposition}
  \label{pr:equivalence}
  Let $P_{\pm,b}=\frac{1}{b^{2}}\mathcal{O}\pm\frac{1}{b}\nabla^{\mathcal{E}}_{\mathcal{Y}}$ and  fix $Q=\mathop{\cup}_{j=1}^{J}\Omega_{j}$ a finite open chart covering of $Q$\,. Let $\sum_{j=1}^{J}\varrho_{j}(q)^{2}\equiv 1$ be a subordinate quadratic partition of unity, $\varrho_{j}\in \mathcal{C}^{\infty}_{0}(\Omega_{j};[0,1])$\,. There exists $C_{0}\geq 1$\,, determined by the geometric data
  $(g,E,\nabla^{E},g^{E})$\,, and now the partition of unity $(\varrho_{j})_{1\leq j\leq J}$\,, such that for all $b>0$\,, $\lambda\in \R$  and for $\kappa_{b}=C_{0}(1+b^{2})$ the following equivalence of norms
  \begin{equation}
    \label{eq:equivalence}
    \left(
      \frac{\norm{(\frac{\kappa_{b}}{b^{2}}+P_{\pm,b}-i\lambda)u}^{2}_{L^{2}(X;\mathcal{E})}}
      {\sum_{j=1}^{J}\norm{(\frac{\kappa_{b}}{b^{2}}+P_{\pm,b}-i\lambda)(\varrho_{j}u)}^{2}_{L^{2}(X;\mathcal{E})}}
      \right)^{\pm 1}\leq 4
  \end{equation}
  holds for all $ u\in \mathcal{C}^{\infty}_{0}(X;\mathcal{E})$\,.
\end{proposition}

\begin{proof}
  It's a straightforward application of Corollary \ref{Cor:equivalenceOfQuantities}. We have to check the assumption \eqref{eq:equivH}  which says
  \begin{align*}
     \forall u\in \mathcal{C}^{\infty}_{0}(X;\mathcal{E})\,,\quad \frac{r}{2} \sum_{j\in J} \| (\frac{\kappa_{b}}{b^{2}} + P_{\pm,b} - i \lambda) \varrho_{j}u \|^{2}_{L^{2}(X;\cal{E})} &\geq  2 \sum_{j_{1},j\in J} \|[(\frac{\kappa_{b}}{b^{2}} + P_{\pm,b} - i \lambda), \varrho_{j_{1}}] \varrho_{j} u \|^{2}_{L^{2}(X;\cal{E})} \\
   &+ 4 \sum_{j_{1},j_{2},j\in J} \|[[(\frac{\kappa_{b}}{b^{2}} + P_{\pm,b} - i \lambda),\varrho_{j_{2}}],\varrho_{j_{1}} ]\varrho_{j}u \|^{2}_{L^{2}(X;\cal{E})}\,,
  \end{align*}
  for some $r\in[0,1)$\,. Because  the operator $P_{\pm,b}$ is a first-order differential operator in the $q$ variable  the first commutator equals 
  $$[\frac{\kappa_{b}}{b^{2}} + P_{\pm,b} - i \lambda, \varrho_{j_{1}}] = \pm \frac{1}{b}\mathcal{Y}\varrho_{j_{1}} = \pm\frac{1}{b}g^{\ell k}(q)p_{k} \frac{\partial \varrho_{j_{1}}}{\partial q^{\ell}}, $$
  for $j_{1}\in J$ where the right-hand side is written local canonical coordinate $(q,p)$. Moreover the double commutators indexed by $j_{1},j_{2}\in J$ all vanish.\\
  We deduce the existence of $C_{0}'>0$ such that
  $$  \forall u\in \mathcal{C}^{\infty}_{0}(X;\mathcal{E})\,,\quad\|[(\frac{\kappa_{b}}{b^{2}} + P_{\pm,b} - i \lambda), \varrho_{j_{1}}] \varrho_{j} u \|^{2}_{L^{2}(X;\cal{E})} \leq C_{0}'\frac{1}{b^{2}} \| |p|_{q} \varrho_{j} u \|_{L^{2}(X;\cal{E})}^{2} .$$
  and the summation over $j_{1},j\in J$\,, combined with the inequality \eqref{eq:secondIPPineq}, yields
  $$ \forall u\in \mathcal{C}^{\infty}_{0}(X;\mathcal{E})\,,\quad \sum_{j_{1},j\in J}  \|[(\frac{\kappa_{b}}{b^{2}} + P_{\pm,b} - i \lambda), \varrho_{j_{1}}] \varrho_{j} u \|^{2}_{L^{2}(X;\cal{E})} \leq C_{0}' ~ |J| ~ \frac{16b^{2}}{\kappa_{b}} \sum_{j} \| (\frac{\kappa_{b}}{b^{2}}+P_{\pm, b} - i \lambda) \varrho_{j}u \|^{2}_{L^{2}(X,\cal{E})}\,.$$
  With $C_{0}' ~ |J| ~ \frac{16b^{2}}{\kappa_{b}}\leq C_{0}' ~ |J| ~ \frac{16}{C_{0}}$\,, choosing $C_{0}\geq 1$ large enough guarantees the assumption \eqref{eq:equivH} with $r=\frac{1}{2}$\,. 
\end{proof}

\subsection{Changing locally the connections}
\label{sec:locchangeconnec}

With Proposition~\ref{pr:equivalence} the analysis of $P_{\pm,b}$ can be localized in a chart open domain $\Omega_{j}$\,. Additionally it can be assumed that there is a well defined local  frame $(f^{1}(q),\ldots,f^{N}(q))$ of the restricted bundle $E\big|_{\Omega_{j}}$\,. In this frame  a trivial connection $\nabla^{E,j}$ on $E\big|_{\Omega_{j}}$ and therefore a corresponding flat connection on $\mathcal{E}\big|_{T^{*}\Omega_{j}}$ is defined by pull-back.
The operator $P_{\pm, b}^{j}$ defined locally can be identified with a scalar operator according to
\begin{align}
\label{eq:defPbj}
  P_{\pm,b}^{j}\left[\sum_{\ell=1}^{N}u_{\ell}f^{\ell}\right]&=
\left(\frac{1}{b^{2}}\mathcal{O}\pm \frac{1}{b}\nabla^{\mathcal{E},j}_{Y}\right)\left[\sum_{\ell=1}^{N}u_{\ell}f^{\ell}\right]
  \\
  &=\sum_{\ell=1}^{N}
\left[\frac{-g_{ik}(q)\frac{\partial^{2}}{\partial p_{i}\partial p_{k}}+ g^{ik}(q)p_{i}p_{k}}{2b^{2}}(u_{\ell})
  \pm \frac{1}{b}g^{ik}(q)p_{i}e_{k}(u_{\ell})\right]f^{\ell} \,.
\end{align}
\begin{proposition}
  \label{pr:locscalPb}
  Under the assumptions of Proposition~\ref{pr:equivalence} with  and with the additional condition that $E$ is trivialized by  a local frame $(f^{1}(q),\ldots, f^{N}(q))$ over $\Omega_{j}$ for every $j\in \left\{1,\ldots,J\right\}$\,, let $P_{\pm, b}^{j}$ be defined by \eqref{eq:defPbj}. There exists $C_{0}\geq 1$\,,  determined by the geometric data
  $(g,E,\nabla^{E},g^{E})$ and the partition of unity $(\varrho_{j})_{1\leq j\leq J}$\,, such that for all $b>0$\,, $\lambda\in \R$  and for $\kappa_{b}=C_{0}(1+b^{2})$ the following equivalence of norms
  \begin{equation}
    \label{eq:locscalPb}
    \left(
      \frac{\norm{(\frac{\kappa_{b}}{b^{2}}+P_{\pm, b}-i\lambda)u}^{2}_{L^{2}(X;\mathcal{E})}}
      {\sum_{j=1}^{J}\norm{(\frac{\kappa_{b}}{b^{2}}+P^{j}_{\pm,b}-i\lambda)(\varrho_{j}u)}^{2}_{L^{2}(X;\mathcal{E})}}
      \right)^{\pm 1}\leq 12
  \end{equation}
  holds for all $ u\in \mathcal{C}^{\infty}_{0}(X;\mathcal{E})$\,.
\end{proposition}

\begin{proof}
 For a given $j\in \{1,\ldots, J\}$ and for $v\in \mathcal{C}^{\infty}_{0}(\Omega_{j};\mathcal{E})$ we have 
   $$ (P_{\pm,b} - P_{\pm,b}^j) (v) = \pm\frac{1}{b}(\nabla_{\mathcal{Y}}^{\mathcal{E}}-\nabla^{\mathcal{E},j}_{\mathcal{Y}})(v)
   =\pm\frac{1}{b} g^{ik}(q)p_{i}\pi_{X}^{*}(\nabla^{E}_{\frac{\partial}{\partial q^{k}}}-\nabla^{E,j}_{\frac{\partial}{\partial q^{k}}})(v)
   $$
   where
$$
(\nabla^{E}_{\frac{\partial}{\partial q^{k}}}-\nabla^{E,j}_{\frac{\partial}{\partial q^{k}}})[v_{\ell}(q,p)f^{\ell}(q)]=v_{\ell}(q,p)F^{\ell}_{\ell',k}(q)f^{\ell'}(q)
$$
with $F^{\ell}_{\ell',k}\in \mathcal{C}^{\infty}(\Omega_{j};\R)$\,.\\
Applied to $v=\varrho_{j}u$ this gives the 
    upper bound 
    $$ \|(P_{\pm,b} - P_{\pm, b}^j) )(\varrho_j u) \|^2 \leq \frac{C_0'}{b^2} \| |p|_q (\varrho_j u) \|_{L^{2}(X;\cal{E})}^2,$$
    where $C_0'$ depends only on the metric $g^E$ and the connection $\nabla^E$ given on the vector bundle $E$ and is uniform with respect to $j\in \left\{1,\ldots J\right\}$\,.
    The same argument as in the proof o Proposition~\ref{pr:equivalence} shows that 
    $$ \| (P_{\pm,b} - P_{\pm,b}^j) (\varrho_j u) \|^2_{L^2(X;\cal{E})} \leq \frac{1}{4} \| (\frac{\kappa_b}{b^2}+ P_{\pm,b} -i \lambda) (\varrho_j u) \|^2_{L^2(X;\cal{E})} ,$$
    when the constant $C_{0}$ in $\kappa_{b}=C_{0}(1+b^{2})$ is chosen large enough.
    The parallelogram identity implies
    \begin{align*}
&\| (\frac{\kappa_b}{b^2}+ P_{\pm,b}^{j} -i \lambda) (\varrho_j u) \|^2_{L^2(X;\cal{E})}\leq 3
\| (\frac{\kappa_b}{b^2}+ P_{\pm,b} -i \lambda) (\varrho_j u) \|^2_{L^2(X;\cal{E})}
      \\
      \text{and}~&
\| (\frac{\kappa_b}{b^2}+ P_{\pm,b} -i \lambda) (\varrho_j u) \|^2_{L^2(X;\cal{E})}\leq 3
\| (\frac{\kappa_b}{b^2}+ P_{\pm,b}^{j} -i \lambda) (\varrho_j u) \|^2_{L^2(X;\cal{E})}\,.
    \end{align*}
    By summation over $j\in \left\{1,\ldots,J\right\}$ we obtain
    $$
\left(\frac{\sum_{j=1}^{J}\| (\frac{\kappa_b}{b^2}+ P_{\pm,b}^{j} -i \lambda) (\varrho_j u) \|^2_{L^2(X;\cal{E})}}{\sum_{j=1}^{J}\| (\frac{\kappa_b}{b^2}+ P_{\pm,b} -i \lambda) (\varrho_j u) \|^2_{L^2(X;\cal{E})}}\right)^{\pm 1}\leq 3\,.
$$
The inequality \eqref{eq:locscalPb} is obtained by taking the product with the result \eqref{eq:equivalence} of Proposition~\ref{pr:equivalence}.
  \end{proof}
  \section{Sobolev spaces}
\label{sec:spacesWs}
Like for the operator $P_{\pm,b}$ we firstly reduce the characterization of $u\in\tilde{\mathcal{W}}^{k}(X;\mathcal{E})$\,, $k\in\mathbb{N}$\,, to a local problem with the possibility of replacing the connection $\nabla^{\mathcal{E}}$ by a trivial connection in a given local frame. Then $\mathcal{W}^{s}(X;\mathcal{C})$ and its norm will be expressed in terms of the functional calculus of pseudo-differential elliptic self-adjoint operator $W^{2}$ in the class $\mathrm{OpS}^{2}_{\Psi}(Q;\mathbb{C})$ presented in Appendix~\ref{sec:pseudodiff}. Finally general spaces $\tilde{\mathcal{W}}^{s_{1},s_{2}}(X;\mathcal{E})$\,, $s_{1},s_{2}\in \mathbb{R}$ are introduced by using the functional calculus of two commuting self-adjoint operators.

\subsection{First properties of $\tilde{\mathcal{W}}^k$, $k\in\N$}
\label{sec:firstWk}
We collect rather immediate consequences of the Definition~\ref{de:tW}.
The norm $\|u\|_{\tilde{\mathcal{W}}^{k}}$ is 
\begin{equation}
  \label{eq:normeNkW}
\|u\|_{\tilde{\mathcal{W}}^{k}}=\max_{N_{1}+\frac{N_{2}+N_{3}}{2}\leq k} P_{k,N_{1},N_{2},N_{3}}(u)
\end{equation}
where $P_{k,N_{1},N_{2},N_{3}}(u)$ is the smallest constant $C\geq 0$ such that
\begin{align*}
  & \forall (T_{1}^{H},\ldots T_{N_{1}}^{H})\in \mathcal{C}^{\infty}_{Q}(X;T^{H}X)^{N_{1}}\,, \forall (T_{1}^{V},\ldots, T_{N_{2}}^{V})\in \mathcal{C}^{\infty}_{Q}(X;T^{V}X)^{N_{2}}\,,\\
  &\hspace{2cm}\|\langle p\rangle_{q}^{N_{3}}\nabla_{T_{1}^{H}}^{\mathcal{E}}\ldots \nabla_{T_{N_{1}}^{H}}^{\mathcal{E}}\nabla_{T_{1}^{V}}^{\mathcal{E}}\ldots \nabla_{T_{N_{2}}^{V}}^{\mathcal{E}}u\|_{L^{2}}\leq C\prod_{n_{1}=1}^{N_{1}}\|T_{n_{1}}^{H}\|_{\mathcal{C}^{k}}\prod_{n_{2}=1}^{N_{2}}\|T_{n_{2}}^{V}\|_{\mathcal{C}^{k}}\,.
\end{align*}
Because
$$
\langle p\rangle_{q}^{N_{3}}\nabla_{T_{1}^{H}}^{\mathcal{E}}\ldots \nabla_{T_{N_{1}}^{H}}^{\mathcal{E}}\nabla_{T_{1}^{V}}^{\mathcal{E}}\ldots \nabla_{T_{N_{2}}^{V}}^{\mathcal{E}}u
=
\nabla_{T_{1}^{H}}^{\mathcal{E}}\ldots \nabla_{T_{N_{1}}^{H}}^{\mathcal{E}}\langle p\rangle_{q}^{N_{3}}\nabla_{T_{1}^{V}}^{\mathcal{E}}\ldots \nabla_{T_{N_{2}}^{V}}^{\mathcal{E}}u\,,
$$
$\|u\|_{\tilde{\mathcal{W}}^{k}}$ also equals
\begin{equation}
  \label{eq:redN1}
\|u\|_{\tilde{\mathcal{W}}^{k}}=\max_{\substack{N_{2}+N_{3}=j\leq 2k \\ N_{1}\leq k-2j}}
\quad \sup_{(T_{1}^{V},\ldots T_{N_{2}}^{V})\in \mathcal{C}^{\infty}_{Q}(X;T^{V}X)^{N_{2}}}\frac{P_{k,N_{1},0,0}(\langle p\rangle_{q}^{N_{3}}\nabla_{T_{1}^{V}}^{\mathcal{E}}\ldots \nabla_{T_{N_{2}}^{V}}^{\mathcal{E}}u)}{\prod_{n_{2}=1}^{N_{2}}\|T_{n_{2}}^{V}\|_{\mathcal{C}^{k}}}\,.
\end{equation}

\begin{proposition}
  \label{pr:contTheta}
  Let $k\in \N$\,, let $\theta\in \mathcal{C}^{\infty}(Q;\mathrm{End}(E))$ and fix any $\mathcal{C}^{k}$-norm on $Q$\,. The multiplication by $\pi_{X}^{*}\theta\in \mathcal{C}^{\infty}(X;\mathrm{End}(\mathcal{E}))$\,, $\pi_{X}^{*}\theta(x)=\pi_{X}^{*}(\theta(\pi_{X}(x)))$\,, is a bounded operator in $\tilde{\mathcal{W}}^{k}(X;\mathcal{E})$ with
  \begin{equation}
    \label{eq:thetaW}
\forall u\in \tilde{\mathcal{W}}^{k}(X;\mathcal{E})\,,\quad
\|(\pi_{X}^{*}\theta) u\|_{\tilde{\mathcal{W}}^{k}}\leq C_{k}\|\theta\|_{\mathcal{C}^{k}}\|u\|_{\tilde{\mathcal{W}}^{k}}\,.
\end{equation}
\end{proposition}
\begin{proof}
  Because $\langle p\rangle_{q}^{N_{3}}\nabla_{T_{1}^{V}}^{\mathcal{E}}\ldots \nabla^{\mathcal{E}}_{T_{N_{2}}^{V}}(\pi_{X}^{*}\theta)= (\pi_{X}^{*}\theta)\langle p\rangle_{q}^{N_{3}}\nabla_{T_{1}^{V}}^{\mathcal{E}}\ldots \nabla^{\mathcal{E}}_{T_{N_{2}}^{V}}$ and owing to \eqref{eq:redN1} the problem is reduced to
$$
P_{k,N_{1},0,0}((\pi_{X}^{*}\theta)v)\leq C_{k}\|\theta\|_{\mathcal{C}^{k}}P_{k,N_{1},0,0}(v)
$$
for all $N_{1}\in \left\{0,\ldots,k\right\}$\,.\\
It is obviously true for $N_{1}=0$\,. If it is true for $N_{1}\in \left\{0,\ldots, k-1\right\}$ then for $T_{N_{1}+1}^{H}\in \mathcal{C}^{\infty}_{Q}(X;TX^{H})$ with $T_{N_{1}+1}=\pi_{X,*}T_{N_{1}+1}^{H}\in \mathcal{C}^{\infty}(Q;TQ)$\,, we write
$$
\nabla_{T_{N_{1}+1}^{H}}^{\mathcal{E}}(\pi_{X}^{*}\theta)v=(\pi_{X}^{*}(\theta)\nabla_{T_{N_{1}+1}^{H}}^{\mathcal{E}}v)+
[\pi_{X}^{*}(\nabla_{T_{N_{1}+1}}^{\mathrm{End}(E)}\theta)]v\,.
$$
We get
\begin{align*}
  P_{k-1,N_{1},0,0}(\nabla_{T_{N+1}^{H}}^{\mathcal{E}}(\pi_{X}^{*}\theta)v)
  &\leq
C_{k-1}\left[\|\theta\|_{\mathcal{C}^{k-1}}P_{k-1,N_{1},0,0}(\nabla_{T_{N_{1}+1}^{H}}^{\mathcal{E}}v)+
  \|\nabla_{T_{N_{1}+1}}^{\mathrm{End}(E)}\theta\|_{\mathcal{C}^{k-1}}
                                                      P_{k-1,N_{1},0,0}(v)\right]\\
  &\leq
    2C_{k-1}\|\theta\|_{\mathcal{C}^{k}}P_{k,N_{1}+1,0,0}(v)\,.
\end{align*}
The obvious inequality
$$
P_{k,N_{1}+1,0,0}(w)\leq \sup_{T_{N_{1}+1}^{H}\in \mathcal{C}^{\infty}_{Q}(X;T^{H}V)}
\frac{P_{k-1,N_{1},0,0}(\nabla_{T_{N_{1}+1}^{H}}^{\mathcal{E}}w)}{\|T_{N_{1}+1}^{H}\|_{\mathcal{C}^{k}}}
$$
ends the proof by induction.
\end{proof}
The previous statement contains two particular cases which allow the local scalar characterization of $u\in \tilde{\mathcal{W}}^{k}(X;\mathcal{E})$ with a simpler norm.
\begin{proposition}
  \label{pr:Wkred}
  Fix $k\in \mathbb{N}$ and consider two different connections $\nabla^{E,1}$ and $\nabla^{E,2}$ on the vector bundle $E\stackrel{\pi_{E}}{\to}Q$ with the associated connections $\nabla^{\mathcal{E},1}$ and $\nabla^{\mathcal{E},2}$ on $\mathcal{E}\stackrel{\pi_{\mathcal{E}}}{\to}X$\,. Let
  $\tilde{\mathcal{W}}_{\nabla^{j}}^{k}(X;\mathcal{E})$ and $\|~\|_{\tilde{\mathcal{W}}^{k}_{\nabla^{j}}}$ be the corresponding $\tilde{\mathcal{W}}^{k}(X;\mathcal{E})$ spaces and norms according to Definition~\ref{de:tW}.
  \begin{description}
  \item[1)] The space $\mathcal{W}^{k}(X;\mathcal{E})$ is a $\mathcal{C}^{\infty}(Q;\mathbb{R})$ module and for any finite atlas $Q=\mathop{\bigcup}_{j=1}^{J}\Omega_{j}$ and for any subordinate partition of unity $\sum_{j=1}^{J}\varrho_{j}(q)\equiv 1$\,, a section $u$ belongs to
$\mathcal{W}^{k}(X;\mathcal{E})$ if and only if, for every $1\leq j\leq J$\,,  $\varrho_{j}u\in \mathcal{W}^{k}_{\Omega_{j}-\mathrm{comp}}(T^{*}\Omega_{j}; \mathcal{E}\big|_{T^{*}\Omega_{j}})$  and the norm $\|u\|_{\tilde{\mathcal{W}}^{k}}$ is equivalent to $\max_{1\leq j\leq J}\|\varrho_{j}u\|_{\tilde{\mathcal{W}}^{k}}$\,.
\item[2)] For two different connections $\nabla^{E,1}$ and $\nabla^{E,2}$\,, the two spaces $\tilde{\mathcal{W}}_{\nabla^{1}}^{k}(X;\mathcal{E})$ and $\tilde{\mathcal{W}}^{k}_{\nabla^{2}}(X;\mathcal{E})$ are equal and the norms $\|~\|_{\tilde{\mathcal{W}}^{k}_{\nabla^{1}}}$ and $\|~\|_{\tilde{\mathcal{W}}^{k}_{\nabla^{2}}}$ are equivalent.
  \end{description}
\end{proposition}
\begin{proof}
\noindent\textbf{1)} Simply apply Proposition~\ref{pr:contTheta} with $\theta\in \mathcal{C}^{\infty}(Q;\mathbb{R})$\,. The definition of $\tilde{\mathcal{W}}^{k}_{\Omega_{j}-\mathrm{comp}}(T^{*}\Omega_{j};\mathcal{E}\big|_{T^{*}\Omega_{j}})$ and the other statements are explained in Appendix~\ref{sec:Qloccomp}. Simply use the triangular inequality for $\|u\|_{\tilde{\mathcal{W}}^{k}}\leq J\max_{1\leq j\leq J}\|\varrho_{j}u\|_{\tilde{\mathcal{W}^{k}}}$\,.\\

\noindent\textbf{2)} Remember $\nabla_{T}^{E,2}-\nabla_{T}^{E,1}=R(T)\in \mathrm{End}(E)$ with $\|R(T)\|_{\mathcal{C}^{k}}\leq C_{k}\|T\|_{C^{k+1}}$ for any $T\in \mathcal{C}^{\infty}(Q;TQ)$\,. Let $P_{k,N_{1},N_{2},N_{3}}^{j}$ be the norms involved in \eqref{eq:normeNkW}\eqref{eq:redN1} for the two associated connections $\nabla^{\mathcal{E},\ell}$ for $\ell=1,2$\,.
We can make an induction proof with respect to $N_{1}\in \left\{0,\ldots,k\right\}$ like in Proposition~\ref{pr:contTheta} after reducing the problem to $N_{2}=N_{3}=0$ by noticing
$$
\forall T^{V}\in \mathcal{C}^{\infty}_{Q}(X;TX^{V})\,,\quad
\nabla_{T^{V}}^{\mathcal{E},2}=\nabla_{T^{V}}^{\mathcal{E},1}
$$
and by using
$$
\nabla_{T^{H}}^{2}-\nabla_{T^{H}}^{1}=\pi_{X}^{*}(R(T))
$$
for any $T^{H}\in \mathcal{C}^{\infty}_{Q}(X;TX^{H})$ with $\pi_{X,*}T^{H}=T\in \mathcal{C}^{\infty}(Q;TQ)$\,.\\
Actually the induction proof relies on
\begin{align*}
  P^{2}_{k-1,N_{1},0,0}(\nabla_{T_{N_{1}+1}^{H}}^{\mathcal{E},2}v)&\leq P^{2}_{k-1,N_{1},0,0}(\nabla_{T_{N_{1}+1}^{H}}^{\mathcal{E},1}v)+
                                                                    P^{2}_{k-1,N_{1},0,0}([\pi^{*}(R(T_{N_{1}+1}^{H})]v)\\
  &\underbrace{\leq}_{\text{induction}}C_{k}\left[P^{1}_{k-1,N_{1},0,0}(\nabla_{T_{N_{1}+1}^{H}}^{\mathcal{E},1}v)+
    P^{1}_{k-1,N_{1},0,0}([\pi^{*}(R(T_{N_{1}+1}^{H})]v)\right]\\
&\underbrace{\leq}_{\text{Prop.~\ref{pr:contTheta}}}
   C_{k}'\|T_{N_{1}+1}^{H}\|_{\mathcal{C}^{k}}
   P^{1}_{k,N_{1}+1,0,0}(v)\,,
\end{align*}
and leads to $\|u\|_{\tilde{\mathcal{W}}^{k}_{\nabla^{2}}}\leq C_{k}''\|u\|_{\tilde{\mathcal{W}^{k}}_{\nabla^{1}}}$\,. The result follows by symmetry.
\end{proof}
The atlas $(\Omega_{j})_{1\leq j\leq J}$ can be chosen such that  $E\big|_{\Omega_{j}}$ is trivial with a local frame $(f_{j}^{1},\ldots,f_{j}^{N})$\,, $f_{j}^{n}\in \mathcal{C}^{\infty}(\Omega_{j};E\big|_{\Omega_{j}})$\,. Above $\Omega_{j}$ the connection $\nabla^{E}$ can be replaced by the trivial connection $\nabla^{j,E}$ given by
$$
\forall T\in \mathcal{C}^{\infty}(Q;TQ)\,,\quad \nabla^{j}_{T}(\sum_{n=1}^{N}v_{n}(q)f_{j}^{n})=\sum_{n=1}^{N}(Tv_{n})f_{j}^{n}\,.
$$
For a section $u=\sum_{n=1}^{N}u_{n,j}(x)f_{j}^{n}$ of $\mathcal{E}\big|_{T^{*}\Omega_{j}}$ we get
$$
\nabla^{j,\mathcal{E}}_{T_{1}^{H}}\ldots \nabla^{j,\mathcal{E}}_{T_{N_{1}}^{H}}
\nabla^{j,\mathcal{E}}_{T_{1}^{V}}\ldots \nabla^{j,\mathcal{E}}_{T_{N_{2}}^{V}}u=
\sum_{n=1}^{N}(T_{1}^{H}\ldots T_{N_{1}}^{H}T_{1}^{V}\ldots T_{N_{2}}^{V}v_{n,j})f^{n}_{j}\,.
$$
Proposition~\ref{pr:Wkred} now implies that the $\tilde{\mathcal{W}^{k}}$-norm of 
$$
u=\sum_{j=1}^{J}\varrho_{j}(q)u=\sum_{j=1}^{J}\sum_{n=1}^{N}u_{n,j}f^{n}_{j}
$$
with $u_{n,j}\in \tilde{\mathcal{W}}^{k}_{\Omega_{j}-\mathrm{comp}}(T^{*}\Omega_{j};\mathbb{C})\subset \tilde{\mathcal{W}}^{k}(X;\mathbb{C})$\,, is equivalent to
$$
\max_{\substack{1\leq j\leq J\\ 1\leq n\leq N}}\|u_{n,j}\|_{\tilde{\mathcal{W}}^{k}(X;\mathbb{C})}\,.
$$
The  $\tilde{\mathcal{W}}^{k}$-spaces for $k\in \mathbb{N}$ and their norms is are thus fully understood in a local scalar setting.

\subsection{Pseudo-differential definition of $\tilde{\mathcal{W}^{s}}$\,, $s\in \mathbb{R}$}
\label{sec:pseudWs}

The end of Subsection~\ref{sec:firstWk} reduced the description of $\tilde{\mathcal{W}}^{k}(X;\mathcal{E})$ to the local description of $\tilde{\mathcal{W}}^{k}(X;\mathbb{C})$\,. We can thus focus on scalar sections, and we now give a pseudo-differential and a global characterization.\\
We need the pseudo-differential calculus in $\mathop{\bigcup}_{m\in \mathbb{R}}\mathrm{OpS}^{m}_{\Psi}(Q;\mathbb{C})$ introduced in Appendix~\ref{sec:pseudodiff}.
We recall that $a\in S^{m}_{\Psi}(Q;\mathbb{C})$ if in doubly canonical coordinates $(q,p,\xi,\eta)$ in $T^{*}(T^{*}\Omega)$ associated with the local coordinates  $q=(q^{1},\ldots,q^{d})$ on a chart open set $\Omega\subset Q$\,, the uniform estimate
$$
|\partial_{q}^{\alpha}\partial_{p}^{\beta}\partial_{\xi}^{\gamma}\partial_{\eta}^{\delta}a(q,p,\xi,\eta)|\leq C_{\alpha,\beta,\gamma,\delta}(1+|\xi|^{2}+|p|^{4}+|\eta|^{4})^{\frac{m-|\gamma|-\frac{|\beta|+|\delta|}{2}}{2}}
$$
and that, the quantization is the standard one,given by the local kernel on $T^{*}\Omega\times T^{*}\Omega$:
$$
[a(q,p,D_{q},D_{p})](q,p,q',p')=\int_{\mathbb{R}^{2d}}e^{i[(q-q').\xi+(p-p').\eta]}a(q,p,\xi,\eta)~\frac{d\xi d\eta}{(2\pi)^{2d}}\,.
$$
Actually the general quantization of $a\in S^{m}_{\Psi}(Q;\mathbb{C})$ is defined by introducing a partition of unity $\sum_{j=1}^{J}\hat{\varrho}_{j}(q)\equiv 1$ and cut-off functions $\hat\chi\in \mathcal{C}^{\infty}_{0}(\Omega_{j};[0,1])$\,, $\hat\chi\equiv 1$ on $\mathrm{supp}\,\hat{\varrho}_{j}$\,, and  by setting $a_{\hat{\varrho},\hat{\chi}}(x,D_{x})=\sum_{j=1}^{J}(\hat{\varrho}_{j}(q)a)(x,D_{x})\circ \hat{\chi}_{j}(q)$\,, and the set of pseudo-differential operators is
$$
\mathrm{OpS}^{m}_{\Psi}=\left\{a_{\hat\varrho,\hat\chi}(x,D_{x})+R\,, a\in S^{m}_{\Psi}(Q;\mathbb{C})\,, R\in \mathcal{R}(Q;\mathbb{C})\right\}\,,
$$
with $\mathcal{R}(Q;\mathbb{C})=\mathcal{L}(\mathcal{S}'(T^{*}Q;\mathbb{C});\mathcal{S}(T^{*}Q;\mathbb{C}))\,.
$
It is proved in Appendix~\ref{sec:pseudodiff} that this pseudo-differential calculus has the same properties as the usual pseudo-differential calculus, with a different homogeneity which takes into account the global estimates as $p\to \infty$\,.\\
In canonical coordinates $(q,p)$ associated with the local coordinates
$(q^{1},\ldots,q^{d})$ on $\Omega$ we know the two frames $(e_{1},\ldots,e_{d})$ (resp. $(\hat{e}^{1},\ldots,\hat{e}^{d})$) of $T(T^{*}\Omega)^{H}$ (resp. $T(T^{*}\Omega)^{V}$) given by
\begin{eqnarray*}
  &&e_{i}=\frac{\partial}{\partial q^{i}}+\Gamma_{i\ell}^{k}(q)p_{k}\frac{\partial}{\partial p_{\ell}}\\
  \text{resp.}&&
                 \hat{e}^{i}=\frac{\partial}{\partial p_{i}}\,.
\end{eqnarray*}
As differential operators, the locally defined operators
$e_{i}$\,, $\partial_{q^{i}}$\,, $p_{k}\partial_{p_{\ell}}$\,,  $\mathcal{O}=\frac{-g_{ij}(q)\partial_{p_{i}}\partial_{p_{j}}+g^{ij}(q)p_{i}p_{j}}{2}$ belong to $\mathrm{OpS}^{1}_{\Psi, \Omega-\mathrm{loc}}(\Omega;\mathbb{C})$\,, while $p_{k}\times$\,, $\langle p\rangle_{q}\times$ and $\partial_{p_{k}}$ belong to $\mathrm{OpS}^{1/2}_{\Psi,\Omega-\mathrm{loc}}(\Omega;\mathbb{C})$\,.\\
Thus any $T^{H}\in \mathcal{C}^{\infty}_{Q}(X;TX^{H})$ is a differential operator that belongs  to $\mathrm{OpS}_{\Psi}^{1}(Q;\mathbb{C})$\,, and any $T^{V}\in \mathcal{C}^{\infty}_{Q}(X;TX^{V})$ belongs to $\mathrm{OpS}^{1/2}_{\Psi}(Q;\mathbb{C})$\,.\\
Let us introduce another operator which involves the scalar horizontal Laplacian $\Delta_{H}$\,.
We follow \cite{BeBo}: On $X$ with the decomposition $TX=TX^{H}\oplus TX^{V}\sim TQ\oplus T^{*}Q$ given by the Levi-Civita connection associated with $g$\, we put the riemannian metric $g\oplus^{\perp}g^{-1}$ and consider the associated total Laplacian $\Delta_{x}$\,. The projection $\pi_{X}:X=T^{*}Q\to Q$ is now a riemannian submersion with totally geodesic fibers and the horizontal Laplacian $\Delta_{H}=\Delta_{x}-\Delta_{p}$ equals in local canonical coordinates
$$
\Delta_{H}=g^{ij}(q)(e_{i}e_{j}-\Gamma_{ij}^{k}(q)e_{k})\,.
$$
Because the volume of $g\oplus^{\perp} g^{-1}$ is equal to the symplectic volume $dqdp$ and $\Delta_{x}$ and $\Delta_{p}$ are symmetric on $\mathcal{S}(T^{*}Q;\mathbb{C})$\,, the operator $\Delta_{H}$ is symmetric on $\mathcal{S}(T^{*}Q;\mathbb{C})$ for the $L^{2}(T^{*}Q,dqdp;\mathbb{C})$ scalar product. By introducing the adjoint differential operator $e_{i}^{*}=-\partial_{q}^{i}-\Gamma_{ii'}^{k}p_{k}\partial_{p_{i'}}-\Gamma_{ii'}^{'}$ and owing to the symmetry or by explicit computations with
$$
\partial_{q^{i}}g^{ij}=\partial_{q^{i}}g^{-1}(dq^{i},dq^{j})
=g^{-1}(-\Gamma_{i,k}^{i}dq^{k},dq^{j})+g^{-1}(dq^{i}, -\Gamma_{i,k}^{j}dq^{k})
=-\Gamma_{ik}^{i}g^{kj}-\Gamma_{ik}^{j}g^{ik}\,,
$$
the horizontal Laplacian is also given by
$$
-\Delta_{H}=e_{i}^{*}g^{ij}(q)e_{j}\,,
$$
without a divergence term because integrations are made with respect to the symplectic volume $dqdp$\,.
\begin{definition}
\label{de:W2} The operator $W^{2}$ is the closure in $L^{2}(X,dqdp;\mathbb{C})$ of the differential operator $C_{g}-\Delta_{H}+C_{g}\mathcal{O}^{2}:\mathcal{S}(X;\mathbb{C})\to \mathcal{S}(X;\mathbb{C})\subset L^{2}(X,dqdp;\mathbb{C})$ for $C_{g}\geq 1$ large enough.
\end{definition}
Notice that because the flow $\exp(te_{i})$ sends isometrically $T^{*}_{q}Q$ to $T^{*}_{\exp(t\partial_{q^{i}})q}Q$\,, the commutations
$$
[e_{i},-\Delta_{p}]=[e_{i},|p|_{q}^{2}]=[e_{i},\mathcal{O}]=[\Delta_{H},\mathcal{O}]=[W^{2},\mathcal{O}]=0
$$
hold true on $\mathcal{S}_{\Omega-\mathrm{loc}}(\Omega;\mathbb{C})$\,.\\
As a consequence of Appendix~\ref{sec:pseudodiff} we have a simple characterization of $\tilde{\mathcal{W}}^{s}(X;\mathbb{C})$\,, in the case when $E=Q\otimes \mathbb{C}$\,. The general case can then be deduced either by the localization at the end of the previous paragraph or by the approach proposed afterwards. Both are equivalent.
\begin{proposition}
  \label{pr:characWs}
  For any $s\in \mathbb{R}$\,, the space $\tilde{\mathcal{W}}^{s}(X;\mathbb{C})$ is characterized by
$$
\tilde{\mathcal{W}}^{s}(X;\mathbb{C})=\left\{u\in \mathcal{S}'(X;\mathbb{C})\,, \forall A\in \mathrm{OpS}^{s}_{\Psi}(Q;\mathbb{C})\,,\; Au\in L^{2}(X,dqdp;\mathbb{C})\right\}
$$
For $C_{g}\geq 1$ large enough, $C_{g}-\Delta_{H}+C_{g}\mathcal{O}^{2}:\mathcal{S}(X;\mathbb{C})\to L^{2}(X,dqdp;\mathbb{C})$ is a non~negative essentially self-adjoint operator, with self-adjoint extension $W^{2}$ and $D(W^{2})=\tilde{\mathcal{W}}^{2}(X;\mathbb{C})$\,.\\
For any $s\in \mathbb{R}$\,, $W^{s}=(W^{2})^{s/2}$ is an elliptic operator in $\mathrm{OpS}^{s}_{\Psi}(Q;\mathbb{C})$ and the norm on $\tilde{\mathcal{W}}^{s}(X;\mathbb{C})$ can be chosen as $\|W^{s}u\|_{L^{2}(X,dqdp;\mathbb{C})}$\,.
\end{proposition}

\begin{proof}[Proof]
  \noindent\textbf{a)} If we start from the above definition of $\tilde{\mathcal{W}}^{s}(X;\mathbb{C})$ the problem is reduced to the ellipticity and the identification of the principal symbol of $W^{2}=C_{g}-\Delta_{H}+C_{g}\mathcal{O}^{2}$\,. Because it is a differential operator $W^2=\sum_{j=1}^{J}\hat{\varrho}_{j}(q)W^{2}\hat{\chi}_{j}(q)$\,, and we obtain $W^{2}=a_{\hat{\varrho},\hat{\chi}}(x,D_{x})$  with $a-a_{2}\in S^{1}_{\Psi}(Q;\mathbb{C})$ and
  $$
a_{2}(x,\Xi)=C_{g}+|\xi+\Gamma_{..}^{k}(q)p_{k}\eta|_{q}^{2}+\frac{C_{g}}{4}(|p|_{q}^{2}+|\eta|_{q}^{2})^{2}\geq C_{g}+|\xi|_{q}^{2}-2|\Gamma_{..}^{k}(q)p_{k}\eta|_{q}^2+\frac{C_{g}}{4}(|p|_{q}^{2}+|\eta|_{q}^{2})^{2}
$$
We used the notation $|\tau|_{q}^{2}=g^{ij}(q)\tau_{i}\tau_{j}$ for $\tau=\xi+\Gamma_{..}^{k}(q)p_{k}\eta$\,, $\tau=p$ and $|\eta|_{q}^{2}=g_{ij}(q)\eta^{i}\eta^{j}$\,.
The ellipticity comes from
$$
a_{2}(x,\Xi)\geq C_{g}+\varepsilon_{g}|\xi|^{2}-\frac{1}{\varepsilon_{g}}|p|^{2}|\eta|^{2}+\frac{C_{g}\varepsilon_{g}}{4}(|p|^{4}+|\eta|^{4})
$$
for some $\varepsilon_{g}>0$ given by $g$ and the fixed open covering $\mathop{\bigcup_{j=1}^{J}}\Omega_{j}$\,.
The ellipticity $a_{2}\geq C_{g}+\varepsilon_{g}(|\xi|^{2}+|p|^{4}+|\eta|^{4})$ holds true if  $C_{g}\varepsilon_{g}^{2}-1\geq 8\varepsilon_{g}^{2}$\,.\\
The operator $W^{2}$ is symmetric with
\begin{equation}
  \label{eq:positW2}
\langle u\,,\, W^{2}u\rangle
=C_{g}\|u\|_{L^{2}}^{2}+\sum_{j=1}^{J}\int_{T^{*}\Omega_{j}}g^{ii'}(q)\left[\overline{(e_{i}\theta_{j}(q)u)}(e_{i'}\theta_{j}(q)u)-(\partial_{q^{i}}\theta_{j})(\partial_{q^{i}}\theta_{j})|u|^{2}\right]~dqdp +C_{g}\|\mathcal{O}u\|_{L^{2}}^{2}
\end{equation}
for all $ u\in \mathcal{S}(X;\mathbb{C})$ when $\sum_{j=1}^{J}\theta_{j}^{2}(q)\equiv 1$ with $\theta_{j}\in \mathcal{C}^{\infty}_{0}(\Omega_{j};[0,1])$\,.
It is bounded from below by $1$ for $C_{g}>0$ large enough.\\ It suffices to apply Proposition~\ref{pr:HeSj} of Appendix~\ref{sec:pseudodiff}.\\

\noindent\textbf{b)} For the identification of the two definitions of $\tilde{\mathcal{W}}^{s}(X;\mathbb{C})$ and the equivalence of the norms, it suffices to consider the case $s=k\in \mathbb{N}$\,, because all the other cases will follow by interpolation and duality.
\\
We start from the Definition~\ref{de:tW} and Proposition~\ref{pr:Wkred}-\textbf{1)} which says that the $\tilde{\mathcal{W}}^{k}(X;\mathbb{C})$ norm of $u$ is equivalent to
$$
\max_{1\leq j\leq J}\|\varrho_{j}u\|_{\tilde{\mathcal{W}}^{k}}
$$
with $\sum_{j=1}^{J}\varrho_{j}(q)\equiv 1$\,, $\varrho_{j}\in \mathcal{C}^{\infty}_{0}(\Omega_{j};[0,1])$\,. Additionally the definition of $\tilde{\mathcal{W}}^{k}_{\Omega_{j}-\mathrm{comp}}(\Omega_{j};\mathbb{C})$ ensures that it is independent of the choice of a coordinate system $(q^{1},\ldots,q^{d})$ with equivalent norm for two different choices.
So let us work on $T^{*}\Omega=\Omega\times\mathbb{R}^{d}\subset \mathbb{R}_{q}^{d}\times \mathbb{R}_{p}^{d}$ and let us consider functions  $u\in L^{2}_{\Omega-\mathrm{comp}}(T^{*}\Omega;\mathbb{C})$ with a $\Omega$-support included in a fixed compact set $K\subset\subset \Omega$ (a neighborhood of $\mathrm{supp}\,\varrho$  with $\varrho=\varrho_{j}$ when $\Omega=\Omega_{j}$). Any vector field $T^{H}\in \mathcal{C}^{\infty}_{\Omega-\mathrm{comp}}(T^{*}\Omega^{d};TX^{H})$ (resp. $T^{V}\in \mathcal{C}^{\infty}_{\Omega-\mathrm{comp}}(T^{*}\mathbb{\Omega}^{d};TX^{V})$) can be written
$$
T^{H}=\sum_{i=1}^{d}t^{i}(q)e_{i}\quad\text{resp.}\quad T^{V}=\sum_{i=1}^{d}t_{i}(q)\hat{e}^{i} 
$$
with $\max_{1\leq i\leq d}\|t^{i}\|_{\mathcal{C}^{k}} \asymp \|T^{H}\|_{\mathcal{C}^{k}}$ (resp. $\max_{1\leq i\leq d}\|t_{i}\|_{\mathcal{C}^{k}}\asymp \|T^{V}\|_{\mathcal{C}^{k}}$)\,.\\

We deduce
$$
  \frac{\|\langle p\rangle_{q}^{N_{3}}T_{1}^{H}\ldots T_{N_{1}}^{H}T_{1}^{V}\ldots T_{N_{2}}^{V}u\|^{2}_{L^{2}}}{\prod_{n_{1}=1}^{N_{1}}\|T_{n_{1}}^{H}\|_{\mathcal{C}^{k}}\prod_{n_{2}=1}^{N_{2}}\|T_{n_{2}}^{V}\|_{\mathcal{C}^{k}}}\leq C_{K,k}\max_{\substack{1\leq i_{1},\ldots,i_{N_{1}}\leq d\\ 1\leq j_{1},\ldots,j_{N_{2}}\leq d}}\|\langle p\rangle_{q}^{N_{3}}e_{i_{1}}\ldots e_{i_{N_{1}}}\hat{e}^{j_{1}}\ldots \hat{e}^{j_{N_{2}}}u\|_{L^{2}}
  $$
where
$$
\|\langle p\rangle_{q}^{N_{3}}e_{i_{1}}\ldots e_{i_{N_{1}}} \hat{e}^{j_{1}}\ldots \hat{e}^{j_{N_{2}}}u\|_{L^{2}}
=\|\langle p\rangle_{q}^{N_{3}}(\chi e_{i_{1}})\ldots (\chi e_{i_{N_{1}}})(\chi \hat{e}^{j_{1}})\ldots (\chi \hat{e}^{j_{N_{2}}})u\|_{L^{2}}
$$
for some  $\chi=\chi(q)\in \mathcal{C}^{\infty}_{0}(\Omega;[0,1])$ such that $\chi\equiv 1$ in a neighborhood of $K\supset \Omega-\mathrm{supp}\,u$\,. Because $\chi(q)e_{i}\in \mathcal{C}^{\infty}_{\Omega-\mathrm{comp}}(T^{*}\Omega; TX^{H})$ and $\chi(q)\hat{e}^{j}\in \mathcal{C}^{\infty}_{\Omega-\mathrm{comp}}(T^{*}\Omega; TX^{V})$ the right-hand side is bounded by $C_{K,\chi,k}\|u\|_{\tilde{\mathcal{W}}^{k}}$\,. \\
By taking the supremum with respect to $N_{1}+\frac{N_{2}+N_{3}}{2}\leq k$\,, $T_{1}^{H},\ldots, T_{N_{1}}^{H}\in \mathcal{C}^{\infty}_{\Omega-\mathrm{comp}}(T^{*}\Omega;TX^{H})$ and $T_{1}^{V},\ldots, T_{N_{2}}^{V}\in \mathcal{C}^{\infty}_{\Omega-\mathrm{comp}}(T^{*}\Omega;TX^{V})$\,.
The  norm $\|u\|_{\tilde{\mathcal{W}}^{k}}$ for $u\in L^{2}(T^{*}\Omega;\mathbb{C})$ with $\Omega$-support included in $K$\,, is thus equivalent to
\begin{equation}
  \label{eq:normmax}
  \max_{\substack{1\leq i_{1},\ldots,i_{N_{1}}\leq d\\ 1\leq j_{1},\ldots,j_{N_{2}}\leq d\\ N_{1}+\frac{N_{2}+N_{3}}{2}\leq k}}\|\langle p\rangle_{q}^{N_{3}}e_{i_{1}}\ldots e_{i_{N_{1}}}\hat{e}^{j_{1}}\ldots \hat{e}^{j_{N_{2}}}u\|_{L^{2}}\asymp
  \max_{\substack{1\leq i_{1},\ldots,i_{N_{1}}\leq d\\ N_{1}+\frac{|\beta|+N_{3}}{2}\leq k}}
  \|\langle p\rangle^{N_{3}}e_{i_{1}}\ldots e_{i_{N_{1}}}\partial_{p}^{\beta}u\|_{L^{2}}
\end{equation}
where we have replaced $\langle p\rangle_{q}=(1+g^{ij}(q)p_{i}p_{j})^{1/2}$ by the equivalent quantity $\langle p\rangle=(1+\sum_{i}p_{i}^{2})^{1/2}$\,.\\
From
$$
[e_{i},f^{k}(q)p_{k}]=(\partial_{q^{i}}f^{k})(q)p_{k}+(f^{k}\Gamma_{ik}^{\ell})(q)p_{\ell}\,,
$$
we get by induction
$$
e_{i_{1}}\ldots e_{i_{N_{1}}}\partial_{p}^{\beta}-\partial_{q^{i_{1}}}\ldots\partial_{q^{i_{N_{1}}}}\partial_{p}^{\beta}=\sum_{\substack{|\alpha|\leq N_{1}-1\\ |\alpha|+\frac{|\gamma|+|\beta'|}{2}=N_{1}+\frac{|\beta|}{2}}}f_{\alpha,\beta',\gamma}(q)p^{\gamma}\partial_{q}^{\alpha}\partial_{p}^{\beta'}
$$ and we deduce
\begin{multline*}
\varepsilon^{N_{1}}\|\langle p\rangle^{N_{3}}e_{i_{1}}\ldots e_{i_{N_{1}}}\partial^{\beta}_{p}u\|_{L^{2}}-\varepsilon^{N_{1}}\|\langle p\rangle_{q}^{N_{3}}\partial_{q^{i_{1}}}\ldots \partial_{q^{i_{N_{1}}}}\partial^{\beta}_{p}u\|_{L^{2}}
\\
\leq C_{K,k}\varepsilon\max_{\substack{|\alpha|\leq N_{1}-1\\
  |\alpha|+\frac{N_{3}'+|\beta'|}{2}\leq N_{1}+\frac{|\beta|}{2}}}\varepsilon^{N_{1}-1}\|\langle p\rangle^{N_{3}'}\partial_{q}^{\alpha}\partial_{p}^{\beta'}u\|_{L^{2}}\,.
\end{multline*}
Choosing $\varepsilon=\varepsilon_{K,k}$ for $\varepsilon_{K,k}>0$ small enough implies that the norm $\|u\|_{\tilde{\mathcal{W}}^{k}}$ is equivalent to
$$
\max_{|\alpha|+\frac{N_{3}+|\beta|}{2}\leq k}\|\langle p\rangle^{N_{3}}\partial_{q}^{\alpha}\partial_{p}^{\beta}\chi(q)u\|_{L^{2}}
\quad\text{or}
\quad \sqrt{\sum_{|\alpha|+\frac{N_{3}+|\beta|}{2}\leq k}\|\langle p\rangle^{N_{3}}\partial_{q}^{\alpha}\partial_{p}^{\beta}\chi(q)u\|^{2}_{L^{2}}}\,.
$$
But according Appendix~\ref{sec:pseudodiff} and in particular Proposition~\ref{pr:eqnormes}\,, it is equivalent to the norm $\|W^{k}\chi(q)u\|_{L^{2}}$\,.
\end{proof}

Let us extend now this result to $\tilde{\mathcal{W}}^{s}(Q;\mathcal{E})$\,.
For a self-adjoint non negative scalar operator $A\in \mathrm{OpS}^{m}_{\Psi}(Q;\mathbb{C})$\,, like $W^{2}$ with $m=2$ or $W^{s}$\,, $s=m\in \mathbb{R}$\,, it is not possible to define directly its action on sections of $\mathcal{E}$\,. However a localization technique makes it possible, up to lower order corrections.\\
We fix, as we did in Appendix~\ref{sec:pseudodiff}, the atlas covering $Q=\mathop{\bigcup}_{j=1}^{J}\Omega_{j}$ by assuming that for every $j\in \left\{1,\ldots, J\right\}$ the two properties are satisfied:
\begin{itemize}
\item the open set $\tilde{\Omega}_{j}=\mathop{\bigcup_{\Omega_{j'}\cap \Omega_{j}\neq \emptyset}}\Omega_{j'}$ is a chart open set;
\item the restricted vector bundle $E\big|_{\tilde{\Omega}_{j}}$ admits an orthonormal frame $(f_{j}^{1},\ldots,f_{j}^{N})$ for the metric $g_{E}$\,.
\end{itemize}
If $\sum_{j=1}^{J}\theta_{j}^{2}(q)\equiv 1$ is a quadratic partition of unity with $\theta_{j}\in \mathcal{C}^{\infty}_{0}(\Omega_{j};[0,1])$ we set
$$
A_{\theta}=\sum_{j=1}^{J}\theta_{j}(q)\circ A_{\mathrm{sc}\,j}\circ \theta_{j}(q)\,,
$$
where $A_{\mathrm{sc},j}$ is the scalar pseudo-differential operator in the othonormal local frame $(f_{j}^{1},\ldots,f_{j}^{N})$ above $\tilde{\Omega}_{j}$:
$$
A_{\mathrm{sc,j}}(u_{k}f_{j}^{k}(q))(q,p)
=[Au_{k}](q,p)f_{j}^{k}(q)\,.
$$
When $A=a_{\varrho,\chi}(x,D_{x})+R=\sum_{j_{1}=1}^{J}(\varrho_{j_{1}}(q)a)(x,D_{x})\circ\chi_{j_{1}}(q)+R\in \mathrm{OpS}^{m}_{\Psi}(Q;\mathbb{C})$ we obtain
$$
A_{\theta}=\sum_{j_{1}=1}^{J}\varrho_{j_{1}}(q)\sum_{\Omega_{j}\cap \Omega_{j_{1}}\neq\emptyset }\theta_{j}(q)\circ [a(x,D_{x})]_{\mathrm{sc},j}\circ \theta_{j}(q)\circ \chi_{j_{1}}(q)+R_{\theta}\,.
$$
If $U_{j_{1},j_{2}}(q)$ is the unitary matrix of $(f_{j_{1}}^{1},\ldots, f_{j_{1}}^{N})$ in the frame $(f_{j_{2}}^{1},\ldots, f_{j_{2}}^{N})$\,, the operator
$$
\sum_{\Omega_{j}\cap \Omega_{j_{1}}\neq\emptyset}
\theta_{j}(q)\circ [a(x,D_{x})]_{\mathrm{sc},j}\circ \theta_{j}(q)
$$
with $\mathop{\bigcup}_{\Omega_{j}\cap \Omega_{j_{1}}\neq \emptyset}\Omega_{j}\subset \tilde{\Omega}_{j_{1}}\subset \mathbb{R}^{d}$\,, is nothing but the operator $\tilde{a}_{j_{1}}(x,D_{x})$ with $a_{j_{1}}\in S(\Psi^{m},g_{\Psi};\mathbb{C}^{N})$ given by
$$
\sum_{\Omega_{j}\cap \Omega_{j_{1}}\neq \emptyset} U_{j_{1},j}(q)\sharp ([(\theta_{j}(q)a)\sharp \theta_{j}(q)]\otimes \mathrm{Id}_{\mathbb{C}^{N}})\sharp U_{j,j_{1}}(q)\,,
$$
where we recall $(a\sharp b)(x,D_{x})=a(x,D_{x})\circ b(x,D_{x})$\,.\\
Owing to the exact chain rules $U_{j_{3},j_{1}}(q)=U_{j_{3},j_{2}}\circ U_{j_{2},j_{1}}(q)$ and $U_{j_{3},j_{1}}(q)=U_{j_{3},j_{2}}(q)\sharp U_{j_{2},j_{1}}(q)$ and the exact commutation $\theta_{j_{3}}(q)\sharp U_{j_{2},j_{1}}(q)=U_{j_{2},j_{1}}(q)\sharp \theta_{j_{3}}(q)$\,, we can write
$$
A_{\theta}=\sum_{j_{1}=1}^{J}(\varrho_{j_{1}}(q)a_{\theta})(x,D_{x})\circ \chi_{j_{1}}(q)+R_{\theta}=(a_{\theta})_{\varrho,\chi}(x,D_{x})+R_{\theta}
$$
with $a_{\theta}\in S^{m}_{\Psi}(Q;\mathrm{End}\,,\mathcal{E})$ and $R_{\theta}\in \mathcal{R}(Q;\mathcal{\mathcal{E}})$\,.\\
Additionally,  if $A=a_{m,\varrho,\chi}(x,D_{x})+A_{m-1}$\,, $A_{m-1}\in \mathrm{OpS}^{m-1}_{\Psi}(Q;\mathbb{C})$\,, then $A_{\theta}=(a_{m}\otimes \mathrm{Id}_{\mathrm{\mathcal{E}}})_{\varrho,\chi}(x,D_{x})+A_{\theta,m-1}$ with $A_{\theta,m-1}\in \mathrm{OpS}^{m-1}_{\Psi}(Q;\mathrm{End}\,\mathcal{E})$\,. In particular the principal symbol does not depend on the chosen orthonormal frames $(f_{j}^{1},\ldots,f_{j}^{N})$\, $j=1,\ldots,J$\,.\\
If $A$ is self-adjoint and elliptic, the same holds for $A_{\theta}$ with
\begin{eqnarray*}
  D(A_{\theta})=\tilde{\mathcal{W}}^{m}(X;\mathcal{E})&=&\left\{u\in \mathcal{S}'(X;\mathcal{E})\,,\forall j\in \left\{1,\ldots,J\right\}\,,\quad u\big|_{\Omega_{j}}=u_{k}(x)f_{j}^{k}(q)\,, u_{k}\in \tilde{\mathcal{W}}^{m}_{\Omega_{j}-\mathrm{loc}}(T^{*}\Omega_{j};\mathbb{C})\right\}\\
  &=& \left\{u\in \mathcal{S}'(X;\mathcal{E})\,,\, \forall B\in \mathrm{OpS}^{m}_{\Psi}(Q;\mathrm{End}\,\mathcal{E})\,,\, Bu\in L^{2}(X,dqdp;\mathcal{E})\right\}\,.
\end{eqnarray*}
We can conclude with the following summary.
\begin{proposition}
  \label{pr:identWsE}
   Once the quadratic partition of unity $\sum_{j=1}^{J}\theta_{j}^{2}(q)\equiv 1$ and the local orthonormal frames $(f_{j}^{1},\ldots,f_{j}^{N})$\,, $1\leq j\leq J$, are fixed and for $C_{g}\geq C_{g,\theta}$ chosen large enough, the operator 
   $$W^{2}_{\theta}=\sum_{j=1}^{J}\theta_{j}(q)(W^{2})_{\mathrm{sc},j}\theta_{j}(q)\quad\text{with}\quad D(W^{2}_{\theta})=\left\{s\in L^{2}(X,dqdp;\mathcal{E})\,,\, W^{2}_{\theta}s\in L^{2}(X,dqdp;\mathcal{E})\right\}
   $$ is self-adjoint and bounded from below by $1$\,.\\
  For $s\in \mathbb{R}$\,, the space $\tilde{\mathcal{W}}^{s}(X;\mathcal{E})$ introduced in Definition~\ref{de:tW} for $s=k\in \mathbb{N}$ and then extended by interpolation and duality, equals
  \begin{eqnarray*}
    \tilde{\mathcal{W}}^{s}(X;\mathcal{E}) &=& \left\{u\in \mathcal{S}'(X;\mathcal{E})\,,\, \forall B\in \mathrm{OpS}^{s}_{\Psi}(Q;\mathrm{End}\,\mathcal{E})\,,\, Bu\in L^{2}(X,dqdp;\mathcal{E})\right\}\\
             &=&
                 \left\{u\in \mathcal{S}'(X;\mathcal{E})\,,\quad (W^{2}_{\theta})^{s/2}u\in L^{2}(X,dqdp;\mathcal{E})\right\}\\
                                           &=&
 \left\{u\in \mathcal{S}'(X;\mathcal{E})\,,\,\forall j\in \left\{1,\ldots,J\right\}\,, u\big|_{\Omega_{j}}=u_{k}(x)f_{j}^{k}(q)\,,\, W^{s}u_{k}\in L^{2}_{\Omega_{j}-\mathrm{loc}}(T^{*}\Omega_{j},dqdp;\mathcal{E})\right\}
    \\                               
    &=& \left\{u\in \mathcal{S}'(X;\mathcal{E})\,,\quad (C_{s}+(W^{|s|})_{\theta})^{\mathrm{sign}\,s}u\in L^{2}(X,dqdp;\mathcal{E})\right\}\,,
  \end{eqnarray*}
where the constant $C_{s}>0$ is chosen large enough.
\end{proposition}
\begin{proof}
  We already know that $W^{2}=C_{g}-\Delta_{H}+C_{g}\mathcal{O}^{2}$ is elliptic, self-adjoint and bounded from below  by $1$ for $C_{g}\geq 1$ large enough with domain $D(W^{2})=\tilde{\mathcal{W}}^{2}(X;\mathbb{C})$\,.\\
  With the previous discussion this proves that $W^{2}_{\theta}$ is elliptic and self-adjoint.
  The same computation as \eqref{eq:positW2} shows that $W^{2}_{\theta}\geq 1$\,: Actually the derivatives of the unitary matrix associated with the change of frames, $\partial_{q^{i}}U_{j_{1},j_{2}}(q)$\,, bring lower order terms which are absorbed if $C_{g}=C_{g,\theta}$ is chosen large enough.\\
  For $W^{2}_{\theta}$ with a scalar principal symbol $a_{2}\otimes\mathrm{Id}_{\mathcal{E}}\geq \frac{1}{\kappa}\Psi^{2}\otimes \mathrm{Id}_{\mathcal{E}}$, Proposition~\ref{pr:HeSj} applies and $(W^{2}_{\theta})^{s/2}=f_{s}(W^{2}_{\theta})$ with $f_{s}\in S(\langle t\rangle^{s/2},\frac{dt}{\langle t\rangle^{2}})$ is elliptic with the principal symbol $f_{s}(a_{2})\otimes\mathrm{Id}_{\mathcal{E}}=a_{2}^{s/2}\otimes \mathrm{Id}_{\mathcal{E}}$\,.\\
The local characterization with $u\big|_{\Omega_{j}}=u_{k}(x)f_{j}^{k}(q)$ has been explained and with the reduction of the previous paragraph and Proposition~\ref{pr:characWs} it shows that $D((W^{2}_{\theta})^{k/2})$ coincides with $\tilde{\mathcal{W}}^{k}(X;\mathcal{E})$ when $k\in \mathbb{N}$\,. This ends the identifications of the general spaces $\tilde{\mathcal{W}}^{s}(X;\mathcal{E})$ for $s\in \mathbb{R}$\,.\\
Because $W^{|s|}=(W^{2})^{|s|/2}=f_{|s|}(W^{2})$ is elliptic with the principal symbol $a_{2}^{|s|/2}$ for $s\neq 0$\,, $(W^{|s|})_{\theta}$ is elliptic with the principal symbol $a_{2}^{|s|/2}\otimes \mathrm{Id}_{\mathcal{E}}$\,. It is self-adjoint with the same domain, $\tilde{\mathcal{W}}^{|s|}(X;\mathcal{E})$\,, as $(W^{2}_{\theta})^{|s|/2}$\,. It is bounded from below by Garding inequality. Adding a constant $C_{s}$ ensures that $(C_{s}+(W^{|s|})_{\theta})$ is bounded from below by $1$ and invertible.
\end{proof}

\subsection{Spaces $\tilde{\mathcal{W}}^{s_{1},s_{2}}(X;\mathcal{E})$}
A priori $\mathcal{Y}=g^{ij}(q)p_{i}e_{j}$ belongs to $\mathrm{OpS}_{\Psi}^{3/2}(Q;\mathbb{C})$ but it has locally some specific structure made of $e_{i}\in \mathrm{OpS}_{\Psi}^{1}(\Omega;\mathbb{C})$ and followed by a multiplication by $p_{i}$\,.
We start with a simple commutation result.
\begin{proposition}
\label{pr:commut}
  The self-adjoint operator $(W^{2}_{\theta},D(W^{2}_{\theta})=\tilde{\mathcal{W}}^{2}(X;\mathcal{E}))$ modelled on $W^{2}=C_{g}-\Delta_{H}+C_{g}\mathcal{O}^{2}$ and introduced in Proposition~\ref{pr:identWsE} and the vertical harmonic oscillator $\mathcal{O}$ with the maximal domain
  $D(\mathcal{O})=\left\{u\in L^{2}(X,dqdp;\mathcal{E})\,,\quad \mathcal{O}u\in L^{2}(X,dqdp;\mathbb{C})\right\}$ make a pair of strongly commuting self-adjoint operators:
  For any Borel functions $f,g:\mathbb{R}\to \mathbb{C}$\,, $f(W^{2}_{\theta})g(\mathcal{O})=g(\mathcal{O})f(W^{2}_{\theta})$ on the intersection of their domain.
\end{proposition}
\begin{proof}
  The space $L^{2}(X,dqdp;\mathcal{E})$ is isomorphic to the direct integral $\int_{Q}^{\oplus}L^{2}(\mathbb{R}^{d},dp)\,d\mathrm{vol}_{g}(q)$ after the pointwise gauge transformation $(q,p)\mapsto (q,g(q)^{-1/2}.p)$\,. In this direct integral decomposition the operator $\mathcal{O}$ is nothing but $\int_{Q}^{\oplus}O~d\mathrm{vol}_{g}(q)$ where $O=\sum_{j=1}^{d}\frac{-\partial_{p_{j}}^{2}+p_{j}^{2}}{2}$ is the euclidean harmonic oscillator.\\
  The associated unitary group $e^{itO}$ satisfies $e^{itO}\partial_{q^{i}}e^{-itO}=\partial_{q^{i}}$\,, $e^{itO}p_{i}e^{-itO}=\cos(t)p_{i}-\sin(t)D_{p_{i}}$ and $e^{itO}D_{p_{i}}e^{-itO}=\sin(t)p_{i}+\cos(t)D_{p_{i}}$\,. We deduce that for any $t\in \mathbb{R}$\,, $e^{itO}$ is continuous from $\tilde{\mathcal{W}}^{2}(X;\mathbb{C})$ into itself, and therefore as a scalar operator from $\tilde{\mathcal{W}}^{2}(X;\mathcal{E})$ into itself.\\
  Because the unitary transform $U_{\Phi}:L^{2}(X;\mathcal{E})\to \int_{Q}^{\oplus}L^{2}(\mathbb{R}^{d})~d\mathrm{vol}_{g}(q)$ given by $(U_{\Phi}u)(q,p')=u(q,g(q)^{1/2}.p)$ is a special case of Proposition~\ref{pr:quantchang} (it suffices to consider locally the effect on the scalar components). It is an isomorphism of $\mathcal{W}^{2}(X;\mathcal{E})=D(W^{2}_{\theta})$ and $e^{it\mathcal{O}}$ is continuous from $D(W^{2}_{\theta})$ into itself for any $t\in \mathbb{R}$\,.\\
  Because the scalar operator $W^{2}$ and $\mathcal{O}$ commute on $\mathcal{S}(X;\mathbb{C})$ we deduce that $[W^{2}_{\theta},\mathcal{O}]=0$ on $\mathcal{S}(X;\mathcal{E})$ which is a core for $W^{2}_{\theta}$\,.\\
  We have all the ingredients of \cite{ABG} in order to conclude that
  $$
\mathrm{ad}_{\mathcal{O}}W^{2}_{\theta}=i\frac{d}{dt}e^{-it\mathcal{O}}W_{\theta}^{2}e^{it\mathcal{O}}\big|_{D(W^{2}_{\theta})}=0\,,
$$
and $W^{2}_{\theta}$ and $\mathcal{O}$ strongly commute.
\end{proof}
This leads to the introduction of the following, double indexed, spaces.
\begin{definition}
  \label{de:Ws1s2}
  For any $s_{1},s_{2}\in \mathbb{R}$\,, the space $\tilde{\mathcal{W}}^{s_{1},s_{2}}(X;\mathcal{E})$ is the space associated with the functional calculus of the two commuting self-adjoint operators $\mathcal{O}$ and $W^{2}_{\theta}$ and endowed with the Hilbert norm
$$
\|u\|_{\tilde{\mathcal{W}}^{s_{1},s_{2}}}=\|\mathcal{O}^{s_{1}/2}(W^{2}_{\theta})^{s_{2}/2}u\|_{L^{2}}\,.
$$
\end{definition}
In particular the space $\tilde{\mathcal{W}}^{1,s}(T^{*}Q;\mathcal{E})$ of Definition~\ref{de:tW} with the norm
$$
\|u\|_{\tilde{\mathcal{W}}^{1,s}}=\|\mathcal{O}^{1/2}u\|_{\tilde{\mathcal{W}}^{s}}=\|(W^{2}_{\theta})^{s/2}\mathcal{O}^{1/2}u\|_{L^{2}}
$$
is the particular case $s_{1}=1$\,, $s_{2}=s$\,. Clearly the spaces $\tilde{\mathcal{W}}^{s_{1},s_{2}}(X;\mathcal{E})$ contain a finer description of the regularity properties. With
$$
\|u\|_{\tilde{\mathcal{W}}^{s_{1},s_{2}}}^{2}=\langle (W^{2}_{\theta})^{s_{2}/2}\mathcal{O}^{(s_{1}-1)/2}u\,,\, \mathcal{O}(W^{2}_{\theta})^{s_{2}/2}\mathcal{O}^{(s_{1}-1)/2}u\rangle_{L^{2}}\leq \|(W^{2}_{\theta})^{(s_{2}+1/2)/2}\mathcal{O}^{(s_{1}-1)/2}u\|_{L^{2}}^{2}=\|u\|_{\tilde{\mathcal{W}}^{s_{2}+1/2,s_{1}-1}}^{2}\,.
$$
for $s_{1}\geq 1$\,, we deduce the continuous embeddings $\tilde{\mathcal{W}}^{0,s_{2}+s_{1}/2}(X;\mathcal{E})\subset \mathcal{W}^{s_{1},s_{2}}(X;\mathcal{E})$ for $s_{1}\geq 0$ and by duality
$\tilde{\mathcal{W}}^{s_{1},s_{2}}(X;\mathcal{E})\subset \tilde{\mathcal{W}}^{s_{1},s_{2}+s_{1}/2}(X;\mathcal{E})$ for $s_{1}<0$\,. We will essentially work with $s_{1}\in \left\{0,1\right\}$\,.\\
As a first order differential operator with respect to $q$ the operator $\nabla_{\mathcal{Y}}^{\mathcal{E}}$ can be written
$$
\nabla_{\mathcal{Y}}^{\mathcal{E}}=\sum_{j=1}^{J}\theta_{j}(q)\nabla_{\mathcal{Y}}^{\mathcal{E}}\theta_{j}(q)=\sum_{j=1}^{J}\theta_{j}(q)[g^{ii'}(q)p_{i'}\nabla_{e_{i'}}\big|_{T^{*}\tilde{\Omega}_{j}}]\theta_{j}(q)\,,
$$
where $g^{ii'}(q)p_{i}\nabla_{e_{i'}}^{\mathcal{E}}\big|_{\mathcal{T}^{*}\tilde{\Omega}_{j}}$ is expressed with the local coordinates in $\tilde{\Omega}_{j}$\,.\\
With the cut-off function $\tilde{\chi}_{j}\in \mathcal{C}^{\infty}_{0}(\tilde{\Omega}_{j};[0,1])$ such that $\tilde{\chi}_{j}\equiv 1$ in a neighborhood of $\mathrm{supp}\,\theta_{j'}(q)$ when $\Omega_{j}\cap \Omega_{j'}\neq \emptyset$\,, we can introduce the local scalar operator
\begin{eqnarray}
  \label{eq:hatpji}
  &&\hat{p}_{j,i}=\tilde{\chi}_{j}(q)p_{i}\otimes \mathrm{Id}_{\mathcal{E}}\quad,\quad \hat{D}_{j,i}=\tilde{\chi}_{j}(q)D_{p_{i}}\,,
  \\
  \label{eq:hatEji}  
  \text{while}
  &&\hat{E}_{j,i}=\theta_{j}(q)g^{ii'}(q)\nabla_{e_{i'}}^{\mathcal{E}}\theta_{j}(q)\in \mathrm{OpS}^{1}_{\Psi}(Q;\mathrm{End}\,\mathcal{E})\\
  \label{eq:hatEji12}
  \text{with}&&
                \hat{E}_{j,i}-(\theta_{j}(q)g^{ii'}(q)e_{i'}\theta_{j}(q)\otimes \mathrm{Id}_{\mathcal{E}})\in \mathrm{OpS}^{0}_{\Psi}(Q;\mathrm{End}\,\mathcal{E})\,.
\end{eqnarray}
We have in particular
$$
\nabla^{\mathcal{E}}_{\mathcal{Y}}=\sum_{j=1}^{J}\sum_{i=1}^{d}\hat{E}_{j,i}\circ \hat{p}_{j,i}\,.
$$
\begin{proposition}
  \label{pr:nablaY}
  Let $\hat{p}_{j,i}$\,, $\hat{D}_{j,i}$ and $\hat{E}_{j,i}$\,, $j\in \left\{1,\ldots,J\right\}$\,, $i\in \left\{1,\ldots,d\right\}$\,, be the operators defined by \eqref{eq:hatpji} and \eqref{eq:hatEji}\,.
  For any $s\in \mathbb{R}$ we have the estimates:
  \begin{itemize}
  \item $\|\hat{p}_{j,i}\|_{\mathcal{L}(\tilde{\mathcal{W}}^{1,s};\tilde{\mathcal{W}}^{0,s})}+\|\hat{D}_{j,i}\|_{\mathcal{L}(\tilde{\mathcal{W}}^{1,s};\tilde{\mathcal{W}}^{0,s})}\leq C_{g,s}$\,;
  \item $\|(W^{2}_{\theta})^{s/2}\hat{p}_{j,i}(W^{2}_{\theta})^{-s/2}-\hat{p}_{j,i}\|_{\mathcal{L}(\tilde{\mathcal{W}}^{1,0};\tilde{\mathcal{W}}^{0,1})}+\|(W^{2}_{\theta})^{s/2}\hat{D}_{j,i}(W^{2}_{\theta})^{-s/2}-\hat{D}_{j,i}\|_{\mathcal{L}(\tilde{\mathcal{W}}^{1,0};\tilde{\mathcal{W}}^{0,1})}\leq C_{g,s}$\,;
  \item $\|(W^{2}_{\theta})^{s/2}\nabla_{\mathcal{Y}}^{\mathcal{E}}(W^{2}_{\theta})^{-s/2}-\nabla_{\mathcal{Y}}^{\mathcal{E}}\|_{\mathcal{L}(\tilde{\mathcal{W}}^{1,0};L^{2})}\leq C_{g,s}$\,.
  \end{itemize}
\end{proposition}
\begin{proof}
  All the operators and commutators are well defined continuous operators onn the space of smooth rapidly decaying (w.r.t $p$) sections, $\mathcal{S}(X;\mathcal{E})$\,. The estimates are then extended by density.\\
  For $A=\hat{p}_{j,i}$ or $\hat{D}_{j,i}$ we know $A\in \mathrm{OpS}^{1/2}_{\Psi}(Q;\mathcal{E})$ while $A$ and $(W^{2}_{\theta})^{\pm s}\in \mathrm{OpS}^{\pm s}_{\Psi}(Q;\mathrm{End}\,\mathcal{E})$ have scalar principal symbols.
  We deduce
  $$
(W^{2}_{\theta})^{s}A(W^{2}_{\theta})^{-s}-A\in \mathrm{OpS}^{-1/2}_{\Psi}(Q;\mathrm{End}\,\mathcal{E})\subset \mathcal{L}(\tilde{\mathcal{W}}^{1,0}(X;\mathcal{E});L^{2}(X;\mathcal{E}))\,,
$$
and $\|A\|_{\mathcal{L}(\tilde{\mathcal{W}}^{1,s};\tilde{\mathcal{W}}^{0,s})}=\|(W^{2}_{\theta})^{s/2}A(W^{2}_{\theta})^{-s/2}\|_{\mathcal{L}(\tilde{\mathcal{W}}^{1,0};L^{2})}\leq C_{g,s}$\,.\\
For the second estimate we need a more accurate decomposition of $(W^{2}_{\theta})^{s/2}A(W^{2}_{\theta})^{-s/2}-A$\,. Let us write $A=a(q,p,D_{p})$ with the local coordinate writing, $a(q,p,\eta)=\tilde{\chi}_{j}(q)p_{i}$ when $A=\hat{p}_{j,i}$ and $a(q,p,\eta)=\tilde{\chi}_{j}(q)\eta_{i}$ when $A=\hat{D}_{j,i}$\,, and let $w(q,p,\xi,\eta)=(C_{g}+|\xi-\Gamma_{..}^{k}p_{k}\eta|_{g}^{2}+C_{g}/4(|p|_{g}^{2}+|\eta|_{g}^{2})^{2})^{1/2}$ be the principal scalar symbol of $W^{2}_{\theta}$\,. If we forget the tensor product with $\mathrm{Id}_{\mathcal{E}}$\,, we have
$$
(W^{2}_{\theta})^{\pm s}-w^{\pm s}(q,p,D_{q},D_{p})=R^{\pm s-1}\in \mathrm{OpS}^{\pm s -1}_{\Psi}(Q;\mathrm{End}\,\mathcal{E})
$$
and
\begin{eqnarray*}
&&
   (W^{2}_{\theta})^{s}A(W^{2}_{\theta})^{-s}-A=\underbrace{w^{s}(q,p,D_{q},D_{p})\circ A\circ w^{-s}(q,p,D_{q},D_{p})-A}_{=A_{1,s}}+ A_{2,s} +R_{s}
  \\
  \text{with}&& A_{2,s}= (W^{2}_{\theta})^{s/2}R^{-s-1}A+R^{s-1}w^{-1}(q,p,D_{q},D_{p})A\in \mathcal{L}(\tilde{\mathcal{W}}^{1,0};\tilde{\mathcal{W}}^{0,1})\,,\\
  \text{and}&&
R_{s}=(W^{2}_{\theta})^{s/2}[A,R^{-s-1}]+R^{s-1}[A,w^{-s}(q,p,D_{q},D_{p})]\in 
\mathrm{OpS}^{-3/2}_{\Psi}(Q;\mathrm{End}\,\mathcal{E})\subset \mathcal{L}(\tilde{\mathcal{W}}^{1,0};\tilde{\mathcal{W}}^{0,1})\,.
\end{eqnarray*}
By pseudo-differential calculus the symbol of $iA_{1,s}$ equals
\begin{multline*}
  w^{s}\partial_{\eta}a.\partial_{p}(w^{-s})-w^{s}\partial_{p}a.\partial_{\eta}(w^{-s})-w^{s}\partial_{q}a.\partial_{\xi}w^{-s}+r_{s}\\
  =\frac{s}{2}\sum_{k=1}^{d}-(w^{-1}\partial_{\eta_{k}}a)\sharp (w^{-1}\partial_{p}w^{2})+(w^{-1}\partial_{p_{k}}a)\sharp (w^{-1}\partial_{\eta_{k}}w^{2})+ w^{s}\partial_{\xi}w^{-s}\sharp \partial_{q}a +r_{s}'
\end{multline*}
with $r_{s},r_{s}'\in S^{-3/2}_{\Psi}(Q;\mathbb{C})$\,.\\
An explicit computation shows that the operators $(w^{-1}\partial_{p_{k}}w^{2})(q,p,D_{q},D_{p})$\,,\, $(w^{-1}\partial_{p_{k}}w^{2})(q,p,D_{q},D_{p})$ and $\partial_{q}a(q,p,D_{p})$ belong to $\mathcal{L}(\tilde{\mathcal{W}}^{1,0}(X;\mathcal{E});L^{2}(X;\mathcal{E}))$\,.\\
The operators $(w^{-1}\partial_{\eta_{k}}a)(q,p,D_{q},D_{p})$\,, $(w^{-1}\partial_{p_{k}}a)(q,p,D_{q},D_{p})$ belong to $\mathcal{L}(L^{2}(X;\mathcal{E});\tilde{\mathcal{W}}^{0,1}(X;\mathcal{E}))$\,.\\
Finally the remainder $r_{s}'(q,p,D_{q},D_{p})\in \mathrm{OpS}^{-3/2}_{\Psi}(Q;\mathbb{C})\subset \mathcal{L}(\tilde{\mathcal{W}}^{1,0}(X;\mathcal{E});\tilde{\mathcal{W}}^{0,1}(X;\mathcal{E}))$\,.\\
This ends the proof of
$$
\|(W^{2}_{\theta})^{s}A(W^{2}_{\theta})^{-s}-A\|_{\mathcal{L}(\tilde{\mathcal{W}}^{1,0};\tilde{\mathcal{W}}^{0,1})}\leq C_{g,s} \quad\text{for}~A=\hat{p}_{j,k}~\text{or}~\hat{D}_{j,i}\,.
$$
We split $(W^{2}_{\theta})\nabla_{\mathcal{Y}}^{\mathcal{E}}(W^{2}_{\theta})^{-s}-\nabla_{\mathcal{Y}}^{\mathcal{E}}$ into
\begin{multline*}
\sum_{j=1}^{J}(W^{2}_{\theta})^{s}\hat{E}_{j,i}\hat{p}_{j,i}(W^{2}_{\theta})^{-s}-\hat{E}_{j,i}\hat{p}_{j,i}
=\sum_{j=2}^{J}\left[(W^{2}_{\theta})^{s}\hat{E}_{j,i}(W^{2}_{\theta})^{-s}-\hat{E}_{j,i}\right]\circ (W^{2}_{\theta})^{s}\hat{p}_{j,i}(W^{2}_{\theta})^{-s}
\\
+\hat{E}_{j,i}\circ \left[(W^{2}_{\theta})^{s}\hat{p}_{j,i}(W^{2}_{\theta})^{-s}-\hat{p}_{j,i}\right]
\end{multline*}
The factor $\left[(W^{2}_{\theta})^{s}\hat{E}_{j,i}(W^{2}_{\theta})^{-s}-\hat{E}_{j,i}\right]$ belongs to $\mathrm{OpS}^{0}_{\Psi}(Q;\mathrm{End}\,\mathcal{E})\subset \mathcal{L}(L^{2}(X;\mathcal{E});L^{2}(X;\mathcal{E}))$ while the operator $(W^{2}_{\theta})^{s}\hat{p}_{j,i}(W^{2}_{\theta})^{-s}$ belongs to $\mathcal{L}(\tilde{\mathcal{W}}^{1,0}(X;\mathcal{E});L^{2}(X;\mathcal{E}))$\,.\\
The operator $(W^{2}_{\theta})^{s}\hat{p}_{j,i}(W^{2}_{\theta})^{-s}-\hat{p}_{j,i}$ belongs to $\mathcal{L}(\tilde{\mathcal{W}}^{1,0}(X;\mathcal{E});\tilde{\mathcal{W}}^{0,1}(X;\mathcal{E}))$ while the factor $\hat{E}_{j,i}$ belongs to $\mathrm{OpS}^{1}_{\Psi}(Q;\mathrm{End}\,\mathcal{E})\subset \mathcal{L}(\tilde{\mathcal{W}}^{0,1}(X;\mathcal{E});L^{2}(X;\mathcal{E}))$\,.
\end{proof}
\section{A priori estimates on the scalar GKFP operator}
\label{sec:localization}
In this section we work directly with the localized scalar version of GKFP operators. The results of this section will then applied to the operators  $P_{\pm,b}^j$'s of Subsection~\ref{sec:locchangeconnec}.
From now on, we focus the analysis to the case $\pm=+$, because the other case $\pm=-$ is the same, and we write simply $P_{b}^j$  and all the forthcoming related operators without the $\pm$ index.\\
The chart coordinates open set $\Omega$  in $Q$ is fixed and any coordinate system allows the identification $T^*\Omega=\Omega\times\R^d\subset \R^{2d}_{q,p}$\,. The symplectic volume on $\Omega\times\R^d$\,, is the usual Lebesgue measure $dqdp$ and the corresponding $L^2(\Omega\times\R^d,dqdp;\C)$-norm will be denoted simply by $\norm{~}_{L^2}$\,.
We consider a scalar GKFP operator 
\begin{eqnarray*}
&&    \mathcal{P}_b  = \frac{1}{b^2} \mathcal{O} + \frac{1}{b} \mathcal{Y}, \ \ b \in (0,\infty),
\\
\text{with}
&&
\mathcal{Y}=g^{ij}(q)p_j e_i\quad,\quad \mathcal{O}=\frac{-g_{ij}(q)\partial_{p_i}\partial_{p_j}+g^{ij}(q)p_ip_j}{2}\,,
\end{eqnarray*}
with  the domain
\begin{align}
	D(\mathcal{P}_b) = C^\infty_0(\Omega\times\R^d; \C).
\end{align}
By assuming $\Omega\subset\subset\Omega_1$ where $\Omega_1$ is a bigger chart coordinates open subset of $Q$\,, we can assume $g\big|_{\Omega_{1/2}}=\tilde{g}\big|_{\Omega_{1/2}}$ where $\tilde{g}$ is a riemannian metric on $\R^d$ which is euclidean outside a compact set, and $\Omega_{1/2}$ is an open neighborhood of $\Omega$ such that $\Omega\subset\subset \Omega_{1/2}\subset\subset \Omega_1$\,.\\
Alternatively the local scalar GKFP operators $\mathcal{P}_{b}$ can be introduced directly on $\Omega\times\R^d\subset \R^{2d}$ with a metric $g$ which is a compactly supported  perturbation of the euclidean metric.

\subsection{Dyadic partition of unity}\label{subSec:DyadicPartitionOfUnity}
By following \cite{Leb1}\cite{Leb2} or \cite{BCD}, let $\theta, \tilde{\theta} \in C^\infty_0(\R)$ be such that $\supp{\theta} \subset \left[\frac{1}{4}, 4 \right]$, $\supp{\tilde{\theta}} \subset \left[0, 4 \right]$, and
\begin{equation}
  \label{eq:dyadic}
  \forall t\in [0,\infty), \phantom{ccc}  \tilde{\theta}^{2}(4t^{2}) + \sum_{\ell=0}^{\infty} \theta^{2} \left(2^{-2 \ell} t^2 \right) = 1.
\end{equation}
For $x \in T^*\Omega$ and $\ell \in \N \cup \{-1\}$, set
\begin{align}
	\theta_\ell(x) = \begin{cases}
 					\theta \left(2^{-2 \ell} \abs{p}^2_q\right), \ \ &\ell > -1, \\
 					\tilde{\theta}(4 \abs{p}^2_q), \ \ &\ell = -1.
 					\end{cases}
\end{align}
The collection $\{\theta_\ell\}_{\ell = -1}^\infty$ constitutes a quadratic dyadic partition of unity for $T^*\Omega$ in the sense that
\begin{align}
	\forall x \in T^*\Omega,\ \ \sum_{\ell = -1}^\infty \theta^2_\ell(x) = 1\,,
\end{align}
with
\begin{align}
	\supp{\theta_\ell} \subset \{2^{\ell-1} \le \abs{p}_q \le 2^{\ell+1} \} \ \ \ \mathrm{whenever} \ \ \ \ell>-1\,,
\end{align}
and
\begin{align}
	\supp{\theta_{-1}} \subset \{0 \le \abs{p}_q \le 1\}, \ \mathrm{and} \ \ \theta_{-1}(x) = 1 \ \mathrm{for} \ 0 \le \abs{p}_q \le \frac{1}{2}.
\end{align}
Notice that because  $\theta_{\ell}$ is a function of  $\abs{p}^2_q$\,, $\theta_{\ell}$   satisfies
\begin{align} \label{Y_kills_theta_ell}
	\forall \ell \in \N \cup \{-1\}, \ \ \mathcal{Y} \theta_\ell \equiv 0.
\end{align}

 When $(q,p)$ are canonical coordinates on $T^* \Omega=\Omega\times\R^d$\,, we also observe
\begin{align} \label{p_derivative_estimates_thetas}
   \forall \alpha\in \mathbb{N}^d, \ \exists C_{\alpha}>0\,, \ \forall \ell \in \N \cup \{-1\}\,, &\quad \sup_{x\in T^* \Omega} |\p_{p}^{\alpha}\theta_{\ell}(x)| \leq C_{\alpha} 2^{-|\alpha|\ell}.
\end{align}

\begin{proposition}\label{pr:equivNormAfterDyadicPartition}
There exists a constant $C_{g,\theta, \tilde{\theta}}\geq 1$ depending only on the metric $g$ and the functions $\theta$ and $\tilde{\theta}$ so that
  \begin{equation}\label{eq:EquivDyadicPartition}
    \frac{1}{4} \sum_{\ell} \norm{\left(\frac{\kappa_{b}}{b^{2}}+\mathcal{P}_{b}-\frac{i\lambda}{b} \right) \theta_{\ell} u}_{L^2}^{2} \leq \norm{\left(\frac{\kappa_{b}}{b^{2}}+\mathcal{P}_{b}-\frac{i\lambda}{b}\right) u}^{2}_{L^2} \leq \frac{5}{2}\sum_{\ell} \norm{\left(\frac{\kappa_{b}}{b^{2}}+\mathcal{P}_{b}-\frac{i\lambda}{b}\right) \theta_{\ell} u}^{2}_{L^2}
 \end{equation}
 holds for all $u\in \mathcal{C}^\infty_0(\Omega\times\R^d;\mathbb{C})$ and all $(\lambda,b)\in \R\times \R_+$ when $\kappa_b = C_{g,\theta, \tilde{\theta}}(1+b^2)$.
\end{proposition}

\begin{proof}
Thanks to (\ref{Y_kills_theta_ell}), we have the commutator identities
\begin{align} \label{first_commutator_identity}
 &\left[\frac{\kappa_{b}}{b^{2}}+\mathcal{P}_{b}-\frac{i\lambda}{b},\theta_{\ell_{1}} \right]  = \frac{1}{b^{2}}[\mathcal{O},\theta_{\ell_{1}}]= -\frac{1}{b^{2}}g_{ij}(q)\frac{\partial \theta_{\ell_{1}}}{\partial p_{i}} \frac{\partial}{\partial p_{j}}-\frac{1}{2b^{2}}g_{ij}(q)\frac{\partial^{2} \theta_{\ell_{1}}}{\partial p_{i} \partial p_{j}},
 \end{align}
 and
 \begin{align} \label{second_commutator_identity}
\left[\left[\frac{\kappa_{b}}{b^{2}}+\mathcal{P}_{b}-\frac{i\lambda}{b},\theta_{\ell_{1}}\right],\theta_{\ell_{2}}\right]  =\frac{1}{b^{2}}[[\mathcal{O},\theta_{\ell_{1}}],\theta_{\ell_{2}}] = -\frac{1}{b^{2}}g_{ij}(q) \frac{\partial \theta_{\ell_{1}}}{\partial p_{i}} \frac{\partial \theta_{\ell_{2}}}{\partial p_{j}}.
\end{align}
  for any $\ell_1, \ell_2 \in \N \cup \{-1\}$. From  (\ref{p_derivative_estimates_thetas}), (\ref{first_commutator_identity}), (\ref{second_commutator_identity}) and the integration by parts inequality of Proposition~\ref{pr:IppIneqWithRealPart}, we deduce that there is a constant $C'_{g, \theta, \tilde{\theta}} \ge 1$, depending only on the metric $g$ and the functions $\theta$ and $\tilde{\theta}$ such that $\kappa_b = C_{g, \theta, \tilde{\theta}}(1+b^2)$, with $C_{g, \theta, \tilde{\theta}} = C_{0}+32C'_{g, \theta, \tilde{\theta}}$ and $C_0\geq 1$ fixed in Proposition~\ref{pr:IppIneqWithRealPart}, implies
  \begin{align}
  \begin{split}
    \sum_{\ell,\ell_{1}} \norm{\left[\frac{\kappa_{b}}{b^{2}}+\mathcal{P}_{b}-\frac{i\lambda}{b},\theta_{\ell_{1}} \right] \theta_{\ell} u}^{2}_{L^2}  & \leq C'_{g,\theta, \tilde{\theta}} 
  \left(
    \sum_{\ell} \left(\frac{1}{b^{4}}\|D_{p}\theta_{\ell}u\|^{2}_{L^2} + \frac{1}{b^{4}}\| \theta_{\ell}u\|^{2}_{L^2} \right) \right) \\
  & \leq C'_{g,\theta, \tilde{\theta}}  \left(\frac{1}{\kappa_{b}}+\frac{1}{\kappa_{b}^{2}}\right) \sum_{\ell} \norm{ \left(\frac{\kappa_{b}}{b^{2}}+\mathcal{P}_{b}-\frac{i\lambda}{b}\right) \theta_{\ell}u}^{2}_{L^2}
  \end{split}
  \end{align}
and
\begin{align}
    \sum_{\ell,\ell_{1},\ell_{2}} \norm{\left[ \left[\frac{\kappa_{b}}{b^{2}}+\mathcal{P}_{b}-\frac{i\lambda}{b},\theta_{\ell_{1}} \right],\theta_{\ell_{2}} \right] \theta_{l}}^{2}_{L^2}\leq \frac{C'_{g,\theta, \tilde{\theta}}}{b^{4}}  \|u\|^{2}_{L^2} & \leq \frac{C'_{g,\theta, \tilde{\theta}}}{\kappa_{b}^{2}}\sum_{\ell} \norm{ \left(\frac{\kappa_{b}}{b^{2}}+\mathcal{P}_{b}-\frac{i\lambda}{b}\right) \theta_{\ell}u}^{2}_{L^2}
\end{align}
for all $\lambda \in \R$, $b>0$, and $u \in C^\infty_0(T^* \Omega; \C)$. The equivalence (\ref{eq:EquivDyadicPartition}) then follows from Corollary \ref{Cor:equivalenceOfQuantities} with $r=\frac{1}{2}$.
\end{proof}

For every $\ell \geq -1$, we define the change of variable
$$
\Phi_{\ell}:\begin{array}[t]{ccl}
\Omega\times\R^d & \rightarrow & \Omega\times \R^d \\
(q,p) & \mapsto & (q,2^{\ell}p)
\end{array}.
$$ 
The change of variable in the integral give
\begin{equation}
  \|(\frac{\kappa_{b}}{b^{2}}+\mathcal{P}_{b}-\frac{i\lambda}{b})\theta_{\ell}u\|_{L^2}=  \| (\frac{\kappa_{b}}{b^{2}}+\mathcal{P}_{b,\ell}-\frac{i\lambda}{b})u_{\ell} \|_{L^2}
\end{equation}
with $ u_{\ell}(q,p) = 2^{\frac{\ell d}{2}}\theta(|p|^{2}_q)u(q,2^{\ell}p)$ (for $\ell=-1$ replace $\theta$ by $\tilde{\theta}$) and $\mathcal{P}_{b,\ell}=\Phi_{\ell}^{*}\mathcal{P}_{b}(\Phi_{\ell}^{-1})^{*}. $
  After the change of variable, operators are changed by
\begin{align}
\label{eq:defPbell}
  \mathcal{P}_{b,\ell} & = \frac{1}{b^{2}} \mathcal{O}_{\ell} + \frac{1}{b}\mathcal{Y}_{\ell}, \\
\label{eq:defOell}
      \mathcal{O}_{\ell} & =\Phi_{\ell}^{*}\mathcal{O}(\Phi_{\ell}^{-1})^{*} = \frac{1}{2}(2^{-2\ell}g_{ij}(q)D_{p_{i}}D_{p_{j}}+2^{2\ell}g^{ij}(q)p_{i}p_{j}) \\
\label{eq:defYell}
\mathrm{and} \quad \mathcal{Y}_{\ell}  & = \Phi_{\ell}^{*}\mathcal{Y}(\Phi_{\ell}^{-1})^{*}=  2^{\ell} g^{ij}(q) p_{j}(\frac{\partial}{\partial q^{i}} + \Gamma^{m}_{ik}(q)p_{m}\frac{\partial}{\partial p_{k}}).
\end{align}
The equivalence \eqref{eq:EquivDyadicPartition} can be rewritten as 
\begin{multline}
\label{eq:equivcalPbell}
\forall u\in \mathcal{C}^{\infty}_{0}(\Omega \times \mathbb{R}^d;\C)\,,\\
\frac{1}{4} \sum_\ell \| (\frac{\kappa_b}{b^2} + \mathcal{P}_{b,\ell} - \frac{i \lambda}{b}) u_{\ell}\|^2_{L^2} \leq \|(\frac{\kappa_b}{b^2} + \mathcal{P}_b - \frac{i \lambda}{b}) u \|^2_{L^2} \leq \frac{5}{2}\sum_\ell \| (\frac{\kappa_b}{b^2} + \mathcal{P}_{b,\ell} - \frac{i \lambda}{b}) u_{\ell}\|^2_{L^2},
\end{multline}
where now $u_\ell\in \mathcal{C}^{\infty}_{0}(S_{\ell,2};\C)$ with
\begin{equation}
\label{eq:defSell}
    S_{\ell,R}=\left\{
    \begin{array}[c]{ll}
    \{x=(q,p)\in\Omega\times\R^d\,,  \frac{1}{R}< |p|_q <R\}&\text{for}~\ell\geq 0\\ 
    \{x=(q,p)\in\Omega\times \R^d\,, |p|_q< 1\}&\text{for}~\ell=-1\,,
\end{array}
\right.
\end{equation}
for any fixed $R>1$\,.

\subsection{Partition with a $2^\ell$-dependent grid in the open set $\Omega\subset\R^d_q$}
\label{sec:gridpart}
Here the integer $\ell\geq -1$ is fixed and the localization will be done with some translation invariance in $\R^d$ by using a regular grid with a spacing of size $A2^{-\ell}$\,.\\
Let us start with the translation invariant partition of unity
\begin{equation} \label{eq:DefinitionOfPsi}
\sum_{m\in\Z^d}\psi^2(q-m)\equiv 1\quad\text{with}~ \psi\in \mathcal{C}^{\infty}_0(\R^d;[0,1])\,.    
\end{equation}

For any $A>0$ and $\ell\in\Z$\,, $\ell\geq -1$\,, it can be written
$$
\sum_{m\in\Z^d}\psi^2\left(\frac{q-q_{m,\ell,A}}{A2^{-\ell}}\right)\equiv 1\quad \text{with}~q_{m,\ell,A}=A2^{-\ell}m\,.
$$
Accordingly we set
$$
\psi_{m,\ell,A}(q)=\psi\left(\frac{q-q_{m,\ell,A}}{A2^{-\ell}}\right)
$$
and we get
\begin{align}
  & \mathrm{Supp}(\psi_{m,\ell,A})  = q_{m,\ell,A}+A2^{-\ell}\mathrm{Supp}(\psi)\,, \\
  \forall \alpha , \beta \in \mathbb{N}^{d}\,, \quad & |\alpha|>0 \Longrightarrow D_{p}^{\alpha}D_{q}^{\beta}\psi_{m,\ell,A} = 0\,, \\
  \label{eq:derpsimellA}
  \forall \beta \in \mathbb{N}^{d}\,, \exists C_{\beta,\psi}>0\,, \quad & |A^{|\beta|}2^{-\ell|\beta|} D_{q}^{\beta}\psi_{m,\ell,A}| \leq C_{\beta,\psi}\,.
\end{align}

\begin{proposition}\label{pr:controlOfErrorAfterGridPartitionOfUnity}
Let $\mathcal{P}_{b,\ell}$ and $S_{\ell,2}$ be defined respectively by \eqref{eq:defPbell} and \eqref{eq:defSell} for $\ell\in\Z$\,, $\ell\geq -1$\,.
There exists a constant $C_{g,\psi}>0$ depending only on the metric $g$ and the function $\psi$ such that
\begin{eqnarray*}
    &&
    \frac{1}{2}\sum_{m\in \Z^d}\| (\frac{\kappa_{b}}{b^{2}}+\mathcal{P}_{b,\ell} - \frac{i\lambda}{b}) \psi_{m,\ell,A}u \|^{2}_{L^2}-\frac{C_{g,\psi}}{A^2b^2}\norm{2^{2\ell}\psi_{m,\ell,A}u}_{L^2}^2
    \leq
    \| (\frac{\kappa_{b}}{b^{2}}+\mathcal{P}_{b,\ell} - \frac{i\lambda}{b}) u \|^{2}_{L^2}\\
    \text{and}&&
    \| (\frac{\kappa_{b}}{b^{2}}+\mathcal{P}_{b,\ell} - \frac{i\lambda}{b}) u \|^{2}_{L^2} \leq 
    2 \sum_{m\in\Z^d}\| (\frac{\kappa_{b}}{b^{2}}+\mathcal{P}_{b,\ell}-\frac{i\lambda}{b}) \psi_{m,\ell,A} u \|^{2}_{L^2}+\frac{C_{g,\psi}}{A^2b^2}\norm{2^{2\ell}\psi_{m,\ell,A}u}_{L^2}^2  
\end{eqnarray*}
holds for all $u\in \mathcal{C}^{\infty}_0(S_{\ell,2};\C)$ and for all $(\lambda,b)\in\R\times (0,+\infty)$\,.
\end{proposition}

\begin{proof}
The computation of commutators gives
\begin{align}
\label{eq:firstcommPbell}
  & [\frac{\kappa_{b}}{b^{2}} + \mathcal{P}_{b,\ell}-\frac{i\lambda}{b},\psi_{m_{1},\ell,A}] = \frac{1}{b}[\mathcal{Y}_{\ell},\psi_{m_{1},\ell,A}] = \frac{1}{b}g^{ij}(q)2 ^{\ell}p_{j}\frac{\partial \psi_{m_{1},\ell,A}}{\partial q^{i}}(q), \\
  & [[\frac{\kappa_{b}}{b^{2}} + \mathcal{P}_{b,\ell}-\frac{i\lambda}{b},\psi_{\ell,m_{1},A}],\psi_{m_{2},\ell,A}] = \frac{1}{b}[[\mathcal{Y}_{\ell},\psi_{m_{1},\ell,A}],\psi_{m_{2},\ell,A}]=0.
\end{align}
Because $\mathrm{supp}\, u\subset S_{\ell,2}\subset \{x=(q,p), |p|_q\leq 2\}$\,, the estimate \eqref{eq:derpsimellA} of the derivatives of $\psi_{m,\ell,A}$ implies that the right-hand side of 
\eqref{eq:firstcommPbell} satisfies 
$$
\sum_{m_1\in \mathbb{Z}^d}\norm{\frac{1}{b}g^{ij}(q)2 ^{\ell}p_{j}\frac{\partial \psi_{m_{1},\ell,A}}{\partial q^{i}}(q)\psi_{m_2,\ell,A}u}_{L^2}^2 \leq \frac{\tilde{C}_{g,\psi}2^{4\ell}}{(Ab)^2}\norm{\psi_{m_2,\ell,A}u}_{L^2}^2
$$ 
for all $m_1,m_2\in \Z^d$ with a constant $\tilde{C}_{g,\psi}>0$ depending only on $g$ and $\psi$\,.
Because the double commutator  vanishes, we apply  the formulas \eqref{eq:upperBound} and \eqref{eq:lowerBound}. This yields the result with the constant $C_{g,\psi}=4\tilde{C}_{g,\psi}^2$\,.
\end{proof}

Since
\begin{align}\label{set:supportAfterGridPartition}
\{ x=(q,p)\in S_{\ell,2} \,, q\in \mathrm{supp}(\psi_{m,\ell,A})    \}\subset (B(q_{m,\ell,A}, \hat{C}_{g,\psi}A2^{-\ell})\times \R^d)\cap S_{\ell,2}\
\end{align}
when $u \in C^\infty_0(S_{\ell,2}; \C)$, the problem is thus reduced to finding lower bounds for $\norm{(\frac{\kappa_{b}}{b^{2}}+\mathcal{P}_{b,\ell} - \frac{i\lambda}{b}) v}_{L^2}^2$ when $v \in C^\infty_0 \left(B \left(q_{m,\ell,A}, \hat{C}_{g,\psi}A2^{-\ell}\right)\times \R^d \cap S_{\ell,2}\right)$. When $A2^{-\ell}$ is small enough, \eqref{set:supportAfterGridPartition} is contained in a $q$-ball with radius below the injectivity radius of the metric $\tilde{g}$ on $\R^d$, and  normal coordinates around $q_{m,\ell,A}$ can be used.
Note also that the ball $B(q_{m,\ell,A}, \hat{C}_{g,\psi}A2^{-\ell})$ can be equivalently taken for the euclidean metric or the metric $\tilde{g}$ by possibly adapting the constant $\hat{C}_{g,\psi}$\,.

\subsection{Use of normal coordinates}\label{subsec:NormalCoordinate}
Due to (\ref{set:supportAfterGridPartition}), we are interested in $\left(\frac{\kappa_b}{b^2}+\mathcal{P}_{b, \ell} - \frac{i \lambda}{b}\right) v$ when $v \in C^\infty_0(B(q_{m,\ell,A}, \hat{C}_{g,\psi}A2^{-\ell})\times \R^d)\cap S_{\ell,2}$\,.
For $2^{-\ell}A\leq \varepsilon_{g,\psi}$ with 
$\varepsilon_{g,\psi}$ small enough determined by the pair
$(g,\psi)$\,, we may introduce normal coordinates 
$\tilde{q} = (\tilde{q}^1, \ldots, \tilde{q}^d)$ centered at $q_{m, \ell, A}$. The associated canonical coordinates on $B(0, \hat{C}_{g, \psi} A 2^{-\ell}) \times \R^d \subset T^* \R^d$ will be denoted by $(\tilde{q}, \tilde{p})$ with $\tilde{p} = (\tilde{p}_1, \ldots, \tilde{p}_d)$. Since $\mathcal{P}_{b, \ell}$ maintains the same form (\ref{eq:defPbell}) under the coordinate change $(q,p) \mapsto (\tilde{q}, \tilde{p})$, we may assume without loss of generality when considering
\begin{align} \label{quantity_to_bound_below}
    \norm{\left(\frac{\kappa_b}{b^2}+\mathcal{P}_{b,\ell}-\frac{i \lambda}{b}\right) v}_{L^2}
\end{align}
for $2^{-\ell}A \leq \varepsilon_{g, \psi}$ and $v \in C^\infty_0(B(q_{m,\ell,A}, \hat{C}_{g,\psi}A2^{-\ell})\times \R^d \cap S_{\ell,2}; \C)$ that $q_{m, \ell, A} = 0$ and that the metric $g$ satisfies
\begin{align}\label{eq:BoundForTheMetricInNormalCoordinate}
    \forall \alpha \in \N^d, \ \ \p^\alpha_q \left(g_{ij}(q)-\delta_{ij}\right) = \mathcal{O}(\abs{q}^{(2-\abs{\alpha})_{+}}).
\end{align}
We note that since the Christoffel symbols for the metric $g$ are given by
\begin{equation}
  \Gamma_{ik}^{\ell}(q) = \frac{1}{2}g^{\ell j}(q)\left(\frac{\partial g_{j i}}{\partial q^{k}}(q) + \frac{\partial g_{jk}}{\partial q^{i}}(q) - \frac{\partial g_{ik}}{\partial q^{j}}(q)\right) ,
\end{equation}
we have
\begin{equation}\label{eq:BoundForTheChristoffelSymbolInNormalCoordinate}
  \forall \alpha \in \N^d, \ \  \partial_q^{\alpha}\Gamma_{ik}^{\ell}(q) = {\mathcal{O}}(|q|^{(1-|\alpha|_+)}) .
\end{equation}
in these coordinates. By taking $\varepsilon_{g, \psi}$ smaller if necessary, we may restrict our attention to functions $v \in C^\infty_0(B(0, \hat{C}'_{g, \psi} A 2^{-\ell}) \times \R^d \cap S'_{\ell,4}; \C)$, where $B(0, \hat{C}'_{g, \psi} A 2^{-\ell})$ denotes the Euclidean ball in $\R^d$ of radius $\hat{C}'_{g, \psi} A 2^{-\ell}$ centered at the origin $0 \in \R^d$, $\hat{C}'_{g, \psi}>0$ is a constant depending only on $g$ and $\psi$, and $S'_{\ell,4} = \set{(q,p) \in \R^{2d}}{\frac{1}{4}<\abs{p}<4}$. Here $\abs{p} = (p_1^2+\cdots+p_d^2)^{1/2}$.

Let $g(q) = \left(g_{ij}(q)\right)_{1 \le i, j \le d}$ extended to a function defined on the whole space $\R^d$ with the same properties. We now introduce the following non-symplectic change of coordinates on $\R^{2d}$\,:
\begin{equation}
\label{eq:defvphiellm}
  \varphi_{\ell,m} : (q,p) \mapsto 
\left(
  2^{-\ell}q,g(2^{-\ell}q)p
\right),
\end{equation}
and the associated unitary map 
$$
\mathcal{U}_{\ell,m} :\begin{array}[t]{rcl}
     L^2(\mathbb{R}^{2d};\mathbb{C})  & \rightarrow & L^2(\mathbb{R}^{2d} ; \mathbb{C} )  \\
     v & \mapsto & 2^{-\frac{\ell d}{2}}\sqrt{\mathrm{det}(g(2^{-\ell}q))} (v\circ \varphi_{\ell,m}) \;,
\end{array}
$$
which sends $L^2(B(0, \hat{C}'_{g, \psi} A 2^{-\ell})\times\mathbb{R}^d;\mathbb{C})$ into $L^2(B(0, \hat{C}'_{g, \psi} A) \times \mathbb{R}^d ; \mathbb{C} )$\,.\\
By taking $\varepsilon_{g,\psi}$ smaller if necessary, we can assume that the unitary map $\mathcal{U}_{\ell,m}$ and the pull-back $\varphi_{\ell,m}^*$  send $C^\infty_0((B(0, \hat{C}'_{g, \psi} A 2^{-\ell}) \times \R^d) \cap S'_{\ell,4}; \C)$ into $\mathcal{C}^{\infty}_0((B(0,\hat{C}'_{g,\psi}A)\times \mathbb{R}^d) \cap S_{\ell,8}' ;\mathbb{C}  ).$\\
The change of variables in the $L^{2}$-norm gives
\begin{align}\label{eq:ChangeOfVariableInL2NormAfterNormalCoordinate}
  \| (\frac{\kappa_{b}}{b^{2}}+\mathcal{P}_{b,\ell}-\frac{i\lambda}{b})v \|_{L^2} & = \|\left(\mathcal{U}_{\ell,m} (\frac{\kappa_{b}}{b^{2}}+\mathcal{P}_{b,\ell}-\frac{i\lambda}{b}) \mathcal{U}_{\ell,m}^*\right)  \mathcal{U}_{\ell,m}v \|_{L^2} \,,
\end{align}
where 
$$ \mathcal{U}_{\ell,m}\mathcal{P}_{b,\ell} \mathcal{U}_{\ell,m}^* = \sqrt{\mathrm{det}(g(2^{-\ell}q))} \mathcal{P}_{b,\ell,m} \frac{1}{\sqrt{\mathrm{det}(g(2^{-\ell}q))}} \;,$$
with
$$ \mathcal{P}_{b,\ell,m} = \varphi_{\ell,m}^{*} \mathcal{P}_{b,\ell} (\varphi^{-1}_{\ell,m})^{*} \;.$$
A straightforward computation shows that
\begin{align}
\label{eq:defPbellm}
\mathcal{P}_{b,\ell,m} & = \frac{1}{b^{2}}\mathcal{O}_{\ell,m}+\frac{1}{b}\mathcal{Y}_{\ell,m} \,, \\
\mathcal{U}_{\ell,m} \mathcal{P}_{b,\ell} \mathcal{U}_{\ell,m}^* & = \mathcal{P}_{b,\ell,m} + \left[ \sqrt{\mathrm{det}(g(2^{-\ell}q)} , \frac{\kappa_b}{b^2} + \mathcal{P}_{b,\ell,m} - \frac{i\lambda}{b} \right] \frac{1}{\sqrt{\mathrm{det}(g(2^{-\ell}q)}} \label{eq:ConjugationByUnitary}
\end{align}
where
\begin{align}
\label{eq:defOellm}
  \mathcal{O}_{\ell,m} & := \varphi_{\ell,m}^*\mathcal{O}_{\ell} (\varphi_{\ell,m}^{-1})^{*} = \mathcal{U}_{\ell,m} \mathcal{O}_{\ell} \mathcal{U}_{\ell,m}^* = \frac{1}{2}
                            \sum_{i,j}\left(g^{ij}(2^{-\ell}q)\frac{D_{p_{i}}}{2^{\ell}}\frac{D_{p_{j}}}{2^{\ell}} + g_{ij}(2^{-\ell}q)2^{\ell}p_{i}2^{\ell}p_{j}
                         \right)
\end{align}
and
\begin{align}
\label{eq:defYellm}
    \mathcal{Y}_{\ell,m} &:= \varphi_{\ell,m}^*\mathcal{Y}_{\ell}(\varphi_{\ell,m}^{-1})^*= 2^{2\ell}\delta^{ij}p_{j}\frac{\partial}{\partial q^{i}} + f_{k}^{ij}(q,\ell)p_{i}p_{j}\frac{\partial}{\partial p_{k}} \,, \\
    f_{k}^{ij}(q,\ell) &:= 2^{\ell} \sum_{n'}g_{n',j}(2^{-\ell}q)
                       \left(
                       \frac{\partial g^{k n'}}{\partial q^{i}}(2^{-\ell }q) + \sum_{n}\Gamma^{n'}_{in}(2^{-\ell} q)g^{n k}(2^{-\ell}q)
                       \right) \,,  \label{formOff}\\
    \mathcal{U}_{\ell,m}\mathcal{Y}_{\ell} \mathcal{U}_{\ell,m}^* & = \mathcal{Y}_{\ell,m} + \left[  \sqrt{\mathrm{det}(g(2^{-\ell}q)} , \mathcal{Y}_{\ell,m} \right] \frac{1}{\sqrt{\mathrm{det}(g(2^{-\ell }q)}} \label{eq:linkBetweenYlmAndYl} \;.
\end{align}
From \eqref{eq:BoundForTheMetricInNormalCoordinate} and \eqref{eq:BoundForTheChristoffelSymbolInNormalCoordinate} we know that 
\begin{equation} \label{eq:BoundForfijk}
\sup_{q\in B(0,\hat{C}'_{g,\psi}A)} |f_{k}^{ij}(q,\ell)| \leq C_{g,\psi}^{(1)}A \;.    
\end{equation}
Since $\mathrm{det}(g)\circ \varphi_{\ell,m} = \mathrm{det}(g(2^{-\ell}q)) $, we have
\begin{align}
\nonumber  \left[\frac{\kappa_{b}}{b^{2}} + \mathcal{P}_{b,\ell,m}-\frac{i\lambda}{b},\sqrt{\mathrm{det}(g)\circ\varphi_{\ell,m}}\right] & = \frac{1}{b}\left[ \mathcal{Y}_{\ell,m} , \sqrt{\mathrm{det}(g(2^{-\ell}q))} \right] \\
\nonumber  & = \frac{1}{b}\varphi_{\ell,m}^{*}[\mathcal{Y}_{\ell},\sqrt{\mathrm{det}(g)}] (\varphi^{-1}_{\ell,m})^{*} \\
\nonumber        &= \frac{1}{b}[\mathcal{Y}_{\ell}, \sqrt{\mathrm{det}(g)} ]\circ \varphi_{\ell,m}\\
        &= \frac{2^{\ell}}{b}\delta^{ki}p_{k} \frac{\partial \sqrt{\mathrm{det}(g)}}{\partial q^{i}}\circ \varphi_{\ell,m} \;. \label{eq:ErrorAfterCommutingWithTheJacobian}
\end{align}
 is bounded by some constant $C_{g,\psi}^{(2)}\frac{A}{b}$ on $(B(0,\hat{C}'_{g,\psi}A)\times \mathbb{R}^d) \cap S_{\ell,8}'$\,. With the identities \eqref{eq:ChangeOfVariableInL2NormAfterNormalCoordinate} and \eqref{eq:ConjugationByUnitary} we deduce 
\begin{equation}\label{eq:BoundForTheEquivalenceAfterChangeOfVariableVarphi}
  | ~ \| (\frac{\kappa_{b}}{b^{2}}+\mathcal{P}_{b,\ell,m} - \frac{i\lambda}{b})\mathcal{U}_{\ell,m} v\|_{L^2} - \| (\frac{\kappa_{b}}{b^{2}} + \mathcal{P}_{b,\ell} - \frac{i\lambda}{b})v \|_{L^2} ~ | \leq C_{g,\psi}^{(3)}\frac{A}{b}\| \mathcal{U}_{\ell,m} v\|_{L^2}.
\end{equation}

\begin{proposition} \label{pr:equivNormAfterGaugeChangeOfVariables}
There exists a constant $C_{g,\psi}\geq 1$ , determined by the metric $g$ and the function $\psi$\,, such that the inequalities
  \begin{equation}
   \frac{1}{4} \| (\frac{\kappa_{b}}{b^{2}} + \mathcal{P}_{b,\ell,m} - \frac{i\lambda}{b})\mathcal{U}_{\ell,m}v \|_{L^2}^{2} \leq \| (\frac{\kappa_{b}}{b^{2}}+\mathcal{P}_{b,\ell} - \frac{i\lambda}{b})v \|_{L^2}^{2} \leq 4 \| (\frac{\kappa_{b}}{b^{2}} + \mathcal{P}_{b,\ell,m} - \frac{i\lambda}{b})\mathcal{U}_{\ell,m}v \|_{L^2}^{2} ,
  \end{equation}
  hold for all $v\in \mathcal{C}_0^\infty((B(0, \hat{C}'_{g,\psi}A2^{-\ell})\times\mathbb{R}^d) \cap S'_{\ell,4};\mathbb{C})$, all  $(\lambda,b)\in \mathbb{R}\times(0,\infty)$ when $2^{-\ell}A\leq \frac{1}{C_{g,\psi}}$ and $\frac{\kappa_b}{b^2}>C_{g,\psi}\frac{A}{b}$\,.\\
  Additionally, we recall
  $$
  \mathcal{U}_{\ell,m}v\in \mathcal{C}_0^\infty((B(0, \hat{C}'_{g,\psi}A)\times\mathbb{R}^d) \cap S'_{\ell,8};\mathbb{C})\,,
  $$
  for all $v\in \mathcal{C}_0^\infty((B(0, \hat{C}'_{g,\psi}A2^{-\ell})\times\mathbb{R}^d) \cap S'_{\ell,4};\mathbb{C})$\,.
\end{proposition}

\begin{proof}
  The same integration by parts as in the proof of Proposition~\ref{pr:IppIneqWithRealPart}  now gives
\begin{align*}
    \| (\frac{\kappa_{b}}{b^{2}}+\mathcal{P}_{b,\ell,m} - i\lambda) u \|_{L^{2}} \|u\|_{L^{2}} &\geq \frac{1}{4b^{2}} 
  \left[
    C_g\|2^{-\ell}D_{p}u\|_{L^{2}}^{2} + C_g\| 2^\ell|p|u\|_{L^{2}}^{2} + \kappa_{b} \|u\|_{L^{2}}^2
  \right]\\
&  \hspace{3cm}
  +\frac{1}{b}\mathrm{Re}~{\langle u\,,\,  f_{k}^{ij}(q,\ell)p_{i}p_{j}\frac{\partial}{\partial p_{k}}  u\rangle}\\
 &\geq (\frac{\kappa_{b}}{4b^2}-C_{g,\psi}^{(4)} \frac{A}{b})\|u\|_{L^{2}}^2\\
\end{align*}
for all $u\in \mathcal{C}^{\infty}_0((B(0,\hat{C}'_{g,\psi} A)\times\R^d)\cap S'_{\ell,8};\mathbb{C})$\,, owing to \eqref{eq:BoundForfijk}.
For $\kappa_b\geq 8C_{g,\psi}^{(4)}Ab$ we deduce

  $$\forall u\in \mathcal{C}^{\infty}_0((B(0,\hat{C}'_{g,\psi} A)\times\R^d)\cap S'_{\ell,8};\mathbb{C}),\quad  \frac{\kappa_{b}}{8b^{2}}\| u\| \leq \| (\frac{\kappa_{b}}{b^{2}}+\mathcal{P}_{b,\ell,m}-\frac{i\lambda}{b})u\|\,.$$
  The inequality \eqref{eq:BoundForTheEquivalenceAfterChangeOfVariableVarphi} now implies
  \begin{equation}
   | ~ \| (\frac{\kappa_{b}}{b^{2}}+\mathcal{P}_{b,\ell,m} - \frac{i\lambda}{b})\mathcal{U}_{\ell,m}v \|_{L^2} - \| (\frac{\kappa_{b}}{b^{2}} + \mathcal{P}_{b,\ell} - \frac{i\lambda}{b})v \|_{L^2} ~ | \leq \frac{8C_{g,\psi}^{(3)}Ab}{\kappa_{b}} \| (\frac{\kappa_{b}}{b^{2}} + \mathcal{P}_{b,\ell,m} - \frac{i\lambda}{b})\mathcal{U}_{\ell,m}v \|_{L^2}   .
  \end{equation}
  when $\kappa_b\geq 8C_{g,\psi}^{(4)}Ab$\,.
  It suffices to take $C_{g,\psi}=\max(8C_{g,\psi}^{(4)}, 16C_{g,\psi}^{(3)})$\,.
\end{proof}

The next step is to replace the $\mathcal{O}_{\ell,m}$ with the euclidean version defined by
\begin{equation}
  \tilde{\mathcal{O}}_{\ell}=\frac{1}{2}(\delta_{ij}2^{-2\ell}D_{p_{i}}D_{p_{j}}+2^{2\ell}|p|^2) \,,
\end{equation}
and we set according to \eqref{eq:defYellm}
\begin{align}
\label{eq:deftildePbellm}
\tilde{\mathcal{P}}_{b,\ell,m} &= \frac{1}{b^{2}} \tilde{\mathcal{O}}_{\ell} + \frac{1}{b}\mathcal{Y}_{\ell,m} \;,
\\
\label{eq:rappYellm}
\text{with}\quad
\mathcal{Y}_{\ell,m}&=2^{2\ell}\delta^{ij}p_{j}\frac{\partial}{\partial q^{i}} + f_{k}^{ij}(q,\ell)p_{i}p_{j}\frac{\partial}{\partial p_{k}}\,.
\end{align}

\begin{proposition} \label{pr:ChangingOToTheEuclideanO}
There is a constant $C_{g,\psi}\geq 1$\, determined by the metric $g$ and the function $\psi$\,, such that the following inequalities 
$$
\| \left(
    \frac{\kappa_{b}}{b^{2}} + \tilde{\mathcal{P}}_{b,\ell,m} - \frac{i\lambda}{b}
\right) u \|_{L^2} 
- C_{g,\psi} A^{2}2^{-2\ell} \|\left(
  \frac{\kappa_{b}}{b^{2}} + \frac{1}{b^{2}}\tilde{\mathcal{O}}_{\ell}
\right) u \|_{L^2}  
\leq 
\| \left(
  \frac{\kappa_{b}}{b^{2}} + \mathcal{P}_{b,\ell,m}-\frac{i\lambda}{b}
\right) u\|_{L^2}
$$
and
$$
\|\left( 
\frac{\kappa_{b}}{b^{2}} + \mathcal{P}_{b,\ell,m}-\frac{i\lambda}{b}
\right) u\|_{L^2} 
\leq 
\|\left(
    \frac{\kappa_{b}}{b^{2}} + \tilde{\mathcal{P}}_{b,\ell,m} - \frac{i\lambda}{b}
\right) u \|_{L^2} 
+ C_{g,\psi} A^{2}2^{-2\ell} 
\|\left(
  \frac{\kappa_{b}}{b^{2}} + \frac{1}{b^{2}}\tilde{\mathcal{O}}_{\ell}
\right) u \|_{L^2} $$
hold
for all $u\in \mathcal{C}^\infty_0((B(0,\hat{C}'_{g,\psi}A)\times \mathbb{R}^d)\cap S_{\ell,8}',\mathbb{C})$, all $\ell\in\Z$\,, $\ell\geq -1$\,, and all $(\lambda,b)\in \mathbb{R}\times (0,+\infty)$.
\end{proposition}

\begin{proof}
  From the equality
  \begin{equation}
   \mathcal{P}_{b,\ell,m} - \tilde{\mathcal{P}}_{b,\ell,m} = \frac{1}{b^{2}}
   \left(
     \mathcal{O}_{\ell,m} - \tilde{\mathcal{O}}_{\ell}
\right) \;.
  \end{equation}
  The difference between $\mathcal{O}_{\ell,m}$ and $\tilde{\mathcal{O}}_{\ell,m}$ is given by
  \begin{equation}
    \mathcal{O}_{\ell,m}- \tilde{\mathcal{O}}_{\ell} = \frac{1}{2}
  \left[
    \left(
      g^{ij}(2^{-\ell}q)-\delta_{ij}
\right)2^{-2\ell}D_{p_{i}}D_{p_{j}} +2^{2\ell}
\left( g_{ij}(2^{-\ell}q)-\delta^{ij} \right)p_{i}p_{j}
\right] \;.
\end{equation}
A unitary change of scale  replaces $(2^\ell p, 2^{-\ell} D_p)$ by $(p,D_p)$ and the problem is reduce to the comparison of harmonic oscillator hamiltonians in Proposition~\ref{pr:composcharm}
relying on global ellipticity. With $|g^{ij}(2^{-\ell}q)-\delta^{ij}|+|g_{ij}(2^{-\ell}q)-\delta_{ij}|\leq C_{g,\psi,1}A^2 2^{-2\ell}$\,, we obtain
\begin{equation}\label{eq:upperBoundForReplacingOToTheEuclideanO}
  \| \frac{1}{b^{2}} 
\left(
  \mathcal{O}_{\ell,m} - \tilde{\mathcal{O}}_{\ell}
\right) u  \|_{L^2} \leq C_{g,\psi} A^{2}2^{-2\ell} \|
\left(
  \frac{\kappa_{b}}{b^{2}} + \frac{1}{b^{2}}\tilde{\mathcal{O}}_{\ell}
\right) u \|_{L^2} \;,
\end{equation}
for some constant $C_{g,\psi}>0$ determined by $(g,\psi)$\,.
The two inequalities of this proposition follow.
\end{proof}

We set $h=\frac{1}{2^{2\ell}b}$,
\begin{equation}
  \frac{\kappa_{b}}{b^{2}} +  \tilde{\mathcal{P}}_{b,\ell,m} - \frac{i\lambda}{b}=  2^{2\ell}
  \left(
    \frac{\kappa_{b}h}{b} + \hat{\mathcal{P}}_{b,h,f} - i h \lambda
\right)
\end{equation}
where
\begin{equation} \label{perturbation_of_Euclidean_KFP}
  \hat{\mathcal{P}}_{b,h,f} = \frac{1}{2}\delta_{ij}(hD_{p_{i}})(hD_{p_{j}}) + \frac{|p|^2}{2b^2} + \frac{1}{b}\delta^{ij}p_{j}\frac{\partial}{\partial q^{i}} + hf^{ij}_{k} p_{i}p_{j}\frac{\partial}{\partial p_{k}}.
\end{equation}

The problem is now reduced to a careful study of the operator $\hat{ \mathcal{P}}_{b,h,f}$ acting on $\mathcal{C}^\infty_0((B(0,\hat{C}'_{g,\psi}A)\times \mathbb{R}^d)\cap S_{\ell,8}',\mathbb{C})$\,. We will firstly consider the euclidean case in Section~\ref{sec:EuclCase}, where  $f=(f^{ij}_k)_{1\leq i,j,k \leq d}=0$\,, and the results for the general case will be obtained by some accurate perturbative argument in Subsection~\ref{sec:localtPbellm}.
Passing from the local to the global estimates will be developed in the end of Section~\ref{sec:finalproof} but preliminary results are collected in the next paragraph.

\subsection{From the local models to the global estimates for a fixed $\ell$}
\label{sec:locglobfixedell}
We work in the framework of Subsection~\ref{sec:gridpart} and Subsection~\ref{subsec:NormalCoordinate}. Remember the notations and assumptions:
\begin{itemize}
\item the partition of unity $\sum_{m\in\Z^d}\psi^2(q-m)\equiv 1$ and the notations $\psi_{m,\ell,A}(q)=\psi(\frac{q-q_{m,\ell,A}}{2^{-\ell}A})$\,, $q_{\ell,m,A}=A2^{-\ell}m$\,;
    \item $\mathcal{P}_{b,\ell}=\frac{1}{b^2}\mathcal{O}_{\ell}+\frac{1}{b}\mathcal{Y}_{\ell}$ introduced in \eqref{eq:defPbell}\eqref{eq:defOell}\eqref{eq:defYell}\,;
    \item the condition $2^{-\ell}A\leq \varepsilon_{g,\psi}=\frac{1}{C_{g,\psi}}$ for $\varepsilon_{g,\psi}$ small enough which allows in Subsection~\ref{subsec:NormalCoordinate} the use of normal coordinates centered at $q_{\ell,m,A}=A2^{-\ell}m$\,, $m\in\Z^d$\,, and the comparison between the metric $g$ with the euclidean metric;
    \item the unitary transform $\mathcal{U}_{\ell,m}$ associated with the change of variables $  \varphi_{\ell,m} : (q,p) \mapsto 
\left(2^{-\ell}q,g(2^{-\ell}q)p\right)$ written in normal coordinates;
\item the operator $\mathcal{P}_{b,\ell,m}=\frac{1}{b^2}\mathcal{O}_{\ell,m}+\frac{1}{b}\mathcal{Y}_{\ell,m}$ introduced in \eqref{eq:defPbellm}.
\end{itemize}
\begin{proposition}
\label{pr:locglobPbell}
    With the above notations and assumptions, in particular $2^{-\ell}A\leq \frac{1}{C_{g,\psi}}$\,, set for $u\in\mathcal{C}^{\infty}_{0}(S_{\ell,2},\C)$ and for any $m\in \Z^d$\,,
    $$u_{\ell,m}=\mathcal{U}_{\ell,m}(\psi_{m,\ell,A}u)~\in \mathcal{C}^\infty_0((B(0,\hat{C}'_{g,\psi}A)\times \mathbb{R}^d)\cap S_{\ell,8}',\mathbb{C})\,.
    $$
    The following inequalities
    \begin{eqnarray*}
    &&\sum_{m\in \Z^d}\frac{1}{8}\| (\frac{\kappa_{b}}{b^{2}}+\mathcal{P}_{b,\ell,m} - \frac{i\lambda}{b}) u_{\ell,m}\|^{2}_{L^2(\R^{2d})}-\frac{C_{g,\psi}}{A^2b^2}\norm{2^{2\ell}u_{\ell,m}}_{L^2(\R^{2d})}^2
    \leq
    \| (\frac{\kappa_{b}}{b^{2}}+\mathcal{P}_{b,\ell} - \frac{i\lambda}{b}) u \|^{2}_{L^2(\R^{2d})}\,,\\
    \text{and}&&\| (\frac{\kappa_{b}}{b^{2}}+\mathcal{P}_{b,\ell} - \frac{i\lambda}{b}) u \|^{2}_{L^2(\R^{2d})}
    \leq \sum_{m\in \Z^d}8\| (\frac{\kappa_{b}}{b^{2}}+\mathcal{P}_{b,\ell,m} - \frac{i\lambda}{b}) u_{\ell,m}\|^{2}_{L^2(\R^{2d})}
    +\frac{C_{g,\psi}}{A^2b^2}\norm{2^{2\ell}u_{\ell,m}}_{L^2(\R^{2d})}^2
    \end{eqnarray*}
    hold true as soon as $\kappa_b\geq C_{g,\psi}Ab$ and the constant $C_{g,\psi}\geq 1$ is chosen large enough.
\end{proposition}
\begin{proof}
    Simply combine Proposition~\ref{pr:controlOfErrorAfterGridPartitionOfUnity} and Proposition~\ref{pr:equivNormAfterGaugeChangeOfVariables}.
\end{proof}
A similar result can be written for the hamiltonian vector field alone.
\begin{proposition}
\label{pr:locglobYell}
    With the same notations and assumptions as in Proposition~\ref{pr:locglobPbell}
    the following inequalities
    \begin{eqnarray*}
    &&\sum_{m\in \Z^d}\frac{1}{8}\| \frac{1}{b}(\mathcal{Y}_{\ell,m} - i\lambda) u_{\ell,m}\|^{2}_{L^2(\R^{2d})}-C_{g,\psi}(\frac{1}{A^2b^2}+\frac{1}{b^2 2^{2\ell}})\norm{2^{2\ell}u_{\ell,m}}_{L^2(\R^{2d})}^2
    \leq
    \| \frac{1}{b}(\mathcal{Y}_{\ell} - i\lambda) u \|^{2}_{L^2(\R^{2d})}\,,\\
    \text{and}&&\| \frac{1}{b}(\mathcal{Y}_{\ell} - i\lambda) u \|^{2}_{L^2(\R^{2d})}
    \leq \sum_{m\in \Z^d} 8\| \frac{1}{b}(\mathcal{Y}_{\ell,m} - i\lambda) u_{\ell,m}\|^{2}_{L^2(\R^{2d})}
    +C_{g,\psi}(\frac{1}{A^2b^2}+\frac{1}{b^2 2^{2\ell}})\norm{2^{2\ell}u_{\ell,m}}_{L^2(\R^{2d})}^2
    \end{eqnarray*}
    hold true as soon as the constant $C_{g,\psi}\geq 1$ is chosen large enough.
\end{proposition}
\begin{proof}
    In Proposition~\ref{pr:controlOfErrorAfterGridPartitionOfUnity} the operator $(\frac{\kappa_{b}}{b^{2}}+\mathcal{P}_{b,\ell} - \frac{i\lambda}{b})$ can be replaced by $\frac{1}{b}(\mathcal{Y}_{\ell}-i\lambda)$ and this produces the same error term 
    $\frac{C_{g,\psi}}{A^2b^2}\norm{2^{2\ell}u_{\ell,m}}_{L^2}^2$\,. In the same way, the calculation leading to \eqref{eq:BoundForTheEquivalenceAfterChangeOfVariableVarphi} can be done with $(\frac{\kappa_{b}}{b^{2}}+\mathcal{P}_{b,\ell} - \frac{i\lambda}{b})$ replaced by $\frac{1}{b}(\mathcal{Y}_{\ell}-i\lambda)$\,. This second step leads to 
    $$
    \|\frac{1}{b}(\mathcal{Y}_{\ell,m}-i\lambda)\mathcal{U}_{\ell,m}v\|^2_{L^2(\R^{2d})}\leq  2\|\frac{1}{b}(\mathcal{Y}_{\ell}-i\lambda)v\|^2_{L^2(\R^{2d})}
    +2[C_{g,\psi}^{(3)}]^2\frac{A^2}{b^2}\|\mathcal{U}_{\ell,m}v\|^2_{L^2(\R^{2d})}
    $$
    and
    $$
    \|\frac{1}{b}(\mathcal{Y}_{\ell}-i\lambda)v\|^2_{L^2(\R^{2d})}\leq 2 \|\frac{1}{b}(\mathcal{Y}_{\ell,m}-i\lambda)\mathcal{U}_{\ell,m}v\|^2_{L^2(\R^{2d})}
    +2[C_{g,\psi}^{(3)}]^2\frac{A^2}{b^2}\|\mathcal{U}_{\ell,m}v\|^2_{L^2(\R^{2d})}\,.
    $$
     Since $2^{-\ell}A\leq \frac{1}{C_{g,\psi}}\leq 1 $, we conclude that
    $$
    \frac{A^2}{b^2}\|\mathcal{U}_{\ell,m}v\|^2_{L^2(\R^{2d})}\leq \frac{2^{-2\ell}A^2}{b^2 2^{2\ell}}\|2^{2\ell}\mathcal{U}_{\ell,m}v\|^2_{L^2(\R^{2d})}\leq\frac{1}{b^2 2^{2\ell}}\|2^{2\ell}\mathcal{U}_{\ell,m}v\|^2_{L^2(\R^{2d})}\,.
    $$
    
  \end{proof}
The last result of this paragraph is about the comparison between local and global estimates of the $\tilde{\mathcal{W}}^{2/3}$-norm appearing in the lower bound of Theorem~\ref{th:mainOne}.
According to Proposition~\ref{pr:eqnormes}-\textbf{ii)} in Appendix~\ref{sec:pseudodiff}, the $\tilde{\mathcal{W}}^{2/3}$-norm of  $u\in\mathcal{C}^{\infty}_{0}(\Omega\times\R^d;\C)$ can be written
$$
\|u\|_{\tilde{\mathcal{W}}^{2/3}}^2=\norm{(\tilde{\mathcal{O}}_{1})^{2/3}u}_{L^2(\R^{2d})}^2 +\norm{|D_q|^{2/3}u}_{L^2(\R^{2d})}^2\,,
$$
expressed with operators constructed in the euclidean case.
The change of scale $(p,D_p)\to (2^\ell p, 2^{-\ell} D_p)$ leads to consider
\begin{equation}
\label{eq:defnormWell}
\norm{u}_{\tilde{\mathcal{W}}^{2/3},\ell}^2
=\norm{(\tilde{\mathcal{O}_{\ell}})^{2/3}u}_{L^2(\R^{2d})}^2 +\norm{|D_q|^{2/3}u}_{L^2(\R^{2d})}^2\,.
\end{equation}
\begin{proposition}

\label{pr:locglobW23}
  With the same notations and assumptions as in Proposition~\ref{pr:locglobPbell}
    and with $\norm{u}_{\mathcal{W}^{2/3},\ell}$ defined by \eqref{eq:defnormWell}, the following inequalities
        \begin{eqnarray*}
   \|u \|^{2}_{\tilde{\mathcal{W}}^{2/3},\ell}
   \leq C_{g,\psi}\sum_{m\in \Z^d}\| (\tilde{\mathcal{O}}_{\ell})^{2/3} u_{\ell,m}\|^{2}_{L^2(\R^{2d})}
    +\||2^{\ell}D_q|^{2/3}u_{\ell,m}\|^{2}_{L^2(\R^{2d})}
    +\frac{1}{A^{4/3}}\norm{2^{2\ell/3}u_{\ell,m}}_{L^2(\R^{2d})}^2
    \end{eqnarray*}
    hold true as soon as $\kappa_b\geq C_{g,\psi}Ab$ and the constant $C_{g,\psi}\geq 1$ is chosen large enough.
\end{proposition}
\begin{proof}
    With $\norm{\langle D_q\rangle^{2/3}u}_{L^2(\R^{2d})}^2=\langle u\,,\, (1-\Delta_q)^{2/3} u\rangle$ and 
$(1-\Delta_q)^{2/3}=\left[(1+|\xi|^2)^{2/3}\right]^{Weyl}(q,D_q)$\,.\\
and  $\sum_{m\in \Z^{d}} \psi^2(q-m)\equiv 1$\,, Weyl-H{\"o}rmander pseudo-differential calculus with the standard metric $dq^2+\frac{d\xi^2}{\langle\xi\rangle^2}$ (see \cite{HormIII}-Chap~XVIII) , provides a constant $C_{\psi}\geq 1$ such that
      $$ \| \langle D_q \rangle^{\frac{2}{3}} v \|_{L^2}^2 \leq C_{\psi}\sum_{m\in \mathbb{Z}^d}\| \langle D_q \rangle^{\frac{2}{3}} \psi(.-m) v \|_{L^2}^2 \,.$$
      By setting  $h=A2^{-\ell}$\,, $\psi(\frac{q-hm}{h}) = \psi_{m,\ell,A }(q)$\,, 
      $v(q,p)=h^{d/2}u(hq,p)$ and $$v_{m,\ell,A}(q,p)=\psi_{m,\ell,A}(q)u(q,p)=h^{-d/2}\psi(\frac{q-hm}{h})v(h^{-1}q,p)\,,$$ the above inequality becomes 
    $$ \| |hD_q|^{\frac{2}{3}} u \|_{L^2}^2 \leq C_{\psi}\sum_{m\in \mathbb{Z}^d}\| \langle hD_q \rangle^{\frac{2}{3}} v_{m,\ell,A} \|_{L^2}^2
    \leq 2 C_{\psi}\left[\sum_{m\in \mathbb{Z}^d}\| |hD_q|^{\frac{2}{3}} v_{m,\ell,A} \|_{L^2}^2+\|v_{m,\ell, A}\|_{L^2(\R^{2d})}^2\right]\,.$$
 By multiplying both sides of the inequality by $ h^{-\frac{4}{3}}=\frac{2^{4\ell/2}}{A^{4/3}}$\,, we get 
    $$ \|  |D_q|^{\frac{2}{3}} u \|_{L^2}^2 \leq 2C_{\psi}\sum_{m\in \mathbb{Z}^d}\| |D_q|^{\frac{2}{3}} v_{m,\ell,A} \|_{L^2}^2 +\frac{1}{A^{4/3}}\|2^{2\ell/3}v_{m,\ell, A}\|_{L^2(\R^{2d})}^2\;.$$
    By adding $\|(\tilde{\mathcal{O}}_{\ell})u\|_{L^2(\R^{2d})}^{2}=\sum_{m\in\Z^d} \|(\tilde{\mathcal{O}}_{\ell})v_{m,\ell,A}\|_{L^2(\R^{2d})}^{2}$\,, we deduce
    $$
    \|u\|_{\tilde{\mathcal{W}}^{2/3},\ell}^2\leq C_{\psi}\left[\sum_{m\in\Z^d} \|v_{m,\ell,A}\|_{\tilde{\mathcal{W}}^{2/3},\ell}^2 +\frac{1}{A^{4/3}}\|2^{2\ell/3}v_{m,\ell, A}\|_{L^2(\R^{2d})}^2\right]\,,
    $$
    which is not exactly the seeked inequality expressed in terms of the $u_{m,\ell,A}$\,.
By setting $\tilde{v}_{m,\ell,A}=2^{\ell d/2}v_{m,\ell,A}(q, 2^{\ell}p)$ we have on one side
$$
 \|v_{m,\ell,A}\|_{\tilde{\mathcal{W}}^{2/3},\ell}^2= \|\tilde{v}_{m,\ell,A}\|_{\tilde{\mathcal{W}}^{2/3},1}^2
$$
while on the other side
\begin{eqnarray*}
\tilde{v}_{m,\ell, A}(q,p)&=&2^{\ell d/2}\psi_{m,\ell,A}(q) u(q,2^{\ell}p)=2^{\ell d/2}[\mathcal{U}_{\ell,m}^{-1}u_{m,\ell,A}](q, 2^{\ell}p)\\
&=&
\frac{2^{\ell d}}{\sqrt{\mathrm{det}(g(2^{\ell}q))}} u_{m,\ell,A}( 2^{\ell}q,g(2^{\ell}q)^{-1}2^{\ell}p)
=[\hat{\mathcal{U}}_{\ell,m}\hat{u}_{m,\ell,A}](q,p)
\\
\text{with}&&
\hat{u}_{m,\ell,A}(q, p) = 2^{\ell d}u_{m,\ell,A}(2^{\ell}q,2^{\ell}p)\quad,
\quad [\hat{\mathcal U}_{\ell,m}w](q,p)=\frac{1}{\sqrt{\mathrm{det}(g(q))}} w(q,g(q)^{-1}p)\,.
\end{eqnarray*}
Proposition~\label{pr:Wsplatchange} gives $\|\hat{\mathcal U}_{\ell,m}w\|_{\tilde{\mathcal{W}}^{2/3},1}\leq C_{g,\psi}^{(1)} \|w\|_{\tilde{\mathcal{W}}^{2/3},1}$ and
$$
\|\tilde{v}_{m,\ell,A}\|_{\tilde{\mathcal{W}}^{2/3},1}^2\leq [C_{g,\psi}^{(1)}]^2 \|\hat{u}_{m,\ell,A}\|_{\tilde{\mathcal{W}}^{2/3},1}^2
\leq  [C_{g,\psi}^{(1)}]^2 \|(\tilde{\mathcal{O}}_{\ell})^{2/3}u_{m,\ell,A}\|_{L^2(\R^{2d})}^2 + \||2^{\ell}D_q|^{2/3}u_{m,\ell,A}\|_{L^2(\R^{2d})}^2\,,
$$
while
$$
\|v_{m,\ell,A}\|_{L^2(\R^{2d})}=\|\tilde{v}_{m,\ell,A}\|_{L^2(\R^{2d})}=\|\hat{u}_{m,\ell,A}\|_{L^2(\R^{2d})}=\|u_{m,\ell,A}\|_{L^2(\R^{2d})}\,.
$$
This ends the proof after choosing the constant $C_{g,\psi}\leq 1$ large enough.
\end{proof}
\section{Euclidean Case}
\label{sec:EuclCase}
In this section we consider the scalar euclidean case indexed by two parameters $b,h>0$ with the Kramers-Fokker-Planck operator
\begin{equation}
  \hat{\mathcal{P}}_{b,h,0}=\underbrace{\frac{1}{2}(-h^2\Delta_{p} + \frac{|p|^{2}}{b^2})}_{=\hat{\mathcal{O}}_{b,h}} + \frac{1}{b}\underbrace{\sum_{j=1}^{d} p_{j}\frac{\partial}{\partial q^{j}}}_{=i p\cdot D_q}
\end{equation}
on $\R^{2d} = \R^d_q \times \R^d_p$ where $(q,p) = (q^1, \ldots, q^d, p_1, \ldots, p_d)$\,. 

\begin{proposition} \label{pr:SubellipticEstimateEuclidianCase}
    There exists a universal constant $C\geq 1$ such that the inequality 
    \begin{align}
    \nonumber
   C\| (\frac{h}{b}+ \hat{\mathcal{P}}_{b,h,0} -ih\lambda)u \|^{2}_{L^2(\R^{2d})} \geq &\| (\frac{h}{b}+\hat{\mathcal{O}}_{b,h})u\|^{2}_{L^2(\R^{2d})} + \| (\frac{1}{b}p\cdot D_q  -h\lambda)u\|^{2}_{L^2(\R^{2d})} \\
   &\quad + \|(|\frac{h}{b}D_{q}|^{\frac{2}{3}}+\frac{h}{b})u\|^{2}_{L^2(\R^{2d})} + 
 \left\|
   \left(
     h^{2}\frac{|\lambda|}{\sqrt{hb}+|p|}
   \right)^{\frac{2}{3}} u
 \right\|^{2}_{L^2(\R^{2d})}
       \label{eq:LowerBoundWithParameterbh}
\end{align}
holds for all $u \in C^\infty_0(\R^{2d};\C)$\,, all $\lambda \in \R$ and all $(b,h) \in (0,\infty)^2$\,.
\end{proposition}
\begin{proof}
A unitary change of scale transforms  $(q,p,D_q,D_p)$ into $ (\sqrt{\frac{b}{h}}q, \sqrt{bh}p, \sqrt{\frac{h}{b}} D_q, \frac{1}{\sqrt{hb}}D_p)$  and $\hat{\mathcal{P}}_{b,h,0}$ into $\frac{h}{b}\hat{\mathcal{P}}_{1,1,0}$\,. The problem is thus reduced to the case $b=h=1$ and the final result will be obtained after multiplying both sides of the specific inequality by $h^2/b^2$\,.\\
Let us introduce some simplified notations:
\begin{itemize}
    \item The partial Fourier transform with respect to the variable  $q$ is normalized as
    \begin{align}
    \mathcal{F}_{q \mapsto \xi} u(\xi, p) = \int_{\R^d} e^{-i q \cdot \xi} u(q,p) \, dq, \ \ \ \ u \in C^\infty_0(\R^{2d})\,.
\end{align}
It is unitary from $L^2(\R^{2d}_{q,p},dqdp;\C)$ onto $L^2(\R^{2d}_{\xi,p},\frac{d\xi}{(2\pi)^{d}}dp; \C)$\,.
\item The operator $\hat{\mathcal{P}}_{1,1,0}$ is simply denoted by $\hat{P}$ and we set
\begin{align}
    \nonumber
    \hat{P}&=\hat{\mathcal{P}}_{1,1,0}=\underbrace{\frac{1}{2}(-\Delta_{p} + |p|^{2})}_{=\mathcal{O}} + i p\cdot D_q\\
    \label{conjugation_partial_FT}
    \widetilde{P} &:= \mathcal{F}_{q \mapsto \xi} \circ \hat{P} \circ \left(\mathcal{F}_{q \mapsto \xi}\right)^{-1} = \int_{\R^d}^{\oplus}\underbrace{\frac{1}{2}(-\Delta_p + \abs{p}^2)+i (p\cdot\xi)}_{=\widetilde{P}_{\xi}}~\frac{d\xi}{(2\pi)^d}\,,
\end{align}
in the direct integral decomposition $L^2(\R^{2d}_{\xi,p},\frac{d\xi}{(2\pi)^d}dp;\C)=\int_{\R^d}^{\oplus}L^2(\R^d,dp;\C)~\frac{d\xi}{(2\pi)^d}$\,.
\item The harmonic oscillator hamiltonian $\mathcal{O}=\frac{-\Delta_p+|p|^2}{2}$ is decomposed as $\mathcal{O}=\sum_{j=1}^{d}\mathcal{O}_j$ with 
\begin{align*}
    \mathcal{O}_j = -\frac{1}{2} \frac{\p^2}{\p p_j^2} + \frac{1}{2} p_j^2\,.
\end{align*}

\item For any fixed $\xi \in \R^d$, there exists an orthogonal matrix $\mathcal{R}_\xi \in O(d)$ such that
\begin{align} \label{definition_of_R_xi}
    \mathcal{R}^T_\xi(\xi) = \left(\abs{\xi}, 0, \ldots, 0 \right).
\end{align}
The corresponding unitary pullback on functions $(\mathcal{R}^T_\xi)^*: L^2(\R^d) \mapsto L^2(\R^d)$ is given by
\begin{align*}
    (\mathcal{R}^T_\xi)^* u(p) = u(\mathcal{R}^T_\xi p), \ \ \ \ u \in L^2(\R^d),
\end{align*}
and we let $(\mathcal{R}_\xi)^* := \left(\left(\mathcal{R}^T_\xi \right)^*\right)^{-1}$ be the inverse of $\left(\mathcal{R}^T_\xi\right)^*$.
For $\xi \in \R^d$, let
\begin{align} \label{conjugated_scalar_KFP_operator}
    \widetilde{P}_{\xi, \mathcal{R}} := (\mathcal{R}^T_\xi)^* \circ \widetilde{P}_\xi \circ (\mathcal{R}_\xi)^* = \frac{1}{2} \left(-\Delta_p + \abs{p}^2 \right) + ip_1 \abs{\xi}
\end{align}
be the conjugation of $\widetilde{P}_\xi$ by $(\mathcal{R}^T_\xi)^*$\,. From (\ref{conjugated_scalar_KFP_operator}), it is clear that
\begin{align}
    \widetilde{P}_{\xi, \mathcal{R}} = \mathcal{O}_1 + ip_1 \abs{\xi} + \sum_{j \neq 1} \mathcal{O}_j.
\end{align}
\end{itemize}
We now turn our attention to the topic of obtaining lower bounds for the quantity
\begin{align}
    \norm{(1+\hat{P}-i\lambda)u}_{L^2(\R^{2d})}
\end{align}
when $u \in C^\infty_0(\R^{2d})$ and $\lambda \in \R$\,. We begin by observing, via a straightforward calculation, that
\begin{align}
\nonumber
    \| (1+\mathcal{O}_1+i( p_1 |\xi| -\lambda) )u \|_{L^2(\R^d)}^{2}=& 
  \left\|
    \left(1-\frac{1}{2}\frac{\p^2}{\p p_1^2}+i(p_{1}|\xi|-\lambda)\right)u
  \right\|_{L^2(\R^d)}^{2} 
  \\
  \label{eq:ExpansionInNormOfThe1DHatP}
  &\quad + 
  \left\|
    \frac{p_{1}^{2}}{2} u
  \right\|_{L^2_p}^{2} + \| p_{1}u\|_{L^2(\R^d)}^{2} + 
  \frac{1}{2}\left\|
    p_{1} \frac{\partial}{\partial p_{1}}u
  \right\|_{L^2(\R^d)}^{2} - \frac{1}{2}\|u\|_{L^2(\R^d)}^2
\end{align}
for any $u \in C^\infty_0(\R^d)$, $\xi \in \R^d$, and $\lambda \in \R$\,. Meanwhile, an integration by parts argument gives
\begin{equation}\label{eq:InequalityIpp1D}
  \| (1-\frac{1}{2} \frac{\p^2}{\p p^2_1}+i(p_{1}|\xi|-\lambda)) u \|^{2}_{L^2(\R^d)} \geq \|u\|^{2}_{L^2(\R^d)} +\frac{1}{2} \| \frac{\partial}{\partial p_{1}}u\|^{2}_{L^2(\R^d)}
\end{equation}
for every $u \in C^\infty_0(\R^d)$, $\xi \in \R^d$, and $\lambda \in \R$\,. On the other hand, from Proposition \ref{eq:EstimateOn1DComplexAiryOperator}, we know that there is a universal constant $C_0>0$\,, such that
\begin{align} \nonumber
    C_0 \| 1 - \frac{1}{2} \Delta_{p_1} + i(p_1 |\xi|-\lambda) u \|^2_{L^2(\R^d)}\geq &\| \frac{1}{2} \Delta_{p_1} u \|^2_{L^2(\R^d)} 
    + \| (p_1 |\xi| -\lambda ) u \|^2_{L^2(\R^d)}\\
    \label{Airy_inequality}
    &\quad  + (|\xi|^{\frac{2}{3}}+1)^2 \|u\|^2_{L^2(\R^d)}  + \left\|\left(\frac{|\lambda|}{1+|p_1|}\right)^{\frac{2}{3}}u\right\|^2_{L^2(\R^d)} 
\end{align}
for every $u \in C^\infty_0(\R^d)$, $\xi \in \R^d$, and $\lambda \in \R$\,. Since
\begin{align*}
     \|(1+\mathcal{O}_1)u \|_{L^2(\R^d)}^2 =& \frac{1}{2} \|u\|^2_{L^2(\R^d)} + \| p_1 u\|^2_{L^2(\R^d)}\\ 
    &\quad + \|\frac{p_1^2}{2} u \|^2_{L^2(\R^d)} + \frac{1}{2}\|p_1 \frac{\partial}{\partial p_1} u \|^2_{L^2(\R^d)} + \|\frac{\partial}{ \partial p_1} u \|^2_{L^2(\R^d)} + \| \frac{1}{2} \frac{\p^2}{\p p_1^2} u\|^2_{L^2(\R^d)}
\end{align*}
for every $u \in C^\infty_0(\R^d)$, we deduce from (\ref{eq:ExpansionInNormOfThe1DHatP}), (\ref{eq:InequalityIpp1D}), and (\ref{Airy_inequality}) that there is a universal constant $C_1\geq 1$ such that
\begin{align}
\nonumber
    C_1\| (1+\mathcal{O}_1+i(p_{1}|\xi|-\lambda)) u \|^{2}_{L^2(\R^d)} \geq &\| (1+\mathcal{O}_{1})u\|^{2}_{L^2(\R^d)}+\|(p_{1}|\xi| - \lambda)u\|^{2}_{L^2(\R^d)}
    \\
    \label{lower_bound_for_O1}
    &\quad + ( |\xi|^{\frac{2}{3}}+1)^{2} \| u \|^{2}_{L^2(\R^d)}+ 
  \left\|
    \left(
      \frac{|\lambda|}{1+|p_{1}|}
    \right)^{\frac{2}{3}}u
  \right\|^{2}_{L^2(\R^d)}
\end{align}
for every $u \in C^\infty_0(\R^d)$, $\xi \in \R^d$, and $\lambda \in \R$\,.

Next, we observe that elementary algebraic manipulations give the following identity:
\begin{align} \label{harmonic_oscillator_identity}
    (1+\mathcal{O})^2-(1+\mathcal{O}_1)^2 = \left(\sum_{j \neq 1} \mathcal{O}_j \right)^2 + 2(1+\mathcal{O}_1) \sum_{j \neq 1} \mathcal{O}_j\geq \left(\sum_{j \neq 1} \mathcal{O}_j \right)^2\,.
\end{align}
A straightforward computation using (\ref{harmonic_oscillator_identity}) gives that
\begin{align} \label{positive_operator_identity}
\begin{split}
    (1+\widetilde{P}_{\xi, \mathcal{R}}-i\lambda)^*(1+\widetilde{P}_{\xi, \mathcal{R}}-i\lambda) &= (1+\mathcal{O}_1-i(p_1 \abs{\xi}-\lambda))(1+\mathcal{O}_1+i(p_1 \abs{\xi}-\lambda)) \\
    & \ \ \ \ +(1+\mathcal{O})^2-(1+\mathcal{O}_1)^2
\end{split}
\end{align}
holds for every $\xi \in \R^d$ and $\lambda \in \R$. As a consequence, we have
\begin{align} \label{lower_bound_p_tilde}
    \norm{(1+\widetilde{P}_{\xi, \mathcal{R}}-i\lambda)u}^2_{L^2(\R^d)} \geq \norm{(1+\mathcal{O}_1+i(p_1 \abs{\xi}-\lambda))u}^2_{L^2(\R^d)} + \norm{\sum_{j\neq 1}\mathcal{O}_j u}^{2}_{L^2(\R^d)}
\end{align}
for every $u \in C^\infty_0(\R^d)$, $\xi \in \R^d$, and $\lambda \in \R$. Combining (\ref{lower_bound_for_O1}) and (\ref{lower_bound_p_tilde}) then leads to
\begin{align} 
\nonumber
    C_1\| (1+\widetilde{P}_{\xi,\mathcal{R}}-i\lambda)u \|^{2}_{L^2(\R^d)} &\geq \| \sum_{j\neq 1}\mathcal{O}_j u\|^{2}_{L^2(\R^d)} +\|(1+\mathcal{O}_{1})u\|^2_{L^2(\R^d)} +\| (p_{1}|\xi|-\lambda)u\|^{2}_{L^2(\R^d)} 
    \\
    \nonumber
    &\hspace{1cm}+ (|\xi|^{\frac{2}{3}}+1)^{2}\|u\|^{2}_{L^2(\R^d)} + 
 \left\|
   \left(
     \frac{|\lambda|}{1+|p|}
   \right)^{\frac{2}{3}} u
 \right\|^{2}_{L^2(\R^{d})}
 \\
 \nonumber
 &\geq \frac{1}{2}\| (1+\mathcal{O}) u\|^{2}_{L^2(\R^d)} +\| (p_{1}|\xi|-\lambda)u\|^{2}_{L^2(\R^d)} 
    \\
    \label{estimate_for_P_xi_R}
    &\hspace{1cm}+ (|\xi|^{\frac{2}{3}}+1)^{2}\|u\|^{2}_{L^2(\R^d)} + 
 \left\|
   \left(
     \frac{|\lambda|}{1+|p|}
   \right)^{\frac{2}{3}} u
 \right\|^{2}_{L^2(\R^{d})}
\end{align}
for every $u \in C^\infty_0(\R^d)$, $\xi \in \R^d$, and $\lambda \in \R$\,. From (\ref{definition_of_R_xi}), (\ref{conjugated_scalar_KFP_operator}), (\ref{estimate_for_P_xi_R}), and the unitarity of $\left(\mathcal{R}^T_\xi \right)^*$, we see that there is a universal constant $C=2C_1\geq 1$ such that
\begin{align} 
    \nonumber
    C\| (1+\widetilde{P}_{\xi}-i\lambda)u \|^{2}_{L^2(\R^d)} \geq & \| (1+\mathcal{O})u\|^{2}_{L^2(\R^d)} + \| (p\cdot\xi -\lambda)u\|^{2}_{L^2(\R^d)} 
    \\\label{estimate_P_xi}
    &\quad + (|\xi|^{\frac{2}{3}}+1)^{2}\|u\|^{2}_{L^2(\R^d)}
    + 
 \left\|
   \left(
     \frac{|\lambda|}{1+|p|}
   \right)^{\frac{2}{3}} u
 \right\|^{2}_{L^2(\R^d)}
\end{align}
for every $u \in C^\infty_0(\R^d)$, $\xi \in \R^d$, and $\lambda \in \R$. Using (\ref{conjugation_partial_FT}), (\ref{estimate_P_xi}) and the unitarity of $\mathcal{F}_{q\to\xi}$\,, we obtain
  \begin{align*}
    C\| (1+\hat{P}-i\lambda)u \|^{2}_{L^2(\R^{2d})} \geq & \| (1+\mathcal{O})u\|^{2}_{L^2(\R^{2d})} + \| (p\cdot D_{q}-\lambda)u\|^{2}_{L^2(\R^{2d})}
    \\&\quad
 + \|(|D_{q}|^{\frac{2}{3}}+1)u\|^{2}_{L^2(\R^{2d})}    + 
 \left\|
   \left(
     \frac{|\lambda|}{1+|p|}
   \right)^{\frac{2}{3}} u
 \right\|^{2}_{L^2(\R^{2d})}
\end{align*}
for all $u \in C^\infty_0(\R^{2d})$ and $\lambda \in \R$\,. 
The change of scale introduced in the beginning of this proof  says for $\lambda$ replaced by $b\lambda$
 \begin{align*}
    C\| (1+\frac{b}{h}\hat{\mathcal{P}}_{b,h,0}-ib\lambda)u \|^{2}_{L^2(\R^{2d})} \geq & \| (1+\frac{b}{h}\hat{\mathcal{O}}_{b,h})u\|^{2}_{L^2(\R^{2d})} + \| (\frac{1}{h}p\cdot D_{q}-b\lambda)u\|^{2}_{L^2(\R^{2d})}
    \\&\quad
     + \|(|\sqrt{\frac{b}{h}}D_{q}|^{\frac{2}{3}}+1)u\|^{2}_{L^2(\R^{2d})} + 
 \left\|
   \left(
     \frac{|b\lambda|}{1+\frac{|p|}{\sqrt{hb}}}
   \right)^{\frac{2}{3}} u
 \right\|^{2}_{L^2(\R^{2d})}
\end{align*}
for all $u \in C^\infty_0(\R^{2d})$ and $\lambda \in \R$\,. Multiplying both sides by $(\frac{h}{b})^2$ ends the proof.
\end{proof}

\section{Final proof of Theorem~\ref{th:mainOne}}
\label{sec:finalproof}

We now  collect all the information given by the localization techniques of Section~\ref{sec:localization} and the accurate estimates for the euclidean model in Section~\ref{sec:EuclCase}.
The first result will be the derivation of the subelliptic estimate for the local model  $\tilde{\mathcal{P}}_{b,\ell,m}=\frac{1}{b^2}\tilde{\mathcal{O}}_{\ell}+\frac{1}{b}\mathcal{Y}_{\ell,m}$ introduced in \eqref{eq:deftildePbellm}  at the end of Subsection~\ref{subsec:NormalCoordinate} from the subelliptic estimates for the euclidean model.
The second result is about the other local operator  $\mathcal{P}_{b,\ell,m}=\frac{1}{b^2}\mathcal{O}_{\ell,m}+\frac{1}{b}\mathcal{Y}_{\ell,m}$ introduced in \eqref{eq:defPbellm}\eqref{eq:defOellm}\eqref{eq:defYellm}.
These preliminary results hold for all momenta $p\sim 2^{\ell}$ and arbitrary values of the intermediate parameter $A>0$ introduced in the grid partition. 
Then, in the third paragraph, we consider the case of large momenta or large $\ell$ and the summation with respect to the grid index $m\in\Z^d$ will hold for $A\geq 1$ large enough. Here the intermediate parameter $A$ will be fixed to $A=A_{\infty}(b)\geq 1$ large enough according to the value of $b>0$ and the geometric data.\\
Once $A_\infty(b)\geq 1$ is fixed, the fourth paragraph collects the information when the momentum $p\sim 2^{\ell}$ is bounded by $C_{A_\infty(b),b}$\,. For this part the term $p\times p\times \partial_p$\,, estimated by $O(C_{A_\infty(b),b})\partial_p$\,, is controlled  by a simple integration by parts argument provided $\kappa_b$ is large enough. The summation with respect to the grid index $m\in\Z^d$ will be done by choosing another value for the intermediate parameter $A=A_0(b)$ with $A_0(b)>0$ small enough according to the value of $b>0$ and the geometric data.\\
Finally all of the summations with respect to $\ell\geq -1$ are carried out.

\subsection{Comparison of the local model $\widetilde{\mathcal{P}}_{b,\ell,m}$ with the euclidean case}
\label{sec:localtPbellm}
We write general local subelliptic estimates for the local operator $\tilde{\mathcal{P}}_{b,\ell,m}=\frac{1}{b^2}\tilde{\mathcal{O}}_{\ell}+\frac{1}{b}\mathcal{Y}_{\ell,m}$ introduced in \eqref{eq:deftildePbellm}  at the end of Subsection~\ref{subsec:NormalCoordinate} and parametrized by the dyadic scale $2^{\ell}$, $b>0$, the grid index $m\in \Z^d$ and the constant grid scaling $A>0$\,. 

\begin{proposition}\label{pr:LowerBoundInLocalizedGeneralCase}
Let $C\geq 1$ be the universal constant given by Proposition~\ref{pr:SubellipticEstimateEuclidianCase} for the euclidean metric.\\
There exists a constant $C_{g,\psi}^{(3)}>0$\,, depending only on the metric $g$ and the function $\psi$\,, such that for all $(A,b)\in (0,+\infty)^2$\,, $\kappa_b \geq C_{g,\psi}^{(3)}Ab+1$ implies the inequalities
\begin{equation}\label{pr:LowerBoundInLocalizedGeneralCaseIPP}
    \left\| \left(\frac{\kappa_{b}}{b^{2}} +  \tilde{\mathcal{P}}_{b,\ell,m} - \frac{i\lambda}{b} \right) u  \right\|_{L^2(\R^{2d})}^{2} \geq \frac{1}{2^{16}b^4}\|2^{2\ell} u\|_{L^2(\R^{2d})}^2
\end{equation}
and
  \begin{align}\label{eq:subellipticEstimateInAlmostEuclideanCase}
    \begin{split}
4C\left( 1 + C_{g,\psi}^{(3)}b^2A^2\right)\left\| \left(\frac{\kappa_{b}}{b^{2}} +  \tilde{\mathcal{P}}_{b,\ell,m} - \frac{i\lambda}{b} \right) u  \right\|_{L^2(\R^{2d})}^{2} & \geq  \left\| \frac{1}{b^2}\left(\kappa_b+\tilde{\mathcal{O}}_{\ell}\right) u\right\|_{L^2(\R^{2d})}^2 
+ \left\| \frac{1}{b}\left(\mathcal{Y}_{\ell,m}-i \lambda \right)u \right\|_{L^2(\R^{2d})}^2
  \\ &\hspace{-3cm} + \left\| \left( \left|\frac{2^{\ell}}{b^2}D_{q} \right|^{\frac{2}{3}} + \frac{1}{b^2}\right) u  \right\|_{L^2(\R^{2d})}^{2} + \left\| \left( \frac{1}{b^2} \frac{|\lambda|}{1+2^{\ell}|p|} \right)^{\frac{2}{3}} u  \right\|_{L^2(\R^{2d})}^{2} \,,
    \end{split}
\end{align}
\begin{eqnarray*}
    \text{when~either}&& u \in \mathcal{C}^\infty_0((B(0,\hat{C}'_{g,\psi}A)\times \mathbb{R}^d)\cap S_{0,8}',\mathbb{C})\quad\text{and}\quad (\lambda,\ell)\in \mathbb{R}\times \mathbb{N}\,,\\
\text{or}&& u \in \mathcal{C}^\infty_0((B(0,\hat{C}'_{g,\psi}A)\times \mathbb{R}^d)\cap S_{-1,8}',\mathbb{C})\quad \text{and}\quad \lambda\in \mathbb{R}\quad\text{for}\quad  \ell=-1\,.
\end{eqnarray*}
\end{proposition}

Remember the notations of Subsection \ref{subsec:NormalCoordinate}

\begin{itemize}
    \item $ h=\frac{1}{b2^{2\ell}} $ ,
\item $
  \frac{\kappa_{b}}{b^{2}} +  \tilde{\mathcal{P}}_{b,\ell,m} - \frac{i\lambda}{b}=  2^{2\ell}
  \left(
    \frac{\kappa_{b}h}{b} + \hat{\mathcal{P}}_{b,h,f} - i h \lambda
\right)=\frac{1}{bh}\left(
    \frac{\kappa_{b}h}{b} + \hat{\mathcal{P}}_{b,h,f} - i h \lambda
\right)\;,
$
\item $\tilde{\mathcal{O}}_{\ell}=\frac{b}{h}\hat{\mathcal{O}}_{b,h}$ with $\hat{\mathcal{O}}_{b,h}=\frac{-h^2\Delta_p+\frac{|p|^2}{b^2}}{2}$\,,
\item $\frac{1}{b}\mathcal{Y}_{\ell,m}=2^{2\ell}(\frac{1}{b}p\cdot \partial_q +hf^{ij}_{k}(q,\ell)p_i p_j\partial_{p_k})=\frac{1}{bh}(\frac{1}{b}p\cdot \partial_q +hf^{ij}_{k}(q,\ell)p_i p_j\partial_{p_k})$\,,
\item $S'_{1,8}=\{(q,p)\in\R^{2d}, \frac{1}{8}\leq |p|\leq 8\}$ and $S'_{0,8}=\{(q,p)\in\R^{2d}, |p|\leq 8\}$\,.
\end{itemize}

The result of Proposition~\ref{pr:LowerBoundInLocalizedGeneralCase} is actually deduced from the same results for the operator $\hat{\mathcal{P}}_{b,h,f}$\,. This will be done in two steps.

\begin{proposition}\label{proposition:RefinedIntegrationByPart}
There exists a constant $C_{g,\psi}^{(2)}>0$\,, depending on the metric $g$ and the function $\psi$\,, such for all $(A,b)\in (0,+\infty)^2$  and  for $\kappa_b \geq 1+C_{g,\psi}^{(2)}Ab$ the inequalities   
\begin{eqnarray}
\label{eq:IPPPbhf}
  &&\mathrm{Re}~ \, \left\langle \left( \frac{h\kappa_b}{b} + \hat{\mathcal{P}}_{b,h,f} - ih\lambda \right) u , u \right\rangle_{L^2} \geq  \frac{1}{2^7 b^2}  \|u\|_{L^2(\R^{2d})}^2 + \frac{1}{2} \sum_j \|hD_{p_j}u \|_{L^2(\R^{2d})}^2\\
    \label{eq:RefinedIppeq}
    &&  \hspace{-0.5cm}\left\| \left( \frac{h \kappa_{b}}{b} + \hat{\mathcal{P}}_{b,h,f} -i h\lambda \right) u \right\|_{L^2}^{2}\geq \frac{1}{2^{14} b^4} \left\| u \right\|_{L^2(\R^{2d})}^{2} +  \frac{1}{2^8 b^2} \sum_j \left\| hD_{p_{j}}u \right\|_{L^2(\R^{2d})}^{2}.
  \end{eqnarray}
  hold true
\begin{eqnarray*}
    \text{when~either}&& u \in \mathcal{C}^\infty_0((B(0,\hat{C}'_{g,\psi}A)\times \mathbb{R}^d)\cap S_{0,8}',\mathbb{C})\quad\text{and}\quad (\lambda,\ell)\in \mathbb{R}\times \mathbb{N}\,,\\
\text{or}&& u \in \mathcal{C}^\infty_0((B(0,\hat{C}'_{g,\psi}A)\times \mathbb{R}^d)\cap S_{-1,8}',\mathbb{C})\quad \text{and}\quad \lambda\in \mathbb{R}\quad\text{for}\quad  \ell=-1\,.
\end{eqnarray*}
\end{proposition}

\begin{proof}
A straightforward computation gives
  \begin{align}
  \nonumber
    \mathrm{Re}~ \, \left\langle \left(  \frac{h\kappa_b}{b} + \hat{\mathcal{P}}_{b,h,f}-i h\lambda \right) u , u \right\rangle_{L^2} = &  \frac{h\kappa_b}{b} \|u\|_{L^2(\R^{2d})}^2 + \frac{1}{2} \sum_j \left\| hD_{p_j} u \right\|_{L^2(\R^{2d})}^2 + \frac{1}{2}\left\| \frac{|p|}{b} u \right\|_{L^2(\R^{2d})}^2 \\
    & \label{eq:straightfor} \quad + \mathrm{Re}~ \, \left\langle hf^{ij}_k p_ip_j \frac{\partial u} {\partial p_k},u \right\rangle_{L^2},
  \end{align}
  and, owing to $|p|\leq 8$ and $|f_{ij}^{k}(q)|\leq C_{g,\psi}^{(1)}A$ according to \eqref{eq:BoundForfijk},
  \begin{align} \nonumber
    \left|\mathrm{Re}~\, \left\langle hf^{ij}_k p_ip_j \frac{\partial u}{\partial p_k},u \right\rangle_{L^2}\right| &= \left|\frac{h}{2}
    \left\langle
      \left[
        f^{ij}_{k}p_{i}p_{j},\frac{\partial}{\partial p_{k}}
\right]u , u \right\rangle_{L^2}\right|= \left|-\frac{h}{2} 
      \left\langle
        \sum_k\left(
          f^{ik}_{k}p_{i}+f^{kj}_{k}p_{j}
        \right)u , u
      \right\rangle_{L^2}\right|
      \\
      \label{equality:fErrorTerm}&\leq 8h C_{g,\psi}^{(1)} A\norm{u}_{L^2(\R^{2d})}^2
  \end{align}
for any $u$ is supported in $(B(0,\hat{C}'_{g,\psi}A)\times \mathbb{R}^d)\cap S_{\ell,8}'$\, for $\ell=0$ or $\ell=-1$\,.\\
When $\ell=-1$ and $h=\frac{4}{b}$ we simply use
  \begin{align*}
    \mathrm{Re}~ \, \left\langle \left(  \frac{h\kappa_b}{b} + \hat{\mathcal{P}}_{b,h,f}-i h\lambda \right) u , u \right\rangle_{L^2} &
    \geq  4\frac{\kappa_b-8C_{g,\psi}^{(1)}Ab}{b^2} \|u\|_{L^2}^2 + \frac{1}{2} \sum_j \left\| hD_{p_j} u \right\|_{L^2(\R^{2d})}^2
  \\
  &\geq \frac{2}{b^2}\|u\|_{L^2}^2 + \frac{1}{2} \sum_j \left\| hD_{p_j} u \right\|_{L^2(\R^{2d})}^2,
  \end{align*}
if $\kappa_b\geq 1+C_{g,\psi}^{(2)}Ab$ and $C_{g,\psi}^{(2)}=16C_{g,\psi}^{(1)}$\,.\\
When $\ell\geq 0$ and $\mathrm{supp}\,u\subset (B(0,\hat{C}'_{g,\psi}A)\times \mathbb{R}^d)\cap S_{0,8}$\,, we deduce with the lower bound $|p|\geq 2^{-3}$  the inequality
  \begin{align*}
    \mathrm{Re}~ \, \left\langle \left(  \frac{h\kappa_b}{b} + \hat{\mathcal{P}}_{b,h,f}-i h\lambda \right) u , u \right\rangle_{L^2} \geq  (\frac{h\kappa_b}{b}- \frac{hC_{g,\psi}^{(2)}A}{2} + \frac{1}{2^7b^{2}}) \|u\|_{L^2}^2 + \frac{1}{2} \sum_j \left\| hD_{p_j} u \right\|_{L^2(\R^{2d})}^2,
  \end{align*}
  where again $C_{g,\psi}^{(2)}=16 C_{g,\psi}^{(1)}$ is determined by the metric $g$ and the function $\psi$\,. The assumption $\kappa_{b}\geq C_{g,\psi}^{(2)}Ab$ implies
$$
   \mathrm{Re}~ \, \left\langle \left(  \frac{h\kappa_b}{b} + \hat{\mathcal{P}}_{b,h,f}-i h\lambda \right) u , u \right\rangle_{L^2} \geq  (\frac{h\kappa_b}{2b}+\frac{1}{2^7b^{2}}) \|u\|_{L^2(\R^{2d})}^2 + \frac{1}{2} \sum_j \left\| hD_{p_j} u \right\|_{L^2(\R^{2d})}^2
   $$
   and this ends the proof of \eqref{eq:IPPPbhf}.\\
 
 The second inequality \eqref{eq:RefinedIppeq} is deduced from \eqref{eq:IPPPbhf} via the Cauchy-Schwarz inequality like in the proof of Proposition~\ref{pr:IppIneqWithRealPart}.
 
\end{proof}

We are now able to give a lower bound for the operator $\hat{\mathcal{P}}_{b,h,f}$

\begin{proposition}
Let $C\geq 1$ be the universal constant given by Proposition~\ref{pr:SubellipticEstimateEuclidianCase} for the euclidean metric.\\
There exists a constant $C_{g,\psi}^{(3)}>0$\,, depending only on the metric $g$ and the function $\psi$\,, such that for all $(A,b)\in (0,+\infty)^2$\,, 
and $\kappa_b \geq C_{g,\psi}^{(3)}Ab+1$ the inequality 
  \begin{align} 
\nonumber
C\left( 1 + C_{g,\psi}^{(3)}A^2b^2\right)\left\| \left(\frac{ h\kappa_b}{b} + \hat{\mathcal{P}}_{b,h,f} -i h\lambda \right) u  \right\|_{L^2(\R^{2d})}^{2} &\geq \frac{1}{4} \left\|\left(\frac{\kappa_b h}{b}+\hat{\mathcal{O}}_{b,h}\right) u\right\|_{L^2(\R^{2d})}^2\\
&
\hspace{-3cm}
\nonumber
+ \frac{1}{4}\left\| \left(\frac{1}{b} p\cdot\partial_q+hf^{ij}_{k}(q,\ell)p_ip_j\partial_{p_k}-i h \lambda \right)u \right\|_{L^2(\R^{2d})}^2
\\
& \hspace{-3cm}
\label{eq:LowerBoundAfterChangeOfVariablesAndAfterGridPartition}
+\frac{1}{2} \left\| \left( \left|\frac{h}{b}D_{q} \right|^{\frac{2}{3}}+ \frac{h}{b}\right) u  \right\|_{L^2(\R^{2d})}^{2} 
+
\frac{1}{2}\left\| \left( h^2 \frac{|\lambda|}{\sqrt{hb}+|p|} \right)^{\frac{2}{3}} u  \right\|_{L^2(\R^{2d})}^{2} 
\end{align}
holds
 \begin{eqnarray*}
    \text{when~either}&& u \in \mathcal{C}^\infty_0((B(0,\hat{C}'_{g,\psi}A)\times \mathbb{R}^d)\cap S_{0,8}',\mathbb{C})\quad\text{and}\quad (\lambda,\ell)\in \mathbb{R}\times \mathbb{N}\,,\\
\text{or}&& u \in \mathcal{C}^\infty_0((B(0,\hat{C}'_{g,\psi}A)\times \mathbb{R}^d)\cap S_{-1,8}',\mathbb{C})\quad \text{and}\quad \lambda\in \mathbb{R}\quad\text{for}\quad  \ell=-1\,.
\end{eqnarray*}
\end{proposition}

\begin{proof}
  We start with the inequality
\begin{align}
\label{eq:ineq1}
  \left\| \left(\frac{h\kappa_b}{b} + \hat{\mathcal{P}}_{b,h,f} - i h\lambda \right) u  \right\|_{L^2(\R^{2d})}^{2}  \geq \frac{1}{2} \left\| \left(\frac{h\kappa_b}{b}+\hat{\mathcal{P}}_{b,h,0} -i h\lambda \right) u\right\|_{L^2(\R^{2d})}^2- \left\| h f^{ij}_{k}p_{i}p_{j}\frac{\partial u}{\partial p_{k}} \right\|_{L^2(\R^{2d})}^{2}
\end{align}
and 
\begin{equation}
\label{eq:ineq2}
    \| hf^{ij}_k p_ip_j \frac{\partial u}{ \partial p_k} \|_{L^2(\R^{2d})}^2\leq [64 C_{g,\psi}^{(1)} A]^2 \sum_j\|hD_{p_j} u \|_{L^(\R^{2d})}^2\;,
\end{equation}
which comes from $|p|\leq 8$ and \eqref{eq:BoundForfijk}\,.
With $\kappa_b\geq 1$ and $C\geq 1$ given by Proposition~\ref{pr:SubellipticEstimateEuclidianCase}, the inequality \eqref{eq:LowerBoundWithParameterbh} combined with
\begin{align*}
    \norm{(\frac{h\kappa_b}{b}+Q-ih\lambda)u}_{L^2(\R^{2d})}^2&
    =\norm{(\frac{h}{b}+Q-ih\lambda)u}_{L^2(\R^{2d})}^2+\norm{\frac{h(\kappa_b-1)}{b}u}_{L^2(\R^{2d})}^2\\
    &\hspace{2cm}
    +2\underbrace{\langle \frac{h(\kappa_b-1)}{b}u, (\frac{h}{b}+\hat{\mathcal{O}}_{b,h})u\rangle}_{\geq 0}\\
    \text{for}& \quad Q=\hat{\mathcal{P}}_{b,h,0}~\text{or}~ Q=\hat{\mathcal{O}}_{b,h}=\frac{1}{2}(-h^2\Delta_p+\frac{|p|^2}{b^2})\,,
\end{align*}
 implies
\begin{align}
\label{eq:ineq3}
  \begin{split}
  C \left\| \left(\frac{ h\kappa_b}{b} + \hat{\mathcal{P}}_{b,h,0} - i h\lambda \right) u  \right\|_{L^2(\R^{2d})}^{2} \geq & \frac{1}{2} 
  \left\| \left(\frac{h\kappa_b}{b}+\hat{\mathcal{O}}_{b,h}\right) u \right\|_{L^2(\R^{2d})}^{2} 
   +  
\left\| \left(\frac{1}{b} p\cdot\partial_q-i h\lambda \right)u \right\|_{L^2(\R^{2d})}^{2}
    \\ & \quad + \left\| \left( \left|\frac{h}{b}D_{q} \right|^{\frac{2}{3}} + \frac{h}{b}\right) u  \right\|_{L^2(\R^{2d})}^{2}
  + \left\| \left( h^2 \frac{|\lambda|}{\sqrt{hb}+|p|} \right)^{\frac{2}{3}} u  \right\|_{L^2(\R^{2d})}^{2}\,.
  \end{split}
\end{align}
Combining \eqref{eq:ineq1}\eqref{eq:ineq2} and \eqref{eq:ineq3} implies
\begin{align}
\label{eq:ineginterm}
  \begin{split}
  C \left\| \left(\frac{ h \kappa_b}{b} + \hat{\mathcal{P}}_{b,h,f} - ih\lambda \right) u  \right\|_{L^2(\R^{2d})}^{2}  \geq & \frac{1}{4} \left\| \left(\frac{h\kappa_b}{b}+\hat{\mathcal{O}}_{b,h}\right) u \right\|_{L^2(\R^{2d})}^{2} + \frac{1}{2}\left\| \left(\frac{1}{b} p\cdot\partial_q-i h\lambda \right)u \right\|_{L^2(\R^{2d})}^{2} 
  \\ & \quad  + \frac{1}{2}\left\| \left( \left|\frac{h}{b}D_{q} \right|^{\frac{2}{3}} + \frac{h}{b}\right) u  \right\|_{L^2(\R^{2d})}^{2}
  + \frac{1}{2}\left\| \left( h^2 \frac{|\lambda|}{\sqrt{hb}+|p|} \right)^{\frac{2}{3}} u  \right\|_{L^2(\R^{2d})}^{2}
   \\ & \quad  - C [64C_{g,\psi}^{(1)}A]^{2}\sum_{j} \left\| hD_{p_{j}}u \right\|_{L^2(\R^{2d})}^{2} \;.
  \end{split}
\end{align}
Putting
\begin{align*}
    \left\| \left(\frac{1}{b} p\cdot\partial_q-i h\lambda \right)u \right\|_{L^2(\R^{2d})}^{2} &\geq 
    \frac{1}{2}\left\| \left(\frac{1}{b} p\cdot\partial_q +hf^{ij}_{k}p_i p_j \partial_{p_k}-i h\lambda \right)u \right\|_{L^2(\R^{2d})}^{2}
    \\
    &\hspace{1cm} -\left\| h f^{ij}_{k}p_i p_j \partial_{p_k}u \right\|_{L^2(\R^{2d})}^{2}\\
    &\geq \frac{1}{2}\left\| \left(\frac{1}{b} p\cdot\partial_q +hf^{ij}_{k}p_i p_j \partial_{p_k}-i h\lambda \right)u \right\|_{L^2(\R^{2d})}^{2}\\
    &\hspace{1cm}-[64C_{g,\psi}^{(1)}A]^{2}\sum_{j} \left\| hD_{p_{j}}u \right\|_{L^2(\R^{2d})}^{2} 
\end{align*}
into \eqref{eq:ineginterm} gives
\begin{align}
\label{eq:ineginterm2}
  \begin{split}
  C \left\| \left(\frac{ h \kappa_b}{b} + \hat{\mathcal{P}}_{b,h,f} - ih\lambda \right) u  \right\|_{L^2(\R^{2d})}^{2}  \geq & \frac{1}{4} \left\| \left(\frac{h\kappa_b}{b}+\hat{\mathcal{O}}_{b,h}\right) u \right\|_{L^2(\R^{2d})}^{2} + \frac{1}{4}\left\| \left(\frac{1}{b} p\cdot\partial_q  +hf^{ij}_{k}p_i p_j \partial_{p_k}-i h\lambda \right)u \right\|_{L^2(\R^{2d})}^{2} 
  \\ & \quad  + \frac{1}{2}\left\| \left( \left|\frac{h}{b}D_{q} \right|^{\frac{2}{3}} + \frac{h}{b}\right) u  \right\|_{L^2(\R^{2d})}^{2}
  + \frac{1}{2}\left\| \left( h^2 \frac{|\lambda|}{\sqrt{hb}+|p|} \right)^{\frac{2}{3}} u  \right\|_{L^2(\R^{2d})}^{2}
   \\ & \quad  - 2C[64C_{g,\psi}^{(1)}A]^{2}\sum_{j} \left\| hD_{p_{j}}u \right\|_{L^2(\R^{2d})}^{2} \;.
  \end{split}
\end{align}
When $\kappa_b\geq C_{g,\psi}^{(2)}Ab$\, the inequality \eqref{eq:RefinedIppeq} provides
$$
\left\| hD_{p_{j}}u \right\|_{L^2(\R^{2d})}^{2}\leq 2^9b^2\left\| \left(\frac{ h \kappa_b}{b} + \hat{\mathcal{P}}_{b,h,f} - ih\lambda \right) u  \right\|_{L^2(\R^{2d})}^{2} 
$$
and taking
$$
C_{g,\psi}^{(3)}=\max (C_{g,\psi}^{(2)},2^{22}[C_{g,\psi}^{(1)}]^2)\geq 1\quad,\quad \kappa_b\geq C_{g,\psi}^{(3)}Ab+1\,,
$$
yields the result.
\end{proof}

\subsection{Comparison of the local models $\mathcal{P}_{b,\ell,m}$ and $\widetilde{\mathcal{P}}_{b,\ell,m}$}
\label{subsec:compPbellmtPbellm}
We now deduce subelliptic estimates for the local operator $\mathcal{P}_{b,\ell,m}=\frac{1}{b^2}\mathcal{O}_{\ell,m}+\frac{1}{b}\mathcal{Y}_{\ell,m}$ introduced in \eqref{eq:defPbellm}\eqref{eq:defOellm}\eqref{eq:defYellm} from the one obtained in the previous paragraph for $\tilde{\mathcal{P}}_{b,\ell,m}$\,. It is a consequence of the upper bounds for the differences
$(\tilde{\mathcal{P}}_{b,\ell,m}-\mathcal{P}_{b,\ell,m})$
and $(\tilde{\mathcal{O}}_{\ell}-\mathcal{O}_{\ell,m})$ studied in Proposition~\ref{pr:ChangingOToTheEuclideanO}.
\begin{proposition}
\label{pr:subellPbellm}
There exists a constant $C_{g,\psi}^{(4)}\geq 1$ such that for any $A,b>0$\,, $\kappa_b\geq C_{g,\psi}^{(4)}(1+A)(1+b)$ and $2^{2\ell}\geq C_{g,\psi}^{(4)}(1+A)(1+b)A^{2}$ imply
\begin{equation}
\label{eq:subellPbellm1}
C_{g,\psi}^{(4)} \|(\frac{\kappa_b}{b^2}+\mathcal{P}_{b,\ell,m}-\frac{i\lambda}{b})u\|_{L^{2}(\R^{2d})}^2 \geq \frac{1}{b^4}\|2^{2\ell} u\|^2_{L^2(\R^{2d})} 
\end{equation}
and
\begin{align}
\nonumber
C_{g,\psi}^{(4)}(1+A)^2(1+b)^2
\left\| \left(\frac{\kappa_{b}}{b^{2}} +  \mathcal{P}_{b,\ell,m} - \frac{i\lambda}{b} \right) u  \right\|_{L^2(\R^{2d})}^{2} \geq & \left\| \frac{1}{b^2}\left(\kappa_b+\mathcal{O}_{\ell,m}\right) u\right\|_{L^2(\R^{2d})}^2 +
\left\| \frac{1}{b}\left(\mathcal{Y}_{\ell,m}-i \lambda \right)u \right\|_{L^2(\R^{2d})}^2
  \\
  \label{eq:subellPbellm2}
  &\hspace{-5cm} +\left\| \frac{1}{b^2}\left(\kappa_b+\tilde{\mathcal{O}}_{\ell}\right) u\right\|_{L^2(\R^{2d})}^2+ \left\| \left( \left|\frac{2^{\ell}}{b^2}D_{q} \right|^{\frac{2}{3}} + \frac{1}{b^2}\right) u  \right\|_{L^2(\R^{2d})}^{2} + \left\| \left( \frac{1}{b^2} \frac{|\lambda|}{1+2^{\ell}|p|} \right)^{\frac{2}{3}} u  \right\|_{L^2(\R^{2d})}^{2} \,,
\end{align}
when
\begin{eqnarray*}
    \text{when~either}&& u \in \mathcal{C}^\infty_0((B(0,\hat{C}'_{g,\psi}A)\times \mathbb{R}^d)\cap S_{0,8}',\mathbb{C})\quad\text{and}\quad (\lambda,\ell)\in \mathbb{R}\times \mathbb{N}\,,\\
\text{or}&& u \in \mathcal{C}^\infty_0((B(0,\hat{C}'_{g,\psi}A)\times \mathbb{R}^d)\cap S_{-1,8}',\mathbb{C})\quad \text{and}\quad \lambda\in \mathbb{R}\quad\text{for}\quad  \ell=-1\,.
\end{eqnarray*}
\end{proposition}
\begin{proof}
\noindent\textbf{a)} For the inequality \eqref{eq:subellPbellm1} we must reconsider the proof of Proposition~\ref{proposition:RefinedIntegrationByPart} after noticing 
$$
(bh)\mathcal{P}_{b,\ell,m}=\hat{\mathcal{P}}_{b,h,f} +\frac{-h^2(g_{ij}(2^{-\ell}q)-
\delta_{ij})\partial_{p_i}\partial_{p_j}+(g^{ij}(2^{-\ell}q)-\delta^{ij})p_ip_j/b^2}{2}\,.
$$
With $|g(2^{-\ell}q)-\mathrm{Id}|\leq C_{g,\psi}A^2 2^{-2\ell}$ we get 
\begin{multline*}
|\langle u\,,\,\frac{-h^2(g_{ij}(2^{-\ell}q)-\delta_{ij})\partial_{p_i}\partial_{p_j}+(g^{ij}(2^{-\ell}q)-\delta^{ij})p_ip_j/b^2}{2} u\rangle|\\
\leq C'_{g,\psi}A^2{2^{-2\ell}}
\left[\|hD_p u\|_{L^2(\R^{2d})}^2 +\frac{1}{b^2}\|u\|_{L^2(\R^{2d})}^2\right]
\end{multline*}
and
\begin{multline*}
\left|\langle u\,,\,(bh)(\frac{\kappa_b}{b^2}+\mathcal{P}_{b,\ell,m}-i\frac{\lambda}{b})u\rangle -\langle u\,,\,(\frac{h\kappa_b}{b}+\hat{\mathcal{P}}_{b,h,f}-ih\lambda) u\rangle\,\right|
\\
\leq 
C''_{g,\psi}A^2 2^{-2\ell}
\mathrm{Re}~{\langle u\,,\,(\frac{h\kappa_b}{b}+\hat{\mathcal{P}}_{b,h,f}-ih\lambda) u\rangle}\,.
\end{multline*}
The lower bound for $\mathrm{Re}~ \langle u\,,\,(bh)(\frac{\kappa_b}{b^2}+\mathcal{P}_{b,\ell,m}-i\frac{\lambda}{b})u\rangle $
and by Cauchy-Schwarz for $\|(bh)(\frac{\kappa_b}{b^2}+\mathcal{P}_{b,\ell,m}-i\frac{\lambda}{b})u\|^2_{L^2(\R^{2d})}$ are thus deduced  from Proposition~\ref{proposition:RefinedIntegrationByPart} when $2^{2\ell}\geq 2C''_{g,\psi}A^2$\,. The conditions and the inequality \eqref{eq:subellPbellm1} are thus satisfied by taking  $C_{g,\psi}^{(4)}\geq 2C''_{g,\psi}$ large enough.\\
\noindent\textbf{b)} Let us consider \eqref{eq:subellPbellm2}.
We recall the inequality \eqref{eq:upperBoundForReplacingOToTheEuclideanO}
\begin{equation*}
  \| \frac{1}{b^{2}} 
\left(
  \mathcal{O}_{\ell,m} - \tilde{\mathcal{O}}_{\ell}
\right) u  \|_{L^2(\R^{2d})} \leq C_{g,\psi} A^{2}2^{-2\ell} \|
\left(
  \frac{\kappa_{b}}{b^{2}} + \frac{1}{b^{2}}\tilde{\mathcal{O}}_{\ell}
\right) u \|_{L^2(\R^{2d})} \;,
\end{equation*}
which implies 
\begin{eqnarray*}
&&
\|
\left(
  \frac{\kappa_{b}}{b^{2}} + \frac{1}{b^{2}}\tilde{\mathcal{O}}_{\ell}
\right) u \|^2_{L^2(\R^{2d})} \geq \frac{1}{2}
\|
\left(
  \frac{\kappa_{b}}{b^{2}} + \frac{1}{b^{2}}\mathcal{O}_{\ell,m}
\right) u \|^2_{L^2(\R^{2d})}\\
\text{and}&&
\|
\left(
  \frac{\kappa_{b}}{b^{2}} + \frac{1}{b^{2}}\tilde{\mathcal{O}}_{\ell}
\right) u \|^2_{L^2(\R^{2d})} \geq \frac{1}{4}
\|
\left(
  \frac{\kappa_{b}}{b^{2}} + \frac{1}{b^{2}}\mathcal{O}_{\ell,m}
\right) u \|^2_{L^2(\R^{2d})}
+\frac{1}{4}\|
\left(
  \frac{\kappa_{b}}{b^{2}} + \frac{1}{b^{2}}\tilde{\mathcal{O}}_{\ell}
\right) u \|^2_{L^2(\R^{2d})}
\end{eqnarray*}
as soon as $2^{2\ell}\geq \frac{C_{g,\psi} A^2}{\sqrt{2}-1}$\,.\\
Proposition~\ref{pr:LowerBoundInLocalizedGeneralCase} holds under the sufficient condition $\kappa_b\geq C^{(3)}_{g,\psi}Ab+1$ which can be simplified into $\kappa_b\geq C^{(3)}_{g,\psi}(1+A)(1+b)$  while the left-hand side of \eqref{eq:subellipticEstimateInAlmostEuclideanCase} can be replaced by
$$
4CC^{(3)}_{g,\psi} (1+A)^2(1+b)^2\left\| \left(\frac{\kappa_{b}}{b^{2}} +  \tilde{\mathcal{P}}_{b,\ell,m} - \frac{i\lambda}{b} \right) u  \right\|_{L^2(\R^{2d})}^{2}\,.
$$
Therefore Proposition~\ref{pr:LowerBoundInLocalizedGeneralCase} and Proposition~\ref{pr:ChangingOToTheEuclideanO} say
\begin{eqnarray*}
\|(\frac{\kappa_b}{b^2}+\mathcal{P}_{b,\ell,m}-i\frac{\lambda}{b})u\|_{L^2(\R^{2d})}
&\geq& 
\left(1-2\sqrt{CC^{(3)}_{g,\psi}}(1+A)(1+b)C_{g,\psi}A^2 2^{-2\ell}\right)
\|(\frac{\kappa_b}{b^2}+\tilde{\mathcal{P}}_{b,\ell,m}-i\frac{\lambda}{b})u\|_{L^2(\R^{2d})}
\\
&\geq &
\frac{1}{2}
\|(\frac{\kappa_b}{b^2}+\tilde{\mathcal{P}}_{b,\ell,m}-i\frac{\lambda}{b})u\|_{L^2(\R^{2d})}
\end{eqnarray*}
as soon as
$$
2^{2\ell}\geq 4\sqrt{C C^{(3)}_{g,\psi}}C_{g,\psi}(1+A)(1+b)A^{2}\,.
$$
When all these conditions are satisfied this proves the seeked inequality with the upper bound
$$
64 CC^{(3)}_{g,\psi}(1+A)^2(1+b)^2\|(\frac{\kappa_b}{b^2}+\mathcal{P}_{b,\ell,m}-i\frac{\lambda}{b})u\|^2_{L^2(\R^{2d})}\,.
$$
The result is thus proved by choosing
$$
C^{(4)}_{g,\psi}\geq \max(4\sqrt{CC^{(3)}_{g,\psi}}C_{g,\psi}, 64 CC^{(3)}_{g,\psi})
$$
where the right-hand side is larger than $C^{(3)}_{g,\psi}$ and for which 
$C_{g,\psi}^{(4)}(1+A)(1+b)A^{2}\geq \frac{C_{g,\psi}A^{2}}{\sqrt{2}-1}$\,.
\end{proof}
\subsection{Estimate for $|p|\sim 2^{\ell}$ large} 
\label{subsec:estimateForLargeEll}
We now prove subelliptic estimates for $\mathcal{P}_{b,\ell}$ by summing the local subelliptic estimates for $\mathcal{P}_{b,\ell,m}$\,. All the error terms coming from the partition of unity $\sum_{m\in\Z^{d}}\psi^{2}(.-m)\equiv 1$ and  studied in Subsections~\ref{sec:gridpart}, \ref{subsec:NormalCoordinate} and \ref{sec:locglobfixedell}, happen to be relatively small enough when the parameter $A=A_\infty (b)$ is much larger than $1$ and $\ell$ is large enough so that $2^{2\ell}\gg [A_{\infty}(b)]^3\times (1+b)$\,.

\begin{proposition}
\label{pr:subellPbellLarge}
Let $\mathcal{P}_{b,\ell}$ and $S_{\ell,2}=S_{0,2}\subset \Omega\times \R^{d}$ be defined respectively by \eqref{eq:defPbell} and \eqref{eq:defSell} for $\ell\in\N$\,.\\
There exists a constant $C_{g,\psi}^{(5)}\geq 1$ such that $A_{\infty}(b)= C_{g,\psi}^{(5)}(1+b)$\,, $\kappa_{b}\geq C_{g,\psi}^{(5)}A_{\infty}(b)\times (1+b)\geq [C_{g,\psi}^{(5)}]^2(1+b)^2$ and $2^{2\ell}\geq C_{g,\psi}^{(5)}[A_{\infty}(b)]^3(1+b)=[C_{g,\psi}^{(5)}]^{4}(1+b)^4$ imply
\begin{align}
\nonumber
[C_{g,\psi}^{(5)}]^3(1+b)^4
\left\| \left(\frac{\kappa_{b}}{b^{2}} +  \mathcal{P}_{b,\ell} - \frac{i\lambda}{b} \right) u  \right\|_{L^2(\R^{2d})}^{2} \geq & \left\| \frac{1}{b^2}\left(\kappa_b+\mathcal{O}_{\ell}\right) u\right\|_{L^2(\R^{2d})}^2 +
\left\| \frac{1}{b}\left(\mathcal{Y}_{\ell}-i \lambda \right)u \right\|_{L^2(\R^{2d})}^2
  \\
  \label{eq:subellPbell}
  &\hspace{-5cm} +\frac{1}{b^{8/3}}\left[\left\|(\tilde{\mathcal{O}}_{\ell})^{2/3} u\right\|_{L^2(\R^{2d})}^2+ \left\|\left|D_{q} \right|^{\frac{2}{3}} u  \right\|_{L^2(\R^{{2d}})}^{2}\right] + \left\| \left( \frac{1}{b^2} \frac{|\lambda|}{1+2^{\ell}|p|} \right)^{\frac{2}{3}} u  \right\|_{L^2(\R^{2d})}^{2} \,,
\end{align}
when
$
u \in \mathcal{C}^\infty_0(S_{0,2};\C)$\,, $\lambda\in \R$ and $b>0$\,.
\end{proposition}
\begin{proof}
  Our conditions  and $A=A_\infty(b)=C_{{g,\psi}}^{(5)}(1+b)\geq 1$ and $2^{2\ell}\geq C_{g,\psi}^{(5)}A^{3}(1+b)\geq C_{{g,\psi}}^{(5)}A^{2}$ imply $2^{-\ell}A\leq \frac{1}{\sqrt{C_{g,\psi}^{(5)}}}\leq \frac{1}{C_{g,\psi}}$ when $C_{g,\psi}\geq 1$ is the constant of
  Proposition~\ref{pr:locglobPbell} and $C_{{g,\psi}}^{(5)}$ is chosen larger than $C_{g,\psi}^{2}$\,.
With $\kappa_{b}\geq  C_{g,\psi}^{(5)}A(1+b) \geq C_{g,\psi}Ab$\,, Proposition~\ref{pr:locglobPbell} says 
$$
\sum_{m\in \Z^d}\frac{1}{8}\| (\frac{\kappa_{b}}{b^{2}}+\mathcal{P}_{b,\ell,m} - \frac{i\lambda}{b}) u_{\ell,m}\|^{2}_{L^2(\R^{2d})}-\frac{C_{g,\psi}}{A^2b^2}\norm{2^{2\ell}u_{\ell,m}}_{L^2(\R^{2d})}^2
    \leq
    \| (\frac{\kappa_{b}}{b^{2}}+\mathcal{P}_{b,\ell} - \frac{i\lambda}{b}) u \|^{2}_{L^2(\R^{2d})}\,,
    $$
    with
    $$
    \forall m\in\Z^{d}\,,\quad u_{\ell,m}=\mathcal{U}_{\ell,m}(\psi_{m,\ell,A}u)~\in \mathcal{C}^\infty_0((B(0,\hat{C}'_{g,\psi}A)\times \mathbb{R}^d)\cap S_{0,8}',\mathbb{C})\,.
    $$
By choosing $C_{g,\psi}^{(5)}$ larger than the constant $2\times C_{{g,\psi}}^{(4)}$ of Proposition~\ref{pr:subellPbellm}, the condition $\kappa_{b}\geq C_{g,\psi}^{(5)}A(1+b)\geq C_{{g,\psi}}^{(4)}(1+A)(1+b)$ and the inequality 
\eqref{eq:subellPbellm1}
imply
$$
\forall m\in\Z^{d}\,,\quad 
C_{g,\psi}^{(4)} \|(\frac{\kappa_b}{b^2}+\mathcal{P}_{b,\ell,m}-\frac{i\lambda}{b})u_{\ell,m}\|_{L^{2}(\R^{2d})}^2 \geq \frac{1}{b^4}\|2^{2\ell} u_{\ell,m}\|^2_{L^2(\R^{2d})}\,,
$$
With $\frac{C_{g,\psi}^{(4)}C_{g,\psi}b^{2}}{A^{2}}\leq \frac{1}{16}$ when $A\geq C_{g,\psi}^{(5)}(1+b)$ and $C_{g,\psi}^{(5)}$ is chosen larger than $4\sqrt{C_{g,\psi}^{(4)}C_{g,\psi}}$\,, we obtain
$$
\frac{1}{16}\sum_{m\in \Z^d}\| (\frac{\kappa_{b}}{b^{2}}+\mathcal{P}_{b,\ell,m} - \frac{i\lambda}{b}) u_{\ell,m}\|^{2}_{L^2(\R^{2d})}
    \leq
    \| (\frac{\kappa_{b}}{b^{2}}+\mathcal{P}_{b,\ell} - \frac{i\lambda}{b}) u \|^{2}_{L^2(\R^{2d})}\,,
    $$
    and the two lower bounds of $\| (\frac{\kappa_{b}}{b^{2}}+\mathcal{P}_{b,\ell,m} - \frac{i\lambda}{b}) u_{\ell,m}\|^{2}_{L^2(\R^{2d})}$ of Proposition~\ref{pr:subellPbellm} can be used for every $m\in\Z^{d}$\,.
    The first one \eqref{eq:subellPbellm1} already used gives now
$$
\frac{1}{32}\sum_{m\in \Z^d}\| (\frac{\kappa_{b}}{b^{2}}+\mathcal{P}_{b,\ell,m} - \frac{i\lambda}{b}) u_{\ell,m}\|^{2}_{L^2(\R^{2d})}+\frac{1}{C_{g,\psi}^{(4)}b^{4}}\|2^{2\ell}u_{\ell,m}\|^{2}_{L^{2}(\R^{2d})}
    \leq
    \| (\frac{\kappa_{b}}{b^{2}}+\mathcal{P}_{b,\ell} - \frac{i\lambda}{b}) u \|^{2}_{L^2(\R^{2d})}\,.
$$
The second one  \eqref{eq:subellPbellm2} with here $2A\geq (1+A)$ and  combined with the second inequality of Proposition~\ref{pr:locglobYell} it implies
    \begin{eqnarray*}
      2^{10}C_{g,\psi}^{(4)}A^{3}(1+b)\| (\frac{\kappa_{b}}{b^{2}}+\mathcal{P}_{b,\ell} - \frac{i\lambda}{b}) u \|^{2}_{L^2(\R^{2d})}
      &&\geq\|\frac{1}{b^{2}}(\kappa_{b}+\mathcal{O}_{\ell})u\|_{L^{2}(\R^{2d})}^{2}
             +\|\frac{1}{b}(\mathcal{Y}_{\ell}-i\lambda)u\|_{L^{2}(\R^{2d})}^{2}
             \\
      &&\hspace{1cm}
         +\left\| \left( \frac{1}{b^2} \frac{|\lambda|}{1+2^{\ell}|p|} \right)^{\frac{2}{3}} u  \right\|_{L^2(\R^{2d})}^{2}\\
      &&\hspace{-2cm}+\sum_{m\in\Z^{d}}\left(\frac{1}{C^{(4)}_{g,\psi}b^{4}}-\frac{C_{g,\psi}}{A^{2}b^{2}}-\frac{C_{g,\psi}}{b^{2}2^{2\ell}}\right)
         \|2^{2\ell}u_{\ell,m}\|_{L^{2}(\R^{2d})}^{2}
      \\
      &&\hspace{-1cm}  +
\left\|\frac{1}{b^{2}}(\kappa_{b}+\tilde{\mathcal{O}}_{\ell}) u_{\ell,m}\right\|_{L^2(\R^{2d})}^2+ \frac{1}{b^{8/3}}\left\|\left|2^{\ell}D_{q} \right|^{\frac{2}{3}} u_{\ell,m}  \right\|_{L^2(\R^{2d})}^{2}\,.
    \end{eqnarray*}
    Moreover $\kappa_{b}\geq C^{(5)}_{g,\psi}A(1+b)\geq (1+b)^{2}$ implies
    $\frac{1}{b^{2}}(\kappa_{b}+\tilde{\mathcal{O}}_{\ell})\geq 1$ and  interpolation gives
$$
\left\|\frac{1}{b^{2}}(\kappa_{b}+\tilde{\mathcal{O}}_{\ell}) u_{\ell,m}\right\|_{L^2(\R^{2d})}^2\geq \frac{1}{b^{8/3}}\left\|(\kappa_b+\tilde{\mathcal{O}}_{\ell})^{2/3} u_{\ell,m}\right\|_{L^2(\R^{2d})}^2\geq \frac{1}{b^{8/3}}\left\|\tilde{\mathcal{O}}_{\ell}^{2/3}u_{\ell,m}\right\|_{L^2(\mathbb{R}^{2d})}\,.
$$

Proposition~\ref{pr:locglobW23} implies
\begin{eqnarray*}
\sum_{m\in\Z^{d}}\|\tilde{\mathcal{O}}_{\ell}^{2/3}u_{\ell,m}\|_{L^{2}(\R^{2d})}^{2}+
  \||2^{\ell}D_{q}|^{2/3}u_{\ell,m}\|_{L^{2}(\R^{2d})}^{2}
  &\geq &
          \frac{1}{C_{g,\psi}}\left[
\|\tilde{\mathcal{O}}_{\ell}^{2/3}u\|_{L^{2}(\R^{2d})}^{2}+
          \||D_{q}|^{2/3}u\|_{L^{2}(\R^{2d})}^{2}
  \right]\\
  &&-\sum_{m\in\Z^{d}}\frac{1}{A^{4/3}}
     \underbrace{\|2^{2\ell/3}u_{\ell,m}\|_{L^{2}(\R^{2d})}^{2}}_{\leq \|2^{2\ell}u_{\ell,m}\|_{L^{2}(\R^{2d})}^{2}}  
\end{eqnarray*}
The complete error term in the lower bound is now bounded from below
by
$$
\frac{1}{b^{4}}\sum_{m\in\Z^{d}}\left(\frac{1}{C^{(4)}_{g,\psi}}-\frac{C_{g,\psi}b^{2}}{A^{2}}-\frac{C_{g,\psi}b^{2}}{2^{2\ell}}-\frac{b^{4/3}}{A^{4/3}}\right)
\|2^{2\ell}u_{\ell,m}\|_{L^{2}(\R^{2d})}^{2}
$$
which is non negative when
\begin{eqnarray*}
  && A=C_{g,\psi}^{(5)}(1+b) \\
  \text{and} && 2^{2\ell}\geq C_{g,\psi}^{(5)} A^3(1+b) \geq[C_{g,\psi}^{(5)}]^{4}(1+b)^{4}\geq C_{g,\psi}^{(5)}(1+b)^2
\end{eqnarray*}

with $C^{(5)}_{g,\psi}$ large enough. The inequality~\eqref{eq:subellPbell} is then obtained by taking $C_{g,\psi}^{(5)}\geq 2^{8}C_{g,\psi}^{(4)}C_{g,\psi}$\,. 
\end{proof}
\subsection{Estimate for $|p|\sim 2^{\ell}\leq C_{b,g,\psi}$}
We have fixed the value of $A_\infty(b)=C_{g,\psi}^{(5)}(1+b)\gg 1$ in Proposition~\ref{pr:subellPbellLarge} in order to get a result for all $\ell\geq \ell_{b,g,\psi}+1$ with
\begin{equation}
  \label{eq:defellb}
  2^{2\ell_{b,g,\psi}+2}\geq [C_{g,\psi}^{(5)}]^{4}(1+b)^{4}\geq 2^{2\ell_{b,g,\psi}}\,.
\end{equation}
We now consider all the bounded values of $\ell\in \left\{-1,0,1,\ldots,\ell_{b,g,\psi}\right\}$\,. Here the error terms related with the partition of unity $\sum_{m\in \Z^{d}}\psi^{2}(.-m)\equiv 1$ will be relatively small by taking a new, small enough, value for the intermediate parameter $A>0$ denoted by $A_0(b)$ and the parameter $\kappa_{b}\gg 1$\,. 
\begin{proposition}\label{pr:subellPbellSmall}
  Let $\ell_{b,g,\psi}\in\N$ be such that \eqref{eq:defellb} is satisfied.
  For all $\ell\in \left\{-1,0,1,\ldots,\ell_{b,g,\psi}\right\}$ let
  $\mathcal{P}_{b,\ell}$ and $S_{\ell,2}\subset \Omega\times \R^{d}$ be defined respectively by \eqref{eq:defPbell} and \eqref{eq:defSell}.\\
  Take $A_0(b)=\frac{1}{2C_{{g,\psi}}'(1+b)}$ where $C_{{g,\psi}}'=\max(C_{g,\psi},C_{g,\psi}^{(4)})$\,, $C_{g,\psi}\geq 1$ is given by Proposition~\ref{pr:locglobPbell} and Proposition~\ref{pr:locglobW23} while $C_{g,\psi}^{(4)}\geq 1$ is given by Proposition~\ref{pr:subellPbellm}.\\
  There exists a constant $C_{g,\psi}^{(6)}\geq 1$ such that  $\kappa_{b}\geq C_{g,\psi}^{(6)}(1+b)^{5}$ implies
\begin{align}
\nonumber
C_{g,\psi}^{(6)}(1+b)^2
\left\| \left(\frac{\kappa_{b}}{b^{2}} +  \mathcal{P}_{b,\ell} - \frac{i\lambda}{b} \right) u  \right\|_{L^2(\R^{2d})}^{2} \geq & \left\| \frac{1}{b^2}\left(\kappa_b+\mathcal{O}_{\ell}\right) u\right\|_{L^2(\R^{2d})}^2 +
\left\| \frac{1}{b}\left(\mathcal{Y}_{\ell}-i \lambda \right)u \right\|_{L^2(\R^{2d})}^2
  \\
  \label{eq:subellPbellSmall}
  &\hspace{-5cm} +\frac{1}{b^{8/3}}\left[\left\|(\tilde{\mathcal{O}}_{\ell})^{2/3} u\right\|_{L^2(\R^{2d})}^2+ \left\|\left|D_{q} \right|^{\frac{2}{3}} u  \right\|_{L^2(\R^{{2d}})}^{2}\right] + \left\| \left( \frac{1}{b^2} \frac{|\lambda|}{1+2^{\ell}|p|} \right)^{\frac{2}{3}} u  \right\|_{L^2(\R^{2d})}^{2} \,,
\end{align}
when
$
u \in \mathcal{C}^\infty_0(S_{\ell,2};\C)$\,, $\ell\in \left\{-1,0,1,\ldots,\ell_{b,g,\psi}\right\}$\,, $\lambda\in \R$ and $b>0$\,.
\end{proposition}
\begin{proof}
  The proof has the same structure as
  the one of Proposition~\ref{pr:subellPbellLarge}.
  Choose now $A=A_0(b)=\frac{1}{2C_{{g,\psi}}'(1+b)}$\,, where the constant $C_{g,\psi}'$ will be fixed later in the proof.\\
  With  $2^{-\ell}A\leq 2A\leq \frac{1}{C'_{g,\psi}(1+b)}\leq \frac{1}{C_{g,\psi}}$ which combined with $\kappa_{b}\geq C_{g,\psi}^{(6)}(1+b)^{6}\geq \frac{b}{2(1+b)}\geq C_{g,\psi}Ab$ in Proposition~\ref{pr:locglobPbell}, implies
  $$
\sum_{m\in \Z^d}\frac{1}{8}\| (\frac{\kappa_{b}}{b^{2}}+\mathcal{P}_{b,\ell,m} - \frac{i\lambda}{b}) u_{\ell,m}\|^{2}_{L^2(\R^{2d})}-\underbrace{\frac{C_{g,\psi}}{A^{2}b^2}}_{\leq \frac{4[C_{g,\psi}']^{3}(1+b)^{2}}{b^{2}}}\underbrace{\norm{2^{2\ell}u_{\ell,m}}_{L^2(\R^{2d})}^2}_{\times[C_{g,\psi}^{(5)}]^{8}(1+b)^{8}\|u\|_{L^{2}(\R^{2d})}^{2}}
    \leq
    \| (\frac{\kappa_{b}}{b^{2}}+\mathcal{P}_{b,\ell} - \frac{i\lambda}{b}) u \|^{2}_{L^2(\R^{2d})}\,,
    $$
    with
    $$
    u_{\ell,m}=\mathcal{U}_{\ell,m}(\psi_{m,\ell,A}u)~\in \mathcal{C}^\infty_0((B(0,\hat{C}'_{g,\psi}A)\times \mathbb{R}^d)\cap S_{0,8}',\mathbb{C})\,.
    $$
    for all $m\in \Z^{d}$ and all $\ell\in\left\{-1,0,1,\ldots,\ell_{b,g,\psi}\right\}$\,.\\
    The same integration by parts argument as for Proposition~\ref{pr:IppIneqWithRealPart} says that for $\kappa_{b}\geq C_{g,\psi}^{(6)}(1+b)^{6}\geq C_{g,\psi}^{(6)}(1+b)^{2}$ with $C_{g,\psi}^{(6)}\geq 1$ large enough
$$
\|(\frac{\kappa_{b}}{b^{2}}+\mathcal{P}_{b,\ell}-\frac{i\lambda}{b})u\|_{L^{2}(\R^{2d})}^{2}\geq \frac{\kappa_{b}^{2}}{16 b^{4}}\|u\|_{L^{2}(\R^{2d})}^{2}\geq \frac{[C_{g,\psi}^{(6)}]^{2}(1+b)^{10}}{16b^{2}}\|u\|_{L^{2}(\R^{2d})}^{2}\,.
$$
The bound $2^{2\ell}\leq [C_{g,\psi}^{(5)}]^{4}(1+b)^{4}$ implies
\begin{align*}
\frac{[C_{g,\psi}^{(6)}]^{2}(1+b)^{2}}{b^{2}}\sum_{m\in \Z}\|2^{2\ell}u_{\ell,m}\|_{L^{2}(\R^{2d})}^{2}&\leq
\frac{[C_{g,\psi}^{(6)}]^{2}[C_{g,\psi}^{(5)}]^{8}(1+b)^{10}}{b^{2}}\|u\|_{L^{2}(\R^{2d})}^{2}
  \\
                                                                                              &\leq
16 [C_{g,\psi}^{(5)}]^{8}\|(\frac{\kappa_{b}}{b^{2}}+\mathcal{P}_{b,\ell}-\frac{i\lambda}{b})u\|_{L^{2}(\R^{2d})}^{2}\,.
\end{align*}
The additional condition
$[C_{g,\psi}^{(6)}]^{2}\geq 8[C_{g,\psi}']^{3}$ thus implies
$$
16 [C_{g,\psi}^{(5)}]^{8}  \| (\frac{\kappa_{b}}{b^{2}}+\mathcal{P}_{b,\ell} - \frac{i\lambda}{b}) u \|^{2}_{L^2(\R^{2d})}
\geq \frac{4[C_{g,\psi}']^{3}[C_{g,\psi}^{(5)}]^{8}(1+b)^{10}}{b^{2}}\|u\|_{L^{2}(\R^{2d})}^{2}+ \underbrace{8[ C_{g,\psi}^{(5)}]^{8}}_{\geq 1+4}\| (\frac{\kappa_{b}}{b^{2}}+\mathcal{P}_{b,\ell} - \frac{i\lambda}{b}) u \|^{2}_{L^2(\R^{2d})}
$$
and
\begin{align*}
16 [C_{g,\psi}^{(5)}]^{8}  \| (\frac{\kappa_{b}}{b^{2}}+\mathcal{P}_{b,\ell} - \frac{i\lambda}{b}) u \|^{2}_{L^2(\R^{2d})}
  &\geq
    \sum_{m\in \Z^d}\frac{1}{8}\| (\frac{\kappa_{b}}{b^{2}}+\mathcal{P}_{b,\ell,m} - \frac{i\lambda}{b}) u_{\ell,m}\|^{2}_{L^2(\R^{2d})}\\
  &\hspace{2cm}
  +\frac{\kappa_{b}^{2}}{4b^4}\norm{u_{\ell,m}}_{L^2(\R^{2d})}^2
\end{align*}


Because $2^{\ell}\geq 1/2\geq C_{g,\psi}^{(4)}(1+b)A\geq C_{g,\psi}^{(4)}(1+A)(1+b)A^{2}$\,, the condition of Proposition~\ref{pr:subellPbellm} are satisfied.
Multiplying the above inequality by
$8\times C_{g,\psi}^{(4)}4(1+b)^{2}$ which is larger than $8 C_{g,\psi}^{(4)}(1+A)^{2}(1+b)^{2}\geq 1$\,, leads to
\begin{align*}
2^{9} C_{g,\psi}^{(4)}(1+b)^{2}[C_{g,\psi}^{(5)}]^{8}  \| (\frac{\kappa_{b}}{b^{2}}+\mathcal{P}_{b,\ell} - \frac{i\lambda}{b}) u \|^{2}_{L^2(\R^{2d})}
  &\geq
     \sum_{m\in \Z^d}C_{g,\psi}^{(4)}(1+A)^{2}(1+b)^{2}\| (\frac{\kappa_{b}}{b^{2}}+\mathcal{P}_{b,\ell,m} - \frac{i\lambda}{b}) u_{\ell,m}\|^{2}_{L^2(\R^{2d})}\\
  &\hspace{2cm}
    +\frac{2\kappa_{b}^{2}(1+b)^{2}}{b^4}\norm{u_{\ell,m}}_{L^2(\R^{2d})}^2
\end{align*}
Proposition~\ref{pr:subellPbellm} combined with Proposition~\ref{pr:locglobYell} leads to
\begin{eqnarray*}
  2^{9} C_{g,\psi}^{(4)}(1+b)^{2}[C_{g,\psi}^{(5)}]^{8}  \| (\frac{\kappa_{b}}{b^{2}}+\mathcal{P}_{b,\ell} - \frac{i\lambda}{b}) u \|^{2}_{L^2(\R^{2d})}
      &&\geq\|\frac{1}{b^{2}}(\kappa_{b}+\mathcal{O}_{\ell})u\|_{L^{2}(\R^{2d})}^{2}
             +\|\frac{1}{b}(\mathcal{Y}_{\ell}-i\lambda)u\|_{L^{2}(\R^{2d})}^{2}
             \\
      &&\hspace{1cm}
         +\left\| \left( \frac{1}{b^2} \frac{|\lambda|}{1+2^{\ell}|p|} \right)^{\frac{2}{3}} u  \right\|_{L^2(\R^{2d})}^{2}\\
      &&\hspace{-2cm}+\sum_{m\in\Z^{d}}\left(\frac{2\kappa_{b}^{2}(1+b)^{2}}{b^{4}}
         -\frac{C_{g,\psi}2^{4\ell}}{A^{2}b^{2}}-\frac{C_{g,\psi}2^{4\ell}}{b^{2}2^{2\ell}}\right)
         \|u_{\ell,m}\|_{L^{2}(\R^{2d})}^{2}
      \\
      &&\hspace{-1cm}  +
 \frac{1}{b^{8/3}}\left[
\left\|\tilde{\mathcal{O}}_{\ell}^{2/3} u_{\ell,m}\right\|_{L^2(\R^{2d})}^2+ \left\|\left|2^{\ell}D_{q} \right|^{\frac{2}{3}} u_{\ell,m}  \right\|_{L^2(\R^{2d})}^{2}\right]\,
    \end{eqnarray*}
    where we used that
   $\frac{\kappa_{b}+\tilde{\mathcal{O}}_{\ell}}{b^{2}}\geq 1$ when $\kappa_{b}\geq C_{g,\psi}^{(6)}(1+b)^{6}\geq C_{g,\psi}^{(6)}(1+b)^{2}$\,, as in the proof of Proposition~\ref{pr:subellPbellLarge}. 
Proposition~\ref{pr:locglobW23} implies
\begin{eqnarray*}
\sum_{m\in\Z^{d}}\|\tilde{\mathcal{O}}_{\ell}^{2/3}u_{\ell,m}\|_{L^{2}(\R^{2d})}^{2}+
  \||2^{\ell}D_{q}|^{2/3}u_{\ell,m}\|_{L^{2}(\R^{2d})}^{2}
  &\geq &
          \frac{1}{C_{g,\psi}}\left[
\|\tilde{\mathcal{O}}_{\ell}^{2/3}u\|_{L^{2}(\R^{2d})}^{2}+
          \||D_{q}|^{2/3}u\|_{L^{2}(\R^{2d})}^{2}
  \right]\\
  &&-\sum_{m\in\Z^{d}}\frac{1}{A^{4/3}}
     \|2^{2\ell/3}u_{\ell,m}\|_{L^{2}(\R^{2d})}^{2}
  \\
  &\geq&  \frac{1}{C_{g,\psi}}\left[
\|\tilde{\mathcal{O}}_{\ell}^{2/3}u\|_{L^{2}(\R^{2d})}^{2}+
          \||D_{q}|^{2/3}u\|_{L^{2}(\R^{2d})}^{2}
         \right]\\
  &&\hspace{-0.5cm}-\underbrace{2^{4/3}[C_{g,\psi}']^{4/3}(1+b)^{4/3}[C_{{b,\psi}}^{(5)}]^{8/3}(1+b)^{8/3}}_{\leq 4[C_{g,\psi}']^{3}[C_{g,\psi}^{(5)}]^{8}(1+b)^{4}}\|u_{\ell,m}\|_{L^{2}(\R^{2d})}^{2}
\end{eqnarray*}
The complete error term in the lower bound is now bounded from below
by
$$
\sum_{m\in\Z^{d}}\left(\frac{\kappa_{b}^{2}(1+b)^{2}}{b^{4}}-4[C_{g,\psi}']^{3}[C_{g,\psi}^{(5)}]^{8}\left(\frac{(1+b)^{10}+(1+b)^{8}}{b^{2}}+\frac{(1+b)^{4}}{b^{8/3}}\right)\right)
\|u_{\ell,m}\|_{L^{2}(\R^{2d})}^{2}
$$
which is non negative as soon as
$$
\frac{\kappa_{b}^2(1+b)^{2}}{b^{4}}\geq 12[C_{g,\psi}']^{3}[C_{g,\psi}^{(5)}]^{8}\frac{(1+b)^{10}}{b^{2}}\,.
$$
A sufficient condition is $\kappa_{b}\geq C_{b,\psi}^{(6)}(1+b)^{5}$ with $C_{g,\psi}^{(6)}\geq 1$ large enough.\\
For the final writing of the the inequality we also take
$C_{g,\psi}^{(6)}\geq 2^{9}C_{g,\psi}C_{g,\psi}^{(4)}[C_{g,\psi}^{(5)}]^{8}$\,.
\end{proof}
\subsection{Lower bound for $\|(P_b-i\lambda/b)u\|_{L^2}+1/b^2\|u\|_{L^2}$}
\label{sec:totallower}
After the dyadic partition of unity for $ \sum_{\ell=-1}^{\infty}\theta_{\ell}^2(q,p)\equiv 1$ on $\Omega\times \R^d$ of Subsection~\ref{subSec:DyadicPartitionOfUnity}  and by setting
$$
u_{\ell}(q,p)=2^{\ell d/2}\theta_{\ell}(q,2^\ell p) u(q,2^\ell p) \in \mathcal{C}^{\infty}_{0}(S_{2,\ell};\C)\quad \text{for}~u\in\mathcal{C}^{\infty}_0(\Omega\times\R^d;\C)\,,
$$
the results of Proposition~\ref{pr:subellPbellLarge} and Proposition~\ref{pr:subellPbellSmall} give after a rescaling
\begin{eqnarray*}
    C_{g,\psi}^{(7)}(1+b)^4 
    \|(\frac{\kappa_b}{b^2}+\mathcal{P}_b-i\frac{\lambda}{b})\theta_\ell u\|_{L^2(\R^{2d})}^2
&\geq &\|\frac{(\kappa_b+\mathcal{O})}{b^2}\theta_\ell u\|_{L^2(\R^{2d})}^2 
+ \|\frac{1}{b}(\mathcal{Y}-i\lambda)\theta_\ell u\|_{L^2(\R^{2d})}^2\\
&&
+\frac{1}{b^{8/3}}\|\theta_\ell u\|_{\tilde{\mathcal{W}}^{2/3}}^2
+\frac{1}{b^{8/3}}\|\frac{|\lambda|}{\langle p\rangle}\theta_{\ell}u\|_{L^2(\R^{2d})}^2
\end{eqnarray*}
for all $\ell\in\Z$\,, $\ell\geq -1$\,, as soon as $\kappa_b\geq C_{g,\psi}^{(7)}(1+b)^5$ and $C_{g,\psi}^{(7)}\geq 1$ is chosen large enough.
By summation with respect to $\ell\geq -1$\,, Proposition~\ref{pr:equivNormAfterDyadicPartition} and the same result with $\mathcal{P}_b-i\lambda/b$ replaced by $\frac{1}{b^2}\mathcal{O}$ and Proposition~\ref{pr:eqnormes} for $\|~\|_{\mathcal{W}^{2/3}}$ imply
\begin{eqnarray*}
    C_{g,\psi}^{(8)}(1+b)^4 
    \|(\frac{\kappa_b}{b^2}+\mathcal{P}_b-i\frac{\lambda}{b})u\|_{L^2(\R^{2d})}^2
  &\geq &\|\frac{(\kappa_b+\mathcal{O})}{b^2}u\|_{L^2(\R^{2d})}^2 
+ \|\frac{1}{b}(\mathcal{Y}-i\lambda)u\|_{L^2(\R^{2d})}^2\\
&&
+\frac{1}{b^{8/3}}\|u\|_{\tilde{\mathcal{W}}^{2/3}}^2
+\frac{1}{b^{8/3}}\|\frac{|\lambda|}{\langle p\rangle}u\|_{L^2(\R^{2d})}^2
\end{eqnarray*}
for all $u\in \mathcal{C}^{\infty}_{0}(\Omega\times\R^d;\C)$ as soon as $\kappa_b\geq C_{g,\psi}^{(8)}(1+b)^5$ with $C_{g,\psi}^{(8)}\geq 1$ large enough.\\
By taking $C_{g,\mathcal{M}}\geq C_{g,\psi}^{(8)}$ large enough so that the comparison results of Proposition~\ref{pr:locscalPb} and Proposition~\ref{pr:eqnormes}-\textbf{ii)} can be applied with 
\begin{eqnarray*}
    C_{g,\mathcal{M}}(1+b)^2 
    \|(\frac{\kappa_b}{b^2}+P_b-i\frac{\lambda}{b})u\|_{L^2(\R^{2d})}
&\geq &\|\frac{(\kappa_b+\mathcal{O})}{b^2}u\|_{L^2(\R^{2d})} 
+ \|\frac{1}{b}(\nabla_{\mathcal{Y}}^{\mathcal{E}}-i\lambda)u\|_{L^2(\R^{2d})}\\
&&
+\frac{1}{b^{4/3}}\|u\|_{\tilde{\mathcal{W}}^{2/3}}
+\frac{1}{b^{4/3}}\|\frac{|\lambda|}{\langle p\rangle}u\|_{L^2(\R^{2d})}
\end{eqnarray*}
for all $u\in \mathcal{C}^{\infty}_{0}(X;\mathcal{E})$\,, when $\kappa_b\geq C_{g,\mathcal{M}}(1+b)^5$\\
By writing
$$
  C_{g,\mathcal{M}}(1+b)^2 \kappa_b \left( \|(P_b-i\frac{\lambda}{b})u\|_{L^2(\R^{2d})}
+ \frac{1}{b^2}\|u\|_{L^2(\R^{2d})} 
\right)\geq    C_{g,\mathcal{M}}(1+b)^2 \|(\frac{\kappa_b}{b^2}+P_b-i\frac{\lambda}{b})u\|_{L^2(\R^{2d})}
$$
and by noticing that the factor of the left-hand side can be
$$
  C_{g,\mathcal{M}}(1+b)^2 \times   C_{g,\mathcal{M}}(1+b)^5=C_{g,\mathcal{M}}(1+b)^7 
$$
this ends the proof of the inequality \eqref{eq:principalInequalityOne} in  Theorem~\ref{th:mainOne}.\\
The essential maximal accretivity result for $\kappa_{b}\geq C_{0}(1+b^{2})$\,, $C_{0}\geq 1$ determined by $(g,E,g^{E},\nabla^{E})$\,,  was proved in Proposition~\ref{pr:IPPineq} and Corollary~\ref{corollary_ess_max_acc}.

\section{Consequences and optimality of Theorem~~\ref{th:mainOne}}
\label{sec:conseqOpt}
In this subsection we will discuss the optimality of the constants appearing in Theorem \ref{th:mainOne} as well as several consequences and extensions of the subelliptic estimate.

\subsection{About $b$-dependent constants}
\label{sec:bdep}
  Obviously the constant $\frac{1}{C_{g}(1+b)^{7}}$ or the condition $\kappa_{b}\geq C_{g}(1+b)^{5}$ of Theorem~\ref{th:mainOne} can be replaced by a uniform constant and a uniform lower bound for $\kappa_{b}$ when $b\in ]0,b_{0}]$ for some $b_{0}>0$\,. The question arises about the regime $b\to +\infty$\,. We do not claim that neither our lower bound nor the condition on $\kappa_{b}$ are optimal with respect to $b$ as $b\to +\infty$\,, but they cannot be written with uniform constants.\\
  Actually, we show here that scalar GKFP operators admit quasimodes at $\lambda = 0$ of size $\mathcal{O}(b^{-2})$ as $b \rightarrow \infty$\,.
\begin{proposition} \label{quasimode_proposition}
    Let $P_{\pm,b} = \frac{1}{b^2} \mathcal{O}\pm \frac{1}{b}\mathcal{Y}$ be the scalar GKFP operator on $X =  T^*Q$ where the Hermitian bundle is $E = Q \times \C$ with $\nabla^E$ the trivial connection, $g^E$ is the usual pointwise Hermitian inherited from $\C$, and $\mathcal{M}(b)=0$. Then there exists $u \in C^\infty_{0}(X; \mathcal{E}) = C^\infty_{0}(X; \C)$ with $\norm{u}_{L^2(X; \mathcal{E})} = 1$ satisfying
    \begin{align} \label{scalar_quasimode_bound}
        \norm{P_{\pm,b}u}_{L^2(X;\mathcal{E})} \le Cb^{-2}, \ \ b \in (0,\infty),
    \end{align}
    for some constant $C>0$ independent of $b$.
\end{proposition}
\begin{proof}
Let $u \in C^\infty(X; \C)$ be any function of the form
\begin{align}
    u(q,p) = \varphi(\abs{p}^2_q), \ \ \ (q,p) \in X,
\end{align}
where $\varphi\neq 0$ belongs to  $\mathcal{C}^{\infty}_{0}(\R;\C)$. By multiplying $u$ by a non-zero constant if necessary, we may ensure that $\norm{u}_{L^2(X; \C)} = 1$. Since $u$ is a function of the kinetic energy, we have
\begin{align}
    \mathcal{Y}u \equiv 0.
\end{align}
Hence
\begin{align} \label{harmonic_osc_applied_to_quasimode}
    \norm{P_{\pm,b} u}_{L^2(X;\C)} \le \frac{1}{b^2} \norm{\mathcal{O}u} + \frac{1}{b} \norm{\mathcal{Y}u} \le Cb^{-2}, \ \ b \in (0,\infty).
\end{align}
\end{proof}
An immediate consequence of Proposition \ref{quasimode_proposition} is that the best possible constant appearing on the right-hand side of (\ref{th:mainOne}) fails to be independent of $b$ in general. Indeed, let $P_{\pm,b} = \frac{1}{b^2} \mathcal{O}\pm \frac{1}{b} \mathcal{Y}$ be a scalar GKFP operator as in Proposition \ref{quasimode_proposition} and let $C(b)>0$ be the largest constant such that
\begin{align} \label{scalar_subelliptic_est}
    \left\| \left(P_{\pm,b} - \frac{i\lambda}{b}\right)u\right\|_{L^{2}}+\frac{1}{b^2}\norm{u}_{L^2} \geq C(b)
    \left(\left\| \frac{\mathcal{O}}{b^{2}}u \right\|_{L^{2}} + 
    \left\|
      \frac{1}{b}\left( \nabla^{\mathcal{E}}_{\mathcal{Y}} - i\lambda \right)u
    \right\|_{L^{2}} + \frac{1}{b^{4/3}}\left[||u||_{\tilde{\cal W}^{\frac{2}{3}}} + \norm{\left(\frac{\abs{\lambda}}{\langle p \rangle_q} \right)^{2/3} u}_{L^2}\right]
    \right)
\end{align}
holds for every $u\in C^\infty_{0}(X; \C)$. Taking $\lambda = 0$ in (\ref{scalar_subelliptic_est}) gives
\begin{align} \label{restricted_inequality}
    \norm{P_{\pm,b} u}_{L^2(X;\C)} + \frac{1}{b^2} \norm{u}_{L^2(X;\C)} \ge \frac{C(b)}{b^{4/3}} \norm{u}_{\tilde{W}^{\frac{2}{3}}(X;\C)}
\end{align}
If $u \in C^\infty_{0}(X; \C)$ is as in Proposition \ref{quasimode_proposition}, then (\ref{scalar_quasimode_bound}) and (\ref{restricted_inequality}) together give
\begin{align}
    C(b) \le C b^{-\frac{2}{3}}, \ \ b \in (0,\infty),
\end{align}
for some $C>0$ independent of $b$. Because $b^{-2/3} \rightarrow 0$ as $b \rightarrow \infty$, there cannot exist a constant $C_0>0$ such that $C(b) \ge C_0$ for $b \gg 1$\,. In particular, $C(b)$ cannot be constant with respect to $b$.

\subsection{Perturbative Estimate}
\label{sec:perturb}
In this subsection, we consider the stability of the subelliptic estimate~\eqref{eq:principalInequalityOne} under a general class of perturbations.

\begin{proposition}\label{perturbation_result}
Let $P_{\pm,b}=\frac{1}{b^{2}}\mathcal{O}\pm \frac{1}{b}\nabla^{\mathcal{E}}_{\mathcal{Y}}$\,, let $M(b)\in \mathcal{L}(\tilde{\mathcal{W}}^{1,0}(X;\mathcal{E});L^{2}(X,dqdp;\mathcal{E}))$ satisfy
\begin{eqnarray*}
  && M(b)=M_{1}(b)+M_{0}(b)\\
  && \|M_{1}(b)\|_{\mathcal{L}(\tilde{\mathcal{W}}^{1,0};L^{2})}\leq \frac{\nu_{1}(b)}{b}\quad\text{and}\quad
     \|M_{0}(b)\|_{\mathcal{L}(L^{2};L^{2})}\leq
     \nu_{0}\left(1+\frac{1}{b^{2}}\right)\,,
\end{eqnarray*}
and set
$$
P_{\pm,b,M}=P_{\pm,b} +M_{1}(b)+M_{0}(b)\,.
$$
We assume that $\nu_{1}(b)$ and $\nu_{0}$ satisfy
$$
\forall b\in (0,+\infty)\,,\quad  \nu_{1}^{2}(b)b^{2}\leq \frac{C_{g}+8\nu_{0}}{16}(1+b^{2})\,,
$$
where  $C_{g}\geq 1$ is the constant determined by $(g,E,g^{E},\nabla^{E})$ in Theorem~\ref{th:mainOne}.
For $\kappa_{b}\geq (C_{g}+16\nu_{0})(1+b^{5})$\,, the operator $\frac{\kappa_{b}}{b^{2}}+P_{\pm,b,M}$ is closable and its closure $\frac{\kappa_{b}}{b^{2}}+\overline{P}_{\pm,b,M}$ is maximal accretive with $D(\overline{P}_{\pm,b,M})=D(\overline{P}_{\pm,b})$ and
$$
\forall u\in D(\overline{P}_{\pm,b,M})\,,\quad
\mathrm{Re}~ \langle u\,,\, (\frac{\kappa_{b}}{b^{2}}+\overline{P}_{\pm,b,M})u\rangle
\geq \frac{1}{8b^{2}}\left[\|u\|_{\tilde{\mathcal{W}}^{1,0}}^{2}+\kappa_{b}\|u\|_{L^{2}}^{2}\right]\,.
$$
Moreover, the inequalities
\begin{align*}
  \left\|(\overline{P}_{\pm,b,M}- \frac{i\lambda}{b})u\right\|_{L^{2}}  
    +\frac{1+b^{2}}{b^2}\norm{u}_{L^2}
  \geq \frac{(1+b)^{-7}}{8(C_{g}+16\nu_{0})^{2}}
  \Bigg(&
  \| \frac{\mathcal{O}}{b^{2}}u\|_{L^{2}}+ 
    \|\frac{1}{b}(\pm\nabla^{\mathcal{E}}_{\mathcal{Y}} - i\lambda)u
                                            \|_{L^{2}}
  \\
      &+ \frac{1}{b^{4/3}}\left[||u||_{\tilde{\cal W}^{\frac{2}{3}}} + \norm{\left(\frac{\abs{\lambda}}{\langle p \rangle_q} \right)^{2/3} u}_{L^2}\right]
    \Bigg)
\end{align*}
and
\begin{align*}
  \left\|(\overline{P}_{\pm,b,M}- \frac{i\lambda}{b})u\right\|_{L^{2}}  
    +\frac{2\kappa_b}{b^2}\norm{u}_{L^2}
  \geq \frac{1}{4C_{g}(1+b)^7}
  \Bigg(&
  \| \frac{\mathcal{O}}{b^{2}}u\|_{L^{2}}+ 
    \|\frac{1}{b}( \pm\nabla^{\mathcal{E}}_{\mathcal{Y}} - i\lambda)u
                                            \|_{L^{2}}
  \\
      &+ \frac{1}{b^{4/3}}\left[||u||_{\tilde{\cal W}^{\frac{2}{3}}} + \norm{\left(\frac{\abs{\lambda}}{\langle p \rangle_q} \right)^{2/3} u}_{L^2}\right]
    \Bigg)
\end{align*}
hold 
for every $u\in D(\overline{P}_{\pm,b,M})$ and every $(\lambda,b)\in \R\times(0,+\infty)$\,.
\end{proposition}
\begin{proof}
  Let us first check the accretivity of $P_{\pm,b,M}$ on $\mathcal{C}^{\infty}_{0}(X;\mathcal{E})$\,.
  For $u\in \mathcal{C}^{\infty}_{0}(X;\mathcal{E})$\,, write
  \begin{eqnarray*}
    \mathrm{Re}~ \langle u\,,\, [\frac{\kappa_{b}}{b^{2}}+P_{\pm,b}+M_{1}(b)+M_{0}(b)]u\rangle_{L^{2}}&\geq& \frac{1}{4b^{2}}\left[\|u\|_{\tilde{\mathcal{W}}^{1,0}}^{2}+\kappa_{b}\|u\|_{L^{2}}^{2}\right] -\frac{\nu_{1}(b)}{b}\|u\|_{L^{2}}\|u\|_{\tilde{\mathcal{W}}^{1,0}}\\
    &&\hspace{3cm}-\nu_{0}\frac{1+b^{2}}{b^{2}}\|u_{0}\|_{L^{2}}^{2}
    \\
                                                                                               &\geq & \frac{1}{4b^{2}}\left[\|u\|_{\tilde{\mathcal{W}}^{1,0}}^{2}+\kappa_{b}\|u\|_{L^{2}}^{2}\right] -\frac{1}{8b^{2}}\|u\|^{2}_{\tilde{\mathcal{W}}^{1,0}}\\
                                                                                               &&\hspace{3cm}    -(2\nu_{1}^{2}(b)+\nu_{0}\frac{1+b^{2}}{b^{2}})\|u\|_{L^{2}}^{2}\\
    \\ &\geq & \frac{1}{8b^{2}}\|u\|_{\tilde{\mathcal{W}}^{1,0}}^{2}
               +\frac{\kappa_{b}+(C_{g}+8\nu_{0})(1+b^{2})-16\nu_{1}^{2}(b)b^{2}}{8b^{2}}\|u\|_{L^{2}}^{2}\\                                                                &\geq & \frac{1}{8b^{2}}\left[\|u\|_{\tilde{\mathcal{W}}^{1,0}}^{2}+\kappa_{b}\|u\|_{L^{2}}^{2}\right]\,.
  \end{eqnarray*}
  This proves the accretivity of $P_{\pm,b,M}$ which is therefore closable.\\
  From the inequality \eqref{eq:IPPineqTh} for $\overline{P}_{\pm,b}$ we deduce for any $\lambda\in \mathbb{R}$
$$
\|(\frac{\kappa_{b}}{b^{2}}+\overline{P}_{\pm,b}-i\lambda)u\|_{L^{2}}\|u\|_{L^{2}}\geq \mathrm{Re}~\langle u\,,\, (\frac{\kappa_{b}}{b^{2}}+\overline{P}_{\pm,b})u\rangle_{L^{2}}
\geq \frac{1}{4b^{2}}\left[\|u\|_{\tilde{\mathcal{W}}^{1,0}}^{2}+\kappa_{b}\|u\|_{L^{2}}^{2}\right]
$$
and for all $t>0$
$$
\frac{t}{2}\|(\frac{\kappa_{b}}{b^{2}}+\overline{P}_{\pm,b}-i\lambda)u\|_{L^{2}}^{2}
\geq \frac{1}{4b^{2}}\|u\|_{\tilde{\mathcal{W}}^{1,0}}^{2}+\frac{\kappa_{b}-2t^{-1}b^{2}}{4b^{2}}\|u\|_{L^{2}}^{2}\,.
$$
This inequality with $t=\frac{1}{8\nu_{1}^{2}(b)}$ and
$$\kappa_{b}-2t^{-1}b^{2}=\kappa_{b}-16\nu_{1}^{2}(b)b^{2}\geq (C_{g}+16\nu_{0})(1+b^{2})-16\nu_{1}^{2}(b)b^{2}\geq 0
$$
gives
$$
\frac{1}{4}\|(\frac{\kappa_{b}}{b^{2}}+\overline{P}_{\pm,b}-i\lambda)u\|_{L^{2}}^{2}\geq \frac{\nu_{1}^{2}(b)}{b^{2}}\|u\|_{\tilde{\mathcal{W}}^{1,0}}^{2}\geq \|M_{1}(b)u\|_{L^{2}}^{2}\,.
$$
Actually the inequality \eqref{eq:IPPineqTh} also gives
$$
\frac{1}{4}\|(\frac{\kappa_{b}}{b^{2}}+\overline{P}_{\pm,b}-i\lambda)u\|_{L^{2}}
-\|M_{0}(b)u\|_{L^{2}}\geq \frac{\kappa_{b}}{16b^{2}}\|u\|_{L^{2}}-\frac{\nu_{0}(1+b^{2})}{b^{2}}\|u\|_{L^{2}}\geq \frac{(C_{g}-16\nu_{0})(1+b^{2})}{16b^{2}}\|u\|_{L^{2}}\geq 0\,.
$$
Used with $\lambda=0$\,, we have two closed accretive operators $A=(\frac{\kappa_{\pm,b}}{b^{2}}+\overline{P}_{\pm,b})$ and $C=(\frac{\kappa_{b}}{b^{2}}+\overline{P}_{\pm,b,M})$ such that $D=\mathcal{C}^{\infty}_{0}(X;\mathcal{E})$ is dense in $D(A)=D(\overline{P}_{\pm,b})$ and $D(C)=D(\overline{P}_{\pm,b,M})$ while
$$
\forall u\in D\,,\quad \|(A-C)u\|_{L^{2}}\leq \frac{3}{4}\|Au\|_{L^{2}}\,,
$$
with $\frac{3}{4}<1$ and $A$ maximal accretive. Theorem~X.50 of \cite{ReSi} tells us that $C$ is maximal accretive as well with $D(C)=D(A)$\,.\\
We can take $\kappa_{b}=(C_{g}+16\nu_{0})(1+b^{2})$\,, and $\|(A-C)u\|_{L^{2}}\leq \frac{3}{4}\|(A-i\lambda)u\|_{L^{2}}$ leads to 
\begin{eqnarray*}
  2(C_{g}+16\nu_{0})\big[\|(\overline{P}_{\pm,b,M}-i\lambda)u\|_{L^{2}}+\frac{1+b^{2}}{b^{2}}\|u\|_{L^{2}}\big]&\geq& \|(\frac{\kappa_{b}}{b^{2}}+\overline{P}_{\pm,b,M}-i\lambda)u\|_{L^2}+\frac{\kappa_{b}}{b^{2}}\|u\|_{L^{2}}\\
                                                                                                               &\geq& \frac{1}{4}\|(\frac{\kappa_{b}}{b^{2}}+\overline{P}_{\pm,b}-i\lambda)u\|_{L^{2}}+\frac{\kappa_{b}}{b^{2}}\|u\|_{L^{2}}\\
  &\geq & \frac{1}{4}\big[\|(\overline{P}_{\pm,b}-i\lambda)u\|_{L^{2}}+\frac{1}{b^{2}}\|u\|_{L^{2}}\big]\,,
\end{eqnarray*}
and
\begin{eqnarray*}
    \|(\overline{P}_{\pm,b,M}-i\lambda)u\|_{L^2}+\frac{2\kappa_{b}}{b^{2}}\|u\|_{L^{2}}
    \geq
    \frac{1}{4}\big[\|(\overline{P}_{\pm,b}-i\lambda)u\|_{L^{2}}+\frac{1}{b^{2}}\|u\|_{L^{2}}\big]
\end{eqnarray*}
where, in both inequalities, the right-hand side is bounded from below by \eqref{eq:principalInequality} of Theorem~\ref{th:mainOne}.
\end{proof}
\subsection{$\widetilde{\mathcal{W}}^{s}$-versions}
\label{sec:tWsversion}
The subellitpic estimate of Theorem~\ref{th:mainOne} wich is concerned with the case $L^2(X,dqp; \mathcal{E})=\tilde{\mathcal{W}}^{0}(X;\mathcal{E})$ can be extended to $\tilde{\mathcal{W}}^{s}(X;\mathcal{E})$ subelliptic estimates for any $s\in\R$ as follows.
\begin{proposition}
  \label{pr:tWsestimate}
  Let $P_{\pm,b,M}=\frac{1}{b^{2}}\mathcal{O}\pm \frac{1}{b}\nabla^{\mathcal{E}}_{\mathcal{Y}}+M_{1}(b)+M_{0}(b)$ satisfy
  $$
\forall s\in \R\,,\quad
\|M_{1}(b)\|_{\mathcal{L}(\tilde{\mathcal{W}}^{1,s};\tilde{\mathcal{W}}^{s})}\leq \frac{\nu_{1,s}(b)}{b}\quad,\quad \|M_{0}(b)\|_{\mathcal{L}(\tilde{\mathcal{W}}^{s};\tilde{\mathcal{W}}^{s})}\leq \nu_{0,s}\frac{1+b^{2}}{b^{2}}
$$
whenever
$$
\forall b\in (0,+\infty)\,, \quad \nu^{2}_{1,s}(b)b^{2}\leq \frac{C_{g}+8\nu_{0,s}}{16}(1+b^{2})\,.
$$
For every $s\in\R$, there exists $C_{g,s}\geq 1$ determined by the geometric data $(g,E,g^{E},\nabla^{E})$ and the pair $(\nu_{0,s},\nu_{1,s}(.))$ such that,
for $\kappa_{b}\geq C_{g,s}(1+b^{5})$\,, the operator $\frac{\kappa_{b}}{b^{2}}+P_{\pm,b,M}$ is closable in $\tilde{\mathcal{W}}^{s}(X;\mathcal{E})$ and its closure $\frac{\kappa_{b}}{b^{2}}+\overline{P}_{\pm,b,M}^{\tilde{\mathcal{W}^{s}}}$ is maximal accretive with $D(\overline{P}_{\pm,b,M}^{\tilde{\mathcal{W}^{s}}})=D(\overline{P}_{\pm,b}^{\tilde{\mathcal{W}}^{s}})=W^{-s}_{\theta}D(\overline{P_{\pm,b}}^{L^{2}})$ and
$$
\forall u\in D(\overline{P}_{\pm,b,M}^{\tilde{\mathcal{W}}^{s}})\,,\quad
\mathrm{Re}~ \langle u\,,\, (\frac{\kappa_{b}}{b^{2}}+\overline{P}_{\pm,b,M}^{\tilde{\mathcal{W}}^{s}})u\rangle_{\tilde{\mathcal{W}}^{s}}
\geq \frac{1}{8b^{2}}\left[\|u\|_{\tilde{\mathcal{W}}^{1,s}}^{2}+\kappa_{b}\|u\|_{\mathcal{W}^{s}}^{2}\right]\,.
$$
Moreover, the inequalities
\begin{align*}
  \left\|(\overline{P}_{\pm,b,M}^{\tilde{\mathcal{W}}^{s}}- \frac{i\lambda}{b})u\right\|_{\tilde{\mathcal{W}}^{s}}  
    +\frac{1+b^{2}}{b^2}\norm{u}_{\tilde{\mathcal{W}}^{s}}
  \geq \frac{(1+b)^{-7}}{8C_{g,s}^{2}}
  \Bigg(&
  \| \frac{\mathcal{O}}{b^{2}}u\|_{\tilde{\mathcal{W}^{s}}}+ 
    \|\frac{1}{b}(\pm \nabla^{\mathcal{E}}_{\mathcal{Y}} - i\lambda)u
                                            \|_{\tilde{\mathcal{W}}^{s}}
  \\
      &+ \frac{1}{b^{4/3}}\left[||u||_{\tilde{\cal W}^{s+\frac{2}{3}}} + \norm{\left(\frac{\abs{\lambda}}{\langle p \rangle_q} \right)^{2/3} u}_{\tilde{\mathcal{W}}^{s}}\right]
    \Bigg)
\end{align*}
and
\begin{align*}
  \left\|(\overline{P}_{\pm,b,M}^{\tilde{\mathcal{W}}^{s}}- \frac{i\lambda}{b})u\right\|_{\tilde{\mathcal{W}}^{s}}  
    +\frac{\kappa_b}{b^2}\norm{u}_{\tilde{\mathcal{W}}^{s}}
  \geq \frac{1}{4C_{g}(1+b)^7}
  \Bigg(&
  \| \frac{\mathcal{O}}{b^{2}}u\|_{\tilde{\mathcal{W}^{s}}}+ 
    \|\frac{1}{b}( \nabla^{\mathcal{E}}_{\mathcal{Y}} - i\lambda)u
                                            \|_{\tilde{\mathcal{W}}^{s}}
  \\
      &+ \frac{1}{b^{4/3}}\left[||u||_{\tilde{\cal W}^{s+\frac{2}{3}}} + \norm{\left(\frac{\abs{\lambda}}{\langle p \rangle_q} \right)^{2/3} u}_{\tilde{\mathcal{W}}^{s}}\right]
    \Bigg)
\end{align*}
holds 
for every $u\in D(\overline{P}_{\pm,b,M})$ and every $(\lambda,b)\in \R\times(0,+\infty)$\,.
\end{proposition}
\begin{proof}
We use the pseudodifferential operator $W^{2}_{\theta}$ introduced in Propostion~\ref{pr:identWsE} which is self-adjoint with $D(W^{2}_{\theta})=\tilde{\mathcal{W}}^{2}(X;\mathcal{E})$ with an elliptic scalar principal symbol in $S^{2}_{\Psi}(Q;\mathrm{End}\mathcal{E})$ and for which we can write
$$
\|u\|_{\tilde{\mathcal{W}}^{s_{1},s_{2}}}=\|\mathcal{O}^{s_{1}/2}(W^{2}_{\theta})^{s_{2}/2}u\|_{L^{2}}\,,
$$
for any $(s_{1},s_{2})\in \mathbb{R}^{2}$ according to Definition~\ref{de:Ws1s2}.\\
Because $\mathcal{O}$ and $(W^{2}_{\theta})^{s/2}$ commute on $\mathcal{S}(X;\mathcal{E})$\,, according to Proposition~\ref{pr:commut},  when considering the operator $P_{\pm,b,M}:\mathcal{S}(X;\mathcal{E})\to \tilde{\mathcal{W}}^{s}(X;\mathcal{E})$ we may instead work with the operator
$$
(W^2_{\theta})^{-s/2}P_{\pm,b,M}(W^2_{\theta})^{s/2}=\frac{1}{b^{2}}\mathcal{O}\pm \frac{1}{b}\nabla_{\mathcal{Y}}^{\mathcal{E}}\pm\frac{1}{b}(W^2_{\theta})^{-s/2}[\nabla_{\mathcal{Y}}^{\mathcal{E}},(W^2_{\theta})^{s/2}]+ (W^2_{\theta})^{-s/2}(M_{1}(b)+M_{0}(b))(W^2_{\theta})^{s/2}
$$
initially defined from $\mathcal{S}(X;\mathcal{E})\to L^{2}(X;\mathcal{E})$\,.\\
The assumptions ensure 
$$
\|(W^2_{\theta})^{-s/2}(M_{1}(b)(W^2_{\theta})^{s/2}\|_{\mathcal{L}(\tilde{\mathcal{W}}^{1,0};L^{2})}\leq \frac{\nu_{1,s}(b)}{b}\quad\text{and}\quad
\|(W^2_{\theta})^{-s/2}(M_{0}(b)(W^2_{\theta})^{s/2}\|_{\mathcal{L}(L^{2};L^{2})}\leq
\nu_{0,s}(1+\frac{1}{b^{2}})\,.
$$
With $(W^2_{\theta})^{-s/2}[\nabla_{\mathcal{Y}}^{\mathcal{E}},(W^2_{\theta})^{s/2}]=(W^2_{\theta})^{-s/2}\nabla_{\mathcal{Y}}^{\mathcal{E}}(W^2_{\theta})^{s/2}-\nabla_{\mathcal{Y}}^{\mathcal{E}}$\,, Proposition~\ref{pr:nablaY} tells us
\begin{equation}
  \label{eq:commutnabY}
\|(W^2_{\theta})^{-s/2}[\nabla_{\mathcal{Y}}^{\mathcal{E}},(W^2_{\theta})^{s/2}]\|_{\mathcal{L}(\tilde{\mathcal{W}}^{1,0};L^{2})}\leq \tilde{C}_{g,s}
\end{equation}
for some constant $\tilde{C}_{g,s}\geq 1$ determined by $(g,E,g^{E},\nabla^{E})$\,.
It then suffices to  apply Proposition~\ref{perturbation_result} with $M_{1}(b)$ replaced by $M_{1}(b)\pm(W^2_{\theta})^{-s/2}[\nabla_{\mathcal{Y}}^{\mathcal{E}},(W^2_{\theta})^{s/2}]$\,, $\nu_{1}(b)$ replaced by $\nu_{1,s}(b)+\tilde{C}_{g,s}$ and $\nu_{0}$ by $\nu_{0,s}+2\tilde{C}_{g,s}$\,.
\end{proof}

\appendix
\section{Comparison of harmonic oscillator hamiltonians}
For a positive definite symmetric matrix $g=(g_{ij})_{1\leq i,j\leq d}\in \mathcal{M}_{dd}(\R)$ with $g^{-1}=(g^{ij})_{1\leq i,j\leq d}$ let $\mathcal{O}_{g}$ denote the harmonic oscillator hamiltonian
\begin{eqnarray*}
\mathcal{O}_{g}&=&\frac{-g_{ij}\partial_{p_i}\partial_{p_j}+g^{ij}p_ip_j}{2}\\
D(\mathcal{O}_{g})&=&\{u\in L^2(\R^d,dp),  \forall\alpha,\beta\in\N^d, |\alpha|+|\beta|\leq 2, p^{\alpha}\partial_{p}^{\beta}u\in L^2(\R^d,dp)\}=D(\mathcal{O}_{\mathrm{Id}})\,.
\end{eqnarray*}
The following result is a consequence of the ellipticity of $\mathcal{O}_{g}$ in the H{\"o}rmander class $S(\langle p,\eta\rangle^2, 
\frac{dp^2+d\eta^2}{\langle p,\eta\rangle^2})$(see \cite{HormIII}-Chap XVIII) combined with $\|\mathcal{O}_g u\|_{L^2}\geq \frac{d}{2}\|u\|_{L^2}$\,.

\begin{proposition}
\label{pr:composcharm}
For two positive definite symmetric matrices $g_1$ and $g_2$ there exist two constants $C_{g_1,g_2}>0$ and $C_{g_1}>0$ such that
\begin{eqnarray*}
    && \left(\frac{\norm{\mathcal{O}_{g_1}u}_{L^2}}{\norm{\mathcal{O}_{g_2}u}_{L^2}}\right)^{\pm 1}\leq C_{g_1,g_2}\,,\\
    \text{and}&&
    \norm{(\mathcal{O}_{g_2}-\mathcal{O}_{g_1})u}_{L^2}\leq C_{g_1}\norm{g_2-g_1}_{\mathcal{M}_{dd}(\R)}\norm{\mathcal{O}_{g_1}u}_{L^2}
\end{eqnarray*}
hold for all $u\in D(\mathcal{O}_{\mathrm{Id}})$\,.
\end{proposition}

\section{Complex Airy Operator}

\label{sec:euconedim}

We consider here the case of the one dimensional euclidean case of which the properties are due to the fact that the complex Airy operator has a compact resolvent and an empty spectrum.\\
Set
$$
P_{1}(\xi,\lambda)=i(p_{1}\xi-\lambda)-\frac{1}{2}\Delta_{p_{1}}
$$
with $\xi,\lambda\in \mathbb{R}$\,. It is  maximal accretive with $D(P_{1}(\xi,\lambda))=\left\{u\in L^{2}(\mathbb{R},dp), P_{1}(\xi,\lambda)u\in L^{2}(\mathbb{R},dp)\right\}$ in which $\mathcal{C}^{\infty}_{0}(\mathbb{R})$ is dense (with respect to the graph norm).\\
\begin{proposition}\label{eq:EstimateOn1DComplexAiryOperator}
  There exists $C_{0}\geq 1$\,, such that for all $(\xi,\lambda)\in \mathbb{R}^{2}$ the inequality
   \begin{equation}
    \label{eq:hypD1}
C_{0}\|(1+P_{1}(\xi,\lambda))u\|\geq \|\frac{1}{2}\Delta_{p}u\|+\|(p_{1}\xi-\lambda)u\|+(|\xi|^{2/3}+1)\|u\|+\|
\left(
  \frac{|\lambda|}{1+|p_{1}|}
\right)^{\frac{2}{3}}u\|\,,
\end{equation}
holds for all $u\in D(P_{1}(\xi,\lambda))$\,.
\end{proposition}
\begin{proof}
  The lower bound
  $$
  \forall u\in D(P_{1}(\xi,\lambda))\,,
\|(1+P_{1}(\xi,\lambda))u\|^{2}\geq \|u\|^{2}+\|P_{1}(\xi,\lambda)u\|^{2}
$$
is due to the accretivity of $P_{1}(\xi,\lambda)$\,.\\
With
\begin{eqnarray*}
\|P_{1}(\xi,\lambda)u\|^{2}&\geq& \|\frac{1}{2}\Delta_{p_{1}}u\|^{2}+\|(\lambda-p_{1}\xi)u\|^{2}-|\langle \xi u\,,\, D_{p_{1}}u\rangle|
  \\
                           &\geq& \|\frac{1}{2}\Delta_{p_{1}}u\|^{2}+\|(\lambda-p_{1}\xi)u\|^{2}-[c_4\||\xi|^{2/3}u\|^2+\frac{1}{8}\|\Delta_{p_1}u\|^2]
  \\
  &\geq& \frac{1}{2}\|\frac{1}{2}\Delta_{p_{1}}u\|^{2}+\|(\lambda-p_{1}\xi)u\|^{2}-c'_{4}\||\xi|^{2/3}u\|^{2}\,.
\end{eqnarray*}
Consider now the lower bound of $\|P_{1}(\xi,\lambda)u\|$ by $\||\xi|^{2/3}u\|$\,. It is obviously true for $\xi=0$\,. For $\xi\neq 0$\,, the operator
$$
P_{1}(\xi,\lambda)=i(p-p_{0})\xi-\frac{1}{2}\Delta_{p_{1}}\quad,\quad p_{0}=\frac{\lambda}{\xi}\,,
$$
is unitarily equivalent to $|\xi|^{2/3}P_{1}(1,0)=|\xi|^{2/3}(ip_{1}-\frac{1}{2}\Delta_{p_{1}})$ and there exists $c_{1}>0$ such that
$$
\forall u\in D(P_1(\xi,\lambda))\,, \quad \|P_{1}(\xi,\lambda)u\|\geq c_{1}|\xi|^{2/3}\|u\|\,.
$$
There exists a constant $C_{1}\geq 1$ such that
$$
\forall (\xi,\lambda)\in \mathbb{R}^{2}\,,\forall u\in D(P_{1}(\xi,\lambda))\,,\quad
C_{1}\|P_{1}(\xi,\lambda)u\|^{2}\geq \|u\|^{2}+\|\frac{1}{2}\Delta_{p_{1}}u\|^{2}+\|(p_{1}\xi-\lambda)u\|^{2}+\||\xi|^{2/3}u\|^{2}\,.
$$
The lower bound \eqref{eq:hypD1} is then obviously true for $|\lambda|\leq 1$ and it suffices to consider the case $\lambda\geq 1$\,.\\
Take a dyadic partition of unity $\chi_{0}^{2}(\varepsilon p_{1})+\sum_{\ell=1}^{\infty}\chi^{2}(\varepsilon 2^{-\ell}p_{1})\equiv 1$ with $\mathrm{supp\,} \chi_{0}\cup \mathrm{supp\,}\chi\subset [-4,4]$ and 
  for all $ \ell \in \mathbb{N}$, we define $\chi_{\ell}$ by 
$$\forall t\in \mathbb{R}_+,\quad \chi_{\ell}(t) = \left\lbrace
  \begin{array}{ll}
    \chi(\varepsilon 2^{-\ell}t) & \text{if } \ell \neq 0\\
    \chi_0(\varepsilon t) & \text{if } \ell=0  \,.
  \end{array}\right.
$$
We get
$$
\|P_{1}(\xi,\lambda)u\|^{2}-\sum_{\ell=0}^{\infty}\|P_{1}(\xi,\lambda)\chi_{\ell}u\|^{2}\leq C_{\chi}\varepsilon^{2}\left[\|\partial_{p_{1}}u\|^{2}+\|u\|^{2}\right]
\leq 16C_{1}C_{\chi}\varepsilon^{2}\|P_{1}(\xi,\lambda)u\|^{2}\,
$$
for some constant $C_{\chi}>0$ determined by the pair $(\chi_0,\chi)$\,.
By taking $\varepsilon\leq \frac{1}{8\sqrt{C_{1}C_{\chi}}}$ it suffices to consider
$$
\|P_{1}(\xi,\lambda)(\chi_{\ell}u)\|=\frac{1}{2^{2\ell}}\|P_{1}(\xi 2^{3\ell},\lambda 2^{2\ell})u_{\ell}\|
$$
with $u_{\ell}(p_{1})=2^{\ell/2}(\chi_{\ell}u)(2^{\ell}p_{1})$ and $\mathrm{supp\,} u_{\ell}\subset [-4\varepsilon^{-1},4\varepsilon^{-1}]$\,.\\
There are two cases
\begin{itemize}
\item $|\lambda 2^{2\ell}|\geq 2(4\varepsilon^{-1}2^{3\ell}|\xi|)$ and then
$$
\forall p_{1}\in \mathrm{supp\,} u_{\ell}\,,\quad|\lambda 2^{2\ell}-p_{1}\xi 2^{3\ell}|\geq |\lambda|2^{2\ell}-4\varepsilon^{-1}2^{3\ell}|\xi|
\geq \frac{|\lambda| 2^{2\ell}}{2}
$$
This implies
$$
C_{1}\|\frac{1}{2^{2\ell}}P_1(\xi 2^{3\ell},\lambda 2^{2\ell})u_{\ell}\|^{2}\geq \|\frac{|\lambda|2^{2\ell}}{2\times 2^{2\ell}}u_{\ell}\|^{2}\geq \|\frac{|\lambda|^{2}}{2}u_{\ell}\|^{2}
$$
Finally with $|\lambda|\geq 1$ and $|\lambda|\geq |\lambda|^{2/3}\geq \left(\frac{|\lambda|}{2^{\ell}}\right)^{2/3}$ we obtain
$$
4C_{1}\|\frac{1}{2^{2\ell}}P(\xi 2^{3\ell},\lambda 2^{2\ell})u_{\ell}\|^{2}\geq \|\left(\frac{|\lambda|}{2^{\ell}}\right)^{2/3}u_{\ell}\|^{2}\,,
$$
in this case.
\item $|\lambda 2^{2\ell}|\leq 2 (4\varepsilon^{-1}2^{3\ell}|\xi|)$ and then the lower bound
$$
C_{1}\|\frac{1}{2^{2\ell}}P_{1}(\xi 2^{3\ell},\lambda 2^{2\ell})u_{\ell}\|^{2}\geq\||\xi|^{2/3}u_{\ell}\|^{2}
$$
implies with $|\xi|\geq \frac{4^{-1}\varepsilon |\lambda|}{2^{\ell+1}}$
$$
C_{1}\|\frac{1}{2^{2\ell}}P_{1}(\xi 2^{3\ell},\lambda 2^{\ell})u_{\ell}\|^{2}\geq \frac{4^{-4/3}\varepsilon^{4/3}}{16}\|\left(\frac{|\lambda|}{2^{\ell}}\right)^{2/3}u_{\ell}\|^{2} 
$$
\end{itemize}
We conclude with the uniform equivalence
$$
C_{\varepsilon}^{-1}\langle p\rangle\leq 2^{\ell}\leq C_{\varepsilon}\langle p\rangle
$$
on $\mathrm{supp\,}u_{\ell}$\, where $\varepsilon\leq\frac{1}{8\sqrt{C_1 C_\psi}}$ and all the other constants are actually universal constants once  the pair $(\chi_0,\chi)$ for the dyadic partition of unity is fixed. 
\end{proof}

\section{Result used to localize the operator}

In this appendix $M$ is a manifold endowed with a volume density  $d\mathrm{vol}$ and  $\pi_{E}:E\to M$ is a smooth complex vector bundle endowed with a hermitian metric $g^{E}$\,, so that $L^{2}(M;E)$ is well defined with the norm $\|~\|_{L^{2}(M;E)}$ simply denoted by $\|~\|$\,.\\

For a differential operator $P$ acting on $\mathcal{C}^{\infty}_{0}(M;E)$ and a function $\chi\in \mathcal{C}^{\infty}(M)$\,, the equality $\chi P= P\chi -\left[P,\chi\right]$ and the triangular inequality give
$$\forall u \in \mathcal{C}^{\infty}_{0}(M;E) ,\quad ||P\chi u || - ||[P,\chi]u||   \leq || \chi Pu || \leq ||P \chi u || + ||[P,\chi] u ||\,.$$
It then follows that
$$\forall u \in \mathcal{C}^{\infty}_{0}(M;E),\quad \frac{1}{2}||P\chi u ||^{2} - ||[P,\chi]u||^{2} \leq ||\chi P u||^{2} \leq 2||P \chi u||^{2} + 2||[P,\chi]u||^{2} .$$
After three iterations with the additional assumptions that third order commutators vanish, which is relevant for differential operators of order less or equal to $2$\,, we get the following statement.
\begin{proposition}
  Let $P$ be a differential operator acting on $\mathcal{C}^{\infty}_{0}(M;E)$ and let $\chi_{1},\chi_{2}$ and $\chi_{3}$ be three $\mathcal{C}^{\infty}$ functions such that
  \begin{equation}
    \label{eq:Cond3com}
    [[[P,\chi_{1}],\chi_{2}],\chi_{3}]=0\,.
  \end{equation}
  The following inequalities hold for $u$ in  $\mathcal{C}^{\infty}_{0}(M;E)$
  $$ ||\chi_{1}\chi_{2}\chi_{3} Pu ||^{2}  \leq 2||\chi_{2}\chi_{3}P\chi_{1}u ||^{2} + 4 ||\chi_{3}[P,\chi_{1}]\chi_{2}u ||^{2}+8 ||[[P,\chi_{1} ],\chi_{2}]\chi_{3} u ||^{2}  $$
  and
  $$ ||\chi_{1}\chi_{2}\chi_{3} Pu ||^{2}  \geq \frac{1}{2}||\chi_{2}\chi_{3}P\chi_{1}u ||^{2} - 2 ||\chi_{3}[P,\chi_{1}]\chi_{2}u ||^{2} -4 ||[[P,\chi_{1} ],\chi_{2}]\chi_{3} u ||^{2}   .$$
\end{proposition}
Applying the above proposition with a locally finite  quadratic partition of unity $(\chi_{\ell})_{\ell\in L}$\,, $\chi_{\ell}\in \mathcal{C}^{\infty}_{0}$\,, $\sum_{\ell\in L}\chi_{\ell}^{2}\equiv 1$\,,  and summing over theindices
$\ell_{1},\ell_{2}$ and $\ell_{3}$ leads to 

\begin{equation}
  \label{eq:upperBound}
  ||Pu||^{2} \leq 2 \sum_{\ell_{1}}||P\chi_{\ell_{1}} u ||^{2} + 4 \sum_{\ell_{1},\ell_{2}} ||[P,\chi_{\ell_{2}}]\chi_{\ell_{1}}u||^{2} + 8 \sum_{\ell_{1},\ell_{2},\ell_{3}}||[[P,\chi_{\ell_{2}}],\chi_{\ell_{3}}]\chi_{\ell_{1}}u ||^{2}
\end{equation}
and
\begin{equation}
  \label{eq:lowerBound}
  ||Pu||^{2} \geq \frac{1}{2} \sum_{\ell_{1}}||P\chi_{\ell_{1}} u ||^{2} -2 \sum_{\ell_{1},\ell_{2}} ||[P,\chi_{\ell_{2}}]\chi_{\ell_{1}}u||^{2} -4 \sum_{\ell_{1},\ell_{2},\ell_{3}}||[[P,\chi_{\ell_{2}}],\chi_{\ell_{3}}]\chi_{\ell_{1}}u ||^{2}
\end{equation}
for all $u\in \mathcal{C}^{\infty}_{0}(M;E)$\,.\\
With (\ref{eq:upperBound}) and (\ref{eq:lowerBound}) we have

\begin{corollary}\label{Cor:equivalenceOfQuantities}
  Let $(\chi_{\ell})_{\ell\in L}$ be a family of functions such that
  $$\sum_{\ell\in L} \chi_{\ell}^{2} = 1 $$
  and let $P$ be a second order differential operator such that 
  \begin{equation}
      \label{eq:equivH}
      \forall u \in \mathcal{C}^{\infty}_{0}(M;E), \quad\frac{r}{2}\sum_{\ell_{1}}||P\chi_{\ell_{1}} u ||^{2} \geq 2 \sum_{\ell_{1},\ell_{2}} ||[P,\chi_{\ell_{2}}]\chi_{\ell_{1}}u||^{2} +4 \sum_{\ell_{1},\ell_{2},\ell_{3}}||[[P,\chi_{\ell_{2}}],\chi_{\ell_{3}}]\chi_{\ell_{1}}u ||^{2} 
  \end{equation}
for some $r\in (0,1)$\,. Then
  \begin{equation}
    \label{eq:equiv1}
    \forall u\in \mathcal{C}^{\infty}_{0}(M;E)\,,\quad (2+r) \sum_{\ell\in L}||P\chi_{\ell}u||^{2}\geq ||Pu||^{2}\geq \frac{1-r}{2} \sum_{\ell}||P\chi_{\ell}u||^{2}.
  \end{equation}
\end{corollary}

\section{$N-\mathrm{loc}$ and $N-\mathrm{comp}$ functional spaces}
\label{sec:Qloccomp}

Let $f:M\to N$ be a $\mathcal{C}^{\infty}$ map from the manifold $M$ to the manifold $N$\,, let $E\stackrel{\pi_{E}}{\to} M$  be a vector bundle and $\mathcal{F}(M;E)$ be a locally convex space of sections continuously embedded in $\mathcal{D}'(M;E)$ (abbreviated as a functional space of sections) such that for any $\chi\in \mathcal{C}^{\infty}_{0}(N;\mathbb{R})$ the multiplication by $\chi\circ f$ is a continuous
endomorphism of $\mathcal{F}(M;E)$\,. The notation $\mathcal{F}_{f-\mathrm{loc}}(M;E)$
will denote the  set of sections $s$ of $E$ such that
$$
\forall \chi\in \mathcal{C}^{\infty}_{0}(N;\mathbb{R})\,,\quad [\chi\circ f]s\in \mathcal{F}(M;E)\,.
$$
Once $\mathcal{F}_{f-\mathrm{loc}}(M;E)$ is defined $\mathcal{F}_{f-\mathrm{comp}}(M;E)$ is the set of sections $s\in \mathcal{F}_{f-\mathrm{loc}}(M;E)$ such that there exists $\chi\in \mathcal{C}^{\infty}_{0}(N;\mathbb{R})$ with $s=[\chi\circ f]s$\,. For $s\in \mathcal{F}_{f-\mathrm{loc}}(M;E)$  the $f$-support of $s$ is defined by
$$
f-\mathrm{supp}\,s=\mathop{\bigcap}_{
  \scriptsize\begin{array}[c]{c}
    F\subset N\\
    F~\text{closed}\\
    s\big|_{f^{-1}(N\setminus F)\equiv 0}
  \end{array}
}F=\overline{f(\mathrm{supp}\, s)}\,.
$$
For a compact subset $K$ of $N$\,
and $\mathcal{F}_{f-K}(M;E)=\left\{s\in \mathcal{F}_{f-\mathrm{loc}}(M;E)\,, f-\mathrm{supp}\,s\subset K\right\}$ and
$$
\mathcal{F}_{f- \mathrm{comp}}(M;E)=
\mathop{\bigcup_{K~\text{compact in}~N}\mathcal{F}_{f-K}(M;E)}\,.
$$
When the topology on $\mathcal{F}(M;E)$ is known, the topology on $\mathcal{F}_{f-\mathrm{loc}}(M;E)$ is the initial topology for the collection of maps $(s\mapsto [\chi\circ f]s)_{\chi\in \mathcal{C}^{\infty}_{0}(N;\mathbb{R})}$\,. This induces the topology on $\mathcal{F}_{f-K}(M;\mathcal{E})$ and the topology on $\mathcal{F}_{f-\text{comp}}(M;E)$ is the inductive limit topology.\\

We will use this in a particular case.

\begin{definition}
  \label{de:Qloccomp} Let $M=T^{*}N$ or $M=T^{*}(T^{*}N)$  be endowed with the natural projection $\pi:M\to N$ and let $\mathcal{F}(M;E)\subset \mathcal{D}'(M;E)$ be a functional space of sections of the vector bundle $E\stackrel{\pi_{E}}{\to}M$ which is a $\mathcal{C}^{\infty}_{0}(N;\mathbb{R})$-module. We will use in both cases the notation $\mathcal{F}_{N-\mathrm{loc}}(M;E)=\mathcal{F}_{\pi-\mathrm{loc}}(M;E)$\,, $N-\mathrm{supp}\, s=\pi-\mathrm{supp}\,s\subset N$\,,  $\mathcal{F}_{N-K}(M;\mathcal{E})=\mathcal{F}_{\pi-K}(M;E)$\,, $\mathcal{F}_{N-\mathrm{comp}}(M;E)=\mathcal{F}_{\pi-\mathrm{loc}}(M;E)$\,.
\end{definition}

When $N$ is a locally compact manifold, introducing a locally finite atlas and a subordinate partition of unity $\sum_{i\in \mathcal{I}}\chi_{i}(q)\equiv 1$ reduces the characterization of $s\in \mathcal{F}_{N-\mathrm{loc}}(M;E)$\,, $M=T^{*}N$ or $M=T^{*}(T^{*}N)$ to the meaning of $s\in \mathcal{F}_{\Omega-\mathrm{loc}}(M;E)$\,,  $M=T^{*}\Omega=\Omega\times \mathbb{R}^{d}$ or $M=T^{*}(T^{*}\Omega)=\Omega\times \mathbb{R}^{3d}$ and the invariance of $\mathcal{F}_{\Omega-\mathrm{loc}}(M;E)$ via a diffeomorphism $\phi:\Omega\to \Omega$\,.
With the extension by $0$ and the restriction, the embeddings  $\mathcal{F}_{\Omega-\text{comp}}(\Omega;E)\subset \mathcal{F}_{N-\text{comp}}(N;E)\subset \mathcal{F}_{\Omega'-\text{loc}}(\Omega';E)$ hold for  two different open sets $\Omega$ and $\Omega'$ in $N$\,.

\medskip
\noindent\textbf{Example:}\\
The spaces $\mathcal{S}_{N-\mathrm{comp}}(T^{*}N;\mathbb{C})$\,, $\mathcal{S}'_{N-\mathrm{comp}}(T^{*}N;\mathbb{C})$ and their respective duals $\mathcal{S}'_{N-\mathrm{loc}}(T^{*}N;\mathbb{C})$ and \hfill \break $\mathcal{S}_{N-\mathrm{loc}}(T^{*}N;\mathbb{C})$ are well defined for any locally compact manifold $N$\,.\\
If additionally $N$ is compact $\mathcal{S}_{N-\mathrm{loc}}(T^{*}N;\mathbb{C})=\mathcal{S}_{N-\mathrm{comp}}(T^{*}N;\mathbb{C})$
(resp. $\mathcal{S}'_{N-\mathrm{loc}}=\mathcal{S}'_{N-\mathrm{comp}}$) will be simply denoted $\mathcal{S}(T^{*}N;\mathbb{C})$ (resp. $\mathcal{S}'(T^{*}N;\mathbb{C})$)\,.\\
The Schwartz kernel theorem for continuous maps $\mathcal{C}^{\infty}_{0}(T^{*}N;\mathbb{C})\to \mathcal{D}'(T^{*}N;\mathbb{C})$ implies that any continuous map from $A:\mathcal{S}_{N-\text{comp}}(T^{*}N;\mathbb{C})\to \mathcal{S}'_{N-\text{loc}}(T^{*}N;\mathbb{C})$ admits a kernel in $K_{A}\in \mathcal{S}'_{N\times N-\text{loc}}(T^{*}(N\times N);\mathbb{C})$\,.
Additionally $A$ is continous from $\mathcal{S}'_{N-\text{comp}}(T^{*}N;\mathbb{C})$ to $\mathcal{S}_{N-\text{loc}}(T^{*}N;\mathbb{C})$ if and only if its kernel $K_{A}$ belongs to $\mathcal{S}_{N\times N-\mathrm{loc}}(T^{*}(N\times N);\mathbb{C})$\,.

\medskip
\noindent Other $N-\mathrm{loc}$ and $N-\mathrm{comp}$ spaces are introduced in the text.

\section{Some pseudo-differential calculus on $X=T^{*}Q$}
\label{sec:pseudodiff}
The manifold $Q$ is  either a compact manifold  which can be endowed with any riemannian metric or $\mathbb{R}^{d}$\,.
The total space of the cotangent bundle is denoted by $X=T^{*}Q$ and symbols of pseudo-differential operators are defined as functions of $T^{*}X=T^{*}(T^{*}Q)$\,.\\
The pseudo-differential calculus presented here and, of which the global geometrical meaning is checked, implements the idea that $\partial_{q^{i}}$ is an operator of order $1$ while $p_{i}\times$ and $\partial_{p_{i}}$ are of order $1/2$ as presented in \cite{Leb1}\cite{Leb2}. However our presentation, like our definition of the spaces $\tilde{\mathcal{W}}^{k}(X;\mathcal{E})$ in Definition~\ref{de:tW} slightly differ from Lebeau's approach (see Remark~\ref{re:compLeb}).
\subsection{Definitions and properties}
\label{sec:defpseudo}

We give here the definitions and state the main properties but their global geometrical meaning will be consequences of the subsequent paragraphs.
\begin{definition}
  \label{de:doubcancoord}
  For any coordinate system $(q^{1},\ldots, q^{d})$ on a chart open set $\Omega\subset Q$\,, the associated canonical coordinates on $T^{*}\Omega\subset X$ are $(q^{1},\ldots,q^{d},p_{1},\ldots,p_{d})$ with $p=p_{i}dq^{i}\in T^{*}_{q}Q$\,.  Accordingly doubly canonical coordinates on $T^{*}(T^{*}\Omega)\subset T^{*}X$ associated with the coordinates $(q^{1},\ldots,q^{d})$ on $\Omega$ will be $(q,p,\xi,\eta)$ with $(\xi,\eta)=\xi_{i}dq^{i}+\eta^{i}dp_{i}\in T^{*}_{q,p}X$ for the canonical coordinates $(q,p)$ associated with the coordinates $(q^{1},\ldots,q^{d})$\,. Those coordinates will be abbreviated as $X=(x,\Xi)\in \Omega\times\mathbb{R}^{3d}$ with $x=(q,p)\in \Omega\times \mathbb{R}^{d}$ and $\Xi=(\xi,\eta)\in \mathbb{R}^{2d}$\,.
\end{definition}

\begin{definition}
  \label{de:symbol}
  The set $S^{m}_{\Psi}(Q;\mathbb{C})$ is the class of functions $a\in \mathcal{C}^{\infty}(T^{*}X;\mathbb{C})$ such that for any doubly canonical coordinate system $(q,p,\xi,\eta)$ on $T^{*}(T^{*}\Omega)$\,, where $\Omega$ is a chart open set in $Q$\,, the following inequalities hold:
  \begin{multline}
\label{eq:Sloc}
\forall K\subset\subset \Omega\,, \forall (\alpha,\beta,\gamma,\delta)\in \mathbb{N}^{d}\,,\exists C_{K,\alpha,\beta,\gamma,\delta}>0\,, \forall (q,p,\xi,\eta)\in K\times (\mathbb{R}^{d})^{3}\,,\\
|\partial_{q}^{\alpha}\partial_{p}^{\beta}\partial_{\xi}^{\gamma}\partial_{\eta}^{\delta}a(q,p,\xi,\eta)|\leq C_{K,\alpha,\beta,\gamma,\delta}(1+|\xi|^{2}+|p|^{4}+|\eta|^{4})^{\frac{m-|\gamma|-\frac{|\beta|+|\delta|}{2}}{2}}\,.
\end{multline}
The intersection $\bigcap_{m\in \mathbb{R}}S^{m}_{\Psi}(Q;\mathbb{C})$ is denoted by $S^{-\infty}_{\Psi}(Q;\mathbb{C})$\,. The topology on $S^{m}_{\Psi}(Q;\mathbb{C})$ is given by the seminorms $p_{K,\alpha,\beta,\gamma,\delta}(a)$ which are the best constant $C_{K,\alpha,\beta,\gamma,\delta}$ in the above inequality. For any open set $\Omega\subset Q$\,, the spaces $S^{m}_{\Psi,\Omega-\mathrm{loc}}(\Omega;\mathbb{C})$ and $S^{m}_{\Psi, \Omega-\mathrm{comp}}(\Omega;\mathbb{C})$ are defined according to Appendix~\ref{sec:Qloccomp}\,. Finally the equivalence relation $a_{1}\sim a_{2}$ means $a_{1}-a_{2}\in S^{-\infty}_{\Psi}(Q;\mathbb{C})$\,.
\end{definition}
Actually $S^{-\infty}_{\Psi}(Q;\mathbb{C})=\mathcal{S}_{Q-\mathrm{loc}}(T^{*}(T^{*}Q);\mathbb{C})=\mathcal{S}(T^{*}(T^{*}Q);\mathbb{C})$
with the notations of Appendix~\ref{sec:Qloccomp}.

\begin{definition}
\label{def:quantiz}
The quantization of a symbol $a\in S^{m}_{\Omega-\mathrm{comp}}(\Omega;\mathbb{C})$ is given by
$$
a(q,p,D_{q},D_{p})u=\int_{\Omega\times \mathbb{R}^{3d}}e^{i[\xi(q-q')+\eta.(p-p')]}a(q,p,\xi,\eta)u(q',p')~ dq'dp'd\xi d\eta \in \mathcal{S}'_{\Omega-\mathrm{comp}}(T^{*}\Omega;\mathbb{C})
$$
for any $u\in \mathcal{S}'_{\Omega-\mathrm{comp}}(T^{*}\Omega;\mathbb{C})$\,.\\
The global definition  of $a(q,p,D_{q},D_{p})$ for $a\in S^{m}_{\Psi}(Q;\mathbb{C})$ is given by
\begin{equation}
  \label{eq:quantchi}
a(q,p,D_{q},D_{p})u=\sum_{n=1}^{N}(\chi_{n}(q)a)(q,p,D_{q},D_{p}) (\tilde{\chi}_{n}(q)u)
\end{equation}
for some partition of unity $\sum_{n=1}^{N}\chi_{n}\equiv 1$ on $Q$ subordinate to a finite atlas $Q=\bigcup_{n=1}^{N}\Omega_{n}$ with $\tilde{\chi}_{n}\in \mathcal{C}^{\infty}_{0}(\Omega_{n};[0,1])$\,, $\tilde{\chi}_{n}\equiv 1$ in a neighborhood of $\mathrm{supp}\,\chi_{n}$\,.\\
The set $\mathcal{R}(Q;\mathbb{C})$ of regularizing operators is $ \mathcal{L}(\mathcal{S}'(T^{*}Q;\mathbb{C});\mathcal{S}(T^{*}Q;\mathbb{C}))$\,.\\
The set $\left\{a(q,p,D_{q},D_{p})+R \, , a\in S^{m}_{\Psi}(Q;\mathbb{C})  \, , R\in \mathcal{R}(Q;\mathbb{C})\right\}$ is denoted $\mathrm{OpS}^{m}_{\Psi}(Q;\mathbb{C})$ and
$\mathrm{OpS}^{-\infty}_{\Psi}(Q;\mathbb{C})=\mathcal{R}(Q;\mathbb{C})$ with the equivalence relation $A_{1}\sim A_{2}$ in $\mathrm{OpS}^{m}_{\Psi}(Q;\mathbb{C})$ iff $A_{1}=A_{2}+R$ with $R\in \mathcal{R}(Q;\mathbb{C})$\,. 
\end{definition}

This pseudo-differential calculus has the same properties listed below as the classical pseudo-differential calculus and  our approach relies on the global
pseudo-differential calculus when $Q=\mathbb{R}^{d}$ recalled in the next paragraph.

\noindent\textbf{Properties:}
\begin{description}
\item[a)] For any vector bundle $\mathcal{C}^{\infty}$ isomorphism $\Phi: T^{*}Q\to T^{*}Q'$ given by
  $$(q',p')=\Phi(q,p)=(\phi(q), L(q).p)\quad\text{with}~L(q)\in \mathrm{GL}(T_{q}Q;T_{\phi(q)}Q')\,,$$
 the pull-back $\Phi^{*}$ defined by
  $[\Phi^{*}a](x,\Xi)=a(\Phi(x),{}^{t}d\Phi^{-1}(x).\Xi)$ with $x=(q,p)$ and $\Xi=(\xi,\eta)$ defines a continuous isomorphism from $S^{m}_{\Psi}(Q';\mathbb{C})$ to $S^{m}_{\Psi}(Q;\mathbb{C})$\,.\\
  Any $a\in S^{m}_{\Psi}(Q;\mathbb{C})$ equals $\sum_{n=1}^{N}\chi_{n}(q)a$ for any finite partition of unity $\sum_{n=1}^{N}\chi_{n}(q)\equiv 1$\,.  The particular case where $L(q)={}^{t}d\phi(q)^{-1}$ says that $\chi_{n}a\in S^{m}_{\Psi, \Omega_{n}-\mathrm{comp}}(\Omega_{n},\mathbb{C})$ is independent of the choice of the coordinate system $(q^{1},\ldots,q^{d})$\,.
\item[b)] When $a\in S^{m}_{\Psi , \Omega-\mathrm{comp}}(\Omega;\mathbb{C})$ and $\tilde{\chi}_{1}$ and $\tilde{\chi}_{2}$ are two elements of $\mathcal{C}^{\infty}_{0}(\Omega;[0,1])$ such that $\tilde{\chi}_{1}-\tilde{\chi}_{2}\equiv 0$ in a neighborhood of $\Omega-\mathrm{supp}\,a$\,, the operator  $a\circ[\tilde{\chi}_{1}-\tilde{\chi}_{2}]$ belongs to $\mathrm{OpS}^{-\infty}_{\Psi}(Q;\mathbb{C})$\,. This property ensures that the quantization \eqref{eq:quantchi} is actually independent of the choice of the $\tilde{\chi}_{n}$ cut-off functions  as a map from $S^{m}_{\Psi}(Q;\mathbb{C})/S^{-\infty}_{\Psi}(Q;\mathbb{C})$ to $\mathrm{OpS}^{m}_{\Psi}(Q;\mathbb{C})/\mathrm{OpS}^{-\infty}_{\Psi}(Q;\mathbb{C})$\,.
\item[c)] For any $a\in S^{m}_{\Psi}(Q;\mathbb{C})$ the operator $a(q,p,D_{q},D_{p})$ defined by \eqref{eq:quantchi} is continuous from $\mathcal{S}(X;\mathbb{C})$ to $\mathcal{S}(X;\mathbb{C})$ and from $\mathcal{S}'(X;\mathbb{C})$ to $\mathcal{S}'(X;\mathbb{C})$\,.
\item[d)] The set $\mathop{\bigcup}_{m\in \mathbb{R}}\mathrm{OpS}^{m}_{\Psi}(Q;\mathbb{C})$ is an algebra for the composition product with
  \begin{eqnarray*}
    &&
a_{1}(q,p,D_{q},D_{p})\circ a_{2}(q,p,D_{q},D_{p})=[a_{1}a_{2}](q,p,D_{q},D_{p})\quad \mod \mathrm{OpS}^{m_{1}+m_{2}-1}_{\Psi}(Q;\mathbb{C})\,,\\
    &&
       \left[a_{1}(q,p,D_{q},D_{p})\,,\, a_{2}(q,p,D_{q},D_{p})\right]=[\frac{1}{i}\{a_{1}a_{2}\}](q,p,D_{q},D_{p})\quad \mod \mathrm{OpS}^{m_{1}+m_{2}-2}_{\Psi}(Q;\mathbb{C})\,,
  \end{eqnarray*}
  for $a_{k}\in S^{m_{k}}_{\Psi}(Q;\mathbb{C})$ and $\left\{a_{1},a_{2}\right\}=\partial_{\xi}.a_{1}\partial_{q}a_{2}+\partial_{\eta}a_{1}.\partial_{p}a_{2}-\partial_{q}a_{1}.\partial_{\xi}a_{2}-\partial_{p}a_{1}.\partial_{\eta}a_{2}$ in doubly canonical  coordinates.
\item[e)] For any family  $(a_{j})_{j\in \mathbb{N}}$ with $a_{j}\in S^{m-j}_{\Psi}(Q;\mathbb{C})$ there exists $a\in S^{m}_{\Psi}(Q;\mathbb{C})$ such that $a-\sum_{j=0}^{J}a_{j}\in S^{m-J-1}_{\Psi}(Q;\mathbb{C})$\,, which is simply written $a\sim \sum_{j\in \mathbb{N}}a_{j}$\,.
  
  In particular for any $a\in \mathrm{OpS}^{m}_{\Psi}(Q;\mathbb{C})$ which is elliptic ($|a(q,p,\xi,\eta)|\geq C^{-1}(1+|\xi|^{2}+|p|^{4}+|\eta|^{4})^{m/2}$) there exists $b\sim \sum_{n=0}^{\infty}b_{j}$ with $b_{0}=\frac{1}{a}$ such that $b(q,p,D_{q},D_{p})\circ a(q,p,D_{q},D_{p})\sim a(q,p,D_{q},D_{p})\circ b(q,p,D_{q},D_{p})\sim\mathrm{Id}$\,.
\item[f)] For any $a\in S^{0}_{\Psi}(Q;\mathbb{C})$\,,
  $a(q,p,D_{q},D_{p})\in \mathcal{L}(L^{2}(X,dqdp;\mathbb{C}))$ with
  $$\|a(q,p,D_{q},D_{p})\|_{\mathcal{L}(L^{2})}\leq C\sup_{|\alpha|+|\beta|+|\gamma|+|\delta|\leq N_{d}}p_{\alpha,\beta,\gamma,\delta}(a)
  $$
  for some $N_{d}\in \mathbb{N}$ determined by $d=\mathrm{dim}~Q$\,.
\item[g)] When $\Phi:T^{*}Q\to T^{*}Q'$ is a $\mathcal{C}^{\infty}$-isomorphism like in $\textbf{a)}$ and $U_{\Phi}:L^{2}(X',dq'dp';\mathbb{C})\to L^{2}(X,dqdp;\mathbb{C})$ is the unitary map defined by $[U_{\Phi}u](q,p)=\sqrt {|\det(d\phi(q))\det(L(q))}|u(\phi(q),L(q).p)$ then for any $a\in S^{m}_{\Psi}(Q';\mathbb{C})$\,,
  $U_{\Phi}a(q',p',D_{q'},D_{p'})U_{\Phi}^{-1}=b(q,p,D_{q},D_{p})\in \mathrm{OpS}^{m}_{\Psi}(Q;\mathbb{C})$ with $b\sim\sum_{j=0}^{\infty}b_{j}$ and $b_{0}=(\Phi^{*}a)$\,.\\
  When $L(q)={}^{t}d\phi(q)^{-1}$ and $q\in \Omega$ a chart open set in $Q$\,, this result contains the fact that the space of pseudo-differential operator $\mathrm{OpS}^{m}_{\Psi , \Omega-\mathrm{comp}}(\Omega;\mathbb{C})$ does not depend on the choice of coordinates on $Q$ with a functorial transformation of the principal symbol. With the local definition of the quantization \eqref{eq:quantchi} and \textbf{b)}, the space $\mathrm{OpS}^{m}_{\Psi}(Q;\mathbb{C})$ has a global geometric meaning.
  
\item[h)] The vector bundle version of pseudo-differential operators $a(q,p,D_{q},D_{p})\in \mathrm{OpS}^{m}_{\Psi}(Q;\pi_{2}^{*}(\mathrm{End}(E)))$ acting on sections of $\pi_{X}^{*}(E)$ where $X=T^{*}Q\stackrel{\pi_{X}}{\to}Q$ and $T^{*}X\stackrel{\pi_{2}}{\to}Q$ are the natural projection and $E\stackrel{\pi_{E}}{\to}Q$ is a vector  bundle over $Q$\,, is reduced to the case of matricial pseudo-differential operators acting on $\mathbb{C}^{N}$-valued sections via the localized definition \eqref{eq:quantchi}. It has the same properties as the scalar pseudo-differential operators except for the principal symbol of a commutator.  We will use the abbreviation $\mathrm{OpS}^{m}_{\Psi}(Q;\mathrm{End}\,\mathcal{E})$ for $\mathrm{OpS}^{m}_{\Psi}(Q;\pi_{2}^{*}(\mathrm{End}(E)))$\,.
\item[i)] The seminorm topology on $\mathrm{OpS}^{m}_{\Psi}(Q;\mathbb{C})$ and the continuity properties of $(A_{1},A_{2})\to A_{1}\circ A_{2}$ and of $A\to U_{\Phi}\circ A \circ U_{\Phi}^{-1}$ are discussed in Subsection~\ref{sec:global}.
\end{description}

\subsection{Global calculus when $Q=\mathbb{R}^d$}
\label{sec:globpseudodiff}

Let $T^{*}(T^{*}\mathbb{R}^{d})=\mathbb{R}^{4d}_{q,p,\xi,\eta}$ and consider the function
$$
\Psi(q,p,\xi,\eta)=\sqrt{1+|\xi|^{2}+|\eta|^{4}+|p|^{4}}\,.
$$
\\
The symbol class $S(m,g_{\Psi})$ is the set of $a\in \mathcal{C}^{\infty}(\mathbb{R}^{4d};\mathbb{C})$ such that:
$$
\forall \alpha,\beta,\gamma,\delta\in \mathbb{N}^{d}\,,\, \exists C_{\alpha,\beta,\gamma,\delta}>0\,,\, \forall X=(q,p,\xi,\eta)\in \mathbb{R}^{4d}\,,\quad
|\partial_{q}^{\alpha}\partial_{p}^{\beta}\partial_{\xi}^{\gamma}\partial_{\eta}^{\delta}a(X)|\leq C_{\alpha,\beta,\gamma,\delta}m(X)\Psi(X)^{-|\gamma|-\frac{|\beta|+|\delta|}{2}}\,
$$
and the topology on $S(m,g_{\Psi};\mathbb{C})$ is given by the family of the seminorms
$$
p_{m,k}(a)=\sum_{\substack{|\alpha|+|\beta|+|\gamma|+|\delta|\leq k\\X\in \mathbb{R}^{4d}}}\left|\frac{\partial_{q}^{\alpha}\partial_{p}^{\beta}\partial_{\xi}^{\gamma}\partial_{\eta}^{\delta}a(X)}{m(X) \Psi^{-|\gamma|-\frac{|\beta|+|\delta|}{2}}(X)} \right|\,.
$$
We follow the terminology of \cite{Bon}.
\begin{proposition}
  \label{pr:Hormmetric} The metric $g_{\Psi}=dq^{2}+\frac{d\xi^{2}}{\Psi^{2}}+\frac{dp^{2}}{\Psi}+\frac{d\eta^{2}}{\Psi}$ on $T^{*}\mathbb{R}^{2d}_{q,p}=\mathbb{R}^{4d}_{q,p,\xi,\eta}$ is a splitted H{\"o}rmander metric with the gain function $\Psi(q,p,\xi,\eta)$\,.\\
  Additionally it is geodesically temperate with $g^{\sigma}_{\Psi}=\Psi^{2}g_{\Psi}$\,.\\
  For any $s\in \mathbb{R}$\,,\, the function $\Psi^{s}$ is a $g_{\Psi}$-weight for any $s\in \mathbb{R}$ and an elliptic symbol in $S(\Psi^{s},g_{\Psi})$\,.
\end{proposition}
Remember that $\sigma=\sum_{j=1}^{2d}d\Xi_{j}\wedge dx_{j}$ denotes the canonical symplectic form on $\mathbb{R}^{4d}_{x,\Xi}$\,, with here $x=(q,p)$ and $\Xi=(\xi,\eta)$.\\
We follow the usual abusive convention for the presentation of Weyl-Hörmander calculus and write shortly $g_{\Psi,X}(T)$ for the quadratic form applied to the tangent vector $T=(t_{x},t_{\Xi})$ instead of $g_{\Psi,X}(T,T)$\,.
\begin{proof}
  The properties $g^{\sigma}_{\Psi}=\Psi^{2}g_{\Psi}$ and $g_{\Psi,X}(t_{x},-t_{\Xi})=g_{\Psi,X}(t_{x},t_{\Xi})$ ($g_{\Psi}$ is splitted)\,, $\Psi\geq 1$ (H{\"o}rmander uncertainty condition) and $\Psi^{s}\in S(\Psi^{s},g_{\Psi};\mathbb{C})$ are obvious.\\
  The inequality
  \begin{equation}
  \label{eq:rappgXgXprime}
\left(\frac{g_{\Psi,X}}{g_{\Psi,X'}}\right)^{-\pm 1}\leq \max \left\{1,\left(\frac{\Psi(X)}{\Psi(X')}\right)^{\pm 2}, \left(\frac{\Psi(X)}{\Psi(X')}\right)^{\pm 1}\right\}\,,
\end{equation}
says that the slowness and (geodesic) temperance are proved when $\Psi$ is a slow and (geodesically) tempered weight for $g_{\Psi}$\,.
\\
\noindent\textbf{Slowness:} Set $X=(q,p,\xi,\eta)$ and $X'=(q',p',\xi',\eta')$\,.For $g_{\Psi,X}(X'-X)\leq \frac{1}{R^{2}}$ let us prove $\left(\frac{\Psi(X)}{\Psi(X')}\right)^{-1}\leq R^{2}$ for $R>1$ large enough.\\
The assumption implies
$$
|\tilde{\xi}-\tilde{\xi}'|\leq \frac{\sqrt{1+|\tilde{\xi}|^{2}}}{R}=\frac{\langle \tilde{\xi}\rangle}{R}\quad\text{with}~\left\{
    \begin{array}[c]{l}
      \tilde{\xi}=\frac{1}{(1+|p|^{4}+|\eta|^{4})^{1/2}}\xi
      \\
      \tilde{\xi}'=\frac{1}{(1+|p|^{4}+|\eta|^{4})^{1/2}}\xi'\,.
    \end{array}
  \right.
$$
We deduce
$$
|\tilde{\xi}'|^{2}\leq 2|\tilde{\xi}|^{2}+2|\tilde{\xi}'-\tilde{\xi}|^{2}\leq
(1+\frac{2}{R^{2}})\langle \tilde{\xi}\rangle^{2}\,,
$$
and
$$
|\tilde{\xi}|^{2}\leq 2\langle \tilde{\xi}'\rangle^{2}+\frac{2}{R^{2}}\langle \tilde{\xi}\rangle^{2}\,.
$$
This gives for $R\geq 2$\,,
$$
(1-\frac{2}{R^{2}})\langle \tilde{\xi}\rangle^{2}\leq 2\langle \tilde{\xi}'\rangle^{2}\leq 2(1+\frac{2}{R^{2}})\langle \tilde{\xi}\rangle^{2}
$$
and
$$
\left(\frac{\Psi(q,p,\xi,\eta)}{\Psi(q',p,\xi',\eta)}\right)^{\pm 1}\leq 2(1+\frac{2}{R^{2}})\,.
$$
Let us consider now the quantity
$$
\left(\frac{\Psi(q',p,\xi',\eta)}{\Psi(q',p',\xi',\eta')}\right)^{\pm 1}
$$
while noticing that the first result and the assumption $g_{\Psi,X}(X-X')\leq \frac{1}{R^{2}}$ implies
\begin{eqnarray*}
  &&|p-p'|\leq \frac{1}{R}\Psi(q,p,\xi,\eta)^{1/2}\leq \frac{\sqrt{2(1+2/R^{2})}}{R}\Psi(q',p,\xi',\eta)^{1/2}\leq
     \frac{2}{R}(1+|\xi'|^{2}+|p|^{4}+|\eta|^{4})^{1/4}\\
   \text{and}&&
                |\eta-\eta'|\leq \frac{\sqrt{2(1+2/R^{2})}}{R}\Psi(q',p,\xi',\eta)^{1/2}\leq  
                \frac{2}{R}(1+|\xi'|^{2}+|p|^{4}+|\eta|^{4})^{1/4}
                \\
  \text{with}&&
    \frac{2}{R}(1+|\xi'|^{2}+|p|^{4}+|\eta|^{4})^{1/4}\leq \frac{2\langle \xi'\rangle^{1/2}(1+|\tilde{p}|^{2}+|\tilde{\eta}|^{2})^{1/2}}{R}
\end{eqnarray*}
by setting $\tilde{p}=\langle \xi'\rangle^{-1/2}p$ and $\tilde{\eta}=\langle \xi'\rangle^{-1/2}\eta$\,.
By using the same normalization for $(\tilde{p}',\tilde{\eta}')$ we deduce that the vectors $Y=(\tilde{p},\tilde{\eta})$ and $Y'=(\tilde{p}',\tilde{\eta}')$ satisfy

$$
|Y-Y'|\leq \frac{2}{R}\langle Y\rangle
$$
and again
$$
\left(\frac{\Psi(q',p,\xi',\eta)}{\Psi(q',p',\xi',\eta')}\right)^{\pm 1/2}
=\left(\frac{\langle Y\rangle}{\langle Y'\rangle}\right)^{\pm 1}\leq 2(1+\frac{8}{R^{2}})
$$
when $R/2>2$\,.\\
We deduce the uniform inequality
$$
\left(\frac{\Psi(X)}{\Psi(X')}\right)^{\pm 1}
=
\left(\frac{\Psi(q,p,\xi,\eta)}{\Psi(q',p,\xi',\eta)}\right)^{\pm 1}\times
\left(\frac{\Psi(q',p,\xi',\eta)}{\Psi(q',p',\xi',\eta')}\right)^{\pm 1}
\leq 2(1+\frac{2}{R^{2}})4(1+\frac{8}{R^{2}})^{2}\leq 2^{12}\leq R
$$
as soon as $g_{\Psi,X}(X'-X)\leq \frac{1}{R^{2}}$
if $R\geq 2^{12}$\,.\\

\medskip
\noindent\textbf{Geodesic Temperance:} With $g^{\sigma}_{\Psi}\geq dq^{2}+d\xi^{2}+dp^{2}+d\eta^{2}$\,, we get $g^{\sigma}_{\Psi,X}(X-X')\geq |X-X'|^{2}$ and the same inequality holds for the geodesic distance for $g_{\Psi}^{\sigma}$\,,  $d^{\sigma}_{\Psi}(X,X')\geq \left|X-X'\right|$\,.\\
From
$$
\Psi(q',p',\xi',\eta')^{2}\leq 1+|\xi'|^{2}+(|p'|^{2}+|\eta'|^{2})^{2}
\leq 1+2|\xi|^{2}+2|\xi'-\xi|^{2}+\left(2|p|^{2}+2|p'-p|^{2}+2|\eta|^{2}+2|\eta'-\eta|\right)^{2}
$$
we deduce
$$
\Psi(q',p',\xi',\eta')^{2}\leq 64 \Psi(q,p,\xi,\eta)^{2}(1+|\xi'-\xi|^{2})(1+|p'-p|^{2}+|\eta'-\eta|^{2})^{2} 
$$
and the symmetric version results from the exchange $X\leftrightarrow X'$\,.
We obtain
$$
\left(\frac{\Psi(X)}{\Psi(X')}\right)^{\pm 2}\leq 64 (1+|q-q'|^{2}+|\xi-\xi'|^{2}+|p'-p|^{2}+|\eta'-\eta|^{2})^{3}\leq 64(1+|X-X'|^{2})^{3}\,.
$$
With  $|X-X'|^{2}\leq \min (g_{\Psi,X}^{\sigma}(X-X'); d_{\Psi}^{\sigma}(X,X')^{2})$\,, this proves that the weight $\Psi$\,, and  the metric $g_{\Psi}$ owing to \eqref{eq:rappgXgXprime}, are geodesically tempered.
\end{proof}
All the result of \cite{HormIII}-Chap~XVIII can be applied for the Weyl quantization $a^{W}(q,p,D_{q},D_{p})$ when $a\in S(m,g_{\Psi})$ and $m$ is a $g_{\Psi}$-weight\,. Because $g_{\Psi}$ is splitted the Weyl and standard quantizations are equivalent and we recall
$a(x,D_{x})=b^{W}(x,D_{x})$ with
\begin{eqnarray*}
  && a=e^{iD_{x}.D_{\Xi}/2}b=\sum_{n=0}^{N-1}\frac{(iD_{x}.D_{\Xi}/2)^{n}}{n!}b+R_{N,+}(b)\\
  && b=e^{-iD_{x}.D_{\Xi}/2}a=\sum_{n=0}^{N-1}\frac{(-iD_{x}.D_{\Xi}/2)^{n}}{n!}a+R_{N,-}(a)
\end{eqnarray*}

where every $n$-th term is continuous from $S(m,g_{\Psi})$ to $S(m\Psi^{-n},g_{\Psi})$ while the remainders are continous from $S(m,g_{\Psi})$ to $S(m\Psi^{-N},g_{\Psi})$\,. Accordingly if $a_{1}\sharp^{W}a_{2}$ (resp. $a_{1}\sharp a_{2}$) denote the symbols of $a_{1}^{W}(x,D_{x})\circ a_{2}^{W}(x,D_{x})$ (resp. $a_{1}(x,D_{x})\circ a_{2}(x,D_{x})$) we know
\begin{align*}
  a_{1}\sharp^{W} a_{2}(X)&=e^{i\sigma(D_{X_{1}},D_{X_{2}})/2}a_{1}(X_{1})a_{2}(X_{2})\big|_{X_{1}=X_{2}=X}\\
                          &=\sum_{n=0}^{N-1}\frac{(i\sigma(D_{X_{1}},D_{X_{2}})/2)^{n}}{n!}a_{1}(X_{1})a_{2}(X_{2})\big|_{X_{1}=X_{2}=X}+R_{N}^{W}(a_{1},a_{2})\,,
\end{align*}
and respectively
\begin{align*}
a_{1}\sharp a_{2}(X)=e^{iD_{\Xi_{1}}D_{x_{2}}}a_{1}(X_{1})a_{2}(X_{2})\big|_{X_{1}=X_{2}=X}&=\sum_{n=0}^{N-1}\frac{(iD_{\Xi_{1}}D_{x_{2}})^{n}}{n!}a_{1}(X_{1})a_{2}(X_{2})\big|_{X_{1}=X_{2}=X}+R_{N}(a_{1},a_{2})
  \\
                    &=\sum_{n=0}^{N-1}\sum_{|\alpha|\leq n}\frac{1}{i^{|\alpha|}\alpha!}\partial_{\Xi}^{\alpha}a_{1}\partial_{x}^{\alpha}a_{2}+R_{N}(a_{1},a_{2})\,,
\end{align*}
where every $n$-th term is is bilinear continuous from
$S(m_{1},g_{\Psi})\times S(m_{2},g_{\Psi})$ to $S(m_{1}m_{2}\Psi^{-n},g_{\Psi})$ while the remainder is bilinear  continous from $S(m_{1},g_{\Psi})\times S(m_{2},g_{\Psi})$ to $S(m_{1}m_{2}\Psi^{-N},g_{\Psi})$\,. Two differences: $a^{W}(x,D_{x})^{*}=(\bar{a})^{W}(x,D_{x})$ remains true only modulo $S(m\Psi^{-1},g_{\Psi})$ for the classical quantization while $f(x)a(x,D_{x})=(fa)(x,D_{x})$ remains true only modulo $S(m\Psi^{-1},g_{\Psi})$ for the Weyl quantization.\\
In \cite{BoCh} were introduced the general Sobolev spaces $H(m,g)$ for any  H{\"o}rmander metric $g$ and $g$-weight $m$ as Hilbert spaces with the norms $\|u\|_{H(m,g)}=\|M^{W}(x,D_{x})u\|_{L^{2}}$\,, where $M\in S(m,g)$ is any fixed elliptic invertible operator.\\
We are concerned here with a simple case.
\begin{definition}
  \label{de:tWsRd} For $s\in \mathbb{R}$ the space $\tilde{\mathcal{W}}^{s}(\mathbb{R}^{2d};\mathbb{C})$ is nothing but $H(\Psi^{s},g_{\Psi})$ with the norm
  $$
  \|u\|_{\tilde{\mathcal{W}}^{s}}=\|(M_{s})^{W}(x,D_{x})u\|_{L^{2}}\,.
  $$
  with $M_{s}=(C_{s}+\Psi^{|s|})^{\mathrm{sign}\,s}$ for some $C_{s}\geq 1$\,.
\end{definition}
The invertibility of $M_{s}(x,D_{x}):\tilde{\mathcal{W}}^{s}(\mathbb{R}^{2d};\mathbb{C})\to L^{2}(\mathbb{R}^{2d},dqdp;\mathbb{C})$ comes from $(C_{s}+\Psi^{|s|})^{-1}\sharp^{W}(C_{s}+\Psi^{|s|})=1+\mathcal{O}(1/C_{s})$ in $S(1,g_{\Psi})$\,.\\
Because $g_{\Psi}$ is geodesically tempered with $g_{\Psi}^{\sigma}=\Psi^{2}g_{\Psi}$\,, J.M.~Bony provides us a simple version of Beals criterion in \cite{Bon}. In our particular case the symbol class $S^{+}(1,g_{\Psi})$ is nothing but the set of $a\in \mathcal{C}^{\infty}(\mathbb{R}^{2d};\mathbb{C})$ such that for any $i\in \left\{1,\ldots,d\right\}$\, $\partial_{q^{i}}a\in S(\Psi,g_{\Psi})$\,,\, $\partial_{\xi_{i}}a\in S(1,g_{\Psi})$\,, while $\partial_{p_{i}}a$ and $\partial_{\eta^{i}}a$ belong to $S(\Psi^{1/2};g_{\Psi})$\,. An operator $A$ equals $a^{W}(x,D_{x})$ with $a\in S(1,g_{\Psi})$ if and only if all the commutators $\mathrm{ad}_{b_{1}^{W}(x,D_{x})} \dots \mathrm{ad}_{b_{N}^{W}(x,D_{x})}A$ for $b_{n}\in S^{+}(1,g_{\Psi})$ are bounded in $\mathcal{L}(L^{2})$\,.\\
Alternatively, the simpler and original version of  Beals criterion in \cite{Bea} works here according \cite{BoCh} (see \cite{NaNi} for a detailed version of Remark~5.6 in \cite{BoCh}) owing to the three properties
\begin{itemize}
\item The metric is diagonal in the canonical basis $\mathcal{B}$ of 
$\mathbb{R}^{4d}=T^*(\mathbb{R}^{2d})$\,, written as 
$$\mathcal{B}=\left\{\partial_{q^{i}},\partial_{p_{i}}, \partial_{\xi_{i}},\partial_{\eta^{i}}\,, 1\leq i\leq d\right\}\,,$$
while the convex hull  $C_{X,\mathcal{B}}=\left\{\sum_{e\in \mathcal{B}}t_{e}g_{\Psi,X}(e)^{-1/2}e, (t_{e})_{e\in \mathcal{B}}\in[-1,1]^{\sharp \mathcal{B}}\right\}$ satisfies
$$
\exists r\in ]0,1]\,, \forall X\in \mathbb{R}^{4d}=T^{*}(\mathbb{R}^{2d})\,,\quad B_{g_{\Psi,X}}(0,r)\subset C_{X,\mathcal{B}}\subset B_{g_{\Psi},X}(0,2)\,.
$$
\item If $L(e)=(\sigma(e,X))^{W}$ for $e\in \mathcal{B}$ with $\sigma = d\eta\wedge dp + d\xi\wedge dq$ and $X$ the radial vector field\,, we get $L(\partial_{q^{i}})=-D_{q^{i}}$\,, $L(\partial_{p_{i}})=-D_{p_{i}}$\,, $L(\partial_{\xi_{i}})=q^{i}$ and $L(\partial_{\eta^{i}})=p^{i}$\,.
  \item For a finite familly $E=(e_{n})_{1\leq n\leq N}$ of elements of $\mathcal{B}$ the weight $m_{E}(X)$ equals
$$
m_{E}(X)=\prod_{n=1}^{N_{E}}g_{\Psi,X}(e_{k})^{1/2}=\Psi(X)^{-N_{1}-N_{2}/2} \quad\text{with}~
\left\{
  \begin{array}[c]{l}
    N_{1}=\sharp\left\{k, e_{k}\in \left\{\partial_{\xi_{i}}\right\}\right\}\\
    N_{2}=\sharp\left\{k, e_{k}\in \left\{\partial_{p_{i}},\partial_{\eta_{i}}\right\}\right\}
  \end{array}
\right.
$$
\end{itemize}
The Beals criterion thus says that $A=a^{W}(q,p,D_{q},D_{p})$ with $a\in S(\Psi^{s},g_{\Psi})$ for some $g_{\Psi}$-weight $m$\,, if and only if all the commutators
$$
\mathrm{ad}_{q}^{\alpha}\mathrm{ad}_{p}^{\beta}\mathrm{ad}_{D_{q}}^{\gamma}\mathrm{ad}_{D_{p}}^{\delta}A
$$
initially defined as continuous operators on $\mathcal{S}(\mathbb{R}^{2d};\mathbb{C})$ (or on $\mathcal{S}'(\mathbb{R}^{2d};\mathbb{C})$)
actually belongs to
$$
\mathcal{L}(\tilde{\mathcal{W}}^{s_{0}}(\mathbb{R}^{2d};\mathbb{C});\tilde{\mathcal{W}}^{s_{0}-s+|\alpha|+\frac{|\beta|+|\delta|}{2}}(\mathbb{R}^{2d};\mathbb{C}))
$$
for some $s_{0}\in \mathbb{R}$ (and equivalently for all $s_{0}\in \mathbb{R}$).
Additionally the topology on $\mathcal{S}(\Psi^{s};g_{\Psi})$ is equivalently defined by the family of seminorms $(q_{\Psi^{s},k})_{k\in \mathbb{N}}$ or $(\tilde{q}_{\Psi^{s}, k})_{k\in\mathbb{N}}$\,,
\begin{eqnarray*}
  && q_{\Psi^{s},k}(A)=\max_{|\alpha|+|\beta|+|\gamma|+|\delta|\leq k}\|\mathrm{ad}_{q}^{\alpha}\mathrm{ad}_{p}^{\beta}\mathrm{ad}_{D_{q}}^{\gamma}\mathrm{ad}_{D_{p}}^{\delta}A\|_{\mathcal{L}(L^{2};\tilde{\mathcal{W}}^{-s+|\alpha|+\frac{|\beta|+|\delta|}{2}})}\,.
\end{eqnarray*}
Beals criterion is especially convenient for the link between a global pseudo-differential calculus and functional analysis.
\begin{proposition}
  \label{pr:Bealsfunction}
  Let $A=a^{W}(x,D_{x})$ be a self-adjoint operator in $L^{2}(\mathbb{R}^{2d},dqdp;\mathbb{C})$ with an elliptic (and real) symbol $a\in S(\Psi^{\mu},g_{\Psi})$ (ellipticity means here $a\geq \frac{1}{C}\Psi^{\mu}$ uniformly on $\mathbb{R}^{2d}$) such that $D(A)=\tilde{\mathcal{W}}^{\mu}(\mathbb{R}^{2d};\mathbb{C})$\,. Then for any $f\in S(\langle t\rangle^{s}, \frac{dt^{2}}{\langle t\rangle^{2}}; \mathbb{C})$\,, the operators $f(A)$ and $f(A)-f(a)^{W}(q,p,D_{q},D_{p})$ are pseudo-differential operators with symbols respectively in $S(\Psi^{\mu s},g_{\Psi})$ and $S(\Psi^{\mu s-1},g_{\Psi})$\,.\\
  If additionally $A\geq C\mathrm{Id}_{L^{2}}$ with $C>0$\,, then the same result holds for $A^{s}$ and $A^{s}-(a^{s})^{W}(q,p,D_{q},D_{p})$\,.
\end{proposition}
\begin{proof}
  The proof of Bony in \cite{Bon}-Theorem~3.8 relies on Helffer-Sj{\"o}strand functional calculus formula very convenient with Beals criterion (seminorms on $S(\Psi^{s},g_{\Psi})$ are expressed in terms of norms of commutators). We refer the reader to \cite{Bon}\cite{DiSj}\cite{HeSj} or to the end of Subsection~\ref{sec:global} for a more detailed use of Helffer-Sj{\"o}strand formula. The only thing which was not verified in \cite{Bon}, because it is about a more general framework, is  the principal symbol statement. Actually we can focus on $\mu\geq 0$ and when one knows that $(z-A)^{-1}=b_{z}^{W}(x,D_{x})$ with  seminorms of $b_{z}$ estimated by $p_{\Psi^{-\mu},k}(b)\leq C_{k}\frac{\langle z\rangle^{N_{k}}}{|\mathrm{Im}\, z|^{N_{k}}}$ it suffices to write $\left(\frac{1}{z-a}\right)^{W}\circ(z-A)=1+r^{W}_{z}(x,D_{X})$ with seminorms of $r_{z}\in S(\Psi^{-1},g_{\Psi})$ estimated by $p_{\Psi^{-1}, k}(r_{z})\leq C'_{k}\frac{\langle z\rangle^{N'_{k}}}{|\mathrm{Im}z|^{N'_{k}+1}}$\,. We deduce $\left(\frac{1}{z-a}\right)^{W}-(z-A)^{-1}=c_{z}^{W}(x,D_{x})$ with $p_{\Psi^{-\mu-1},k}(c_{z})$ estimated by $\frac{\langle z\rangle^{N''_{k}}}{|\mathrm{Im}\,z|^{N''_{k}+1}}$\,. Inserting this into Helffer-Sj{\"o}strand formula proves the result for $s<0$\,, by simple integration. For $s\geq  0$\,, write $f(A)=(i+A)^{N}f_{N}(A)$ with $f_{N}(t)=(i+t)^{-N}f(t)\in S(\langle t\rangle^{s-N},\frac{dt^{2}}{\langle t\rangle^{2}})$ and $N$ large enough.
\end{proof}
The former result provides us an easy way for comparing various simple definitions of the spaces $\tilde{\mathcal{W}}^{s}(\mathbb{R}^{2d};\mathbb{C})$ and the equivalence of the norms.
\begin{proposition}
  \label{pr:eqnormes} Consider the self-adjoint operator $A=1-\Delta_{q}^{2}+\left(\frac{-\Delta_{p}+|p|^{2}}{2}\right)^{2}$ in $L^{2}(\mathbb{R}^{2d},dqdp;\mathbb{C})$ with domain $D(A)=\tilde{\mathcal{W}}^{2}(\mathbb{R}^{2d};\mathbb{C})$\,. Let $\sum_{\ell=-1}^{\infty}\theta_{\ell}^{2}(t)\equiv 1$  on $[0,+\infty)$ be a quadratic dyadic partition of unity like in \eqref{eq:dyadic}. 
  Then the following squared norms on $\tilde{\mathcal{W}}^{s}(\mathbb{R}^{2d};\mathbb{C})$ equivalent with $\|~\|^{2}_{\tilde{\mathcal{W}}^{s}}$:
  \begin{description}
  \item[i)] $\|A^{s/2}u\|^{2}_{L^{2}(\mathbb{R}^{2d},dqdp;\mathbb{C})}$ for any  $s\in \mathbb{R}$\,;
  \item[ii)] $\|(\frac{-\Delta_p+|p|^2}{2})^s u\|^2_{L^{2}(\mathbb{R}^{2d},dqdp;\mathbb{C})}+\||D_s|^s u\|^2_{L^{2}(\mathbb{R}^{2d},dqdp;\mathbb{C})}$ for $s\geq 0$\,;
  \item[iii)] $\sum_{\ell=-1}^{\infty}\|\theta_{\ell}(|p|^{2})u\|^{2}_{\tilde{\mathcal{W}}^{s}}$\,;
  \item[iv)] for $s=k\in \mathbb{N}$\,, $\sum_{|\alpha|+\frac{|\beta|+|\gamma|}{2}\leq k}\|\partial_{q}^{\alpha}p^{\beta}\partial_{p}^{\gamma}u\|^{2}_{L^{2}(\mathbb{R}^{2d},dqdp;\mathbb{C})}$\,;
    \item[v)] for $s=k\in \mathbb{N}$\,, $\sum_{|\alpha|+\frac{N_{3}+|\gamma|}{2}\leq k}\|\langle p\rangle^{N_{3}}\partial_{q}^{\alpha}\partial_{p}^{\gamma}u\|^{2}_{L^{2}(\mathbb{R}^{2d},dqdp:\mathbb{C})}$\,;
  \end{description}
\end{proposition}
\begin{proof}
  The equivalence with  \textbf{i)} is a direct application of Proposition~\ref{pr:Bealsfunction} because $C+A=a^{W}(q,p,D_{q},D_{p})$ with $a=C+1+|\xi|^{2}+\frac{1}{4}(|p|^{2}+|\eta|^{2})^{2}~\mod S(\Psi^{1},g_{\Psi})$ which and $a$ is elliptic for $C>0$ large enough. For $s\in \mathbb{R}$ the functional calculus says that $\|A^{s/2}u\|_{L^{2}}$ is equivalent to $\|(C+A)^{s/2}u\|_{L^{2}}$\,. But for $C\geq C_{s}>0$ large enough $(C+A)^{s/2}=a_{s}^{W}(q,p,D_{q},D_{p})$ with $a_{s}$ elliptic in $S(\Psi^{s},g_{\psi})$ and $\|(C+A)^{s/2}u\|_{L^{2}}$ is equivalent to $\|u\|_{\tilde{\mathcal{W}}^{s}}$\,.\\
  The statement \textbf{ii)} is actually a consequence of the functional calculus with $(1+t^2+{t'}^2)^{s}\asymp t^{2s}+{t'}^{2s}$ for all $t\geq 0, t'\geq d/2$ when  $s\geq 0$ is fixed\,.\\
  For $s=k$ the squared norms of \textbf{iv)} and \textbf{v)} are equivalent to $\langle u\,,\, B_{iv}u\rangle_{L^{2}}$ and $\langle u\,,\, B_{v}u\rangle$ with
$$
B_{iv}=C+\sum_{|\alpha|+\frac{|\beta|+|\gamma|}{2}\leq k}D_{p}^{\gamma}D_{q}^{2\alpha}p^{2\beta}D_{p}^{\gamma}\quad,\quad
B_{v}=C+\sum_{|\alpha|+\frac{N_{3}+|\gamma|}{2}\leq k}D_{p}^{\gamma}D_{q}^{2\alpha}\langle p\rangle^{2N_{3}}D_{p}^{\gamma}
$$
which both have a symbol elliptic in  $S(\Psi^{2k},g_{\Psi})$ for $C>0$ large enough. By Proposition~\ref{pr:Bealsfunction}, the operators $(C+B_{iv})^{1/2}$ and $(C+B_{v})^{1/2}$ have an elliptic symbol in $S(\Psi^{k},g_{\Psi})$ for $C>0$ large enough and the two norms of \textbf{iv)} and \textbf{v)} are equivalent to $\|u\|_{\tilde{\mathcal{W}}^{k}}$ for $k\in \mathbb{N}$\,.\\
Finally for \textbf{iii)}, the partial Fourier $F_{q\to\xi}$ transform with respect to $q$ sends $L^{2}(\mathbb{R}^{2d},dqdp;\mathbb{C})$ onto the direct integral $\int_{\mathbb{R}^{d}}L^{2}(\mathbb{R}^{d},dp;\mathbb{C})~\frac{d\xi}{(2\pi)^{d}}$ and for $s=k\in\mathbb{N}$ the squared norm in $\tilde{\mathcal{W}}^{k}(\mathbb{R}^{2};\mathbb{R})$   is equivalent to 
$$
N_{C,k}(u)^{2}=C\|u\|_{L^{2}}^{2}+\sum_{|\alpha|+\frac{|\beta|+|\gamma|}{2}\leq k}\|\xi^{\alpha}p^{\beta}\partial_{p}^{\gamma}u\|_{L^{2}}^{2}$$
and to $\sum_{2|\alpha|+|\beta|+|\gamma|\leq 2}N_{C,k-1}(\xi^{\alpha}p^{\beta}\partial_{p}^{\gamma}u)^{2}$\,. 
We are considering operators of order $2$ in $\partial_{p}$ and Corollary~\ref{Cor:equivalenceOfQuantities} can be applied with a recurrence with respect to $k\in \mathbb{N}$ and gives
$$
\frac{1}{C_{k}}\sum_{\ell=-1}^{\infty}N_{C,k}(\theta_{\ell}(|p|^{2})u)\leq N_{C,k}(u)^{2}\leq C_{k}\sum_{\ell=-1}^{\infty}N_{C,k}(\theta_{\ell}(|p|^{2})u)^{2}\,,
$$
for all $u\in \mathcal{C}^{\infty}_{0}(\mathbb{R}^{2d};\mathbb{C})$ and by density for all $u\in \tilde{\mathcal{W}}^{k}(\mathbb{R}^{2d};\mathbb{C})$\,. The result for $s\in \mathbb{R}$ follows  by interpolation and duality.
\end{proof}
\begin{remark}
\label{re:2ndmicro}
  The equivalence with \textbf{ii)} could be done with a semi-classical calculus $a^{W}(\sqrt{h}p,\sqrt{h}D_{p})$ with the semiclassical parameter $h=\frac{1}{\sqrt{C+|\xi|^{2}}}$ with the symbol classes $S(\langle p,\eta\rangle^{\mu_{1}}\langle p\rangle^{\mu_{2}}\langle \eta\rangle^{\mu_{3}},\frac{dp^{2}}{\langle p\rangle^{2}}+\frac{d\eta^{2}}{\langle \eta\rangle^{2}})$ with $1+\frac{h^{2}}{4}(-\Delta_{p}+|p|^{2})^{2}=a_{1}(\sqrt{h}p,\sqrt{h}D_{p})$\,, $a_{1}$ elliptic in $S(\langle p,\eta\rangle^{4},\frac{dp^{2}}{\langle p\rangle^{2}}+\frac{d\eta^{2}}{\langle \eta\rangle^{2}})$ and $\chi_{\ell}(|p|^{2})\in S(1,\frac{dp^{2}}{\langle p\rangle^{2}}+\frac{d\eta^{2}}{\langle \eta\rangle^{2}})$ (see e.g. \cite{Rob} or \cite{NaNi}). The necessity of two H{\"o}rmander metrics suggests the link with the second microlocalization of \cite{BoLe}. The chosen elementary method suffices here.
\end{remark}

\subsection{Localization and geometric invariance} 
\label{sec:geominv}
For the localization in $\Omega\times \mathbb{R}^{3d}$\,, $q\in \Omega$\,, $\Omega$ open set of $\mathbb{R}^{d}$\,, it is more convenient to work with the classical  quantization:
$$
[a(q,p,D_{q},D_{p})](q,p,q',p')=\int_{\mathbb{R}^{2d}}e^{i[\xi.(q-q')+\eta.(p-p')}a(q,p,\xi,\eta)~\frac{d\xi d\eta}{(2\pi)^{2d}}\,,
$$
for which $\varrho(q)a(x,D_{x})=[\varrho(q)a](x,D_{x})$\,.\\
Every $a\in \mathcal{S}'_{\Omega-\mathrm{loc}}(\Omega\times\R^{3d};\mathbb{C})$ gives rise to a kernel in $\mathcal{S}'_{\Omega\times \Omega-\mathrm{loc}}(\Omega\times \Omega\times\mathbb{R}^{2d};\mathbb{C})$ and therefore a continuous  operator $A_{a}$ from $\mathcal{S}_{\Omega-\mathrm{comp}}(\Omega\times \mathbb{R}^{d};\mathbb{C})$ to $\mathcal{S}'_{\Omega-\mathrm{loc}}(\Omega\times\mathbb{R}^{d};\mathbb{C})$ and $a\mapsto A_{a}$ is a bijection. Consider the symbol classe $S^{m}_{\Psi,\Omega-\mathrm{loc}}(\Omega;\mathbb{C})$ characterized by \eqref{eq:Sloc} with the associated space $S^{m}_{\Psi,\Omega-\mathrm{comp}}(\Omega;\mathbb{C})$ and the spaces $\tilde{\mathcal{W}^{s}}_{\Omega-{\mathrm{loc}}}(\Omega;\mathbb{C})$\,, $\tilde{\mathcal{W}}^{s}_{\Omega-\mathrm{comp}}(\Omega;\mathbb{C})$\,. For two open sets $\Omega$ and $\Omega'$ of $\mathbb{R}^{d}$ and $\varrho\in \mathcal{C}^{\infty}_{0}(\Omega;\mathbb{C})$ we have the following continuous embeddings when the letter $E$ in $E(\Omega)$ stands for $S^{m}_{\Psi}$\,, $\tilde{\mathcal{W}}^{s}$\,,\, $\mathcal{S}$ or $\mathcal{S}'$: 
\begin{eqnarray*} 
  && E_{\Omega-\mathrm{comp}}(\Omega)\subset E_{\mathbb{R}^{d}-\mathrm{comp}}(\mathbb{R}^{d})\subset E_{\Omega'-\mathrm{loc}}(\Omega')\\
  && \varrho(q)E_{\Omega'-\mathrm{comp}}(\Omega')\subset \varrho(q)E_{\mathbb{R}^{d}-\mathrm{comp}}(\mathbb{R}^{d})\subset E_{\Omega-\mathrm{comp}}(\Omega)\,.
\end{eqnarray*}
Notice also for $\bullet=\mathrm{loc}$ or $\mathrm{comp}$:
$$
\mathop{\bigcap}_{s\in \mathbb{R}}\tilde{\mathcal{W}^{s}}_{\Omega-\bullet}(\Omega\times\mathbb{R}^{d};\mathbb{C})=\mathcal{S}_{\Omega-\bullet}(\Omega\times\mathbb{R}^{d};\mathbb{C})\quad\text{and}\quad
\mathop{\bigcup}_{s\in \mathbb{R}}\tilde{\mathcal{W}^{s}}_{\Omega-\bullet}(\Omega\times\mathbb{R}^{d};\mathbb{C})=\mathcal{S}_{\Omega-\bullet}'(\Omega\times\mathbb{R}^{d};\mathbb{C})
$$
In particular symbols $a\in S^{m}_{\Psi,\Omega-\mathrm{comp}}(\Omega;\mathbb{C})$ can be viewed as symbols $a\in S(\Psi^{m},g_{\Psi})$\,. Therefore $a(x,D_{x})$ defines a continuous operator from $\tilde{\mathcal{W}}^{s}_{\Omega-\mathrm{comp}}(\Omega\times\mathbb{R}^{d};\mathbb{C})$ to $\tilde{\mathcal{W}}^{s}_{\Omega-\mathrm{comp}}(\Omega\times\mathbb{R}^{d};\mathbb{C})$\,. For $\chi\in \mathcal{C}^{\infty}_{0}(\Omega;[0,1])$ such that $\chi\equiv 1$ on a neighborhood $\Omega_{\chi}$  of $\Omega-\mathrm{supp}\, a$ we have
\begin{eqnarray*}
&&a(x,D_{x})\chi(q):\tilde{\mathcal{W}}^{s}_{\Omega-\mathrm{loc}}(\Omega\times\mathbb{R}^{d};\mathbb{C})\to
\tilde{\mathcal{W}}^{s}_{\Omega-\mathrm{comp}}(\Omega\times\mathbb{R}^{d};\mathbb{C})\\
\text{and}&&
             a(x,D_{x})\chi(q)|_{\tilde{\mathcal{W}}^{s}_{\Omega_{\chi}-\mathrm{comp}}}=a(x,D_{x})\big|_{\tilde{\mathcal{W}}^{s}_{\Omega_{\chi}-\mathrm{comp}}}\,.
\end{eqnarray*}
For two different choices of $\chi$ and $\chi'$ which are equal to $1$ in a neighborhood of $\Omega-\mathrm{supp}\,a$\,, $a\in S^{m}_{\Psi,\Omega-\mathrm{comp}}(\Omega;\mathbb{C})$\,, the differences $a(x,D_{x})\chi(q)-a(x,D_{x})\chi'(q)=a(x,D_{x})(\chi(q)-\chi'(q))=b(x,D_{x})$ with $b\in S(\Psi^{-\infty},g_{\Psi})$ and therefore is continuous from $\mathcal{S}'_{\Omega'-\mathrm{loc}}(\Omega'\times\mathbb{R}^{d};\mathbb{C})$ to $\mathcal{S}_{\Omega-\mathrm{comp}}(\Omega\times\mathbb{R}^{d};\mathbb{C})$\,, where $\Omega'=\Omega$ or any other open subset of $\mathbb{R}^{d}$\,.

\begin{definition}
  \label{def:ROmega} For an open set $\Omega$ the set of regularizing operators is
  $$
  \mathcal{R}(\Omega;\mathbb{C})=\mathcal{L}(\mathcal{S}'_{\Omega-\mathrm{loc}}(\Omega\times\mathbb{R}^{d};\mathbb{C});\mathcal{S}_{\Omega-\mathrm{comp}}(\Omega\times\mathbb{R}^{d};\mathbb{C}))\,.
  $$
  For two operators $A,B$ continuous from $\mathcal{S}_{\Omega-\mathrm{comp}}(\Omega\times\mathbb{R}^{d};\mathbb{C})$ to $\mathcal{S}'_{\Omega-\mathrm{comp}}(\Omega\times\mathbb{R}^{d};\mathbb{C})$ the equivalence $A\sim B$ is defined by $A-B\in \mathcal{R}(\Omega;\mathbb{C})$\,.\\
  The set $\mathrm{Op S}^{m}_{\Psi}(\Omega;\mathbb{C})$ is the set of sums $a(x,D_{x})\chi(q)+R$ with $a\in S^{m}_{\Psi,\Omega-\mathrm{comp}}(\Omega;\mathbb{C})\subset S(\Psi^{m},g_{\Psi})$\,, $\chi\in \mathcal{C}^{\infty}_{0}(\Omega;[0,1])$ with $\chi\equiv 1$ on a neighborhood of $\Omega-\mathrm{supp}\, a$\,, and $R\in \mathcal{R}(\Omega;\mathbb{C})$\,. The set $\mathrm{OpS}^{-\infty}_{\Psi}(\Omega;\mathbb{C})$ is $\mathcal{R}(\Omega;\mathbb{C})$\,.
\end{definition}
With the previous remarks $\mathop{\bigcup}_{m\in \mathbb{R}}\mathrm{OpS}^{m}_{\Psi}(\Omega;\mathbb{C})$ is an algebra and clearly $\mathop{\bigcap}_{m\in \mathbb{R}}\mathrm{OpS}^{m}_{\Psi}(\Omega;\mathbb{C})=\mathcal{R}(\Omega;\mathbb{C})=\mathrm{OpS}^{-\infty}_{\Psi}(\Omega;\mathbb{C})$\,. Moreover if $A_{j}=a_{j}(x,D_{x})\chi_{j}(q)+R_{j}\in \mathrm{OpS}^{m}_{\Psi}(\Omega_{j};\mathrm{C})$ for $j=1,2$ then $A_{1}\circ A_{2}\in \mathrm{OpS}^{m}_{\Psi}(\Omega_{1}\cup \Omega_{2};\mathbb{C})$ with $A_{1}\circ A_{2}\sim a_{1}(x,D_{x})\circ a_{2}(x,D_{x})\chi(q)=(a_{1}\sharp a_{2})(x,D_{x})\chi(q)$ for any $\chi\in \mathcal{C}^{\infty}_{0}(\Omega_{1}\cap \Omega_{2};\mathbb{C})$ such that $\chi\equiv 1$ on a neighborhood of $\Omega-\mathrm{supp}\, a_{1}\cap \Omega-\mathrm{supp}\,a_{2}$\,.\\
\begin{definition}
  \label{def:asympsum}
  For $A\in \mathcal{L}(\mathcal{S}_{\Omega-\mathrm{comp}}(\Omega\times\mathbb{R}^{d};\mathbb{C});\mathcal{S}'_{\Omega-\mathrm{loc}}(\Omega\times\mathbb{R}^{d};\mathbb{C}))$\,,
the notation $A\sim \sum_{j=0}^{\infty}a_{j}(x,D_{x})$ is thought of as a localized asymptotic sum, for a sequence $a_{j}\in S^{m_{j}}_{\Psi,\Omega-\mathrm{loc}}(\Omega;\mathbb{C})$ with $\lim_{j\to\infty}m_{j}=-\infty$ and $(m_{j})_{j\in\mathbb{N}}$ decreasing. It means  that for any pair $\varrho,\chi\in \mathcal{C}^{\infty}_{0}(\Omega;[0,1])$\,, with $\chi\equiv 1$ on a neighborhood of $\Omega-\mathrm{supp}\,\varrho$\,, there exists $a_{\varrho}\in S^{m_{0}}_{\Psi}(\Omega;\mathbb{C})$ such that $\varrho(q)A\sim a_{\varrho}(x,D_{x})\chi(q)$ and for any $J\in \mathbb{N}$\,, $a_{\varrho}-\sum_{j=0}^{J}\varrho(q)a_{j}\in S^{m_{J+1}}_{\Psi}(\Omega;\mathbb{C})$\,.
\end{definition}
Notice that if $A=\tilde{\varrho}(q)A\tilde{\chi}(q)$ for some $\tilde{\varrho},\tilde{\chi}\in \mathcal{C}^{\infty}_{0}(\Omega;\mathbb{C})$ and $a_{j}=\tilde{\varrho}(q)a_{j}$ for all $j\in\mathbb{N}$\,, the above condition is reduced to $a-\sum_{j=0}^{J}a_{j}\in S^{m_{J+1}}(\Omega;\mathbb{C})$ with $a$ independent of $\varrho$\,. Moreover the localized definition means that we can always consider this simpler case.\\
The previous definition is justified by the following standard result.
\begin{proposition}
  \label{pr:asympsum}
  For any sequence $a_{j}\in S(\Psi^{m_{j}},g_{\Psi})$ with $m_{j}$ decreasing and $\lim_{j\to\infty}m_{j}=-\infty$\,, there exists $a\in S(\Psi^{m_{0}},g_{\Psi})$ such that $a-\sum_{j=0}^{J}a_{}\in S(\Psi^{m_{J+1}},g_{\Psi})$\,.\\
  For $A\in \mathcal{L}(\mathcal{S}_{\Omega-\mathrm{comp}}(\Omega\times\mathbb{R}^{d};\mathbb{C});\mathcal{S}'_{\Omega-\mathrm{loc}}(\Omega\times\mathbb{R}^{d};\mathbb{C}))$\,, $A\sim \sum_{j=0}^{\infty}a_{j}(x,D_{x})$ is equivalent to the apparently weaker condition $A-\sum_{j=0}^{J}a_{j}(x,D_{x})\in \mathcal{L}(\tilde{\mathcal{W}}^{-\mu_{J}}_{\Omega-\mathrm{comp}}(\Omega\times\mathbb{R}^{d};\mathbb{C});\tilde{\mathcal{W}}^{\mu_{J}}_{\Omega-\mathrm{loc}}(\Omega\times\mathbb{R}^{d};\mathbb{C}))$ with $\lim_{J\to\infty}\mu_{J}=-\infty$\,.
\end{proposition}
\begin{proof}
  The first statement can be reduced to the case where $m_{j}=m_{0}-j$ by putting together $b_{n}=\sum_{m_{0}-n-1<m_{j}\leq m_{0}-n}a_{j}$\,. Then use the standard Borel summation in $S(\Psi^{m_{0}},g_{\Psi})$ by taking $a=\sum_{n=0}^{\infty}(1-\chi)\left(\frac{\Psi}{N_{n}(1+p_{\Psi^{m_{0}-n},n}(b_{n}))}\right) b_{n}$ for $\chi\in \mathcal{C}^{\infty}_{0}(\mathbb{R};[0,1])$ equal to $1$ in a neighborhood of $0$ and the sequence $(N_{n})_{n\in \mathbb{N}}$ being increasing fast enough such that for every $k\in \mathbb{N}$\,,
  $$\sum_{n=k}^{\infty}p_{\Psi^{m_{0}-k},k}\left[(1-\chi)\left(\frac{\Psi}{N_{n}(1+p_{\Psi^{m_{0}-n},n}(b_{n})}\right)b_{n}\right]<+\infty\,.
  $$
  For the second statement, fix $\varrho,\chi\in \mathcal{C}^{\infty}(\Omega;[0,1])$ with $\chi\equiv 1$ in a neighborhood of $\varrho$\,. Then take $a_{\varrho}\in S(\Psi^{m_{0}},g_{\Psi})$ such that $a_{\varrho}-\sum_{j=0}^{j}\varrho(q)a_{j}\in S(\Psi^{m_{J+1}},g_{\Psi})$\,. For any $J\in \mathbb{N}$\,, the difference $D=\varrho(q)A\chi(q)-a_{\varrho}(x,D_{x})\chi(q)$ equals
  $$
D=\varrho(q)A \chi(q)-a_{\varrho}(x,D_{x})\chi(q)=\varrho(q)(A-\sum_{j=0}^{J}a_{j}(x,D_{x}))\chi(q)+(a_{\varrho}(x,D_{x})-\sum_{j=0}^{m_{j}}(\varrho(q)a_{j})(x,D_{x}))\chi(q)
$$
and belongs to $\mathcal{L}(\tilde{\mathcal{W}}^{-\mu_{J}}(\mathbb{R}^{2d};\mathbb{C});\tilde{\mathcal{W}}^{\mu_{j}}(\mathbb{R}^{2d};\mathbb{C}))$ for all $J\in \mathbb{N}$ with $D=\tilde{\varrho}(q)D\tilde{\chi}(q)$ for some pair $\tilde{\varrho},\tilde{\chi}\in \mathcal{C}^{\infty}_{0}(\Omega;[0,1])\subset \mathcal{C}^{\infty}(\mathbb{R}^{d};[0,1])$\,. This implies that $D=\varrho(q)A\chi(q)-a_{\varrho}(x,D_{x})\chi(q)$ is continuous from $\mathcal{S}'(\mathbb{R}^{2d};\mathbb{C})$ to $\mathcal{S}(\mathbb{R}^{2d};\mathbb{C})$
and with $D=\tilde{\varrho}(q)D\tilde{\chi}(q)$ it means $D\in \mathcal{R}(\Omega;\mathbb{C})$\,.
\end{proof}
In particular when $A_{k}\in \mathrm{OpS}^{m_{k}}_{\Psi}(\Omega;\mathbb{C})$  with $A_{k}=a_{k}(x,D_{x})\chi_{k}(q)+R_{k}$\,, $a_{k}\in S^{m_{k}}_{\Psi,\Omega-\mathrm{comp}}(\Omega;\mathbb{C})$ we can write as usual
$$
A_{1}\circ A_{2}\sim \sum_{\alpha \in \mathbb{N}^{2d}}\frac{1}{i^{|\alpha|}\alpha!}\partial_{\xi}^{\alpha}a_{1}\partial_{x}^{\alpha}a_{2}\,.
$$

With the previous localization method, the global differential calculus of Subsection~\ref{sec:globpseudodiff} is well defined with all its properties, if the following two conditions are satisfied:
\begin{itemize}
\item the class of symbols $S^{m}_{\Psi,\Omega-\mathrm{comp}}(\Omega;\mathbb{C})\subset S^{m}_{\Psi,\mathbb{R}^{d}-\mathrm{comp}}(\mathbb{R}^{d};\mathbb{C})$ is sent onto $S^{m}_{\Psi,\phi(\Omega)-\mathrm{comp}}(\phi(\Omega;\mathbb{C}) \subset  \linebreak    S^m_{\Psi,\mathbb{R}^d-\mathrm{comp}}(\mathbb{R}^d;\mathbb{C})$ by the canonical transformation $\Phi_{*}: \mathbb{R}^{4d}_{q,p,\xi,\eta}\to R^{4d}_{q,p,\xi,\eta}$ induced by a diffeomorphism $\phi:\mathbb{R}^{d}\to \mathbb{R}^{d}$ with $\Phi(q,p)=(\phi(q),{}^{t}d\phi(q)^{-1}.p)$ and $\Phi_{*}:T^{*}\mathbb{R}^{2d}\to T^{*}\mathbb{R}^{2d}$\,;
\item when $U_{\Phi}$ is the unitary transform in $L^{2}(\mathbb{R}^{2d},dqdp;\mathbb{C})$ given by
  $$
  (U_{\Phi}u)(q,p)=u\circ{\Phi(q,p)}=u(\phi(q),{}^{t}d\phi(q)^{-1}.p)
  $$
  satisfies $U_{\Phi}a(x,D_{x})U_{\Phi}^{-1}\sim \sum_{n=0}^{\infty}b_{n}(x,D_{x})$ in $\mathrm{OpS}^{m}_{\Psi}(\Omega,\mathbb{C})$ with $b_{n}\in S^{m-n}_{\Psi, \Omega-\mathrm{comp}}(\Omega;\mathbb{C})$\,,  $b_{0}=a\circ \Phi_{*}=\Phi^{*}a$ in
  $S^{m}_{\Psi,\Omega-\mathrm{comp}}(\Omega;\mathbb{C})$\,, for any $a\in S^{m}_{\Psi,\phi(\Omega)-\mathrm{comp}}(\phi(\Omega);\mathbb{C})$\,.
\end{itemize}
In particular all the sets $\Omega\times\mathbb{R}^{d}_{p}$ and $\Omega\times \mathbb{R}^{3d}_{p,\xi,\eta}$ can be replaced by $T^{*}\Omega$ and $T^{*}(T^{*}\Omega)$\,, with a natural geometrical meaning.\\
Actually we will consider more general changes of variables on $\mathbb{R}^{2d}$ which can be viewed as vector bundle isomorphisms of $\mathbb{R}^{2d}=T^{*}\mathbb{R}^{d}$\,.
We consider the following change of variables
\begin{equation}
  \label{eq:defPhi}
(\tilde{q},\tilde{p})=\Phi(q,p)=(\phi(q), L(q).p)
\end{equation}
which is a $\mathcal{C}^{\infty}$-diffeomorphism, a bijection such that $d\phi(q)$ and $L(q)$ belong to $GL_{d}(\mathbb{R})$ for all $q$
with the following estimates
\begin{equation}
\label{eq:estimPhi}
\forall \alpha\in \mathbb{N}^{d}\,,\exists C_{\alpha}>0\,,\quad
     \|\partial_{q}^{\alpha}d\phi\|_{L^{\infty}}+\|\partial_{q}^{\alpha}(d\phi)^{-1}\|_{L^{\infty}}+\|\partial_{q}^{\alpha}L\|_{L^{\infty}}+\|\partial_{q}^{\alpha}L^{-1}\|_{L^{\infty}}\leq C_{\alpha}\,.
   \end{equation}
   Note that $\Phi^{-1}$ takes the same form with 
   $$\Phi^{-1}(\tilde{q},\tilde{p})=(q,p)=(\phi^{-1}(\tilde{q}),[L(\phi^{-1}(\tilde{q}))]^{-1}\tilde{p})\,.$$
It is given by a change of variable $\phi:\mathbb{R}^{d}_{q}\to \mathbb{R}^{d}_{q}$ when $L(q)={}^{t}d\phi(q)^{-1}$\,.\\
   Its differential is given by
   $$
d\Phi(q,p).
\begin{pmatrix}
  t_{q}\\t_{p}
\end{pmatrix}
=
\begin{pmatrix}
  d\phi(q).t_{q}\\
  (dL(q).p).t_{q}+L(q).t_{p}
\end{pmatrix}
=
\begin{pmatrix}
  d\phi(q)&0\\
  dL(q).p&L(q)
\end{pmatrix}
\begin{pmatrix}
  t_{q}\\
  t_{p}
\end{pmatrix}\,.
$$
and
$$
{}^{t}d\Phi(q,p)^{-1}=
\begin{pmatrix}
  {}^{t}d\phi(q)^{-1}& -{}^{t}d\phi(q)^{-1}{}^{t}[L(q).p]{}^{t}L(q)^{-1}\\
  0& {}^{t}L(q)^{-1}
\end{pmatrix}
=
\begin{pmatrix}
  L_{1}(q)& (L_{2}(q).p)\\
  0&L_{3}(q)
\end{pmatrix}\,,
$$
where all linear maps $L_{1},L_{2}$ and the bilinear map $L_{2}$ have uniformly bounded derivatives of any order with respect to $q\in \mathbb{R}^{d}$\,.\\
The canonical transformation $\Phi_{*}:T^{*}\mathbb{R}^{2d}\to T^{*}\mathbb{R}^{2d}$ is thus given by
$$
\begin{pmatrix}
  \tilde{q}\\
  \tilde{p}\\
  \tilde{\xi}\\
  \tilde{\eta}
\end{pmatrix}
=\Phi_{*}
\begin{pmatrix}
  q\\p\\\xi\\\eta
\end{pmatrix}
=
\begin{pmatrix}
  \phi(q)\\
  L(q).p\\
  L_{1}(q).\xi+L_{2}(q).p.\eta\\
 L_{3}(q).\eta
\end{pmatrix}
$$
   With the transformation $\Phi$ given by  \eqref{eq:defPhi} we associate the unitary transform  $U_{\Phi}$ in $L^{2}(\mathbb{R}^{2d},dqdp;\mathbb{C})$:
   \begin{equation}
     \label{eq:defUPsi}
(U_{\Phi}u)(q,p)=|\det(d\phi(q))\det(L(q))|^{1/2}u(\phi(q),L(q)p)=J^{1/2}(q)u(\phi(q),L(q)p)\,.
\end{equation}
Let us consider the simplest versions of spaces of functions.
\begin{proposition}
  \label{pr:invfunc} For the tranformation $\Phi$ given by \eqref{eq:defPhi} the following properties hold for any $s\in \mathbb{R}$:
  \begin{description}
  \item[i)] The operator $U_{\Phi}$ ( resp. $U_{\Phi}^{-1}$) is continuous from $\mathcal{S}(\mathbb{R}^{2d};\mathbb{C})$ and from $\mathcal{S}'(\mathbb{R}^{2d};\mathbb{C})$ onto itself. For any open subset $\Omega\in \mathbb{R}^{d}$ it is continuous from $\mathcal{S}^{\dagger}_{\phi(\Omega)-\bullet}(\phi(\Omega)\times \mathbb{R}^{d};\mathbb{C})$ (resp. $\mathcal{S}^{\dagger}_{\Omega-\bullet}(\Omega\times \mathbb{R}^{d};\mathbb{C})$) onto  $\mathcal{S}^{\dagger}_{\Omega-\bullet}(\Omega\times \mathbb{R}^{d};\mathbb{C})$ (resp. $\mathcal{S}^{\dagger}_{\phi(\Omega)-\bullet}(\phi(\Omega)\times \mathbb{R}^{d};\mathbb{C})$) where, with respective correspondence, $\mathcal{S}^{\dagger}$ stands for $\mathcal{S}$ or $\mathcal{S}'$ and $\mathrm{\bullet}$ means $\mathrm{loc}$ or $\mathrm{comp}$\,.
  \item[ii)] The map $a\mapsto a_{\Phi}=a\circ \Phi_{*}$ is continuous from $S(\Psi^{s},g_{\Psi})$ onto itself. For any open subset $\Omega$ in $\mathbb{R}^{d}$ and any $m\in \mathbb{R}$\,, it is continuous from $S^{m}_{\Psi,\phi(\Omega)-\mathrm{comp}}(\Phi(\Omega);\mathbb{C})$ onto $S^{m}_{\Psi,\Omega-\mathrm{comp}}(\Omega;\mathbb{R})$\,.
  \end{description}
\end{proposition}
\begin{proof}
  \noindent\textbf{i)}
  It suffices to consider $J^{-1/2}(q)U_{\Phi}$ because $|\partial_{q}^{\alpha}J^{\pm 1/2}|\leq C_{\alpha}$ for all $\alpha\in \mathbb{N}^{d}$\,.
  We write
  \begin{eqnarray*}
    &&
  \partial_{q^{i}}(J^{-1/2}U_{\Phi}u)(q,p)=\partial_{q^{i}}\phi^{j}(q)(\partial_{q^{i}}u)(\phi(q),L(q).p)
  +(\partial_{q^{i}}L(q).p)_{j}\partial_{p_{j}}u(\phi(q),L(q).p)
    \\
      &&
                    \partial_{p_{j}}(J^{-1/2}U_{\Phi}u)(q,p)=L(q)^{j}_{k}(\partial_{p_{k}}u)(\phi(q),L(q).p)\\
    &&\frac{1}{C_{\Phi}}(1+|q|^{2}+|p|^{2})\leq 1+|\phi(q)|^{2}+|L(q).p|^{2}\leq C_{\Phi}(1+|q|^{2}+|p|^{2})\,,
  \end{eqnarray*}
  where the last inequality is a consequence of $1+|\phi(q)|\leq C_{\phi}(1+|q|^{2})$ owing to $|d\phi|\leq C_{0}$ and its reverse inequality for $\phi^{-1}$\,.\\
  \noindent\textbf{ii)} The formula for $a\circ \Phi_{*}$ is
  $$
a\circ\Phi_{*}=a(\phi(q),L(q).p, L_{1}(q).\xi+L_{2}(q).p.\eta,L_{3}(q).\eta)\,.
$$
Let us first compare $\Psi$ and $\Psi\circ \Phi_{*}$\,:
\begin{align*}
\Psi^{2}\circ \Phi_{*}(q,p,\xi,\eta)&=1+|L(q).p|^{4}+|L_{1}(q).\xi+L_{2}(q).p.\eta|^{2}
                                      +\left|L_{3}(q).\eta\right|^{4}
  \\
                                    &\leq C_{\Phi}(1+|p|^{4}+|\xi|^{2}+|p|^{2}|\eta|^{2}+|\eta|^{4})\\
                                    &\leq 2C_{\Phi}\Psi^{2}(q,p,\xi,\eta)\,.
\end{align*}
Applied to $\Psi^{2}\circ(\Phi^{-1})_{*}$\,, this provides the equivalence
$$
\frac{1}{C'_{\Phi}}\Psi\leq \Psi\circ (\Phi^{\pm 1})_{*}\leq C'_{\Phi}\Psi\,.
$$
The operators $\partial_{q^{i}}$\,,\, $\partial_{p_{i}}$\,, $\partial_{\xi_{i}}$\,,\, $\partial_{\eta^{i}}$ applied to $a\circ \Phi_{*}$ are equivalent to the following $\mathcal{C}^{\infty}_{b}(\mathbb{R}_{q}^{d};\mathbb{R})$ linear combinations (abbreviated as L.C) of elementary operators acting on $a$:
\begin{description}
\item[$\partial_{q^{i}}$:]  L.C of $\partial_{q^{j}}$\,, $p_{k}\partial_{p_{j}}$\,, $\xi_i\partial_{\xi_{j}}$\,, $p_{i}\eta^{j}\partial_{\xi_{\ell}}$\,, $\partial_{\eta_{j}}$\,, which all are continuous from $S(\Psi^{s},g_{\Psi})$ to $S(\Psi^{s},g_{\Psi})$\,.
\item[$\partial_{p_{i}}$:] L.C. of $\partial_{p_{j}}$\,,\, $\eta^{j}\partial_{\xi_{k}}$\,, which are continuous from $S(\Psi^{s},g_{\Psi})$ to $S(\Psi^{s-1/2},g_{\Psi})$\,.
\item[$\partial_{\xi_{i}}$:] L.C. of $\partial_{\xi_{j}}$ which are all continuous from $S(\Psi^{s},g_{\Psi})$ to $S(\Psi^{s-1},g_{\Psi})$\,.
\item[$\partial_{\eta^{i}}$:] L.C. of $p_{j}\partial_{\xi_{k}}$ and of $\partial_{\eta^{i}}$ which are all continuous from $S(\Psi^{s},g_{\Psi})$ to $S(\Psi^{s-1/2},g_{\Psi})$\,.
\end{description}
This proves the continuity of $a\mapsto a\circ \Phi_{*}$ from $S(\Psi^{s},g_{\Psi})$ to $S(\Psi^{s},g_{\Psi})$\,.
\end{proof}

Let us consider the functoriality of the transformation of  the quantization rule $a\mapsto a(x,D_{x})\chi(q)+R$ with $a\in S^{m}_{\Psi,\phi(\Omega)-\mathrm{comp}}(\Omega;\mathbb{C})$\,, $\chi\equiv 1$ in a neighborhood of $\phi(\Omega)-\mathrm{supp}\,a$ and $R\in \mathcal{R}(\phi(\Omega);\mathbb{C})$\,.
\begin{proposition}
  \label{pr:quantchang}
  For any $A=a(x,D_{x})\chi(q)+R\in \mathrm{OpS}^{m}_{\Psi}(\phi(\Omega);\mathbb{C})$\,, the operator $U_{\Phi}AU_{\Phi}^{*}$ is equal to $b_{\Phi}(x,D_{x})\chi(\Phi(q))+R_{\Phi}$ with $R_{\Phi}\in \mathcal{R}(\phi(\Omega);\mathbb{C})$ and $b_{\Phi}\in S^{m}_{\Psi,\Omega-\mathrm{comp}}(\Omega;\mathbb{C})$ and satisfies
$$
U_{\Phi}AU_{\Phi}\sim \sum_{n=0}^{\infty} b_{n}(x,D_{x})
$$
according to Definition~\ref{def:asympsum} with $b_{0}=a\circ{\Phi}_{*}$\,.
More precisely when $\Omega$ is a bounded open subset\,, with $S^{m}_{\Psi,\Omega-\mathrm{comp}}(\Omega;\mathbb{C})\subset S(\Psi^{m},g_{\psi})$  and $S^{m}_{\Psi,\phi(\Omega)-\mathrm{comp}}(\phi(\Omega);\mathbb{C})\subset S(\Psi^{m},g_{\psi})$\,,
$$
b_{\Phi}=\sum_{n=0}^{N-1}b_{n}+r_{\Phi,N}(a)
$$
with a continuous map $r_{\Phi,N}:S(\Psi^{m},g_{\psi})\to S(\Psi^{m-N},g_{\psi})$ for every $N\in \mathbb{N}$\,.
\end{proposition}
\begin{remark}
  Except for the principal symbol this result does not say that the transformation $a\mapsto b_{\Phi}$ corresponds $b_{\Phi}=a\circ \Phi_{*}$\,. It works exactly only for functions $a(q,p)$ and in particular for the cut-off functions with respect to $q$\,. But when $\Omega$ is a bounded open subset of $\mathbb{R}^{d}$\,, $U_{\Phi}a(x,D_{x})\chi(q)U_{\Phi}^{-1}=b_{\phi}(x,D_{x})\chi(\phi(q))$ defines a continuous operator from  $S^{m}_{\Psi, \phi(\Omega)-\mathrm{comp}}(\phi(\Omega);\mathbb{C})$ to $S^{m}_{\Psi, \Omega-\mathrm{comp}}(\Omega;\mathbb{C})$\,.
\end{remark}
\begin{proof}
  With the localization we can assume $\Omega=\mathbb{R}^{d}$\,, $a\in S(\Psi^{m};g_{\Psi})$\,, $\mathbb{R}^{d}-\mathrm{supp}\,a$ compact, and $R=0$\,. We introduce another cut-off function $\tilde{\chi}\in \mathcal{C}^{\infty}_{0}(\mathbb{R}^{d};[0,1])$ equal to $1$ on $\Phi(\Omega)$ when $\Omega$ is bounded and, for a more general choice of $\Omega$\,, equal to $1$ in a neighborhood of $\mathrm{supp}\,\chi$\,.\\
  Because the function  $J^{\pm 1/2}(q)\in S(1,g_{\Psi})$\,, the problem is reduced to the study of the operator $$(J^{-1/2}(q)U_{\Phi})\tilde{\chi}(q)a(x,D_{x})\tilde{\chi}(q)(J^{-1/2}(q)U_{\Phi})^{-1}$$ of which the Schwarz kernel is given by the oscillating integral
$$
K(x,y)=\int_{\mathbb{R}^{2d}}e^{i[\xi.(\phi(q)-\phi(q'))+\eta.(L(q).p-L(q').p')]}
\tilde{\chi}(\phi(q))[a(\phi(q),L(q).p,\xi,\eta)]\tilde{\chi}(\phi(q'))~\frac{d\xi d\eta}{(2\pi)^{2d}}\,.
$$

\noindent\textbf{Metrics and cut-off:} On $\mathbb{R}_{x,x',\Xi}^{6d}$ with $x=(q,p)$\,, $x'=(q',p')$ and $\Xi=(\xi,\eta)$ we consider the metrics
\begin{eqnarray*}
  &&G_{f}=dq^{2}+dq'^{2}+\frac{dp^{2}}{f}+\frac{dp'^{2}}{f}+\frac{d\xi^{2}}{f^{2}}+\frac{d\eta^{2}}{f}\\
  \text{for}&& f=\Psi(q,p,\xi,\eta)=(1+|\xi|^{2}+|p|^{4}+|\eta|^{4})^{1/2}\quad\text{and}
\quad f=\widetilde{\Psi}=(1+|\xi|^{2}+|p|^{4}+|p'|^{4}+|\eta|^{4})^{1/2}\,.
\end{eqnarray*}
The metrics $G_{\Psi}$ and $G_{\widetilde{\Psi}}$ are slow (same proof as for $g_{\Psi}$)\,. But the slowness of $g_{\psi}$ implies $$\left(\frac{\Psi(q,p,\xi,\eta)}{\Psi(q',p',\xi,\eta)}\right)^{\pm 1}\leq C_{\Psi}\quad\text{when}\quad |p-p'|\leq C_{\Psi}^{-1}\Psi(q,p,\xi,\eta)^{1/2}$$ and therefore with the above notations
$$
|p-p'|\leq \frac{1}{C'_{\Psi}}\Psi^{1/2}\Rightarrow \left(\frac{\Psi}{\tilde{\Psi}}\right)^{\pm 1}\leq C'_{\Psi}
$$
for some large enough constant $C'_{\Psi}\geq 1$\,, and
$$
c\in S(\Psi^{m},G_{\Psi})\Leftrightarrow c\in S(\widetilde{\Psi}^{m},G_{\tilde{\Psi}})\quad\text{when}~\mathrm{supp}\,c\subset
\left\{(x,y,\Xi)\in \mathbb{R}^{6d}\,,\, |p-p'|< \frac{1}{C'_{\Psi}}\Psi^{1/2}\right\}\,.$$
For $\theta\in \mathcal{C}^{\infty}_{0}(\mathbb{R};[0,1])$ and $\varepsilon>0$\,, $\varepsilon\leq \frac{1}{C'_{\Psi}}$ fixed later consider
the two cut-off functions
$$
\Theta_{1}(x,x',\Xi)=\Theta_1(q,q')=\theta\Big(\frac{|q-q'|^{2}}{\varepsilon^{2}}\Big)\quad\text{and}\quad \Theta_{2}(x,x',\Xi)=\theta\Big(\frac{|p-p'|^{2}}{\varepsilon^{2}\Psi}\Big)\,.
$$
By looking at the region $2^{n+10}\leq \widetilde{\Psi}\leq 2^{n+12}$ contained in a fixed shell for the rescaled variable $(\tilde{p},\tilde{p}',\tilde{\xi},\tilde{\eta})=(2^{-n/2}p,2^{-n/2}p',2^{-n}\xi,2^{-n/2}\eta)$\,, a homogeneity argument gives $\Theta_{2}\in S(1,G_{\tilde{\Psi}})\cap S(1,G_{\Psi})$\,. The following properties become obvious when $a\in S(\Psi^{m},g_{\Psi})$
\begin{align*}
  & b_{\tilde\chi,\phi}=\tilde\chi(\phi(q))[a(\phi(q),L(q).p,\xi,\eta)]\tilde\chi(\phi(q'))\in S(\Psi^{m},G_{\Psi})\,,\\
  & \Theta_{1}\,,\,\Theta_{2}\,, 1-\Theta_{2}\in S(1,G_{\widetilde{\Psi}})\cap S(1,G_{\Psi})\,.
\end{align*}
We now write the kernel $K(x,x')$ or the operator $K:\mathcal{S}_{\mathbb{R}^{d}-\mathrm{loc}}(\mathbb{R}^{2d};\mathbb{C})\to \mathcal{S}'_{\mathbb{R}^{d}-\mathrm{comp}}(\mathbb{R}^{2d};\mathbb{C})$ as 
\begin{eqnarray*}
  &&
K=K_{\mathrm{diag}}+K_{1}+K_{2}
  \\
  \text{with}&&
                K_{1}(x,x')=
                \int_{\mathbb{R}^{2d}}e^{i[\xi.(\phi(q)-\phi(q'))+\eta.(L(q).p-L(q').p')]}(1-\Theta_{1}(q,q'))b_{\tilde\chi,\phi}(x,x',\Xi)~\frac{d\xi d\eta}{(2\pi)^{2d}}\\
&&  K_{2}(x,x')=
  \int_{\mathbb{R}^{2d}}e^{i[\xi.(\phi(q)-\phi(q'))+\eta.(L(q).p-L(q').p')]}\Theta_{1}(q,q')(1-\Theta_{2}(x,x',\Xi)) b_{\tilde\chi,\phi}(x,x',\Xi)~\frac{d\xi d\eta}{(2\pi)^{2d}}\\
&&  K_{\mathrm{diag}(x,x')}=
  \int_{\mathbb{R}^{2d}}e^{i[\xi.(\phi(q)-\phi(q'))+\eta.(L(q).p-L(q').p')]}\Theta_{1}(q,q')\Theta_{2}(x,x',\Xi)b_{\tilde\chi,\phi}(x,x',\Xi)~\frac{d\xi d\eta}{(2\pi)^{2d}}\,,
\end{eqnarray*}
and use the same symbol $K_{\mathrm{diag}}$\,,\, $K_{1,2}$ for the associated operators $\mathcal{S}_{\mathbb{R}^{d}-\mathrm{loc}}(\mathbb{R}^{2d};\mathbb{C})\to \mathcal{S}'_{\mathbb{R}^{d}-\mathrm{comp}}(\mathbb{R}^{2d};\mathbb{C})$ with uniformly controlled supports.\\
\noindent\textbf{Non stationary phase in $q$:} For a given $k\in \mathbb{N}$\,,  $N\geq N_{k,d}$ integrations by parts with
$$
\Big[\frac{1}{|\phi(q)-\phi(q')|^{2}}(\phi(q)-\phi(q')).D_{\xi}\Big]^{N}e^{i\xi.(\phi(q)-\phi(q'))}=e^{i\xi.(\phi(q)-\phi(q'))}\,,
$$
and the lower bound
\begin{equation}
  \label{eq:lowboundnondiag}
\forall (q,q')\in \mathrm{supp}(1-\Theta_{1})\cap V_{\tilde\chi}^{2}\,,\quad |\phi(q)-\phi(q')|\geq \frac{1}{C_{\tilde\chi,\phi}}\,,
\end{equation}
for $\varepsilon\leq \frac{1}{C_{\tilde\chi,\phi}}$\,, small enough, and $V_{\tilde\chi}$ a compact neighborhood of $\mathrm{supp}\,\tilde\chi$\,,
implies that for all $(\alpha_{j},\beta_{j},\gamma_{j})\in \mathbb{N}^{3d}$\,, $2|\alpha_{j}|+|\beta_{j}|+|\gamma_{j}|\leq k$\,, for $j=1,2$\,, the kernel
$$
(-1)^{\alpha_{2}+\gamma_{2}}\partial_{q}^{\alpha_{1}}p^{\beta_{1}}\partial_{p}^{\gamma_{1}}\partial_{q'}^{\alpha_{2}}(p')^{\beta_{2}}\partial_{p'}^{\gamma_{2}}K_{1}(x,x')
$$
of $(\partial_{q}^{\alpha_{1}}p^{\beta_{1}}\partial_{p}^{\gamma_{1}})\circ K_{1}\circ(\partial_{q}^{\alpha_{2}}p^{\beta_{2}}\partial_{p}^{\gamma_{2}})$ belongs to $L^{2}(\mathbb{R}^{4d},dq dp dq' dp';\mathbb{C})$\,. Actually the estimates with powers of $p'$ is deduced from our estimates with powers of $p$ via integration by parts with 
$$ 
D_{\eta}e^{i\eta.(L(q).p - L(q').p')}= (L(q).p - L(q').p')e^{i\eta.(L(q).p - L(q').p')}.
$$
We deduce that $K_{1}$ is Hilbert-Schmidt and therefore bounded operator from $\tilde{\mathcal{W}}^{-k}(\mathbb{R}^{2d};\mathbb{C})$ to $\tilde{\mathcal{W}}^{k}(\mathbb{R}^{2d};\mathbb{C})$ for any $k\in \mathbb{C}$\,. With the fixed compact $\Omega$-support, $K_{1}\in \mathcal{L}(\mathcal{S}'(\mathbb{R}^{2d};\mathbb{C});\mathcal{S}(\mathbb{R}^{2d};\mathbb{C}))$\,. It has a symbol in $\mathcal{S}(\mathbb{R}^{4d};\mathbb{C})\subset S(\Psi^{-\infty},g_{\psi})$ with a compact support in $\Omega$\,.\\

\noindent\textbf{Kuranishi's trick:} Write
$$
\xi.(\phi(q)-\phi(q'))+\eta.(L(q)p-L(q')p')=(\xi,\eta).(\Phi(x)-\Phi(x'))=(\xi,\eta).\left[\int_{0}^{1}d\Phi((1-t)x+tx')~dt\right]
\begin{pmatrix}
  q-q'\\
  p-p'
\end{pmatrix}
$$
and remember with $x_{t}=(1-t)x+tx'$\,, $q_{t}=(1-t)q+tq'$\,, $p_{t}=(1-t)p+tp'$
$$
\int_{0}^{1}d\Phi(x_{t})~dt=
\int_{0}^{1}\begin{pmatrix}
  d\phi(q_{t})&0\\
  dL(q_{t}).p_{t}&L(q_{t})
            \end{pmatrix}~dt
=
\begin{pmatrix}
  \int_{0}^{1}d\phi(q_{t})~dt& 0\\
  [\int_{0}^{1}(1-t)dL(q_{t})~dt].p+[\int_{0}^{1}tdL(q_{t})~dt]p'&\int_{0}^{1}L(q_{t})~dt
\end{pmatrix}
$$
Because $(\mathbb{R}^{2d}_{q,q'}-\mathrm{supp}\,K)\subset \mathrm{supp}\,\tilde\chi\times\mathrm{supp}\,\tilde\chi$\,, we can fix $\varepsilon\leq \frac{1}{C_{\tilde\chi,\phi}}$ small enough  so  that  the inequalities
\begin{equation}
  \label{eq:convineq}
\forall (q,q')\in \mathrm{supp}\,\Theta_{1}\cap V_{\tilde\chi}^{2}\,,\, \forall A\in \mathrm{Conv}(d\phi([q,q'])) \cup \mathrm{Conv}(L([q,q']))\,,\quad
|\det(A)|\geq \frac{1}{C_{\tilde\chi,\Phi}}\,,
\end{equation}
where $\mathrm{Conv}(M([q,q']))$ stands for the convex hull in $\mathcal{M}_{d}(\mathbb{C})$ of the set
$M([q,q'])=\left\{M(q_t, t\in [0,1])\right\}\subset M_{d}(\mathbb{C})$\,, while \eqref{eq:lowboundnondiag} remains valid.
We obtain for $(q,q')$
\begin{eqnarray*}
  &&
 [\xi.(\phi(q)-\phi(q'))+\eta.(L(q).p-L(q').p')]=\left[
     \begin{pmatrix}
       A(q,q')& B(q,q').p+C(q,q').p'\\
       0&D(q,q')
     \end{pmatrix}
          \begin{pmatrix}
            \xi\\\eta
          \end{pmatrix}
          \right] 
\cdot\begin{pmatrix}
  q-q'\\p-p'
\end{pmatrix}
  \\
  \text{with}&&
 E(x,x')=               \begin{pmatrix}
       A(q,q')& B(q,q').p+C(q,q').p'\\
       0&D(q,q')
                \end{pmatrix}^{-1}
          =
          \begin{pmatrix}
       A(q,q')^{-1}& B'(q,q').p+C'(q,q').p'\\
       0&D(q,q')^{-1}
          \end{pmatrix}\\
  \text{and}&&
 A,B,C,D,A^{-1},B^{-1},C',D'\in S(1,dq^{2}+dq'^{2};\mathcal{M}_{d}(\mathbb{R}))\,.
\end{eqnarray*}
We obtain
\begin{align*}
  &K_{2}(x,x')=\int_{\mathbb{R}^{2d}}e^{i[\xi.(q-q')+\eta.(p-p')]}
[\Theta_{1}(1-\Theta_{2})b_{\tilde\chi,\phi}](x,x',E(x,x').
\Xi)~|\det(E^{-1})(q,q')|~\frac{d\xi d\eta}{(2\pi)^{2d}}\,,\\
  &K_{\mathrm{diag}}(x,x')=\int_{\mathbb{R}^{2d}}e^{i[\xi.(q-q')+\eta.(p-p')]}
[\Theta_{1}\Theta_{2}b_{\tilde\chi,\phi}](x,x',E(x,x').
\Xi)~|\det(E^{-1})(q,q')|~\frac{d\xi d\eta}{(2\pi)^{2d}}\,,
\end{align*}
where choosing $\varepsilon\leq \frac{1}{C_{\psi,\phi,\tilde\chi}}$\,, small enough, ensures
\begin{align*}
  &\Theta_{1}(x,x',E(x,x').\Xi)=\theta\Big(\frac{|q-q'|^{2}}{\varepsilon^{2}}\Big)=\Theta_{1}(q,q')\in S(1,G_{\widetilde{\Psi}})\,,\\
  &\Theta_{2}(x,x',E(x,x').\Xi)=\theta\left(\frac{|p-p'|^{2}}{\varepsilon^{2}\Psi(0,p,A(q,q').\xi+B'(q,q').p.\eta+C'(q,q').p'.\eta,D(q,q')^{-1}.\eta)}\right) \in S(1,G_{\widetilde{\Psi}})\,,\\
  & b_{\tilde\chi,\phi}(x,x',E(x,x').\Xi)=\tilde\chi(q)\tilde\chi(q')a(\phi(q),L(q).p,A(q,q').\xi+B'(q,q').p.\eta+C'(q,q').p'.\eta,D(q,q')^{-1}.\eta) 
  \\
&[(1-\Theta_{2})b_{\tilde\chi,\phi}](x,x',E(x,x').\Xi)\in S(\widetilde{\Psi}^{|m|},dx^{2}+dx'^{2}+d\Xi^{2})\quad,\quad [\Theta_{2}b_{\tilde\chi,\phi}](x,x',E(x,x').\Xi)\in S(\widetilde{\Psi}^{m},G_{\widetilde{\Psi}})\,.          
\end{align*}
Actually it suffices to check
$$
\Psi^{2}(0,p,E(x,x').\Xi)\leq C(1+(|\xi|+|p||\eta|+|p'||\eta|)^{2}+|p|^{4}+|p'|^{4}+|\eta|^{4})\leq C'\widetilde{\Psi}^{2}
$$
with the symmetric version by applying the same result to $\Psi^{2}(0,p,E(x,x')^{-1}.\Xi)$ and then to use the equivalence $\left(\frac{\widetilde{\Psi}}{\Psi}\right)^{\pm 1}\leq C_{\varepsilon}$ when $|p-p'|\leq \varepsilon \widetilde{\Psi}^{1/2}$ owing to  the slowness of $G_{\widetilde{\Psi}}$\,.\\

\noindent\textbf{Non stationary phase in $p$:}
Despite the bad a priori estimate of $[(1-\Theta_{2})b_{\tilde\chi,\phi}](x,x',E(x,x').\Xi)$\,, $N\geq N_{k,d}$ integrations by parts for a given $k\in N$ with
\begin{eqnarray*}
  && \left(\frac{1}{|p-p'|^{2}}(p-p').D_{\eta}\right)^{N}e^{i[\xi.(q-q')+\eta.(p-p')]}=
     e^{i[\xi.(q-q')+\eta.(p-p')]}\\
  \text{and}&&
               \forall (x,x')\in \mathrm{supp}\,(1-\Theta_{2})(.,.,E(.,.).)\,,\quad
               |p-p'|\geq \frac{1}{C_{\Psi,\phi,\tilde\chi}}\widetilde{\Psi}^{1/2}\,,
\end{eqnarray*}
leads to the property that the kernel of $(\partial_{q}^{\alpha_{1}}p^{\beta_{1}}\partial_{p}^{\gamma_{1}})\circ K_{2}\circ(\partial_{q}^{\alpha_{2}}p^{\beta_{2}}\partial_{p}^{\gamma_{2}})$ is Hilbert-Schmidt and therefore bounded in $L^{2}(\mathbb{R}^{2d},dqdp;\mathbb{C})$ for $|\alpha_{j}|+\frac{|\beta_{j}|+|\gamma_{j}|}{2}\leq k$\,. We conclude as we did for $K_{1}$ that $K_{2}$ belongs to $\mathcal{L}(\mathcal{S}'(\mathbb{R}^{2d};\mathbb{C});\mathcal{S}(\mathbb{R}^{2d};\mathbb{C}))$\,. It has a symbol in $\mathcal{S}(\mathbb{R}^{4d};\mathbb{C})\subset S(\Psi^{-\infty},g_{\Psi})$ with a compact support in $\Omega$\,.\\

\noindent\textbf{Gauss transform~:} The kernel of $K_{\mathrm{diag}}$ can be written
$$
K_{\mathrm{diag}}(x,x')=b(x,D_{x})\quad\text{with}\quad b(x,\Xi)=e^{iD_{\Xi}.D_{x'}}
\left[[\Theta_{1}\Theta_{2}b_{\tilde\chi,\phi}](x,x',E(x,x')\Xi)\right]\big|_{x=x'}\,,
$$
where the metric $G_{\widetilde{\Psi}}$ is slow on $\mathbb{R}^{6d}$\,, the $B$-dual metric of $G_{\widetilde{\Psi}}$ for $B=
\begin{pmatrix}
  0&0&0\\
  0&0&\frac{1}{2}\mathrm{Id}_{\mathbb{R}^{2d}}\\
  0&\frac{1}{2}\mathrm{Id}_{\mathbb{R}^{2d}}&0
\end{pmatrix}
$ is  the degenerate metric $G^{B}=\widetilde{\Psi}^{2}dq'^{2}+\widetilde{\Psi}dp'^{2}+d\xi^{2}+\widetilde{\Psi}d\eta^{2}$\,. Fortunately $G_{\widetilde{\Psi}}$ is $G^{B}$-temperate along the vector space $V_{0}=\left\{(x,x',\Xi)\in \mathbb{R}^{6d}\,,\, x=x'\right\}$ and $\widetilde{\Psi}\big|_{V_{0}}$ can be replaced with $\Psi$\,.
Theorem~18.4.11 of \cite{HormIII} tells us  that
$b\in S(\Psi^{m},g_{\Psi})$ with the asymptotic expansion
$$
b(x,D_{x})\sim \sum_{n=0}^{\infty}\underbrace{b_{n}(x,D_{x})}_{b_{n}\in S(\Psi^{m-n},g_{\psi})}
$$
and the first term $b_{0}(x,\Xi)=b_{\phi,\tilde\chi}(x,x,\Xi)=a\circ\Phi_{*}(x,\Xi)$\,.
\end{proof}
\begin{remark}
  We could have used the general theory of global Fourier integral operators of J.M.~Bony in \cite{Bon}. At least when $\phi-\mathrm{Id}_{\mathbb{R}^{d}}$ and $L-\mathrm{Id}_{\mathbb{R}^{d}}$ have a compact support, this describes $U_{\Phi}$ as a Fourier integral operator of which the global symbol is a section of the fiber bundle of affine metaplectic operators $\mathcal{M}\to \mathcal{G}_{\phi}$ above the graph of $\phi_{*}$\,, $\mathcal{G}_{\Phi}=\left\{(X,\Phi_{*}(X)),\quad X\in \mathbb{R}^{4d}\right\}$ with a value above $(X_{0},\Phi(X_{0}))$\,, $x_{0}=(q_{0},p_{0},\xi_{0},\eta_{0})$ which is the composition $\tau_{\phi(X_{0})}U_{X_{0}}\tau_{-x_{0}}$ where $\tau_{(x_{1},\Xi_{1})}$ is the phase translation $e^{i(\Xi_{1}.D_{x}-x_{1}.D_{\Xi})}$ and $U_{X_{0}}$ is a metaplectic representation of the linear symplectic map $d\Phi_{*}(X_{0})$\,.\\
  The proposed methods mimics the classical techniques of pseudo-differential calculus for the metric $dq^{2}+\frac{d\xi^{2}}{\langle \xi\rangle^{2}}$ modulo the localization process only in the $q$-variable and for specific linear transformations in the $p$-variable. It is informative from this point of view and provides a more explicit formulation for the functoriality of the principal symbol.
\end{remark}
From Proposition~\ref{pr:invfunc} and the definition of $\tilde{\mathcal{W}}^{s}_{\Omega-\bullet}(\Omega;\mathbb{C})$, $\bullet=\mathrm{loc}$ or $\mathrm{comp}$\,, deduced from Definition~\ref{de:tWsRd}, we obtain the following result.

\begin{proposition}
  \label{pr:diffeoWsloc}
  Consider the unitary map $U_{\Phi}:L^{2}(\mathbb{R}^{2d},dqdp;\mathbb{C})\to L^{2}(\mathbb{R}^{2d},dqdp;\mathbb{C})$ given by \eqref{eq:defPhi}\eqref{eq:estimPhi}\eqref{eq:defUPsi} with the additional assumption that $\phi-\mathrm{Id}_{\mathbb{R}^{d}}$ and $L-\mathrm{Id}_{\mathbb{R}^{d}}$  have a compact support. Then for any $s\in \mathbb{R}^{d}$\,, $U_{\Phi}$ and $U_{\Phi}^{-1}$ are isomorphisms from $\tilde{\mathcal{W}}^{s}(\mathbb{R}^{2d};\mathbb{C})$ into itself, with $\mathbb{R}^{d}_{q}-\mathrm{supp}\, U_{\Phi}^{\pm 1}u=\phi^{\mp 1}(\mathbb{R}^{d}_{q}-\mathrm{supp}\,u)$ for every $u\in \tilde{\mathcal{W}}^{s}(\mathbb{R}^{2d};\mathbb{C})$\,.  
\end{proposition}
\begin{proof}
  The support property is obvious from the definition \eqref{eq:defUPsi} of $U_{\Phi}$\,. With the additional support assumption on $\Phi-\mathrm{Id}_{\mathbb{R}^{2d}}$\,, we can write for any $u\in \tilde{\mathcal{W}}^{s}(\mathbb{R}^{2d};\mathbb{C})$\,, $U_{\Phi}u=U_{\Phi}\chi_{1}(\phi(q))u + \chi_{2}(q)u$ for $\chi_{1},\chi_{2}\in \mathcal{C}^{\infty}(\mathbb{R}^{d};[0,1])$\,, $\chi_{1}+\chi_{2}\equiv 1$ and $\mathrm{supp}\,\chi_{1}$ compact. From the Definition~\ref{de:tWsRd} and the global pseudo-differential calculus in $S(\Psi^{2s},g_{\psi})$ remember
  $$
C_{s}^{-1}\mathrm{Re}~\langle u\,,\, M_{2s}(x,D_{x})u\rangle_{L^{2}}\leq \|u\|_{\tilde{\mathcal{W}}^{s}}^{2}=\|M_{s}^{W}(x,D_{x})u\|^{2}_{L^{2}}\leq C_{s}\mathrm{Re}~\langle u\,,\, M_{2s}(x,D_{x})u\rangle_{L^{2}}\,,
$$
for $M_{s}=(C_{s}+\Psi^{|s|})^{\mathrm{sign}\,s}$ with $C_{s}\geq 1$ large enough.\\ 
We deduce
$$
\|U_{\Phi}\chi_{1}(\phi(q))u\|_{\tilde{\mathcal{W}}^{s}}^{2}\leq C_{s}\mathrm{Re}~ \langle u\,,\, U_{\phi}^{*}(\chi_{1}(q)M_{2s})(x,D_{x})U_{\Phi}\chi_{1}(\phi(q))u\rangle\,.
$$
By Proposition~\ref{pr:quantchang} we know
$$
U_{\phi}^{*}(\chi_{1}(q)M_{2s})(x,D_{x})U_{\Phi}\chi_{1}(\phi(q))=b_{2s}(x,D_{x})\circ \chi_{1}(\phi(q))=c^{W}_{2s}(x,D_{x})\,\quad b_{2s},c_{2s}\in S(\Psi^{2s},g_{\Psi})\,,
$$
and the pseudo-differential calculus in $S(\Psi^{2s},g_{\Psi})$ gives
$$
\|U_{\Phi}\chi_{1}(\phi(q))u\|_{\tilde{\mathcal{W}}^{s}}^{2}\leq C'_{s}\|u\|^{2}_{\tilde{\mathcal{W}}^{s}}\,.
$$
By the triangular inequality $U_{\Phi}:\tilde{\mathcal{W}}^{s}(\mathbb{R}^{2d};\mathbb{C})\to \tilde{\mathcal{W}}^{s}(\mathbb{R}^{2d};\mathbb{C})$ is continuous and we conclude with $U_{\Phi}^{-1}=U_{\Phi^{-1}}$\,.
\end{proof}
All this section gives a meaning to  $\mathcal{S}^{\dagger}_{\Omega-\bullet}(T^{*}\Omega;\mathbb{C})$\,, $S^{m}_{\Psi,\Omega-\mathrm{comp}}(\Omega;\mathbb{C})$\,, $\mathcal{R}(\Omega;\mathbb{C})$\,, $\mathrm{OpS}^{m}_{\Psi}(\Omega;\mathbb{C})$ and  $\tilde{\mathcal{W}}^{s}_{\Omega-\bullet}(T^{*}\Omega;\mathbb{C})$ ( with $\mathcal{S}^{\dagger}=\mathcal{S}$ or $\mathcal{S}'$ and $\bullet$ meaning $\mathrm{loc}$ or $\mathrm{comp}$) when $\Omega$ is a chart open set in the compact manifold $Q$\,.\\
The topology, the continuity properties and the global ellipticity that we need will be better discussed in the global setting which avoids considerations of inductive limit topologies.
\subsection{Globalization on $Q$ and applications}
\label{sec:global}
Let us fix, an atlas covering of $Q$\,, $Q=\mathop{\bigcup}_{j=1}^{J}\Omega_{j}$ such that $\tilde{\Omega}_{j}=\mathop{\bigcup}_{\Omega_{j'}\cap \Omega_{j}\neq \emptyset}\Omega_{j'}$ is still a chart open set.
We take a finite partition of unity $\sum_{j=1}^{J}\varrho_{j}(q)\equiv 1$ subordinate with the atlas covering $Q=\mathop{\bigcup}_{j=1}^{J}\Omega_{j}$ and cut-off functions $\chi_{j}\in \mathcal{C}^{\infty}_{0}(\Omega_{j};[0,1])$ such that $\chi_{1}\equiv 1$ in a neighborhood of $\mathrm{supp}\,\varrho_{j}$\,. Notice that $\chi_{j}\chi_{j'}\neq 0$ implies $\Omega_{i}\cap \Omega_{j}\neq 0$ and in this case $\varrho_{j},\varrho_{j'},\chi_{j},\chi_{j'}$ are cut-off functions in the coordinate charts $\tilde{\Omega}_{j}$ and $\tilde{\Omega}_{j'}$\,.\\
The spaces $S^{m}_{\Psi}(Q;\mathbb{C})$ (resp. $\widetilde{\mathcal{W}}^{s}(T^{*}Q;\mathbb{C})$) are the sets $a=\sum_{j=1}^{J}\varrho_{j}(q)a$ with $\varrho_{j}(q)a\in S^{m}_{\Psi,\Omega_{j}-\mathrm{comp}}(\Omega_{j};\mathbb{C})\subset S(\Psi^{m},g_{\Psi})$ (resp. $\varrho_{j}a \in \widetilde{\mathcal{W}}^{s}_{\Omega_{j}-\mathrm{comp}}(T^{*}\Omega_{j};\mathbb{C})$) with the topologies given by
\begin{eqnarray*}
 && p_{m,k}(a)=\sum_{j=1}^{J}p_{\tilde{\Omega}_{j},\Psi^{m},k}(\varrho_{j}(q)a)\,,\quad k\in \mathbb{N}\\
  \text{resp.}&& \|a\|_{\widetilde{\mathcal{W}}^{s}(Q;\mathbb{C})}^{2}=\sum_{j=1}^{J}\|\varrho_{j}(q)a\|^{2}_{\widetilde{\mathcal{W}}^{s}(T^{*}\Omega;\mathbb{C})}\,.
\end{eqnarray*}
The subscript $_{\tilde{\Omega}_{j}}$ in $p_{\tilde\Omega_{j},\Psi^{m},k}$ refers to the choice of some local coordinates in $\tilde\Omega_{j}$\,. But Proposition~\ref{pr:invfunc} and Proposition~\ref{pr:quantchang} for the conjugation $a\mapsto U_{\Phi}a(x,D_{x})\chi(q)U_{\Phi}^{-1}$ with $\Phi(q,p)=(\phi(q),{}^{t}d\phi(q)^{-1}.p)$\,, says that the seminorms $p_{\tilde{\Omega}_{j},\Psi^{m},k}(\varrho_{j}(q)a)$ can be replaced by $p_{\tilde{\Omega}_{j'},\Psi^{m},k}(\varrho_{j}(q)a)$ for any $j'\in \left\{1,\ldots,J\right\}$ such that $\Omega_{j}\subset \tilde{\Omega}_{j'}$\,.\\
A vector bundle isomorphism on $T^{*}Q$\,, written locally as $\Phi:(q,p)\mapsto (\phi(q),L(q).p)$ with the associated unitary operator $U_{\Phi}$\,,
gives rise to an isomorphism of the space $\tilde{\mathcal{W}}^{s}(Q;\mathbb{C})$ according to the local result of  Proposition~\ref{pr:diffeoWsloc}. This gives a first application, which is not exactly due to the pseudo-differential calculus
\begin{proposition}
  \label{pr:dyadic}~\\
  For any riemannian metric $g=g_{ij}(q)dq^{i}dq^{j}$ on $TQ$ with the dual metric $g^{ij}(q)dp_{i}dp_{j}$ on $T^{*}Q$\,, if $\sum_{\ell=-1}^{\infty}\theta_{\ell}^{2}(t)\equiv 1$ is a quadratic dyadic partition of unity like \eqref{eq:dyadic} and $|p|_{q}^{2}=g^{ij}(q)p_{i}p_{j}$\,, then for every $s\in \mathbb{R}$ the squared norm $\|u\|^{2}_{\tilde{\mathcal{W}}^{s}(Q;\mathbb{C})}$ is equivalent to $\sum_{\ell=-1}^{\infty}\|\theta_{\ell}(|p|_{q}^{2})u\|^{2}_{\tilde{\mathcal{W}}^{s}(Q;\mathbb{C})}$\,.
\end{proposition}
\begin{proof}
  It suffices to use the local gauge transform given by $\Phi:(q,p)\mapsto (q,g^{-1/2}(q).p)$ with $|p|_{q}^{2}={}^{t}p g^{-1}(q)p=|g^{-1/2}(q).p|^{2}$ and to write
$$
\sum_{\ell=-1}^{\infty}\|\theta_{\ell}(|p|_{q}^{2})u\|^{2}_{\tilde{\mathcal{W}}^{s}}\asymp\sum_{\ell=-1}^{\infty}\|U_{\Phi}^{-1}\theta_{\ell}(|p|_{q}^{2})u\|^{2}_{\tilde{\mathcal{W}}^{s}}
\asymp\sum_{\ell=-1}^{\infty}\|\theta_{\ell}(|p|^{2})U_{\Phi}^{-1}u\|^{2}_{\tilde{\mathcal{W}}^{s}}
$$
and Proposition~\ref{pr:eqnormes} gives
$$
\sum_{\ell=-1}^{\infty}\|\theta_{\ell}(|p|_{q}^{2})u\|^{2}_{\tilde{\mathcal{W}}^{s}}\asymp \|U_{\Phi}^{-1}u\|^{2}_{\tilde{\mathcal{W}}^{s}}\asymp \|u\|^{2}_{\tilde{\mathcal{W}}^{s}}\,.
$$
\end{proof}
The intersection $\mathop{\bigcap}_{s\in \mathbb{R}}\widetilde{\mathcal{W}}^{s}(T^{*}Q;\mathbb{C})$ is nothing but $\mathcal{S}(T^{*}Q;\mathbb{C})$\,.\\
On $\mathcal{R}(Q;\mathbb{C})=\mathcal{L}(\mathcal{S}'(T^{*}Q;\mathbb{C});\mathcal{S}(T^{*}Q;\mathbb{C}))\sim \mathcal{S}(T^{*}Q\times T^{*}Q;\mathbb{C})$\,, the Fr{\'e}chet space topology is equivalently defined by the family of (semi)norms
$$
q_{k}(R)=\|R\|_{\mathcal{L}(\widetilde{\mathcal{W}}^{-k}(T^{*}Q;\mathbb{C});\widetilde{\mathcal{W}}^{k}(T^{*}Q;\mathbb{C}))}\,,\quad k\in\mathbb{N}\,.
$$
Before giving an explicit family of seminorms
on
$$
\mathrm{OpS}^{m}_{\Psi}(Q;\mathbb{C})=\left\{\sum_{j=1}^{J}(\varrho_{j}(q)a)(x,D_{x})\circ\chi_{j}(q)+ R\,, a\in S^{m}_{\Psi}(\Omega;\mathbb{C})\,, R\in \mathcal{R}(\Omega;\mathbb{C})\right\}\,,
$$
let us check that any $A\in \mathrm{OpS}^{m}_{\Psi}(Q;\mathbb{C})$ admits a canonical decomposition after fixing some cut-off function on $Q\times Q$\,.\\
Attention must be payed to the following point: although $\varrho_{j}(q)a_{j'}=\varrho_{j'}(q)a_{j}$ for all pairs $(j,j')$ allows to define $a(x,\Xi)=\sum_{j'=1}^{J}\varrho_{j'}(q)a_{j'}(x,\Xi)$\,, the equality $\sum_{j=1}^{J}(\varrho_{j}(q))a_{j}(x,D_{x})\circ\chi_{j}(q)-\sum_{j=1}^{J}(\varrho_{j}(q)a)(x,D_{x})\circ\chi_{j}(q)=R\in \mathcal{R}(Q;\mathbb{C})$ is true with $R=0$ only under the equality of the operators
$$
(\varrho_{j'}(q)a_{j})(x,D_{x})\circ\chi_{j'}(q)=(\varrho_{j}(q)a_{j'})(x,D_{x})\circ\chi_{j}(q)
$$
for all pairs $(j,j')$\,.\\
With the subset $\tilde{\Omega}_{j}=\mathop{\bigcup}_{\Omega_{j'}\cap \Omega_{j}\neq\emptyset}\Omega_{j'}$ take a cut-off function $\tilde{\chi}_{j}\in \mathcal{C}^{\infty}_{0}(\tilde{\Omega}_{j};[0,1])$ such that $\tilde{\chi}_{j}\equiv 1$ on a neighborhood of
$$
\mathop{\bigcup}_{\Omega_{j'}\cap \Omega_{j}\neq \emptyset}\mathrm{supp}\,\chi_{j}
\supset
\mathop{\bigcup}_{\Omega_{j'}\cap \Omega_{j}\neq \emptyset}\mathrm{supp}\,\varrho_{j}
$$
Because for all $j\in \left\{1,\ldots,J\right\}$\,, $\varrho_{j}(q)\chi_{j}(q')$ and $\tilde{\chi}_{j}(q)1_{Q\setminus\tilde{\Omega}_{j}}(q')$ vanish in a neighborhood of the diagonal  $\Delta_{Q}=\left\{(q,q)\,,\, q\in Q\right\}$\,, there exists $\Theta_{1}\in \mathcal{C}^{\infty}(Q\times Q;[0,1])$ such that $\Theta_{1}\equiv 1$ in a neighborhood of $\Delta_{Q}$ and
\begin{eqnarray}
\nonumber
  \varrho_{j}(q)\Theta_{1}(q,q')&=&\varrho_{j}(q)\Theta_{1}(q,q')\chi_{j}(q')=\varrho_{j}(q)\tilde{\chi}_{j}(q)\Theta_{1}(q,q')\chi_{j}(q)\\
\nonumber                                &=&\varrho_{j}(q)\left[\sum_{j'=1}^{J}\varrho_{j'}(q)\tilde{\chi}_{j}(q)\right]\Theta_{1}(q,q')\chi_{j}(q')\\
\label{eq:rhojchij}
  &=&\varrho_{j}(q)\left[\sum_{j'=1}^{J}\varrho_{j'}(q)\tilde{\chi}_{j'}(q)\Theta_{1}(q,q')\right]\chi_{j}(q')\,,
\end{eqnarray}
 where the equalities hold as multiplication operators on $\mathcal{S}'(T^{*}\Omega_{j}\times T^{*}\Omega_{j};\mathbb{C})$\,. Additionally the function $\Theta_{1}$ can be chosen symmetric: $\Theta_{1}(q,q')=\Theta_{1}(q',q)$\,, and we set 
 $$
 \Theta_2(q,q') = 1-\Theta_1(q,q')\,.
 $$ \\
For any $K\in \mathcal{\mathcal{L}}(\mathcal{S}(T^{*}Q;\mathbb{C});\mathcal{S}'(T^{*}Q;\mathbb{C}))$\,, identified with its Schwartz kernel $K(x,x')\in \mathcal{S}'(T^{*}Q\times T^{*}Q;\mathbb{C})$\,, we set
 \begin{eqnarray}
   \label{eq:defdiag}
   K_{\mathrm{diag}}(x,x')=\Theta_{1}(q,q')K(x,x')\quad,\quad
   K_{\mathrm{off}}(x,x')=\Theta_{2}(q,q')K(x,x')\quad,\quad K=K_{\mathrm{diag}}+K_{\mathrm{off}}\,.
 \end{eqnarray}
 Notice that $K\mapsto (K_{\mathrm{diag}},K_{\mathrm{off}})$ is an isomorphism between
 $\mathcal{S}'(T^{*}Q\times T^{*}Q;\mathbb{C})$ and the closed set of $\mathcal{S}'(T^{*}Q\times T^{*}Q;\mathbb{C})\times \mathcal{S}'(T^{*}Q\times T^{*}Q;\mathbb{C})$
$$
\left\{(K_{1},K_{2})\in \mathcal{S}'(T^{*}Q\times T^{*}Q;\mathbb{C})\times \mathcal{S}'(T^{*}Q\times T^{*}Q;\mathbb{C})\,,\, \Theta_{1}(q,q')K_{2}(x,x')-\Theta_{2}(q,q')K_{1}(x,x')=0\right\}\,.
$$
With \eqref{eq:rhojchij}, we have the additional properties
\begin{align*}
K_{\mathrm{diag}}=\sum_{j=1}^{J}\varrho_{j}(q)\circ K_{\mathrm{diag}}
   &=\sum_{j=1}^{J}\varrho_{j}(q)\circ\left[\sum_{j'=1}^{J}(\varrho_{j'}\tilde{\chi}_{j'})(q)\circ K\right]_{\mathrm{diag}}\circ \chi_{j}(q)
  \\
  \text{and}
             \sum_{j=1}^{J}\left[(\varrho_{j}(q)a)(x,D_{x}) \circ \chi_{j}(q)\right]_{\mathrm{diag}}
     &=
      \sum_{j=1}^{J}\varrho_{j}(q)\circ
       \left[\sum_{j'=1}^{J}[(\varrho_{j'}\tilde{\chi}_{j'})(q)a](x,D_{x})
       \right]_{\mathrm{diag}}\circ \chi_{j}(q)
\end{align*}
for some $\tilde{\chi}_{j}\in \mathcal{C}^{\infty}_{0}(\Omega_{j};[0,1])$ such that $\tilde{\chi}_{j}\equiv 1$ in a neighborhood of $\mathrm{supp}\,\chi_{j}$\,.
 \begin{proposition}
   \label{pr:candecomp} Every $A=\sum_{j=1}^{J}(\varrho_{j}(q)a)(x,D_{x})\chi_{q}+R\in \mathrm{OpS}^{m}_{\Psi}(Q;\mathbb{C})$ admits a unique decomposition
$$
A=\underbrace{\sum_{j=1}^{J}(\varrho_{j}(q)a_{A}(x,D_{x}))\chi_{j}(q)}_{=A_{\mathrm{diag}}}+\underbrace{R_{A}}_{=A_{\mathrm{off}}}\,,
$$
with $a_{A}\in S^{m}_{\Psi}(Q;\mathbb{C})$ and $R_{A}\in \mathcal{R}(Q;\mathbb{C})$\,.\\
Additionally this provides a topological direct sum on $\mathrm{OpS}^{m}(Q;\mathbb{C})$\,, because the map
\begin{eqnarray*}
  S^{m}_{\Psi}(Q;\mathbb{C})\times \mathcal{R}(Q;\mathbb{C})&\to& S^{m}_{\Psi}(Q;\mathbb{C})\times \mathcal{R}(Q;\mathbb{C})\\
  (a,R)&\mapsto & (a_{A},R_{A})\quad,\quad A=\sum_{j=1}^{J}(\varrho_{j}(q)a)(x,D_{x})\chi_{j}(q)+R
\end{eqnarray*}
is continuous when $S^{m}_{\Psi}(Q;\mathbb{C})\times \mathcal{R}(Q;\mathbb{C})$ is endowed with the seminorms $(p_{m,k}(a)+q_{k}(R))_{k\in \mathbb{N}}$\,.
\end{proposition}
\begin{remark}
\label{re:local}
When $A$ is a differential operator or more generally a local operator with respect to $q$-variable, then we can write $A=A_{\mathrm{diag}}$ with a vanishing remainder $R_A=0$\,.
\end{remark}
\begin{proof}
  The decomposition of $A=\sum_{j=1}^{J}(\varrho_{j}(q)a)(x,D_{x})\circ \chi_{j}(q)+R=A_{a}+R$
  \begin{eqnarray*}
    &&A=(A_{a}+R)_{\mathrm{diag}}+(A_{a}+R)_{\mathrm{off}}=A_{a,\mathrm{diag}}+R_{\mathrm{diag}}+A_{a,\mathrm{off}}+R_{\mathrm{off}}\,,\\
    \text{with}&&R_{\mathrm{off}}(x,x')=\Theta_{2}(q,q')R(x,x')\in \mathcal{S}(T^{*}Q\times T^{*}Q;\mathbb{C})\,,\\
    && A_{a,\mathrm{off}}(x,x')=\sum_{j=1}^{J}A_{j,\mathrm{off}}(x,x')\quad,\quad
       A_{j,\mathrm{off}}(x,x')=\Theta_{2}(q,q')[(\varrho_{j}(q)a)(x,D_{x})\circ \chi_{j}(q)](x,x')\\
    && R_{\mathrm{diag}}=\sum_{j=1}^{J}\varrho_{j}(q) \circ \left[\sum_{j'=1}^{J}(\varrho_{j'}\tilde{\chi}_{j'})(q)\circ R\right]_{\mathrm{diag}}\circ \chi_{j}(q)\,,\\
    \text{and}&&
                 A_{a,\mathrm{diag}}=
                 \sum_{j=1}^{J}\varrho_{j}(q)\circ \left[\sum_{j'=1}^{J}[(\varrho_{j'}\tilde{\chi}_{j'})(q)a](x,D_{x})\right]_{\mathrm{diag}}\circ \chi_{j}(q)\,.
  \end{eqnarray*}
  The kernel of $A_{j,\mathrm{off}}$ with coordinates in $\Omega_{j}$ is
  $$
\int_{\mathbb{R}^{2d}}e^{i[\xi.(q-q')+\eta.(p-p')}\Theta_{2}(q,q')\varrho_{j}(q)a(q,p,\xi,\eta)\chi_{j}(q')~\frac{d\xi d\eta}{(2\pi)^{d}}\,.
$$  
A non stationary phase argument with $\frac{(q-q')}{|q-q'|^{2}}D_{\xi}e^{i[\xi.(q-q')+\eta.(p-p')}=e^{i[\xi.(q-q')+\eta.(p-p')]}$ with the factors $\varphi_{j}(q)\theta_{j}(q')$ implies that the map
  $a\mapsto A_{j,\mathrm{off}}(x,x')$ is continuous from
  $S(\Psi^{m},g_{\Psi})$ to $\mathcal{S}(\mathbb{R}^{4d};\mathbb{C})$\,. Again with the controlled support, the map $a\mapsto A_{j,\mathrm{off}}$ is continuous from $S^{m}_{\Psi}(Q;\mathbb{C})$ to $\mathcal{R}(Q;\mathbb{C})$\,.\\
  This proves that the map
  $$
(a,A)\mapsto A_{\mathrm{off}}=R_{\mathrm{off}}+\sum_{j=1}^{J}A_{j,\mathrm{off}}
$$
is continuous from $S^{m}_{\Psi}(Q;\mathbb{C})\times \mathcal{R}(Q;\mathbb{C})$ to $\mathcal{R}(Q;\mathbb{C})$\,.\\
For the diagonal part, let us first notice at the operator level that the sum with respect to $j'$ is introduced for
$$
\varrho_{j'}(q)\left[(\varrho_{j}\tilde{\chi}_{j})(q)M\right]_{\mathrm{diag}}
  =\varrho_{j'}(q)\varrho_{j}(q)\tilde{\chi}_{j}(q)\tilde{\chi}_{j'}(q)[M]_{\mathrm{diag}}\chi_{j}(q)\chi_{j'}(q)
  $$
  with $M=R$ or $M=(\tilde{\chi}_{j}\tilde{\chi}_{j'}(q)a)(x,D_{x})$\,.\\
  It remains to identify these operators as the quantization of  symbols.
Every term indexed by $(j,j')$ for a fixed $j'\in \left\{1,\ldots,J\right\}$, has a kernel with a $Q\times Q$ support that is compact in $\tilde{\Omega}_{j'}\times\tilde{\Omega}_{j'}$\,. We choose coordinates in $\tilde{\Omega}_{j'}$ in order to compare the terms of the double sum for different values of $j\in \left\{1,\ldots,J\right\}$\,. For $[(\varrho_{j'}\tilde{\chi}_{j'})(q)R]_{\mathrm{diag}}$ the kernel
$$
(\varrho_{j'}\tilde{\chi}_{j'})(q)R(x,x')\Theta_{1}(q,q')=\varrho_{j'}(q)R(x,x')\Theta_{1}(q,q')\chi_{j'}(q')
$$
belongs to $\mathcal{S}(\Omega_{j'}\times \Omega_{j}\times \mathbb{R}^{2d};\mathbb{C})$\,. With the factors $\varrho_{j'}(q)\chi_{j'}(q')$\,, it can be written $a_{j',\mathrm{reg}}(x,D_{x})$ with $a_{j',\mathrm{reg}}\in \mathcal{S}(\mathbb{R}^{4d};\mathbb{C})\subset S(\Psi^{-\infty},g_{\Psi})$\,.\\
For $A_{a,\mathrm{diag}}$\,, the kernel is localized in the same way and equals
$$
\int_{\mathbb{R}^{2d}}e^{i[\xi.(q-q')+\eta.(p-p')]}\varrho_{j'}(q)a(q,p,\xi,\eta)\Theta_{1}(q,q')\chi_{j'}(q')\frac{d\xi d\eta}{(2\pi)^{2d}}\,,
$$
We obtain
$$
\left[[\varrho_{j'}\tilde{\chi}_{j'}(q)a](x,D_{x})\right]_{\mathrm{diag}}=b_{j'}(x,D_{x})
$$
where $b_{j'}$ is given by the Gauss transform
$$
b_{j'}(x,\xi)=e^{iD_{\Xi}.D_{x'}}\varrho_{j'}(q)a(q,p,\xi,\eta)\Theta_{1}(q,q')\chi_{j'}(q')\big|_{X=X'}=e^{iD_{\xi}.D_{q'}}\varrho_{j'}(q)a(q,p,\xi,\eta)\Theta_{1}(q,q')\chi_{j'}(q')\big|_{q'=q}\,,
$$
where the variables $(p,\eta)$ are now simple parameters.
We deduce that the map $a\mapsto b_{j'}$ is continuous from $S(\Psi^{m},g_{\Psi})$ to $S(\Psi^{m},g_{\Psi})$\,. With $b_{j'}(x,D_{x})=\tilde{\chi}_{j'}(q)b_{j'}(x,D_{x})$ we deduce that $b_{j'}\in S^{m}_{\Psi,\tilde{\Omega}_{j}-\mathrm{comp}}(\tilde{\Omega}_{j};\mathbb{C})\subset S^{m}_{\Psi}(Q;\mathbb{C})$\,.\\
So we have written
\begin{eqnarray*}
  && R_{\mathrm{diag}}=\sum_{j=1}^{J} (\varrho_{j}(q)b_{j,\mathrm{reg}})(x,D_{x})\circ \chi_{j}(q)\quad,\quad A_{\mathrm{diag}}=\sum_{j=1}^{J}(\varrho_{j}(q)b_{j})(x,D_{x})\circ \chi_{j}(q)\\
  \text{with}&&
                b_{j,\mathrm{reg}}\in S^{-\infty}_{\Psi}(Q;\mathbb{C})\quad,\quad b_{j}\in S^{m}_{\Psi}(Q;\mathbb{C})\\
  \text{and}&&
               \forall j,j'\in \left\{1,\ldots,J\right\}\,,\quad
               \varrho_{j'}(q)(b_{j}+b_{j,\mathrm{reg}})=\varrho_{j}(q)(b_{j'}+b_{j',\mathrm{reg}})\,,
  \end{eqnarray*}
The last identity follows the same strategy as at the operator level except that we consider only left multiplications by functions of $q$\,, which commute with the Gauss transform.
\end{proof}
\begin{definition}
  \label{de:globaltop}
  According to Proposition~\ref{pr:candecomp} the  topology on $\mathrm{OpS}^{m}_{\Psi}(Q;\mathbb{C})$ is equivalently defined by the family of (semi)norms $(q_{m,k})_{k\in\mathbb{N}}$ and $(\tilde{q}_{m,k})_{k\in\mathbb{N}}$ with
  $$
  q_{m,k}(A)= p_{m,k}(a_{A})+q_{k}(R_{A})\quad \text{with}~A=\underbrace{\sum_{j=1}^{J}\varrho_{j}(q)a_{A}(x,D_{x})\circ \chi_{j}(q)}_{=A_{\mathrm{diag}}} +\underbrace{R_{A}}_{=A_{\mathrm{off}}}
  $$
  and
  $$
\tilde{q}_{m,k}(A)=\mathrm{inf}\left\{p_{m,k}(a)+q_{k}(R)\,, A=\sum_{j=1}^{J}(\varrho_{j}(q)a)(x,D_{x})\circ \chi_{j}(q)+R\,, a\in S^{m}_{\Psi}(Q;\mathbb{C}, R\in \mathcal{R}(Q;\mathbb{C}))\right\}\,.
$$
\end{definition}
\begin{definition}
  \label{de:ellipticity}
  We now write $a_{\varrho,\chi}(x,D_{x})=\sum_{j=1}^{J}(\varrho_{j}(q)a)(x,D_{x})\circ \chi_{j}(q)$ for $a\in S^{m}_{\Psi}(Q;\mathbb{C})$\,.
  A symbol $a\in S^{m}_{\Psi}(Q;\mathbb{C})$ is said to be elliptic if there exists $\kappa\geq 1$ such that $|a|\geq \frac{1}{\kappa}\Psi^{m}$ for $\Psi\geq \kappa$\,. An operator $A=a_{\varrho,\chi}(x,D_{x})+R\in \mathrm{OpS}^{m}_{\Psi}(Q;\mathbb{C})$ is said to be elliptic if it admits a symbol $a$ which is elliptic.
\end{definition}
The product of two operators $A=a_{\varrho,\chi}(x,D_{x})+R, A'=a'_{\varrho,\chi}(x,D_{x})+R'\in \mathrm{OpS}^{m}_{\Psi}(Q;\mathbb{C})$ equals
$$
A\circ A'=a_{\varrho,\chi}(x,D_{x})\circ a'_{\varrho,\chi}(x,D_{x}) + \underbrace{a_{\varrho,\chi}(x,D_{x})\circ R'+R\circ a_{\varrho,\chi}(x,D_{x})+R\circ R'}_{\in \mathcal{R}(Q;\mathbb{C})}
$$
and the treatment of every non vanishing term $\varrho_{j}(q)a(x,D_{x})\chi_{j}(q)\varrho_{j'}(q)a(x,D_{x})\chi_{j'}(q)$ of the product $a_{\varrho,\chi}(x,D_{x})\circ a'_{\varrho,\chi}(x,D_{x})$ can be studied in the chart open set $\tilde{\Omega}_{j}$ with the expansion of the Gauss transform $e^{iD_{\Xi_{1}}.D_{x_{2}}}a_{1}(X_{1})a_{2}(X_{2})\big|_{X_{1}=X_{2}=X}$ in $S(\Psi^{mm'},g_{\psi})$\,.
We deduce that
$$
A\circ A'=\sum_{n=1}^{N}b_{n}(a,a')_{\varrho,\chi}(x,D_{x})+R_{N+1}(A,A')
$$
with $p_{m+m'-n,k}(b_{n}(a,a'))\leq C_{m,m',n,k}p_{m,\ell_{n,k}}(a)p_{m',\ell_{n,k}}(A')$ and the remainder estimated by
$$
q_{m+m'-N-1,k}(R_{N+1}(A,A'))\leq
C'_{m,m',N+1,k}q_{m,\ell_{N+1,k}}(A)q_{m,\ell_{N+1,k}}(A')\,,
$$
when $p_{m,\ell}(a)\leq C_{m,\ell} q_{m,\ell}(A)$ and $p_{m,\ell}(a')\leq C_{m,\ell} q_{m,\ell}(A')$\,.
According to Remark~\ref{re:local}, differential operators provide a wide family of examples where the latter condition holds true. And this can be extended for operators of which the Schwartz kernel is explicitly localized in a small neighborhood of the $Q$-diagonal.\\
The rough version of this continuity property says
$$
A\circ A'=\sum_{n=1}^{N}B_{n}(A,A')+R_{N+1}(A,A')
$$
with $q_{m+m'-n,k}(B_{n}(A,A'))\leq C_{m,m',n,k}q_{m,\ell_{n,k}}(a)q_{m',\ell_{n,k}}(A')$\,, and $$
q_{m+m'-N-1,k}(R_{N+1}(A,A'))\leq
C'_{m,m',N+1,k}q_{m,\ell_{N+1,k}}(A)q_{m,\ell_{N+1,k}}(A')\,.
$$
Similarly for a vector bundle isomorphism $\Phi:T^{*}Q\to T^{*}Q$\,, the conjugation by the associated unitary transform $A=a_{\varrho,\chi}(x,D_{x})+R\mapsto U_{\Phi}AU_{\Phi}^{-1}$ can be written
$$
U_{\Phi}AU_{\Phi}^{-1}=\sum_{n=1}^{N}[b_{n,\Phi}(a)]_{\varrho,\chi}(x,D_{x})+R_{\Phi,N+1}(A)
$$
with continuity estimates gathered from the local model treated in Proposition~\ref{pr:quantchang}. It can be written more roughly as
\begin{eqnarray*}
  && U_{\Phi}AU_{\Phi}^{-1}=\sum_{n=1}^{N}B_{n,\Phi}(A)+R_{\Phi,N+1}(A)\\
  \text{with}&& q_{m-n,k}(B_{n,\Phi}(A))\leq C_{\Phi,m,k}q_{m,\ell_{n,k}}(A)\\
  \text{and}&& q_{m-N-1,k}(R_{\Phi,N+1}(A))\leq C'_{\Phi,m,N+1}q_{m,\ell_{N+1},k}(A)\,.
\end{eqnarray*}
Because the function $(q,q')\mapsto\tilde{\chi}_{j}(q)\tilde{\chi}_{j}(q')\left[\sum_{\Omega_{j'}\cap \Omega_{j}\neq \emptyset}\varrho_{j'}(q)\Theta_{1}(q,q')\chi_{j'}(q')-\chi_{j'}(q)\Theta_{1}(q,q')\varrho_{j'}(q')\right]$, which is symmetric if $\Theta_{1}(q,q')=\Theta_{1}(q',q)$\,, has a compact support away from the diagonal $\Delta$\,, the decomposition of the formal adjoint can be reduced to the
local model with the formula $b(x,D_{x})^{*}=[e^{iD_{\Xi}.D_{x}}\overline{b}](x,D_{x})$\,.  The formal adjoint $A^{*}$ of the operator $A=a_{\varrho,\chi}(x,D_{x})+R\in \mathrm{OpS}^{m}_{\Psi}(Q;\mathbb{C})$ can be written
$$
A^{*}=\sum_{n=0}^{N}[b_{n}]_{\varrho,\chi}(x,D_{x})+R_{N+1}(A)=\sum_{n=0}^{N}B_{N}(A)+R_{N+1}(A)
$$
with  $b_{0}=\overline{a}$ and the estimates $p_{m-n,k}(b_{n})\leq C_{m,n,k}p_{m,\ell_{n,k}}(a)$ and  $q_{m-n,k}(B_{n})\leq C_{m,n,k}q_{m,\ell_{n,k}}(A)$\, and $q_{m-N-1,k}(R_{N+1})\leq C_{m,N,k}q_{m,\ell_{N+1},k}(A)$\,.\\
All the classical estimates can be decomposed in this way by going back to the $Q=\mathbb{R}^{d}$-model. The condition $p_{m,\ell}(a)\leq C_{m,\ell}q_{m,\ell}(A)$ holds true if one starts with an operator $A=a_{\varrho,\chi}(x,D_{x})+0$\,, exactly given by the local quantization rule with no regularizing global remainder, and $a\not\in S^{-\infty}_{\Psi}(Q;\mathbb{C})$\,. For example this is the case if $A$ is a differential operator. The estimates are then propagated via the operations\,.\\
The local Calderon-Vaillancourt theorem and our choice of the norm $q_{k}$ on
$\mathcal{R}(Q;\mathbb{C})$ gives at once the existence of a $k_{d}\in \mathbb{N}$ determined by the dimension of $Q$\,, such that $\|A\|_{\mathcal{L}(L^{2})}\leq C q_{0,k_{d}}(A)$\,.\\
Similarly the Garding inequality says that if $A\in S^{m}_{\Psi}(Q;\mathbb{C})$ has an elliptic non negative symbol $a\geq \frac{1}{\kappa}\Psi^{m}$ for $\Psi\geq \kappa$\,,   there exists $C_{\kappa}\geq 1$ and $k_{1}\in \mathbb{N}$ such that
$$
\forall u\in \mathcal{S}(T^{*}Q;\mathbb{C})\,,\quad \mathrm{Re}~\langle u\,,\, A u\rangle_{L^{2}}\geq \frac{1}{C_{\kappa}}\|u\|^{2}_{\tilde{\mathcal{W}}^{m/2}}-C_{\kappa}q_{m,k_{1}}(A)\|u\|_{\tilde{\mathcal{W}}^{(m-1)/2}}^{2}\,.
$$
All these properties extend to $\mathrm{OpS}^{m}_{\Psi}(Q;\mathrm{End}\,\mathcal{E})$ with the following constraint for the symbol of the adjoint: The reduction to $a\in S^{m}_{\Psi, \Omega-\mathrm{comp}}(\Omega;\mathcal{M}_{d}(\mathbb{C}))$ is done by chosing $\tilde{\Omega}_{j}$ such that the vector bundle $E\big|_{\tilde{\Omega}_{j}}$ admits an orthonormal frame $(f_{j}^{1},\ldots,f_{j}^{N})$ for the metric $g_{E}$\,. Then the adjoint of $A=a_{m,\varrho,\chi}(x,D_{x})+A_{m-1}\in \mathrm{OpS}^{m}_{\Psi}(Q;\mathrm{End}\,\mathcal{E})$ can be written $$A^{*}=(a_{m}^{*})_{\varrho,\chi}(x,D_{x})+A'_{m-1}\quad\text{with}\quad A'_{m-1}\in \mathrm{OpS}^{m-1}_{\Psi}(Q;\mathrm{End}\,\mathcal{E})$$ and this property is invariant by a change of orthonormal frame.\\
Actually if $U(q)\in \mathcal{U}_{N}(\mathbb{C})$ is the unitary matrix which represent another orthonormal frame $(\tilde{f}_{j}^{N},\ldots,\tilde{f}_{j}^{N})$ in the frame $(f_{j}^{N},\ldots,f_{j}^{N})$ for $E\big|_{\tilde{\Omega}_{j}}$\,, the symbol of $a_{m,\varrho,\chi}(x,D_{x})+A_{m-1}$ equals
$$
b_{m}(x,\Xi)=U(q)a_{m}(x,\Xi)U^{-1}(q)=U(q)a_{m}(x,\Xi)U^{*}(q)\quad\text{with}\quad b_{m}^{*}(x,\Xi)=U(q)a_{m}^{*}(x,\Xi)U^{*}(q)\,.
$$
The norms $(q_{m,k})_{k\in \mathbb{N}}$ are very convenient for handling the ellipticity as it is done in the case of the global pseudo-differential calculus on $\mathbb{R}^{d}$\,.
We focus here on the case of non negative elliptic operators.
\begin{proposition}
  \label{pr:HeSj}~\\
  Let $A\in \mathrm{OpS}^{m}_{\Psi}(Q;\mathrm{End}\,\mathcal{E})$\,, $m>0$\,, be an elliptic operator with
  $A=(a_{m}\otimes \mathrm{Id}_{\mathcal{E}})_{\varrho,\chi}(x,D_{x})+A_{m-1}$\,,  $a_{m}\geq \frac{1}{\kappa}\Psi^{m}$ for $\Psi\geq \kappa$\,, $A_{m-1}\in \mathrm{OpS}^{m-1}_{\Psi}(Q;\mathrm{End}\,\mathcal{E})$\,. If additionally $A$ is symmetric on $\mathcal{S}(T^{*}Q;\mathcal{E})$ then it is self-adjoint with $D(A)=\tilde{\mathcal{W}}^{m}(T^{*}Q;\mathcal{E})$\,, bounded from below, and its resolvent is compact.\\
In the case when $m=2$\,, if $A=a_{\varrho,\chi}(x,D_{x})+R\in \mathrm{OpS}^{2}_{\Psi}(Q;\mathrm{End}\,\mathcal{E})$ fulfills the above conditions,  then for every  $f\in S(\langle t\rangle^{s},\frac{dt^{2}}{\langle t\rangle^{2}})$\,, $s\in \mathbb{R}$\,, the operator $f(A)-f(a_{2})_{\varrho,\chi}(x,D_{x})$ belongs to $\mathrm{OpS}^{2s-1}_{\Psi}(Q;\mathrm{End}\,\mathcal{E})$ while $f(a_{2})\in S^{2s}_{\Psi}(Q;\mathrm{End}\,\mathcal{E})$\,.
\end{proposition}
\begin{proof}
  The first results are the standard ones.\\
We just show how Helffer-Sjöstrand formula can be used in this framework and we focus on the case $m=2$\,. We write $a$ in the form $a=a_{2}\otimes\mathrm{Id}_{\mathcal{E}}+a_{1}$ with $a_{1}\in S^{1}_{\Psi}(Q;\mathrm{End}\,\mathcal{E})$ and $a_{2}\otimes \mathrm{Id}_{\mathcal{E}}$ is simply written $a_{2}$\,.\\
  By the Leibniz formula applied to $1=(z-a_{2})^{-1}\times(z-a_{2})$ for $z\in \mathbb{C}\setminus \mathbb{R}$\,, the seminorms  $p_{-2,k}(\frac{1}{z-a_{2}})$ are estimated by
$$
p_{-2,k}\left(\frac{1}{z-a_{2}}\right)\leq C_{k}p_{2,k}(a) \frac{\langle z\rangle^{k}}{|\mathrm{Im}\, z|^{k+1}}\,.
$$
By taking $B_{z}=\left[\frac{1}{z-a_{2}}+\frac{a_{1}}{(z-a_{2})^{2}}\right]_{\varrho,\chi}(x,D_{x})+0$\,, the second term of the expansion  of  $B_{z}\circ (z-A)$\,, to be studied in $S^{1}_{\Psi}(Q;\mathrm{End}\,\mathcal{E})$\,, can be reduced to
$$
-\frac{1}{(z-a_{2})^{2}}a_{1}\circ a_{1}+
\sum_{|\alpha_{1}|+|\alpha_{2}|= 1}\frac{1}{\alpha_{1}!\alpha_{2}!}\partial_{\xi}^{\alpha_{1}}\partial_{\eta}^{\alpha_{2}}\frac{1}{(z-a_{2})}.\partial_{q}^{\alpha_{1}}\partial_{p}^{\alpha_{2}}[a_{2}+a_{1}]
$$
and it is of order $-2$ with seminorms estimated like the seminorms $p_{-2,k}(\frac{1}{z-a_{2}})$\,. We deduce
$$
B_{z}\circ (z-A)=zB_{z} -B_{z}\circ A= (1-r_{L}(z))
$$
with  $q_{-2,k}(r_{L}(z))\leq C_{k}\frac{\langle z\rangle^{N_{k}}}{|\mathrm{Im}\,z|^{N_{k}+1}}$\,.
By left multiplying with $\sum_{n=0}^{M}r_{L}(z)^{n}$ we obtain
$$
\sum_{n=0}^{M}r_{L}(z)^{n}\circ B_{z}\circ (z-A)=1-r_{L}(z)^{M+1}\,,
$$
with $q_{(-M-1)2,k}[r_{L}(z)^{M+1}]\leq C_{M,k}\frac{\langle z\rangle^{N_{M,k}}}{|\mathrm{Im}\,z|^{N_{M,k}+1}}$ and $\|r_{L}(z)^{M+1}\|_{\mathcal{L}(L^{2};\widetilde{\mathcal{W}}^{(Mm-c_{d})})}\leq C'_{M}\frac{\langle z\rangle^{N_{M}}}{|\mathrm{Im}\,z|^{N_{M}+1}}$\,.
In the right-hand side of 
$$
(z-A)^{-1}=\sum_{n=0}^{M}r_{L}(z)^{n}\circ B(z)+(z-A)^{-1}\circ r_{L}(z)^{M+1}\,,
$$
all terms except the remainder term $(z-A)^{-1}\circ r_{L}(s)^{M+1}$ are known to be pseudo-differential operators $r_{L}(z)^{n}\circ B(z)\in \mathrm{OpS}^{-2(n+1)}_{\Psi}(Q;\mathrm{End}\,\mathcal{E})$ and $q_{-2(n+1),k}(r_{L}(z)^{n}\circ B(z))\leq C_{k}\frac{\langle z\rangle^{N_{n,k}}}{|\mathrm{Im}\,z|^{N_{n,k}+1}}$\,.
But we can do the same for the right-multiplication and obtain similarly:
$$
(z-A)^{-1}=\sum_{n=0}^{M}B_{z}\circ r_{D}(z)^{n}+r_{D}(z)^{M+1}(z-A)^{-1}\,,
$$
with the same upper bounds.\\
We deduce
\begin{eqnarray}
\label{eq:resoexp}  &&
\hspace{-2.3cm}(z-A)^{-1}=\sum_{j=1}^{J}\left[\varrho_{j}(q)(\frac{1}{z-a_{2}}+\frac{a_{1}}{(z-a_{2})^{2}})\right](x,D_{x})\chi_{j}(q)+\sum_{n=1}^{2M}b_{n}(z)+\underbrace{r_{D}(z)^{M+1}(z-A)^{-1}r_{L}(z)^{M+1}}_{=r_{M}(z)}
  \\
\nonumber
  \text{with}&&
                q_{-2(n+1),k}(b_{n})\leq C_{n,k}
                 \frac{\langle z\rangle^{N_{n,k}}}{|\mathrm{Im}\,z|^{N_{n,k}+1}}\quad,\quad
                \|r_{M}\|_{\mathcal{L}(\widetilde{\mathcal{W}}^{-M+c_{d}},\widetilde{\mathcal{W}}^{M-c_{d}})}
                 \leq C_{M}
                 \frac{\langle z\rangle^{N_{M}}}{|\mathrm{Im}\,z|^{N_{M}+2}}\,.
\end{eqnarray}
Inserting \eqref{eq:resoexp} into the Helffer-Sjöstrand formula (see \cite{HeSj}\cite{DiSj}) gives
\begin{eqnarray*}
  &&f(A)=\frac{1}{2i\pi}\int_{\mathbb{C}}\partial_{\bar{z}}\tilde{f}(z)(z-A)^{-1}~dz\wedge d\bar{z}\\
  \text{while}&&
                 f(a_{2})=\frac{1}{2i\pi}\int_{\mathbb{C}}\partial_{\bar{z}}\tilde{f}(z)(z-a_{2})^{-1}~dz\wedge d\bar{z}\quad,\quad
                 f'(a_{2})=\frac{1}{2i\pi}\int_{\mathbb{C}}\partial_{\bar{z}}\tilde{f}(z)(z-a_{2})^{-2}~dz\wedge d\bar{z}\,,\\
  \text{with}&&
                \tilde{f}\in \mathcal{C}^{\infty}(\mathbb{C};\mathbb{C})\quad,\quad \mathrm{supp}\,\tilde{f}\subset\left\{z\in \mathbb{C}\,,\;|\mathrm{Im}\,z|\leq C_{f}\langle z\rangle \right\}\quad,\quad \tilde{f}\big|_{\mathbb{R}}=f\,,\\
  \text{and}&& \forall N\in\mathbb{N}\,,\,\exists C_{N}>0\,,\,
  |\partial_{\bar z}\tilde{f}(z)|\leq C_{N}\frac{|\mathrm{Im}\, z|^{N}}{\langle z\rangle^{N}}\langle z\rangle^{s-1}
\end{eqnarray*}
when $f\in S(\langle t\rangle^{s}, \frac{dt^{2}}{\langle t\rangle})$\,, $s<0$\,, we obtain by integration of the respective terms by choosing $N\geq \max\left\{N_{n,k},N_{M}\right\}$
$$
f(A)=[f(a_{2})]_{\varrho,\chi}(x,D_{x})+ [f'(a_{2})a_{1}]_{\varrho,\chi}(x,D_{x})+
\sum_{n=1}^{2M}[\beta_{n}]_{\varrho,\chi}(x,D_{x})+R_{M}
$$
with $\beta_{n}\in S^{-2(n+1)}_{\Psi}(Q;\mathrm{End}\,\mathcal{E})$ for $n\geq 1$ while $f(a)\in S^{-2s}_{\Psi}(Q;\mathrm{End}\,\mathcal{E})$\,.\\
Let us first conclude for $s\in[-3/2,0[$\,. Take $\beta\in S^{-4}(Q;\mathrm{End}\,\mathcal{E})$ such that $\beta\sim \sum_{n=1}^{\infty}\beta_{n}$\,. For every $M\in \mathbb{N}$\,,
$$
f(A)=[f(a_{2})]_{\varrho,\chi}(x,D_{x})+ [f'(a_{2})a_{1}]_{\varrho,\chi}(x,D_{x})+\beta_{\varrho,\chi}(x,D_{x})
+R_{\beta,2M+1}+R_{M}
$$
with $R_{\beta,2M+1}\in \mathrm{OpS}^{-2(2M+1)}_{\Psi}(Q;\mathrm{End}\,\mathcal{E})$ and $R_{M}\in \mathcal{L}(\widetilde{\mathcal{W}}^{-M+c_{d}};\widetilde{\mathcal{W}}^{M-c_{d}})$\,.
By taking $M$ arbritrarily large, this says that $f(A)-[f(a_{2})]_{\varrho,\chi}(x,D_{x})-[f'(a_{2})a_{1}]_{\varrho,\chi}(x,D_{x})-\beta_{\varrho,\chi}(x,D_{x})$ belongs to $\mathcal{R}(Q;\mathrm{End}\,\mathcal{E})$\,.\\
Because $s\in [-3/2,-0[$\,, we know $f(a_{2})\in S^{2s}_{\Psi}(Q;\mathrm{End}\,\mathcal{E})$ with  $2s\geq -3$ and $f'(a_{2})\in S^{2(s-1)}(Q;\mathrm{End}\,\mathcal{E})$\,, $f'(a_{2})a_{1}\in S^{2s-1}_{\Psi}(Q;\mathrm{End}\,\mathcal{E})$\,,  while $\beta_{\varrho,\chi}\in S^{-4}_{\Psi}(Q;\mathrm{End}\,\mathcal{E})$\,, $-4\leq -3-1\leq  2s-1$\,.\\
Now for a general $s<0$\,, simply write $\langle t\rangle^{s}=\langle t\rangle^{s_{1}n_{1}}$ with $s_{1}\in [-3/2,0[$ and $n_{1}\in \mathbb{N}$\,. The composition of pseudo-differential operators says that the principal symbol of $\langle A\rangle^{s}=\langle A\rangle^{s_{1}}\circ\ldots\circ \langle A\rangle^{s_{1}}$ is $\langle a_{2}\rangle^{s}$\,. Any power $\langle t\rangle^{s}$\,, $s\in \mathbb{R}$\,, can be written $\langle t\rangle^{2N}\langle t\rangle^{s'}$ with $s'<0$ and $N\in \mathbb{N}$\,. With $\langle t\rangle^{2N}=(1+t^{2})^{N}\in \mathbb{R}[t]$\,, $\langle A\rangle^{s}$ is a pseudo-differential operator with principal symbol $\langle a_{2}\rangle^{s}$ for any $s\in \mathbb{R}$\,. Finally a general function $f\in S(\langle t\rangle^{s};\frac{dt^{2}}{\langle t\rangle^{2}})$ is written $\langle t\rangle^{s+3/2}f_{s}(t)$ with $f_{s}\in S(\langle t\rangle^{-3/2},\frac{dt^{2}}{\langle t\rangle^{2}})$\,.\\
For the $\mathrm{End}\,\mathcal{E}$ version it suffices to notice that all the explicit computations above, are done essentially with scalar symbols\,.
\end{proof}

\noindent\textbf{Acknowledgements:} The two first authors benefited during this work from the support of the French ANR-Project QUAMPROCS "Quantitative Analysis of Metastable Processes".
Francis White is grateful to have received funding from the European Union's Horizon 2020 research and innovation programme under the Marie Sklodowska-Curie grant agreement No 101034255.\euflag

\end{document}